\documentclass[11pt]{amsart}
\usepackage{amssymb}
\usepackage{amsmath}
\usepackage{mathabx}
\usepackage[latin1]{inputenc}
\usepackage{enumitem}
\usepackage{mathtools}
\usepackage{color}
\usepackage{hyperref}

\usepackage{graphicx}
\usepackage[export]{adjustbox}
\usepackage{pstricks}
\usepackage{pst-plot}
\usepackage{pst-pdf}

\usepackage{environ}

\NewEnviron{comm}{\ifcsname comment\endcsname \noindent  {\color{blue}\BODY} \else\fi}

\renewcommand{\hat}{\widehat}
\newcommand{\ti}{\widetilde}
\renewcommand{\bar}{\overline}

\def\Z{{\mathbb Z}}\def\T{{\mathbb T}}\def\R{{\mathbb R}}\def\C{{\mathbb C}}
\def\cG{\mathcal{G}}
\def\cT{\mathcal{T}}

\def\z{\zeta}\def\e{\eta}
\def\g{\gamma}
\def\s{\sigma}\def\k{\kappa}
\def\be{\beta}\def\d{\delta}
\def\be{\beta}
\def\th{\theta}

\def\L{\Lambda}\def\D{\Delta}

\def\cU{\mathcal{U}}

\def\diag{\mathrm{diag}}
\def\dist{\mathrm{dist}}

\def\b#1{\lbrace#1\rbrace}
\def\a#1{\left|#1\right|}

\def\l#1{\langle #1\rangle}

\def\<{\langle}
\def\>{\rangle}

\theoremstyle{plain}

\newtheorem*{notation}{Notation}
\newtheorem{question}{Question}

\newtheorem{Main}{Theorem}

\newtheorem{Mainprime}{Theorem}

\def\iff{\Longleftrightarrow}

\def\mn{\medskip\noindent}

\def\cC{\mathcal C}

\newcommand{\ph}{\varphi}

\def\a{\alpha}
\def\w{\omega}
\def\d{\delta}
\def\ti{\tilde}
\def\e{\varepsilon}
\def\pa{\partial}
\def\s{\sigma}
\def\th{\theta}

\let\newpf\proof \let\proof\relax

\def\cL{\mathcal{L}}
\def\cD{\mathcal{D}}

\newcommand{\ba}{\overline{A}}

\newcommand{\cF}{\mathcal{F}}

\newcommand{\cO}{\mathcal{O}}

\newcommand{\cR}{\mathcal{R}}

\def\be{\begin{equation}}
\def\ee{\end{equation}}

\def\ba{{\begin{align}}}
\def\ea{{\end{align}}}

\def\bm{\begin{pmatrix}}
\def\em{\end{pmatrix}}

\def\a{{\alpha}}

\def\g{{\gamma}}

\def\fO{{\frak O}}

\def\fd{{\frak d}}

\def\bD{\mathbb{D}}

\def\0{{\mathbf 0}}

\newtheorem{theo}{Theorem}[section]

\newtheorem*{thmA}{Theorem}%[section]
\newtheorem{lemma}[theo]{Lemma}
\newtheorem{prop}[theo]{Proposition}
\newtheorem{cor}[theo]{Corollary}

\theoremstyle{remark}
\newtheorem{rem}{Remark}[section]

\def \bn {\hfill \\ \smallskip\noindent}

\theoremstyle{definition}
\newtheorem{defin}{Definition}[section]

\newtheorem{assumption}{Assumption}[section]
\newenvironment{proof}{ \noindent{\it Proof.}\quad}{\ \hfill $\Box$\vskip .2cm}

\def\lra{\longrightarrow}

\def\iff{\Longleftrightarrow}

\def\ssm{\smallsetminus}

\renewcommand{\setminus}{\ssm}

\newcommand{\N}{{\mathbb N}}
\newcommand{\Q}{{\mathbb Q}}

\newcommand{\bA}{{\mathbb A}}

\newcommand{\bH}{{\mathbb H}}

\def\B0{{\bold{0}}}
\def\b{\beta}
\def\a{\alpha}
\def\w{\omega}
\def\l{\lambda}
\def\be{\begin{equation}}
\def\ee{\end{equation}}
\def\cA{\mathcal{A}}
\def\cA{\mathcal{A}}
\def\cB{\mathcal{B}}
\def\cC{\mathcal{C}}
\def\cE{\mathcal{E}}
\def\cF{\mathcal{F}}
\def\cK{\mathcal{K}}
\def\cP{\mathcal{P}}

\def\cV{\mathcal{V}}

\def\1{{\bf 1}}

\catcode`\@=12

\def\Empty{}
\newcommand\oplabel[1]{
  \def\OpArg{#1} \ifx \OpArg\Empty {} \else
  	\label{#1}
  \fi}
		
\renewcommand{\ti}{\widetilde}
\renewcommand{\check}{\widecheck}

\def\malpha{\mathring{\alpha}}
\def\mbeta{\mathring{\beta}}

\def\cM{\mathcal{M}}
\def\cW{\mathcal{W}}

\def\toitself{\righttoleftarrow}

\begin{document}

\title[Rotation domains and Herman rings for Hénon maps]{Existence of Exotic rotation domains and Herman rings for quadratic Hénon maps}
\author{Raphaël Krikorian }
%\author{Rapha\"el Krikorian}

\address{
CMLS Ecole Polytechnique} 
\email{raphael.krikorian@polytechnique.edu}

%\date{June 1st, 2025}

%\date{\today}

\thanks{This work was supported by  the project ANR KEN : ANR-22-CE40-0016.}

\maketitle
\begin{abstract}  A quadratic Hénon map is an automorphism of $\C^2$ of the form $h:(x,y)\mapsto (\l^{1/2} (x^2+c)-\l y,x)$. It has a constant Jacobian equal to $\l$ and has two fixed points. If $\lambda$ is on the unit circle (one says $h$ is conservative) these fixed points can be both elliptic or both hyperbolic. In the elliptic case, under an additional Diophantine condition, a simple application of Siegel Theorem shows that $h$ admits quasi-periodic orbits with two frequencies in the neighborhood of its fixed points. Surprisingly, in some hyperbolic cases, Shigehiro Ushiki observed numerically what seems to be quasi-periodic orbits belonging to some ``Exotic rotation domains'' though no Siegel disk is associated to the fixed points.  The aim of this paper is to explain and prove the existence of these ``Exotic rotation domains''. Our method  also applies  to the dissipative case ($|\l|<1$) and allows to prove the existence of attracting Herman rings. The theoretical framework we develop permits to produce numerically these Herman rings that  were never observed before.

\end{abstract}

\tableofcontents

\section{Hénon maps, Exotic Rotation Domains and Herman Rings}
\subsection{Hénon maps}
The Hénon map
$$h^{\textrm{Hénon}}_{\b,c}:\C^2\ni (x,y)\mapsto (e^{i\pi \b}(x^2+c)-e^{2\pi i \beta} y,x)\in \C^2,\qquad \b,c\in\C$$
is a polynomial automorphism of $\C^2$, the inverse of which is also polynomial, with constant Jacobian equal to $b:=e^{2\pi i \b}:$
$$\forall\ (x,y)\in\C^2, \ \det D h^{\textrm{Hénon}}_{\b,c}(x,y)=b=e^{2\pi i \b}.$$
Equivalently, if $dx\wedge dy$ is the canonical symplectic form on $\C^2$ one has
$$(h^{\textrm{Hénon}}_{\b,c})^*(dx\wedge dy)=e^{2\pi i \b}(dx\wedge dy);$$
in other words, $h^{\textrm{Hénon}}_{\b,c}$ is {\it conformal symplectic}.
In particular, if $e^{2\pi i \b}=1$, the map $h^{\textrm{Hénon}}_{\b,c}$ is symplectic.

We shall say that $h^{\textrm{Hénon}}_{\b,c}$ is 
\begin{itemize}
\item {\it Conservative} when $|b|=1$ or equivalently when $\b\in\R$. 
\item {\it Dissipative} otherwise. In this case we shall assume $|b|<1$ or equivalently $\Im\b>0$.
\end{itemize}
When $|e^{2\pi i\b}|=1$ and $c\in\R$ the diffeomorphism $h^{\textrm{Hénon}}_{\b,c}$ is {\it reversible}: if $\sigma^{\textrm{Hénon}}$ is the anti-holomorphic  involution 
\be \sigma^{\textrm{Hénon}}:\C^2\ni (x,y)\mapsto (\bar y, \bar x)\in \C^2,\qquad \sigma^{\textrm{Hénon}}\circ \sigma^{\textrm{Hénon}}=id,\label{involushiki}\ee
the inverse of $h^{\textrm{Hénon}}_{\b,c}$ satisfies
$$(h^{\textrm{Hénon}}_{\b,c})^{-1}=\sigma^{\textrm{Hénon}}\circ (h^{\textrm{Hénon}}_{\b,c})\circ \sigma^{\textrm{Hénon}}.$$

The map $h^{\textrm{Hénon}}_{\b,c}$ has exactly   two fixed points (possibly equal) $(t_{+},t_{+})$ and $(t_{-},t_{-})$ where $t_{\pm}$ are the  roots of the quadratic equation
\be t^2-2t\cos(\pi\b)+c=0.\label{eq:1.1}\ee
The multipliers of $h^{\textrm{Hénon}}_{\b,c}$ at  these fixed points, i.e. the eigenvalues of $Dh^{\textrm{Hénon}}_{\b,c}(t_{\pm},t_{\pm})$, are the roots of
$$\l^2-2t_{\pm} e^{i\pi \b}\l+e^{2\pi i \b}=0.$$
We now choose $t$ one of the two values $t_{\pm}$ (for example $t=t_{+}$) and denote by $\l_{1},\l_{2}$ the eigenvalues of $Dh^{\textrm{Hénon}}_{\b,c}(t,t)$. They satisfy
$$\l_{1}+\l_{2}=2t e^{i\pi\b},\qquad \l_{1}\l_{2}=e^{2\pi i \b}$$
and we shall write them under the form
$$\l_{1}=e^{2\pi i (-\a+\b/2)},\qquad \l_{2}=e^{2\pi i (\a+\b/2)},\qquad \a\in\C.$$
Note that 
$$t=\cos(2\pi\a)$$
so
\be c=-(\cos(2\pi\a))^2+2\cos(2\pi\a)\cos(\pi\b).\label{eq:1.2}\ee

\subsection{Dynamics of Hénon maps}
Hénon maps were introduced in \cite{Hen} by the astronomer and mathematician  Michel Hénon as a discrete 2D simplified model for the Lorenz ODE system\footnote{Which was introduced by the meteorologist Edward N. Lorenz as a finite dimensional model to represent   forced dissipative hydrodynamic flows.  } (\cite{Lo}). Since then they play a central role in dynamics.
\subsubsection{Real Hénon maps} The parameters $b=e^{2\pi i \b}$ and $c$ are then real numbers and when $b\in (0,1)$ the Hénon map is {\it dissipative}.  In the regime $0<b\ll1$ it can be seen as a 2-dimensional version of the 1D  {\it logistic} map  $x\mapsto \l x(1-x)$. Quadratic like mappings of the interval can display {\it chaotic} behavior and indeed, M. Lyubich proved (cf. \cite{Ly}) that such mappings are almost always (w.r.t. to the parameter) either regular (they have an attracting cycle) or stochastic (they have an absolutely continuous invariant measure). We refer to  \cite{Ly} for further references on this topic;  let us just mention that  Jakobson (\cite{Ja}) proved  the existence of a positive measure set of parameters close to $\l=4$ for which the logistic map is stochastic.

In the 2-dimensional case,  Hénon observed numerically  in \cite{Hen} that the Hénon maps  with  some dissipation ($|b|=0.3$)   should have a {\it strange} attractor i.e. a {\it non-uniformly hyperbolic} invariant set (whence the name ``chaotic'' strange attractor). This was proved mathematically  by Benedicks and Carleson  in \cite{BC} (see also \cite{BY} for a different approach).

\subsubsection{Complex Hénon maps} In this case one allows $\b$ and $c$ to take any complex values and the phase space is $\C^2$. Hénon maps are then natural {\it invertible} generalization of 1D complex quadratic (more generally polynomial) maps\footnote{When $b=e^{2\pi i \b}\ne 0$ is set to $b=0$, the dynamics on the $x$ coordinate is that of a quadratic polynomial map.}. In the 1D  quadratic case (or for polynomial maps of degree more than 1), the dynamics is (by definition) regular on the Fatou set and  chaotic on its complement, the Julia set, which is the closure of the set of repelling periodic points. Components of  Fatou sets  are classified: they are eventually periodic (this is D. Sullivan's non wandering  theorem \cite{Su}) and  they are  pre-images of attracting regions of contracting or parabolic periodic points, or pre-images   of periodic {\it Siegel disks}.  By Siegel linearization theorem, any {\it Diophantine} elliptic fixed point $\zeta$, i.e. any fixed point with multiplier $e^{2\pi i \a}$, $\a \in\R$ at $\zeta$ satisfying an arithmetic condition 
$$\limsup_{\substack{k\to\infty\\ k\in \N^*}}\frac{-\ln\min_{l\in\Z}|k\a-l|}{\ln k}<\infty$$
is contained in a Siegel disk\footnote{In fact, the existence of Siegel disk is true under the weaker {\it Brjuno condition} (see \cite{Br}) and, in the case of quadratic maps, equivalent to this condition, see \cite{Y-ast}. },   i.e. a (maximal) nonempty bounded  open simply connected  set on which the dynamics is conjugated to $z\mapsto e^{2\pi i \a}z$, $\a\in \R\setminus\Q$.  

Note that in 1D all Fatou components $\Omega$, unless $\Omega$ is the basin of attraction of a parabolic point,   are {\it recurrent} in the sense that there is a point in $\Omega$ whose limit set  contains a point in $\Omega$.

\medskip In higher dimension the picture is less satisfactory (in particular there may exist wandering components, see   \cite{ABDPR} and also  \cite{BeBi} in the Hénon case) though many fundamental results  have been obtained these last 25 years. Let us mention that after the work of \cite{BS2}, \cite{FoSi}, \cite{Ue} recurrent Fatou components are
classified as attracting basins or basins of rotation
attractors, or rotation domains, with the pending question whether 
Herman rings can appear as attractors. The non-recurrent case was considered in \cite{LyPe} for {\it moderately dissipative} Hénon maps i.e. maps for which the Jacobian $b$ satisfies $|b|<1/4$ (for Hénon maps of degree $d$ the bound is $<1/d^2$) and like in the 1D case,  if $\Omega$  is an invariant non-recurrent Fatou component with bounded forward orbits,
all the orbits in $\Omega$ converge to a parabolic point lying in $\pa\Omega$  with multiplier 1.

\subsection{Rotation domains}
\subsubsection{Definition}
If $h:\C^2\to\C^2$ is a holomorphic map,   the {\it forward Fatou set} $F^+$ of $h$  is by definition  the largest open subset of $\C^2$  such that the forward  iterates of $h$  form a normal family. If $h$ is invertible with inverse $h^{-1}:\C^2\to \C^2$, we define the {\it backward Fatou set} $F^-$ of $h$ as the forward Fatou set of $h^{-1}$.

The {\it boundedness domain} $K^+$ of $h$ and its {\it escape locus} $U^+$ are by definition
$$
\begin{aligned}&K^+=\{(z,w)\in\C^2,\ \{h^n(z,w)\}_{n\in\N}\ \textrm{is\ bounded}\},\\
& U^+=\C^2\setminus K^+.
\end{aligned}
$$
If $h$ is invertible, the sets $K^-$ and $U^-$ are defined similarly with $h$ replaced by $h^{-1}$ and we then set
$$K=K^+\cap K^-.$$
By a theorem of \cite{FM}, if $h$ is a  {\it conservative} Hénon map (or a composition of such maps), one has the equalities
$$\textrm{int}(K^+)=\textrm{int}(K^-)=\textrm{int}(K)$$ 
and the corresponding set is bounded. Also, if $\Omega$ is a connected component of $h$, there exists some $n\in\N^*$ such that
$$f^n(\Omega)=\Omega.$$
As a consequence
$$F^\pm= U^\pm \cup \textrm{int}(K ).$$
A {\it Fatou component} is a connected component of $F^+$. Note that $U^+$ is a (unbounded) Fatou component which is the basin of attraction of a point at infinity.

\begin{defin}[Rotation domain]A rotation domain of a conservative Hénon map is by definition a bounded Fatou component.
\end{defin}
The justification of this terminology is the following. Let $\Omega$ be a bounded Fatou component such that $h(\Omega)=\Omega$  and define $\mathcal{G}$ as the set of all  possible limits $h^{n_{j}}:\Omega\to \bar \Omega$. The set $\mathcal{G}$ has a natural  structure of Abelian group; it is also compact for the compact-open topology and, as such, is a Lie group (it has no small subgroups). The connected component of the identity $\mathcal{G}_{0}$ of $\mathcal{G}$ is thus isomorphic to a torus $(\T^d,+)$. This provides $\Omega$ with a torus action, whence the name ``rotation domain''.

\subsubsection{Classification}\label{sec:classification} By a theorem of \cite{BS2} the torus group $\mathcal{G}_{0}$ can be either $\T$ or $\T^2$. We say accordingly that the rank of $\Omega$ is 1 or 2. 

\begin{enumerate}
\item\label{i1}
If $\Omega$ has rank 1, then for any $(z,w)\in\Omega$, its   orbit  $\cO(z,w):=\cG\cdot (z,w)$ under the group $\cG$ is either a disk or an annulus and the restriction of $h$ to $\cO(z,w)$ is conjugated to $\zeta\mapsto e^{2\pi i a}\zeta$ where $a$ is an irrational (real) number independent of $(z,w)$ (the {\it rotation number} of $\Omega$) (cf. \cite{B}). 

\item\label{i2}
If $\Omega$ has rank 2, then by a result of \cite{BBD}, there exists a (polynomially convex) Reinhardt domain\footnote{This is a domain $D\subset\C^2$ which is invariant by the following  action of $\R^2$: $\R^2\times \C^2\ni ((\th,\phi),(\zeta,\xi))\mapsto (\th,\phi)\cdot(\zeta,\xi):=(e^{i\th}\zeta,e^{i\psi}\xi) \in \C^2$.} $D\subset\C^2$ and a biholomorphism $\psi:\Omega\to D$ such that $\psi\circ h\circ \psi^{-1}:D\to D$ is a linear action  $L:(\zeta,\xi)\mapsto (e^{2\pi i a_{1}}\zeta,e^{2\pi i a_{2}}\xi)$, with $(a_{1},a_{2})\in\R^2$ rationally independent on $\Z$. The (polynomially convex) Reinhardt domain $D$  is topologically isomorphic to 
\begin{enumerate}
\item\label{i2.1} Either a ball; in this case, the restriction of $h$ to $\Omega$ has a unique fixed point. 
\item\label{i2.2} Or the product of a disk by an annulus (i.e. a complex cylinder).  In this case the restriction of $h$ to $\Omega$ has no fixed point. 
\end{enumerate}
\end{enumerate}

Note that Case \ref{i2.1} does occur when  the multipliers $(\l_{1},\l_{2})$ of the Hénon map $h=h^{\textrm{Hénon}}_{\b,c}$ at one of its fixed point $(t,t)$  satisfy a Diophantine condition: this is a consequence of 
Siegel Theorem (\cite{Si}). See \cite{Ri}, \cite{Po1}, \cite{Po2}, \cite{Po3}, for more general versions of Siegel theorem.
\begin{thmA}[Siegel] A holomorphic germ $f:(\C^2,(0,0))\toitself$ of the form $f(z,w)=(e^{2\pi i \a_{1}}z,e^{2\pi i \a_{2}}w)+O^2(z,w)$ where $(\a_{1},\a_{2})\in \R^2$ satisfies a Diophantine condition  $$\forall\ (k_{1},k_{2})\in\Z^2\setminus(0,0),\quad \inf_{l\in\Z}|k_{1}\a_{1}+k_{2}\a_{2}-l|\geq\frac{C}{(|k_{1}|+|k_{2}|)^\tau}$$
($C>0, \tau>0$) is linearizable in a neighborhood of $(0,0)$: there exists $g:(\C^2,(0,0))\toitself$ such that 
$$g\circ f\circ g^{-1}:(z,w)\mapsto (e^{2\pi i \a_{1}}z,e^{2\pi i \a_{2}}w).$$
\end{thmA}
This leads to  the following question formulated by Eric Bedford (cf. \cite{B}):
\begin{question}Can Case \ref{i2.2} occur? In other words, does a rotation domain necessarily contain a fixed point?
\end{question}

\begin{defin}An exotic rotation domain is a rank 2 rotation domain without fixed point.
\end{defin}

\subsection{Shigehiro Ushiki's numerical experiments}\label{sec:ushiki'sexample}
Shigehiro Ushiki discovered numerically 
such exotic rotation domains. See the beautiful pictures on S. Ushiki's web page \cite{Ush}. For example (the values are taken from E. Bedford paper \cite{B}), with
$$\pi\b=1.02773,\qquad c=0.269423$$
$$(x_{0},y_{0})=(\bar \zeta,\zeta),\qquad \zeta=0.36 + 0.298i$$
one observes that the closure of the  orbit $(h^{\textrm{Hénon}}_{\b,c})^{\circ n}(x_{0},y_{0})$ is what seems to be a two-torus; see Figure 4 of \cite{B}. This quasi-periodic motion {\it cannot} be associated to the existence of some Siegel disk because, for these values of $\b$ and $c$, the fixed points of $h^{\textrm{Hénon}}_{\b,c}$ are {\it hyperbolic} (i.e. the eigenvalues of $D^{\textrm{Hénon}}_{\b,c}$ at the fixed points do not lie on the unit circle). Indeed, solving
$$(\cos(2\pi\a))^2-2\cos(2\pi\a)\cos(\pi\b)+c=0$$
gives 
\be \cos(2\pi\a)=\cos(\pi\b)\pm\sqrt{\cos(\pi\b)^2-c}.\label{eq:calphabeta}\ee
Since $\cos(\pi\b)\approx 0.5167\pm 10^{-4} $ we find $\cos(2\pi\a)=0.5167\pm 0.0487 i\pm 10^{-4}$.
We thus have
$$\b\approx (1/3)-6.1\times 10^{-3},\qquad \a \approx (\b/2)\pm 9\cdot 10^{-3}i$$
\begin{rem}
Note that 
$$\tau:=\frac{1}{2}+\frac{(-3.05+9i)\cdot 10^{-3}}{-6.1\cdot 10^{-3}}\approx1+1.47 i$$
\end{rem}
\begin{rem}
When
$$\b=(1/3)+\d\mbeta,\qquad \a= (1/6)+\d \malpha,\qquad \malpha=(\tau-1/2)\mbeta$$
or equivalently
$$\a=\frac{1-\tau}{3}+(\tau-(1/2))\b$$
one finds 
$$c=\frac{1}{4}-\frac{\sqrt{3}}{2}\pi \mbeta \d+((1/4)+3(\tau-1/2)(\tau-3/2))\pi^2\mbeta^2\d^2+O(\d^3).$$
If 
$\tau=1+t$, $t\in\C$
\be c= \frac{1}{4}-\frac{\sqrt{3}}{2}\pi \mbeta \d+3t^2\pi^2\mbeta^2\d^2+O(\d^3).\label{eq:cfuncmtau}\ee
\end{rem}
\begin{question}
Prove mathematically that there are Hénon maps with exotic rotation domains.
\end{question}

\begin{figure}[h]
\vskip -1.5cm
\hskip -2cm
\includegraphics[scale=0.5]{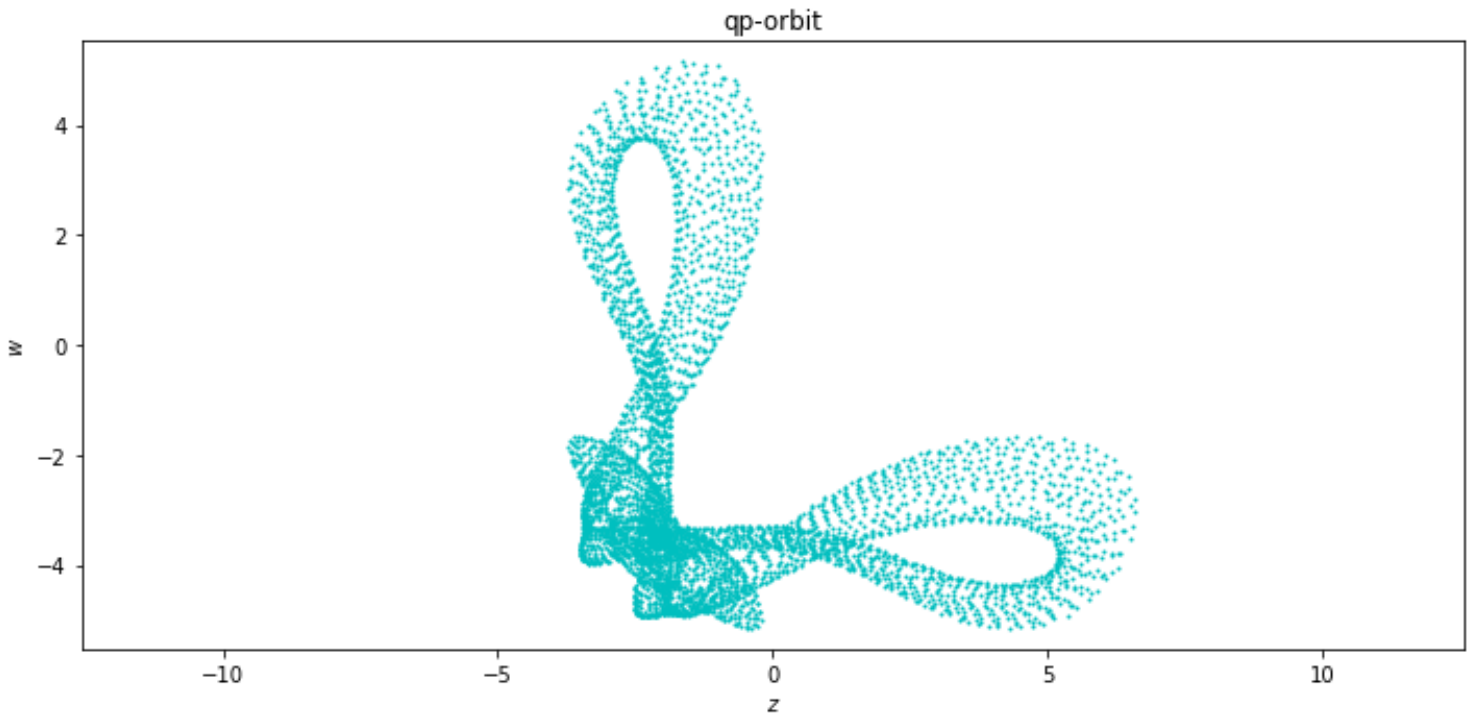}
\caption{S. Ushiki's example. Iteration of the map $h_{\b,c}$ with  
$\beta=0.327136$, $c=0.269343$. The curve represents (after  the scaling $(z,w)\mapsto (20\times (z-0.5),20\times (w-0.58))$)   $(\Re(z),\Re(w))$ after  5000 iterations.
The initial condition is $(z_{*},w_{*})$ avec 
$z_*= 0.3512857-0.352772 \sqrt{-1}$,
$w_*= 0.3856867+0.353207\sqrt{-1}$.
}
\label{fig:0}
\end{figure}

\subsection{Herman rings}

A  Herman ring is an invariant \textit{attracting} annulus. If $\cA$ is this annulus, thus biholomorphic to some 
$$\bA(e^{-s},e^s)\simeq\T_{s}:=(\R+i(-s,s))/\Z\quad (s>0)$$ there exists a open neighborhood $U$ of $\cA$ in $\C^2$ such that for any $(z,w)\in U$ one has
$$\lim_{n\to\infty}\dist((h^{\textrm{Hénon}}_{\b,c})^{\circ n}(z,w),\cA)=0$$
(here $\dist$ is the  distance to a set).

In 1D complex dynamics, these attracting rings cannot exist when the dynamics is a {\it  polynomial} map. Nevertheless, Herman proved  their  existence for some  {\it rational} functions on $\mathbb{P}^1(\C)$, \cite{He}.

\begin{question}Does there exist a {\it dissipative} Hénon map with a Herman ring? 
\end{question}
Until recently\footnote{January 2024.} {\it no numerical experiment} showed evidence for their existence\footnote{See the end of Section \ref{sec:sketch} for a possible  explanation of this fact.} in the case of  Hénon maps\footnote{Let us mention that S. Ushiki found numerically Herman rings for some automorphisms of complex surfaces.  }.  One of the main purpose of this paper is to prove {\it mathematically} their existence \footnote{Let's mention that for strongly dissipative perturbations of 1-dimensional rational maps  the existence of Herman rings is proved in \cite{Yam}.}; as a by-product we can design a {\it systematic} procedure to find them\footnote{Later on, X. Buff, S. Ushiki and H. Inou  also observed numerically Herman rings in the dissipative Hénon case.}.

To conclude this section, let us mention in the {\it moderately dissipative} case  the following result (see \cite{LyPe}).
 Let $\Omega$ be an invariant Fatou component with bounded forward orbits of
a moderately dissipative Hénon mapping $h:\C^2\to \C^2$ of degree $d\geq 2$. Then one of the
following three cases is satisfied:
\begin{enumerate}
\item All orbits in $\Omega$ converge to an attracting fixed point $p\in\Omega$. The component $\Omega$ is
biholomorphically equivalent to $\C^2$.
\item All orbits in $\Omega$ converge to a properly embedded submanifold $\Sigma \subset \Omega$, and $\Sigma$  is
biholomorphically equivalent to either the unit disk or an annulus. The manifold
$\Sigma$  is invariant under $h$ and $h$ acts on $\Sigma$ as an irrational rotation.
\item All orbits in $\Omega$  converge to a fixed point $p\in \pa\Omega$. The eigenvalues  $\l_{1}$ and $\l_{2}$ of
$Dh (p)$ satisfy $|\l_{1}| < 1$ and $ |\l_{2}| = 1$, and $\Omega$ is biholomorphically equivalent to $\C^2$.
\end{enumerate}

In our examples, the dissipation is quite small ($\Im \b$ is positive but small). 
It would be interesting to investigate whether one can produce examples of Herman rings in the moderately dissipative case. 

\begin{question}Can Herman rings exist in the moderately dissipative case?
\end{question}

Before concluding this section let us mention that it would be interesting to study the existence of Exotic rotation domains or Herman rings for surface automorphisms. Numerical simulations by S. Ushiki suggest  they may exist. The existence of Siegel domains is already proved in many interesting situations (K3 surfaces\footnote{For more informations on dynamics of automorphisms of these surfaces see \cite{Ca}.}), see for example \cite{McM}.  

\begin{figure}[h]
\includegraphics[scale=0.5]{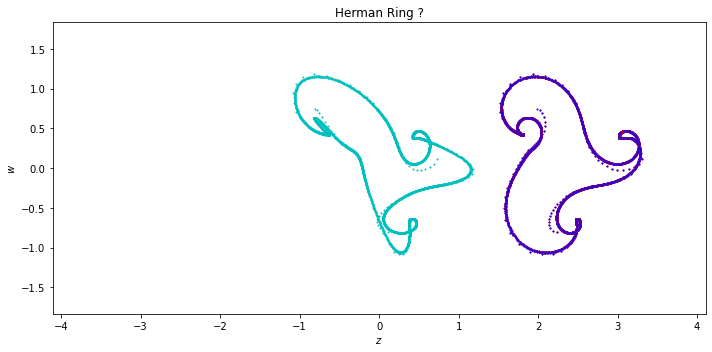}
\caption{A Herman ring for the Hénon map $h:(x,y)\mapsto (e^{i\pi \b}(x^2+c)-e^{2\pi i \beta} y,x)$,  $\beta= 0.3289999+0.0043333\sqrt{-1}$, $c=0.2619897-0.0088858\sqrt{-1}$. Initial condition $(z_{*},w_{*})$,  $z_{*}=0.44672099-0.16062292\sqrt{-1}$,
$w_{*}= 0.3961953+0.149208\sqrt{-1}$. $N=5000$ iterations. The cyan curve is the projection $(\Im z,\Im w)$ and the red and blue curves (that coincide) the projections $(\Re z,\Im z)$, $(\Re w, \Im w)$.\label{fig:7}}
\end{figure}

\begin{figure}[h]
\includegraphics[scale=0.6]{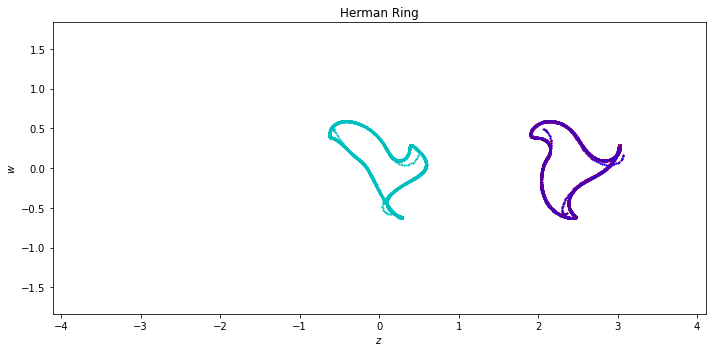}
\caption{A Herman ring for the Hénon map $h:(x,y)\mapsto (e^{i\pi \b}(x^2+c)-e^{2\pi i \beta} y,x)$,
$\beta= 0.33121126+0.00218737 \sqrt{-1}$
$c= 0.2557783-0.00497994 \sqrt{-1}$
 Initial condition $(z_{*},w_{*})$,
$z_{*}= 0.471458035-0.113447719\sqrt{-1}$
$w_{*}= 0.41305318+0.0975217\sqrt{-1}$
Number of iteration $N= 7000$.
The cyan curve is the projection $(\Im z,\Im w)$ and the red and blue curves (that coincide and give the violet curve) the projections $(\Re z,\Im z)$, $(\Re w, \Im w)$. The picture is scaled by a factor 5\label{fig:7}}
The rotation number  on the curve should be $0.0016946$.
\end{figure}

\section{Results}
Let $(t,t)$ be one of the two fixed points of the Hénon map
\be h^{\textrm{Hénon}}_{\b,c}:\C^2\ni (x,y)\mapsto (e^{i\pi \b}(x^2+c)-e^{2\pi i \beta} y,x)\in \C^2,\qquad \b,c\in\C\label{hhenon}\ee
and  let 
\be e^{2\pi i(\a+\b/2)},e^{2\pi i(-\a+\b/2)}\quad \textrm{be\   the\  eigenvalues\  of}\  Dh^{\textrm{Hénon}}_{\b,c}(t,t).\label{eigenvalues}\ee

Conversly, given $\mbeta, \tau\in \C$, $\d\in\R$,  we can define
\be
\left\{
\begin{aligned}
&\b=\frac{1}{3}+\d \mbeta\\
&\a=\frac{1}{6}+\d\times (\tau-1/2)\mbeta\\
&c=-(\cos(2\pi\a))^2+2\cos(2\pi\a)\cos(\pi\b)
\end{aligned}
\right.
\label{alpha,betaMain}
\ee
and consider the Hénon map with parameters $\b,c$ which has eigenvalues (\ref{eigenvalues}).

We shall concentrate on the regime  where  $\d$ is  small. 

\subsection{Existence of Exotic rotation domains}
\subsubsection{On reversibility}
As we mentioned before, when $\b$ and $c$ are real the map $h^{\textrm{Hénon}}_{\b,c}$ is reversible and Ushiki proved conversely that $h^{\textrm{Hénon}}_{\b,c}$ is reversible  with respect to the involution (\ref{involushiki}) if and only if $\b$ and $c$ are real.
In particular, if we denote ${\rm Rev}_{\d}$
the set of $(\tau,\mbeta)\in \C^2$ for which  $h^{\textrm{Hénon}}_{\b,c}$  is reversible w.r.t. (\ref{involushiki}) we have
\be \R^2\subset {\rm Rev}_{\d}.\label{eq:Revreal}\ee
For $\d>0,\mbeta\in\R$ and $\tau\in\C$ let 
\be c_{\d}(\tau,\mbeta)\label{defctaumbeta}
\ee
be the value of $c$ given by (\ref{alpha,betaMain}).
One thus has
$${\rm Rev}_{\d}=\{ (\tau,\mbeta)\in \C\times \R\mid c_{\d}(\tau,\mbeta)\in\R\}.$$ 
One can prove
\begin{lemma}\label{lemma:defintaudelta}For each $\d$ small enough  the following holds. There exists  a $C^2$ function $ (-1,1)\times (-1,1)\ni (t,\mbeta)\mapsto \tau_{\d}(t,\mbeta)\in\C$ such that for any $(t,\mbeta)\in (-1,1)\times (-1,1)$
$$c_{\d}(\tau_{\d}(t,\mbeta),\mbeta)\in\R.$$
Furthermore, the $C^2$-norm of 
$$\tau_{\d}(t,\mbeta)-(1+it)$$
goes to zero as $\d$ goes to zero.
\end{lemma}
\begin{proof}
A computation shows that if $\tau=1+it/2$
($2(\tau-1/2)=1+it$)
\begin{comm}
\begin{align*}
\cos(2\pi\a)&=\frac{1}{2}\cos((1+it) \pi \d \mbeta)-\frac{\sqrt{3}}{2}\sin((1+it) \pi \d \mbeta)\\
&\frac{1}{2}(\cos(\pi \d \mbeta)\cosh (t\pi\d\mbeta)-i\sin(\pi \d \mbeta)\sinh(t\pi\d\mbeta))\\
& -\frac{\sqrt{3}}{2}(\sin(\pi \d \mbeta)\cosh(t\pi\d\mbeta)+i\cos(\pi \d \mbeta)\sinh(t\pi\d\mbeta))\\
&=(\frac{1}{2}(\cos(\pi \d \mbeta)-\frac{\sqrt{3}}{2}(\sin(\pi \d \mbeta))\cosh(t\pi\d\mbeta)\\
&-i(\frac{\sqrt{3}}{2}\cos(\pi \d \mbeta)+\frac{1}{2}\sin(\pi \d \mbeta))\sinh(t\pi\d\mbeta)\\
&=\cos(\pi\b)\cosh(t\pi\d\mbeta)-i\sin(\pi\b)\sinh(t\pi\d\mbeta)
\end{align*}
\end{comm}
$$\cos(2\pi\a)=\cos(\pi\b)\cosh(t\pi\d\mbeta)-i\sin(\pi\b)\sinh(t\pi\d\mbeta).$$
\begin{comm}
\begin{align*}
\Im(2\cos(2\pi\a)\cos(\pi\beta))=&-2(\frac{\sqrt{3}}{2}\cos(\pi \d\mbeta)+\frac{1}{2}\sin(\pi \d \mbeta))\sinh(t\pi\d\mbeta)\cos(\pi\beta)\\
&=-2\sin(\pi\b)\cos(\pi\b)\sinh(t\pi\d\mbeta)
\end{align*}
\begin{align*}
\Im((\cos(2\pi\a))^2)=&-(\frac{\sqrt{3}}{2}\cos(\pi \d \mbeta)+\frac{1}{2}\sin(\pi \d \mbeta))(\frac{1}{2}(\cos(\pi \d \mbeta)-\frac{\sqrt{3}}{2}(\sin(\pi \d \mbeta))\sinh(2t\pi\d\mbeta)\\
&=-(\frac{\sqrt{3}}{2}\cos(\pi \d \mbeta)+\frac{1}{2}\sin(\pi \d \mbeta))\cos(\pi\b)\sinh(2t\pi\d\mbeta)\\
&=-\sin(\pi\b)\cos(\pi\b)\sinh(2t\pi\d\mbeta)
\end{align*}
$$\Im c=-\sin(2\pi\b)(\sinh(t\pi\d\mbeta)-(1/2)\sinh(2t\pi\d\mbeta))$$
$$\sinh z-(1/2)\sinh(2z)= -z^3/2$$
$$2e^{x+iy}-2e^{-x-iy}=e^{2x+2iy}-e^{-2x-2iy}$$
\begin{multline*}\cos(2\pi \a)^2=(\cos(\pi\b)\cosh(t\pi\d\mbeta)-i\sin(\pi\b)\sinh(t\pi\d\mbeta))^2=\\
\cos^2(\pi\b)\cosh^2(t\pi\d\mbeta)-\sin^2(\pi\b)\sinh^2(t\pi\d\mbeta)-(i/2) \sin(2\pi\b)\sinh(2t\pi \d\mbeta)
\end{multline*}
\end{comm}
\begin{multline*}c=-\cos(2\pi \a)^2+2\cos(2\pi\a)\cos(\pi\b)=\\-\cos^2(\pi\b)\cosh^2(t\pi\d\mbeta)+\sin^2(\pi\b)\sinh^2(t\pi\d\mbeta)+(i/2) \sin(2\pi\b)\sinh(2t\pi \d\mbeta)\\+2\cos^2(\pi\b)\cosh(t\pi\d\mbeta)-i\sin(2\pi\b)\sinh(t\pi\d\mbeta)
\end{multline*}
\begin{multline*}c=-\cos(2\pi \a)^2+2\cos(2\pi\a)\cos(\pi\b)=\\ \biggl(-\cos^2(\pi\b)\cosh^2(t\pi\d\mbeta)+\sin^2(\pi\b)\sinh^2(t\pi\d\mbeta)+2\cos^2(\pi\b)\cosh(t\pi\d\mbeta)\biggr)+\\ i\biggl(( 1/2)\sin(2\pi\b)\sinh(2t\pi \d\mbeta)-\sin(2\pi\b)\sinh(t\pi\d\mbeta)\biggr)
\end{multline*}
This shows that $t\mapsto c$ is of the form
$$c(t)=\sum_{k=0}^\infty a_{2k}(\d\mbeta t)^{2k}+i\sum_{k=0}^\infty a_{2k+1}(\d\mbeta t)^{2k+1}$$
where the coefficients $a_{k}=a_{k,\d,\mbeta}$ are real. 
One computes
\begin{align*}
&a_{0}=\cos^2(\pi\b)\\
&a_{1}=0\\
&a_{2}=\sin^2(\pi\b)\\
&a_{3}=\sin(2\pi\d\mbeta)/2.
\end{align*}
In particular if $t=x+iy$ we find
$$\Im c(t)=(\pi\d\mbeta) x(2a_{2}\pi\d\mbeta y+a_{3}(\pi\d\mbeta)^2x^2-3a_{3}(\pi\d\mbeta)^2y^2+Q_{}(\pi\d\mbeta x,\pi\d\mbeta y))$$
where  $Q(x,y)=\sum_{(k,l)\in\N^2, k+l\geq 3}q_{k,l}x^ky^l$  is a convergent series with real coefficients.
So $\Im c=0$ if and only if $x=0$ or 
$$y=\frac{\pi\d\mbeta}{2a_{2}}\biggl(-a_{3}x^2+3a_{3}y^2-\sum_{ \substack{(k,l)\in\N^2, \\k+l\geq 3}}(\pi \d \mbeta)^{k+l-3}q_{k,l}x^ky^l\biggr).$$
The Contraction Mapping Theorem shows that if $\d$ is small enough there exists a $C^2$ (in fact real analytic) function $x\mapsto y_{\d}(x,\mbeta)=O(x^2)$  solution of this fixed point problem, henceforth of
$$c_{\d}(1+i(x+iy_{\d}(x,\mbeta))/2,\mbeta)\in\R.$$
Setting 
$\tau_{\d}(t,\mbeta)=1+it-y_{\d}(2t,\mbeta)$ we get the conclusion.

\end{proof}
\begin{figure}[h]
%\hspace{-4cm}
\includegraphics[scale=.25, left]{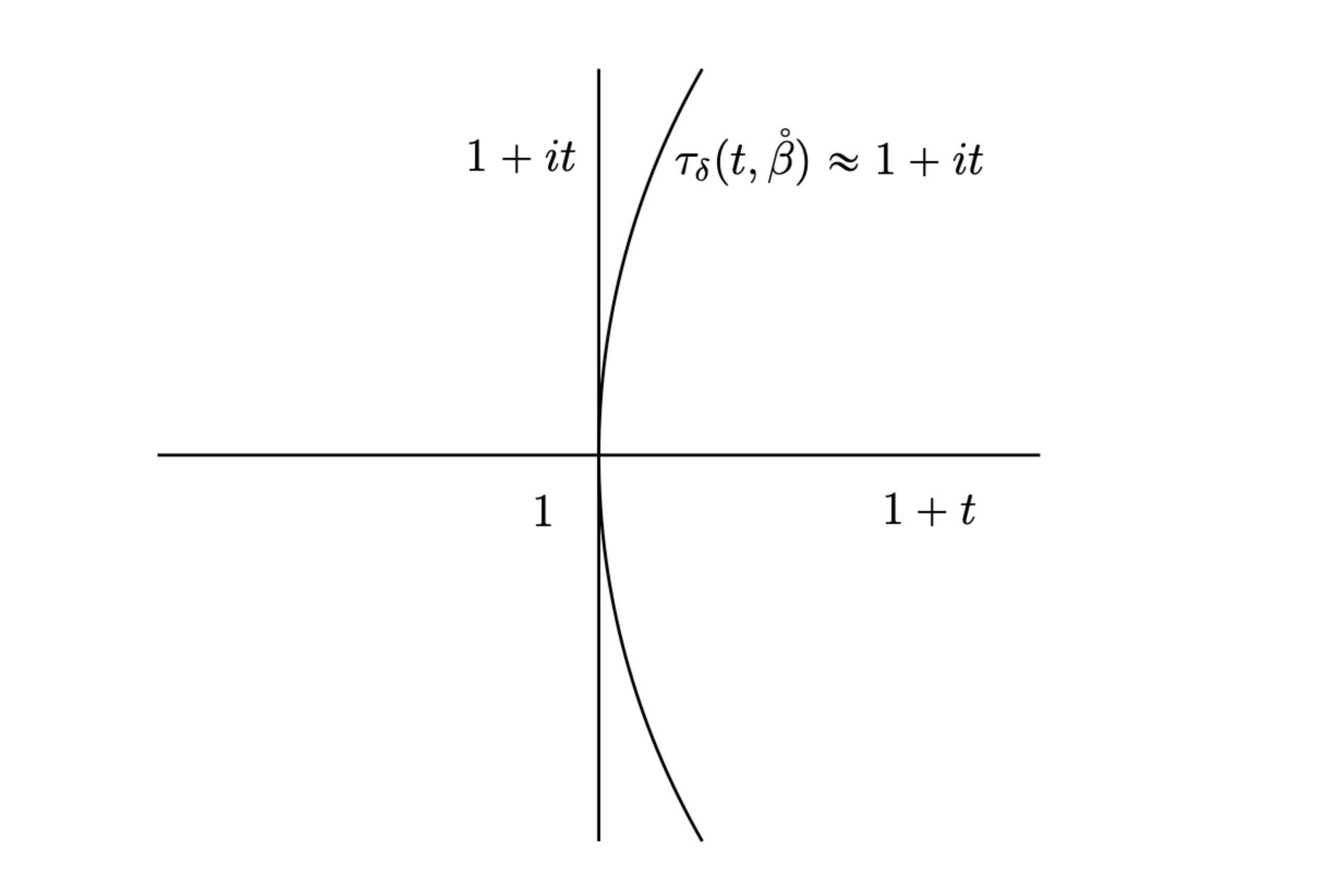}
\caption{Domain of reversibility. }\label{fig:cross-im}
\end{figure}

We thus have in addition to (\ref{eq:Revreal}) 
$$\forall (t,\mbeta)\in (-1,1)\times (-1,1),\quad (\tau_{\d}(t,\mbeta),\mbeta)\in {\rm Rev}_{\d}.$$

The proof of  Lemma \ref{lemma:defintaudelta} shows that  for $\d$ small enough and $\mbeta$ fixed (in $(1/10,9/10)$ for example) the values of $\tau$ for which $(\tau,\mbeta)\in {\rm Rev}_{\d}$ is in a neighborhood of $\tau=1$ the   union of a horizontal segment $\tau=1+t$, $t\in \R$, and an almost vertical curve tangent at $\tau=1$ to the vertical line $1+it$, $t\in\R$.

\subsubsection{Elliptic vs. hyperbolic case}
In the reversible situation 
there are  two interesting  cases (in the following discussion $\a\approx \b/2$): 
\begin{itemize}
\item The {\it a priori} elliptic  (or stable) case: $\a$ and $\b$ in (\ref{eigenvalues}) are real numbers; in this case the two fixed points of $h^{\textrm{Hénon}}_{\b,c}$ are {\it elliptic} and when $(\a,\b)$ satisfies a Diophantine condition they belong to Siegel disks.

If $\d$ is small enough and $(\tau,\mbeta)\in (0,2)\times (-1,1)$ the corresponding  Hénon map $h^{\textrm{Hénon}}_{\b,c}$ is {\it a priori} elliptic.
\item The {\it a priori} hyperbolic (or unstable) case: $\b$ is real but the imaginary part of $\a$ doesn't vanish; in this case the two fixed points of $h^{\textrm{Hénon}}_{\b,c}$ are {\it hyperbolic}. There does not exist any Siegel disk containing either of  these fixed points.

If $\d$ is small enough and $(t,\mbeta)\in (-1,1)^2$ the Hénon map associated to $(\tau_{\d}(t,\mbeta),\mbeta)$ is {\it a priori} hyperbolic.
\end{itemize}

One of the main result of this paper is the existence, in both the elliptic  and hyperbolic case of Exotic rotation domains.

\begin{Main}[Existence of Exotic Rotation Domains, Elliptic case ]\label{main:A}
There exists $\d_{0}>0$ such that for any $\d\in (0,\d_{0})$ the following holds. There exists a   positive measure set $E^{\rm ell}_{\d}\subset (-1,1)^2$ such that for $(\tau,\mbeta)\in E^{\rm ell}_{\d}$ the Hénon map $h^{\textrm{Hénon}}_{\b,c}$ with 
\begin{align*}
&\b=(1/3)+\d \mbeta\\
&c=c_{\d}(\tau,\mbeta)
\end{align*}
has an exotic rotation domain. 
One can choose $E^{\rm ell}_{\d}$ so that each fixed point of $h^{\textrm{Hénon}}_{\b,c}$ belongs to a Siegel disk.
\end{Main}

\begin{Mainprime}[Existence of Exotic Rotation Domains, Hyperbolic case ]\label{main:Aprime}
There exists $\d_{0}>0$ such that for any $\d\in (0,\d_{0})$ the following holds. There exists a   positive measure set $E^{\rm hyp}_{\d}\subset (-1,1)^2$ such that for $(t,\mbeta)\in E^{\rm hyp}_{\d}$ the Hénon map $h^{\textrm{Hénon}}_{\b,c}$ with 
\begin{align*}
&\b=(1/3)+\d \mbeta\\
&c=c_{\d}(\tau_{\d}(t,\mbeta),\mbeta)
\end{align*}
has an exotic rotation domain. 
The fixed points of $h^{\textrm{Hénon}}_{\b,c}$ are hyperbolic.
\end{Mainprime}

\subsection{Existence of Herman rings in the dissipative case}
Assume the imaginary part of $\b$ is psoitive. This is the case where one hopes to find Herman rings.

In the dissipative case, if $\mbeta$ has small imaginary part compared with $\tau$, one has two cases: assuming $\d$ small enough
\begin{itemize}
\item if $\tau-1\in (-1,1)\setminus(-\rho,\rho)$ ($0<\Im\beta\ll \rho$) is such that  $\a$ is Diophantine, each fixed point belong to  the basin of   attraction of some (1-dimensional) complex disk: one thus has attracting Siegel disks.
\item if, for example, $|\Im\tau| >\rho>0$, $0<\Im\beta\ll \rho$, the fixed point are hyperbolic and no quasi-periodic orbit exist in their neighborhood.
\end{itemize}

\begin{Main}[Existence of Herman Rings]\label{main:B}There exists $\mbeta_{0}>0$, $\ph_{0}\in (0,1/10)$, $\d_{0}>0$ such that for any $\mbeta\in (\mbeta_{0} /2,\mbeta_{0})$, $\ph\in (\ph_{0}/2,\ph_{0})$, $\d\in (0,\d_{0})$  the following holds. There exist nonempty open  intervals $I,J\subset \R$ ($J$ containing $\mbeta$), a  $C^1$-embedding $\check \tau: I\to \C$  and a positive Lebesgue measure set $A\subset I$ such that  for any $\tau=\check\tau(\a)$, $\a\in A$, and any $\mbeta\in J$,  the Hénon map $h^{\textrm{Hénon}}_{\b,c}$ with 
\begin{align*}
&\b=(1/3)+\d \mbeta\\
&c=c_{\d}(\tau,\mbeta)
\end{align*}
has an attracting Herman ring with rotation number $\a$. 
\end{Main}

\subsection{Where are these invariant objects loacated?}
\subsubsection{$\tau$ close to 1}
When $\tau$ is close to 1 (for example $|\tau-1|\leq 10^{-4}$ if $|\mbeta|$ is in $(1/2,2)$), one can prove that the point 
$$\bm z_{*}\\ w_{*}\em=\bm 1/2\\ 1/2\em +2.354\times (\pi\sqrt{3} \mbeta \d)^{2/3}\bm e^{2\pi i /3}\\ 1\em$$
is a reasonable initial condition which is close to the invariant annuli of Theorems \ref{main:A}, \ref{main:Aprime}. See Subsection \ref{sec:locinvan}. Note that the same thing holds for the annuli of Theorem \ref{main:B} except that one has to change $\mbeta$ into some $\mbeta e^{i\ph}$ ($\ph>0$ if $\mbeta>0$) so that the frequency on this annulus has vanishing imaginary part. In other words, one has to choose $\ph$ so that
$$e^{i\ph}(\d\mbeta\times(-0.834+0.183\times (\tau-1)^2/2)+h.o.t.)$$
has a vanishing imaginary part (see (\ref{eq:16.approxfreq})).

\subsubsection{More general case}
In fact, our method allows to prove existence of Exotic rotation domains or Herman rings for $\tau$ not so close to 1. 
See Subsection \ref{sec:16.4}.

\bigskip 
\subsection*{Acknowledgments} 
The author wishes to thank  Pierre Berger who brought to his attention the problem of the existence of Exotic rotation domains and Herman rings. He is grateful to Eric Bedford  and Xavier Buff for stimulating discussions  at preliminary stages  of this work (and later), to Prof. Ushiki for his interest  and support and  to Misha Lyubich and Dima Dudko for their patient listening of the strategy of the proof.   The author  benefited from financial support  of the ANR project  KEN  (ANR-22-CE40-0016), of the ERC project ``Emergence of Wild
Differentiable Dynamical Systems'' and of the French-Japanese Workshop on Real and Complex Dynamics (April 2023, September 2024). The author also thanks, the organizers and participants  of these workshops in Kyoto and Sapporo, in particular Yutaka Ishii,  Yushi Nakano, Mitsuhiro Shishikura, Masato Tsuji,  Shigehiro Ushiki, as well as E. Bedford, D. Dudko, M. Lyubich at  the IMS in Stony Brook, for their kind hospitality.

\medskip
\section{Sketch of the proof}\label{sec:sketch}

Recall the notations (\ref{hhenon}), (\ref{eigenvalues}),  (\ref{alpha,betaMain}).

After a simple preliminary conjugation by  an  affine map of $\C^2$ we are reduced to  the dynamics of the quadratic polynomial map 
\be h^{\textrm{mod}}_{\a,\b}:\C^2\ni \bm z\\ w\em\mapsto  \bm \l_{1}z\\\ \l_{2}w\em+\frac{q(\l_{1}z+\l_{2}w)}{\l_{1}-\l_2}\bm 1\\ -1\em\in \C^2\label{def:f:sketch}\ee
where 
$$\begin{cases}
& \l_{1}=e^{2\pi i (-\a+\b/2)},\qquad \l_{2}=e^{2\pi i (\a+\b/2)},\qquad \a\in\C\\
&q(z)=e^{i\pi \b}z^2.
\end{cases}$$
which  has an obvious fixed point at the origin. 

Furthermore, this map is {\it conformal-symplectic} in the sense that 
$$(h^{\textrm{mod}}_{\a,\b})^*dz\wedge dw=e^{2\pi i \b}dz\wedge dw.$$
It is in fact {\it exact}-conformal-symplectic: it can be written
$$h^{\rm mod}_{\a,\b}= \iota_{F}\circ \diag(\l_{1},\l_{2})$$ 
where $\iota_{F}$ denotes some {\it exact symplectic mapping} (see Subsection  \ref{sec:5.2} ) associated to a holomorphic  observable  
$$F(z,w)=i\mu_{\d} \frac{(z+w)^3}{3}+O^4(z,w)$$
with
\be\mu_{\d}=\frac{1}{2\sin(2\pi\a)} .\label{def:mu:skectch}\ee

\subsection{ Resonant Birkhoff Normal Forms}\label{RBNF}

\medskip
A natural idea is then  to apply techniques from Birkhoff Normal Form theory to reduce as much as possible the term $F$ to a simpler one. This means that we try to find successive symplectic (or conformal symplectic) changes of coordinates that kill as much terms in $F$ as possible. In the absence of {\it resonances}, one could, for any arbitrary $N\in\N$, reduce $F$ to $O^N(z,w)$. 

However, in the regime we are considering 
$$\b=(1/3)+\d\mbeta,\qquad \a= (1/6)+\d \malpha,
$$
(where $\d$ is small) resonances are indeed present due to the approximate equalities
\be \begin{cases}&\a\approx \beta/2\\
&(4-1)\times \beta\approx 1.
\end{cases}\label{resushiki}\ee

The resonant terms cannot be eliminated but one can still perform a   {\it Resonant Birkhoff Normal Form} procedure. This way we arrive, after some conjugations, to a diffeomorphism defined in a neighborhood of the origin which is of the form 
$$\iota_{Y_{}}^{-1}\circ h^{\rm mod}_{\a,\b}\circ \iota_{Y_{}}=\diag(\l_{1},\l_{2})\circ \iota_{F_{BNF}}$$
where $F_{BNF}$ is
\begin{multline*}F_{BNF}(z,w)=-2\pi i \malpha \d zw+b_{2,1}^{BNF}z^2w+b_{0,4}^{BNF}w^4\\ +\sum_{k=3}^{3m}b^{BNF}_{k,1}z^kw+\sum_{n=2}^{m}b^{BNF}_{0,3n+1}w^{3 n+1}+O^{3m+2}(z,w).
\end{multline*}
See Proposition \ref{prop:BNF}.

\begin{figure}[h]
\includegraphics[scale=0.5]{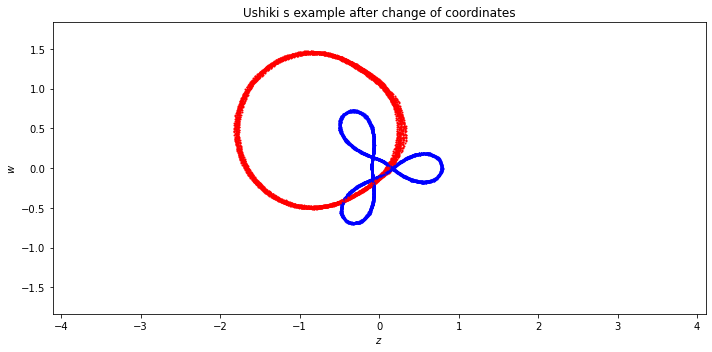}
\caption{S. Ushiki's example after a change of coordinates (BNF and scaling). Parameters $\mbeta=(-1.8592)/3,\qquad
\malpha=(-0.8846+2.67\sqrt{-1})/3$, $\delta=0.01$; initial condition $(z_{*},w_{*})$,  $z_{*}=2.3+3.5\sqrt{-1}$, $
w_{*}=-3.8+7.2\sqrt{-1}$.  5000 iterations. The red (resp. blue) curve is the projection of the orbit on the $z$-coordinate (resp. $w$-coordinate). Scaling factor  of the picture $0.1$. }\label{fig:1}
\end{figure}

\subsection{Reduction to the dynamics of a vector field}\label{RDVF}

After a well chosen dilation, the dynamics of $ \diag(\l_{1},\l_{2})\circ \iota_{F_{BNF}}$ takes   the form
\be \diag(1,e^{2\pi i /3})\circ \phi^1_{\d X_{0}}\circ \iota_{O(\d^2)}\label{sketch:eq1}\ee
where $\phi^1_{\d X_{0}}$ is the time-1 of the vector field 
\be X_{0}(z,w)=2\pi i \bm (1-\tau)z+\mu z^2+\nu w^3\\ \tau w-2\mu zw\em.\label{vfmod:sketch}
\ee
where $\tau$ is defined by the relation
$$\malpha=(\tau-1/2)\mbeta.$$
and $\mu$ and $\nu $ are
\be \mu=\frac{1}{\sqrt{3}}\approx 0.577,\qquad \nu=-(2/3)\frac{1}{\sqrt{3}} \approx -0.3849.\label{munu:sketch}\ee

The vector field $X_{0}$ has constant divergence equal to $2\pi i\mbeta$ and commutes with $\diag(1,e^{2\pi i/3})$. An important consequence of this last fact is that one can control the dynamics of (\ref{sketch:eq1}) at least for times $n=O(\d^{-(1+\e)})$: 
\be \biggl(\diag(1,e^{2\pi i /3})\circ \phi^1_{\d X_{0}}\circ \iota_{O(\d^2)}\biggr)^{\circ 3n}= \phi^{3n}_{\d X_{0}} \circ \iota_{O(n\d^2)}.\label{sketch:eq2}\ee

\begin{figure}[h]
\includegraphics[scale=0.5]{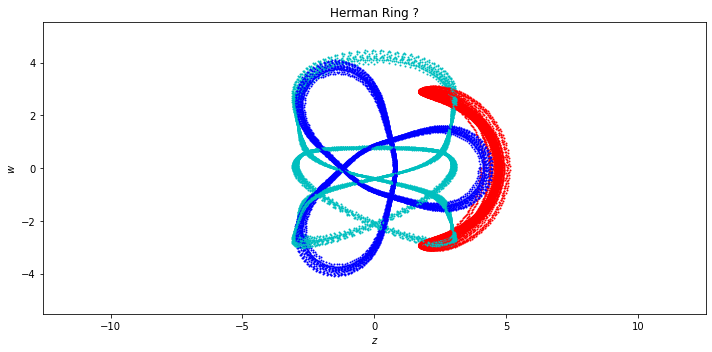}
\caption{A  Herman ring  in the reduced model $h^{\rm mod}_{\alpha,\beta}$ (scaling factor  $0.5$). Parameters $\mbeta=0.311841+(1/3)\times 10^{-3}\sqrt{-1}$, $\malpha=(\tau-(1/2))\times \mbeta$, $\tau=0.4-.0071\sqrt{-1}$, $\delta=10^{-3}$. Initial condition $(z_{*},w_{*})$,  $z_{*}= 8.0734+0.00195\sqrt{-1}$, 
$w_{*}= 7.904-0.204\sqrt{-1}$.  10000 iterations. The red (resp. blue) curve is the projection of the orbit on the $z$-coordinate (resp. $w$-coordinate).  The cyan curve is the projection $(\Im z, \Re w)$.}\label{fig:6}
\end{figure}

\begin{figure}[h]
\includegraphics[scale=0.5]{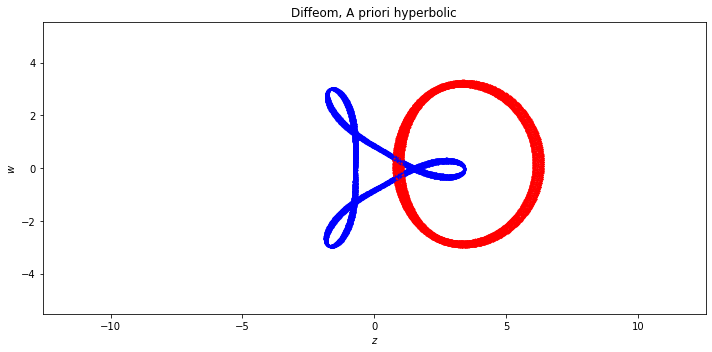}
\caption{Another Ushiki's example after a change of coordinates (scaling factor  $1$). Parameters  $\mbeta=0.311841$, $\malpha=((1/2)+10^{-1}\times\sqrt{-1}))\times \mbeta$, $\delta=0.01$; initial condition 
$(z_{*},w_{*})$,  $z_{*}=1.6+2.3\sqrt{-1}$, $
w_{*}=-1.59-2.19\sqrt{-1}$.  10000 iterations. The red (resp. blue) curve is the projection of the orbit on the $z$-coordinate (resp. $w$-coordinate). }\label{fig:2}
\end{figure}

\subsection{ The dynamics of $X_{0}$ and the Invariant annulus theorem}

\medskip
It turns out that the vector field $X_{0}$ has an  ``unexpected'' (we call it {\it exotic} in Subsection \ref{sec:14.1.1}) non trivial periodic orbit 
$$(\phi^t_{X_{0}}(\zeta_{0}))_{t\in\R}$$ 
with period $1/g_{0}(\tau)\in \R$ when $\tau$ lies in a  complex neighborhood of 1 and on the ``cross'' 
\be C_{0}:=\{\Re \tau=1\}\cup \{\Im\tau=0\}.\label{deformedcross}\ee

This fact is {\it a priori} not completely obvious to establish; nevertheless,  one can give a rigorous mathematical, though ``abacus''-assisted, proof of its existence\footnote{A ``geometric'' proof would of course be highly desirable.}. This is done in Section \ref{sec:perorbthm}.

Note that since we are dealing with holomorphic vector fields, the existence of a periodic orbit implies the existence of an embedded 1-dimensional {\it annulus}  $\cA^{\rm vf}_{\tau}\simeq \T_{s}=(\R+i(-s,s))/\Z$ (just slightly complexify the time $t$ to see this), invariant by the flow of $X_{0}$ and on which the dynamics of $X_{0}$ is conjugate to $g_{0}(\tau)\pa_{\th}$ with $g_{0}(\tau)\in \R$. Note that when $g_{0}(\tau)$ is {\it real}, the orbits of $g_{0}(\tau)\pa_{\th}$ on $\T_{s}$ are ``horizontal'' circles. 

If one believes in the fact that the vector field $X_{0}$ is a good approximation of the discrete dynamics we are studying, one understands that a modification of the parameter $\tau$ giving a non zero imaginary part to $g_{0}(\tau)$ may destroy this situation: the orbits of $g_{0}(\tau)\pa_{\th}$ on $\T_{s}$ then spiral and after a time leave the domain of validity of the model. This explains why Exotic rotation domains of Herman rings are not so easy to observe numerically: the vanishing of $\Im g_{0}(\tau)$ must be quite sharp.

\begin{figure}[h]
\includegraphics[scale=0.5]{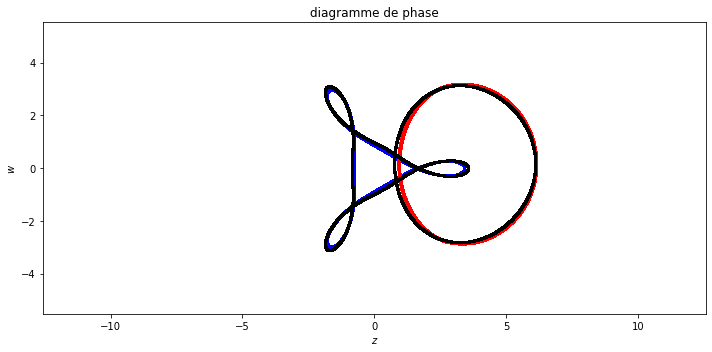}
\caption{Vector field version approximation of the previous diffeomorphism. Same parameters, same initial conditions. 
 The red (resp. blue) curve is the projection of the orbit on the $z$-coordinate (resp. $w$-coordinate). The black curves are $t\mapsto z(t)=\sum_{k=-2}^2 z_{k}e^{3ik\omega t}$, $t\mapsto w(t)=\sum_{k=-2}^1 w_{k}e^{i(3k+1)\omega t}$ for adequate choices of $z_{l},w_{l},\omega$. }\label{fig:3}
\end{figure}

\subsection{Improved vector field approximation}

\medskip For technical reasons we need a better vector field approximation than (\ref{sketch:eq1}) where the exponent 2 is replaced by an exponent $p$ large enough:
\be \diag(1,e^{2\pi i /3})\circ \phi^1_{\d X_{\d}}\circ \iota_{O(\d^p)}\label{sketch:eq2}\ee
and where the vector field $X_{0}$ is replaced by the vector field
$$X_{\d}=X_{0}+O(\d).$$
This vector field is constructed in Section \ref{sec:vfapprox} so that it keeps  the same $\diag(1,e^{2\pi i /3})$-symmetry property. Furthermore, because the linearization of $X_{0}$ along its periodic orbit is non-degenerate, one can prove that for $\tau$ in a neighborhood of 1 and on a slightly deformed cross $C_{\d}\approx C_{0}$ (cf. (\ref{deformedcross})) the vector field $X_{\d}$ has a periodic orbit  $(\phi^t_{X_{\d}}(\zeta_{\d}))_{t\in\R}$ with real period $1/g_{\d}(\tau)$. This is done in Section \ref{sec:invannulusthm}.

\subsection{From the dynamics of the vector field to the discrete dynamics: renormalization and commuting pairs}
\medskip The periodic orbit $(\phi^t_{X_{\d}}(\zeta_{\d}))_{t\in\R}$ allows us to understand {\it first returns} of the discrete dynamics $h_{\d}= \phi^{1}_{\d X_{\d}} \circ \iota_{O(\d^p)}$ in some well chosen {\it boxes} $\cW_{\d}$ (of size $\d$)  where it can be {\it renormalized} (see Section \ref{sec:renormandcommpairs}). The dynamics of $h_{\d}:= \phi^{1}_{\d X_{\d}} \circ \iota_{O(\d^p)}$ is thus reduced  to the study of a {\it commuting pair} $(h_{\d},h^q_{\d})$ ($q$ some integer related to first return times) i.e. a pair of commuting holomorphic diffeomorphisms defined on a neighborhood of $\cW_{\d}$.  After some further conjugation/dilation  this pair can be brought to a commuting pair defined  on a domain   $((-1-\nu,2+\nu)+i(-s,s))\times \bD(0,s)$ ($\nu>0,s>0$) and of the form
$$\bm (z,w)\mapsto (z+1,w)+{\rm small}\\ (z,w)\mapsto (z+\check \a, e^{2\pi i \check \b}w)+{\rm small}\em.$$
This pair can be {\it normalized} i.e. conjugated  to  the nicer form 
$$\bm (z,w)\mapsto (z+1,w)\\ (z,w)\mapsto (z+\check \a, e^{2\pi i \check \b}w)+{\rm small}\em.$$
In this form the second diffeomorphism $(z,w)\mapsto (z+\check \a, e^{2\pi i \check \b}w)+{\rm small}$ commutes with $(z,w)\mapsto (z+1,w)$ and is hence ``1-periodic'' in the $z$-variable a fact which is useful if one wants to use Fourier analysis (see Section \ref{sec:11}). 

Note that this normalization procedure is a kind of {\it uniformization} that we have to prove in a 2-dimensional holomorphic setting (see Appendix \ref{sec:3.3}). See \cite{Y-an} and \cite{AK} for related normalization procedures in the 1-dimensional holomorphic setting and  \cite{KaKri} in the smooth real  2-dimensional one. Renormalization of commuting pairs (``cylinder renormalization'') is also used in   \cite{GaiRadYam}, \cite{GaiYam}.  

In fact, the commuting pairs we shall be working with are {\it partially normalized} ones (see Section \ref{sec:panormpair}) i.e. commuting pairs of the form
\be \bm (z,w)\mapsto (z+1,e^{2\pi i \d \mbeta}w)\\ (z,w)\mapsto (z+\ti \a, e^{2\pi i q\d\mbeta}w)\circ \iota_{O(\d^p)}\em.\label{sketch:pnormalizedpair}\ee
which preserve some conformal symplectic structure.

\subsection{KAM-Siegel Theorem for commuting pairs}
\medskip
Once we have a partially normalized commuting pairs (\ref{sketch:pnormalizedpair}) we are in position to  prove a {\it linearization}
result, similar to Siegel linearization theorem, that says that the pair (\ref{sketch:pnormalizedpair}) can be conjugated to a pair of the form
\be \bm (z,w)\mapsto (z+1,e^{2\pi i \d \mbeta}w)\\ (z,w)\mapsto (z+\ti \a, e^{2\pi i q\d\mbeta}w)\em\label{sketch:pnormalizedpairred}\ee
(the value of $\ti\a$ is may have changed).
 The proof is based (like for Siegel theorem) on a KAM scheme (here performed on partially normalized commuting pairs), the only difference lying in the fact that one has to pay attention to keeping the frequencies {\it real} and avoiding resonances. Like in most\footnote{Note that this is not necessary when one wants to prove the classic Siegel linearization theorem.} KAM linearization problems  we  thus have to do some parameter exclusion (on $\tau$ and $\mbeta$)  which takes two different guises  according to whether we are in the conservative case (Theorems \ref{main:A},  \ref{main:Aprime}) on Exotic rotation domains) or dissipative case (Theorem \ref{main:B} on Herman rings).     In the conservative case, an important feature is the use of the  {\it reversibility} of the initial Hénon map.

\subsection{ Proving the existence of Exotic rotation domains or Herman rings}
\medskip The conjugation of the pair $(h_{\d},h_{\d}^q)$ to (\ref{sketch:pnormalizedpairred}) which is defined on the small  box $\cW_{\d}$ is useful to get more {\it global} information on the dynamics of $h_{\d}$. In the conservative case (the frequencies are real) it yields the existence of an $h_{\d}$-invariant rotation domain diffeomorphic to the product of an annulus by a  disk (and which contains an invariant circle) where the dynamics is conjugate to $(\zeta_{1},\zeta_{2})\mapsto (e^{2\pi i a_{1}}\zeta_{1},e^{2\pi i a_{2}}\zeta_{2})$, while  in the dissipative case it yields   a basin of attraction of an $h_{\d}$-invariant attracting circle. This analysis is carried out in Section \ref{sec:criterion}.

To prove  these domains are invariant by the map $ \diag(1,e^{2\pi i /3})\circ \phi^1_{\d X_{\d}}\circ \iota_{O(\d^p)}$ (see (\ref{sketch:eq2})) we exploit the fact that the invariant circle they contain is almost invariant by  $\diag(1,e^{2\pi i /3})$.

Finally to prove they are {\it exotic} (i.e. do not come from Siegel disks or Herman disks associated to the fixed points) we compare the frequency on the invariant circle to those of the fixed points (in the {\it a priori} elliptic case, since in the {\it a priori} hyperbolic case there is nothing to prove). We refer to Sections \ref{sec:proofmainA} and \ref{sec:proofmainA:B} for more details.

\subsection{On the proof of the existence of Exotic periodic orbits for $X_{0}$}\label{PEERD}
\medskip As we mentioned,  an important point is the  proof of the existence of a periodic orbit for $X_{0}$ when $\tau$ lies in the cross $C_{0}$ (at least close to 1); this  is done the following way. 

We just need to prove the result for $\tau=1$. Numerical experiments show that the vector field $X_{0}$ has what seems to be a periodic orbit with a nice $\diag(1,e^{2\pi i/3})$-symmetry. But, this is somehow surprising because the fact that $X_{0}$ commutes with $\diag(1,e^{2\pi i/3})$ does not imply such a symmetry. This suggests to look for periodic orbits $p(t)=(z(t),w(t))$  of $X_{0}$ which have this symmetry, namely
$$
z(t)=\sum_{k\in\Z}z_{3k}e^{3ki (2\pi g) t}
\qquad w(t)=\sum_{k\in\Z}w_{3k+1}e^{ (3k+1)i (2\pi g) t}
\qquad g\in \C.
$$
One can  find {\it approximate} periodic solutions to the differential equation $\dot p=X_{0}(p)$ by projecting on a finite dimensional space of harmonics ($|k|\leq N$, we choose $N=12$) and by fixing the value of $w_{1}$ to the value 1.4.  Note that fixing the value of $w_{1}$ amounts to choosing a ``height''  in the searched for $X_{0}$-invariant annulus: indeed, when the time $t$ is complexified to  $t+is$, $s$ small, the value of all the coefficients $z_{3k}$ and $w_{3k+1}$ are changed to $z_{3k}e^{-6\pi g ks}$ and $w_{3k+1}e^{-2\pi (3k+1)gs}$. To find an approximate solution to some good order we use a Newton scheme which is easy to implement. 

To prove that this {\it approximate} periodic solution is close to an {\it exact} periodic solution we have to study the {\it linearization of the flow} of $X_{0}$ along this approximate periodic orbit. This leads to a l{\it inear} differential equation with {\it periodic} coefficients. But understanding a linear ODE with periodic coefficients can be done by having information on the  {\it Floquet} decomposition of its {\it resolvent matrix} (see Subsection \ref{sec:15.6}). Here again we end up with an infinite dimensional algebraic problem that can be projected on a finite dimensional space and approximately solved. This gives us enough information to control the {\it linearization of the flow} of $X_{0}$ and prove the existence of a true periodic solution for $X_{0}$ when $\tau=1$. 

The preceding procedure allows us to prove that   the derivative of the function $\tau\mapsto g_{0}(\tau)$ doesn't vanish identically on a neighborhood of $\tau=1$.  More precisely we can compute the approximate value of the derivative at $\hat \tau=1$ of the function $\hat g_{0}$ defined by $\hat g_{0}(\tau-\tau^2/2)=g_{0}(\tau)$.

To keep as much as possible estimates under control, we write  all the Implicit function  or Inverse mapping theorems we implicitly use,  as contracting fixed point problems. 

\bigskip

\begin{rem}
The discussion of subsections \ref{RBNF} ,  \ref{RDVF} adapts to other kind of resonances. For example one can choose
$$ \a=\a_{\d}=\frac{1}{4}+\d\malpha,\qquad \b=\b_{\d}=\frac{1}{2}+\d\mbeta$$
\be\begin{cases}&\a\approx \beta/2\\
&(3-1)\times \beta\approx 1.
\end{cases}\label{resorder2}\ee

\medskip

After one step of BNF we see that
$$\Phi_{Y}^{-1}\circ h^{\rm mod}_{\a,\b}\circ \Phi_{Y}=\diag(\l_{1},\l_{2})\circ \iota_{b_{2,1} z^2w+b_{0,3}w^3+O^4(z,w)}$$
with
\begin{align*}b_{2,1}&=i\frac{\mu_{\d}}{3\l_{1}\l_{2}}\times (3\l_{1}^2\l_{2})\\
&=i\mu+O(\d)
\end{align*}
\begin{align*}b_{0,3}&=i\frac{\mu_{\d}}{3\l_{1}\l_{2}}\times (\l_{2}^3)\\
&=i(1/3)\mu+O(\d)
\end{align*}
where
$$ \mu=\mu_{0}=\frac{1}{2\sin(2\pi /4)}=\frac{1}{2}.$$
The relevant vector field  in subsection \ref{PEERD} is then 
$$X_{0}(z,w)=2\pi i \bm (1-\tau)z+\mu z^2+\mu w^2\\ \tau w-2\mu zw\em.$$
One can find periodic solutions of this vector field by looking for $(z,w)$ of the form
$$
z(t)=\sum_{k\in\Z}z_{2k}e^{2ki (2\pi g) t}
\qquad w(t)=\sum_{k\in\Z}w_{2k+1}e^{ (2k+1)i (2\pi g) t}
\qquad g\in \C.
$$
The techniques developed in this paper also  yield the existence of ERD and Herman rings for the specific resonance (\ref{resorder2}).

\medskip

More generally it would  be interesting to investigate the following problems:
\begin{itemize}\item  Which resonances give rise to ERD and Herman rings?
\item Can one prove the existence of a {\it real} Hénon map ($b$ and $c$ are real) with a Herman ring?\footnote{A good choice could  be   $\a=\a_{\d}=\frac{1}{2}+\d\malpha$, $\b=\b_{\d}=1+\d\mbeta$. }
\end{itemize}

\end{rem}

\bigskip
\section{Notations and preliminaries}

We denote for $z\in \C$ and $\rho>0$, $\bD(z,\rho)=\{\zeta\in\C\mid |\zeta-z|<\rho\}$ and for $d\in\N$,
$\bD_{\C^d}(\zeta,\rho)$, ($z=(z_{1},\ldots, z_{d})\in\C^d$, $\rho>0$), the polydisk
$$\bD_{\C^d}(\zeta,\rho)=\prod_{k=1}^d\bD(z_{k},\rho).$$
We shall sometimes use the notation 
$$\bD_{\R^d}(z,\rho)=\bD_{\C^d}(z,\rho)\cap\R^d.$$

Let $U$ be a nonempty open set of $\C^d$. We denote $\cO(U)$ the set of holomorphic functions $F:U\to \C$. With the norm
$$\|F\|_{U}=\sup_{\zeta\in U}|F(\zeta)|$$
it is a Banach space. If $\e>0$ we set
\be \cB_{\e}(U)=\{F\in \cO(U)\mid \|F\|_{U}<\e\}.\label{defcBepsilonbis}\ee

Let $\d>0$. We denote $\cU_{\d}(U)$  the open set (possibly empty) containing all the  $\zeta\in U$ for which the polydisk $\bD_{d}(\zeta,\d)$ is included in $U$.
One has for $\d_{1},\d_{2}>0$
\be \cU_{\d_{1}}(\cU_{\d_{2}}(U))\supset \cU_{\d_{1}+\d_{2}}(U).\label{Udeltan+1}\ee
By Cauchy estimates one has for any $F\in \cO(U)$
\be \|\pa F\|_{\cU_{\d}(U)}\lesssim \d^{-1}\|F\|_{U}\label{eq:Cauchyestimate}\ee
where we denote by $\pa F(z_{1},\dots,z_{d})$ any derivative $\pa_{z_{i}}F(z_{1},\ldots,z_{d})$.

\subsection{Notations $\frak{O}$, $\frak{d}$}\label{sec:notationfrakO}
Let $U$ be an open set of $\C^d$, functions $F_{1},\ldots,F_{n}\in\cO(U)$ and $l\in\N^*$. We define the  relation 
$$G=\fO_{l}(F_{1},\ldots,F_{n})$$
as  follows:
there exist $a\in\N^*$, $C>0$ and   $Q(X_{1},\ldots,X_{n})$ a homogeneous polynomial   of degree $l$  in the variables $(X_{1},\ldots,X_{n})$ such that for any  $\d>0$ satisfying 
\be C \d^{-a}\max_{1\leq i\leq n}\|F_{i}\|_{U}\leq 1\label{2.35bis}
\ee
one has $G\in\cO(\cU_{\d}(U))$ and  
\be 
\|G\|_{\cU_{\d}(U)}\leq C\delta^{-a} Q(\|F_{1}\|_{U},\ldots,\|F_{n}\|_{U}).\label{2.36}
\ee
\medskip

When we want to keep track of the exponent $a$ appearing in (\ref{2.35bis}), (\ref{2.36}) we shall use the symbol $\fO_{l}^{(a)}$.

When $\d$ satisfies (\ref{2.35bis}) we write  
\be \d=\fd^{a,C}(F_{1},\ldots,F_{n};U) \label{notation:delta}\ee
and we use the short hand notation
\be\d=\fd(F_{1},\ldots,F_{n};U)\quad\textrm{or}\quad \d=\fd(F_{1},\ldots,F_{n}) \label{notation:deltabis}\ee
to say that (\ref{notation:delta}) holds for some positive constants  $a,C$ large enough and  independent of $F_{1},\ldots,F_{n}$.

For example, the Cauchy estimate (\ref{eq:Cauchyestimate}) can be written
$$\pa F=\fO_{1}(F)$$
on some domain $\cU_{\nu}(U)$ for $\nu=\frak{d}(F)$.

For $s,\rho>0$ we set
$$W_{s,\rho}:=\T_{s}\times\bD(0,\rho)$$
and if $\nu>0$
$$e^{-\nu}W_{s,\rho}=\T_{e^{-\nu }s}\times \bD(0,e^{-\nu}\rho).$$

The interest of these notations lies in the following proposition.
\begin{prop}[Quadratic convergence]\label{lemma:quadraticconv}Assume $F_{0}\in \cO(U)$ is an observable defined on an open set $U$ of $\C^d$ and  that  $(F_{n})_{n\in\N}$, satisfies 
$$F_{n+1}=\fO_{2}(F_{n}).$$
Then, 
if $\|F_{0}\|_{U}$ is small enough, there exists $\d_{\infty}>0$ such that $\cU_{\d_{\infty}}(U)\ne\emptyset$ and 
$$\begin{cases}&F_{n}\in \cO(\cU_{\d_{\infty}}(U))\\
&\lim_{n\to \infty}\|F_{n}\|_{\cU_{\d_{\infty}}(U)}=0.
\end{cases}
$$ 
One has also for some $\rho>0$, $\|F_{n}\|_{\cU_{\d_{\infty}}(U)}\leq e^{-\rho 2^n}$.
\end{prop}
\begin{proof}We first choose $\d_{\infty}$ such that $\cU_{\d_{\infty}}(U)\ne \emptyset$ and we define for $\nu_{n}=2^{-(n+1)}$
$$\d_{n}=\nu_{n}\d_{\infty}$$
so that
$$\sum_{n=0}^\infty \d_{n}=\d_{\infty}.$$

 By assumption there exists $C>0, a>0$ such that if $C\d_{n}^{-a}\|F_{n}\|_{U_{n}}\leq 1$ one has 
$$\|F_{n+1}\|_{\cU_{\d_{n}}(U_{n})}\leq C\d_{n}^{-a}\|F_{n}\|_{U_{n}}^2.$$ 
So, if we define $U_{n+1}=\cU_{\d_{n}}(U_{n})$ and $\e_{n}=\|F_{n}\|_{U_{n}}$
$$\e_{n+1}\leq C\d_{\infty}^{-a}2^{a(n+1)}\e_{n}^2$$
provided
\be C\d_{\infty}^{-a}2^{a(n+1)}\e_{n}\leq 1.\label{estCdelta}\ee
A computation shows that if 
$$-\rho:=\ln\e_{0}+\ln (C\d_{\infty}^{-a})+a\ln 2\sum_{n=0}^\infty n2^{-(n+1)}$$
is negative enough, one has for all $n\geq 0$
$$\e_{n}\leq e^{-\rho 2^n}$$
and at the same time (\ref{estCdelta}) is satisfied.

We conclude by observing (use (\ref{Udeltan+1})) that $U_{n}\supset \cU_{\d_{\infty}}(U)\ne \emptyset$.
\end{proof}

\subsection{Exact symplectic maps}\label{sec:exsymp}
If $F:(\C^2,(0,0))\to\C$ is a holomorphic germ we define the so-called {\it canonical}   diffeomorphism
$$\iota_{F}:(\C^2,(0,0))\ni (z,w)\mapsto \iota_{F}(z,w)\in (\C^2,(0,0))$$
by
\be \iota_{F}(z,w)=(\ti z,\ti w)\iff 
\begin{cases}
&\ti z=z+\pa_{\ti w}F( z, \ti w)\\
&w=\ti w+\pa_{ z}F( z, \ti w).
\end{cases}\label{defexsymplmap}
\ee
It preserves the symplectic form $dz\wedge dw$ and it is in fact an exact symplectic diffeomorphism with respect to the Liouvlle 1-form $wdz$,  which means that the 1-form $(f_{F})^{*}(wdz)-wdz$ is exact: indeed
$$\ti wd\ti z-wdz=d(-F+(\ti z-z)\ti w).$$
Note that on simply connected domains, a map is symplectic if and only if it is exact-symplectic.

If $X$ is a vector field we denote by $\phi^t_{X}$ its  time-$t$ map (when it is defined). If $f$ is a diffeomorphism defined on a suitable domain one has
$$f\circ \phi^1_{X}\circ f^{-1}=\phi^1_{f_{*}X}$$
where 
$$f_{*}X=Df\circ f^{-1}\cdot X\circ f^{-1}.$$

If $A\in \cO(\C^2,(0,0))$ we define the symplectic vector field 
$$X=J\nabla A=(\pa_{w}A)\pa_{z}-(\pa_{z}A)\pa_{w}$$ 
and set\footnote{In what follows $J=\bm 0&1\\ -1&0\em$ and $\nabla A=\bm \pa_{z}A\\ \pa_{w} A\em$.}
$$\Phi_{A}=\phi^1_{J\nabla A}.$$

If $\l,\mu$ are complex numbers we denote by $\diag(\l,\mu)$ the linear map $\C^2\ni(z,w)\mapsto (\l z, \mu w)\in \C^2$.

If $A$ is an observable and $\l_{1},\l_{2}$ are in $\C^*$ we define
\be \ti A:=\diag(\l_{1},\l_{2})_{*}A:(z,w)\mapsto (\l_{1}\l_{2})A(\l_{1}^{-1}z,\l_{2}^{-1}w).\label{notationdiag}\ee

The {\it divergence} of a vector field $X=X_{z}\pa_{z}+X_{w}\pa_{w}$ is by definition the function
$${\rm div} X=\pa_{z}X_{z}+\pa_{w}X_{w}.$$
The divergence of a symplectic vector field $X=J\nabla A$ vanishes.

\subsection{Estimates on composition}
Here are some useful lemmas.

\begin{lemma}Let $F\in \cO(U)$.
If $\|F\|_{U}$ is small enough, there exists $\d>0$, $\d=\frak{d}(F)$, such that $\iota_{F}$ is holomorphic and defined on $\cU_{\d}(U)$ and its image  $\iota_{F}(\cU_{\d}(U))$ contains $\cU_{2\d}(U)$.
\end{lemma}
\begin{proof}We refer to \cite{K1}.
\end{proof}

Let $F_{1},F_{2}\in \cO(U)$ small enough.
\begin{lemma}\label{lemma:8.1}
\begin{enumerate}
\item If $F_{1},F_{2}$ are small enough one has on $\cU_{\d}(U)$, $\d=\frak{d}(F_{j})$, $j=1,2$
$$\iota_{F_{1}}\circ \iota_{F_{2}}=\iota_{F_{1}+F_{2}}\circ \iota_{\fO_{2}(F_{1},F_{2})}.$$ 
\item If $U=\bD(0,\rho)\times \bD(0,\rho))$ and $F_{1}=O(w^{p_{1}}),F_{2}=O(w^{p_{2}})$ one has 
$$\iota_{F_{1}}\circ \iota_{F_{2}}=\iota_{F_{1}+F_{2}}\circ \iota_{O(w^{(p_{1}+p_{2}-1)})}.$$ 
\end{enumerate}
\end{lemma}
\begin{proof}We refer to \cite{K1}.
\end{proof}

The following lemma is easy to prove:
\begin{lemma}\label{lemma:prel4.3}If $A,B\in \cO(U)$ and $X,Y:U\to \C^2$ are two holomorphic vector fields one has on $\cU_{\nu}(U)$ ($\nu=\frak{d}(A,B,X,Y)$)
\begin{enumerate}
\item $\iota_{A}=\iota_{B}\implies A=B+{\rm cst}$.
\item  If 
\be \ti A:=\diag(\l_{1},\l_{2})_{*}A:(z,w)\mapsto (\l_{1}\l_{2})A(\l_{1}^{-1}z,\l_{2}^{-1}w),\label{notationdiag}\ee
one has  
$$\begin{cases}&\diag(\l_{1},\l_{2})\circ \iota_{A}\circ \diag(\l_{1},\l_{2})^{-1}=\iota_{\ti A}\\
&\diag(\l_{1},\l_{2})\circ \Phi_{A}\circ \diag(\l_{1},\l_{2})^{-1}=\Phi_{\ti A}.
\end{cases}
$$
\item $\iota_{A}=\Phi_{A}\circ \iota_{\fO_{2}(A)}$. 
\item $\phi^1_{X}\circ \phi^1_{Y}=\phi^1_{X+Y}\circ (id+\fO_{2}(X,Y))$.
\end{enumerate}
\end{lemma}

\subsection{Results on approximation by vector fields}
The following two corollaries will be useful in Section \ref{sec:vfapprox}.
\begin{cor} [Approximation by vector fields]\label{cor:avf}For $A\in \cO(U)$, $A=O(\d)$  small enough, such that $\diag(\l_{1},\l_{2})_{*}A=A$,  there exists $A_{n}\in \cO(\cU_{\nu}(U))$ ($\nu=\frak{d}(A)$) such that $\diag(\l_{1},\l_{2})_{*}A_{n}=A_{n}$ and  on $\cU_{\nu}(U)$ one has 
$$\iota_{A}=\Phi_{A_{n}}\circ \iota_{O(\d^n)}.$$
\end{cor}
\begin{proof}By induction on $n$: if the corollary holds at step $n$ one has $\iota_{A}=\Phi_{A_{n}}\circ \iota_{B_{n}}$ with $B_{n}=O(\d^{n})$. Because $\iota_{A}$ and $\Phi_{A_{n}}$ commute with $\diag(\l_{1},\l_{2})$ the same thing  holds for $\iota_{B_{n}}$ hence $\diag(1,j)_{*}B_{n}=B_{n}$. We then write 
$$\iota_{B_{n}}=\Phi_{B_{n}}\circ \iota_{\fO_{2}(B_{n})}$$
and
\begin{align*}
\iota_{A}&=\Phi_{A_{n}}\circ\Phi_{B_{n}}\circ \iota_{\fO_{2}(B_{n})}\\
&=\Phi_{A_{n}+B_{n}}\circ \iota_{\fO_{2}(A_{n},B_{n})}\circ \iota_{\fO_{2}(B_{n})}\\
&=\Phi_{A_{n}+B_{n}}\circ \iota_{O(\d^{n+1})}.
\end{align*}
If one sets $A_{n+1}=A_{n}+B_{n}$, one has $\diag(\l_{1},\l_{2})_{*}A_{n+1}=A_{n+1}$.
\end{proof}

\begin{cor}[Baker-Campbell-Hausdorff]\label{cor:BCH} If the vector fields $X=O(\d)$ and $Y=O(\d)$ satisfies ${\rm div} \ X={\rm cst}$, ${\rm div}\ Y=0$, then   one has 
$$\phi^1_{X}\circ \phi^1_{Y}=\phi^1_{P_{n}(X,Y)}\circ \iota_{R_{n}(X,Y)}$$
where the vector field  
\begin{align*}P_{n}(X,Y)&=X+Y+\frak{O}_{2}(X,Y)\\
&=O(\d)
\end{align*}
satisfies ${\rm div} P_{n}(X,Y)={\rm div} X$ and  the observable $R_{n}$ verifies   $R_{n}(X,Y)=O(\d^n)$.  

Moreover, if $\diag(\l_{1},\l_{2})_{*}X=X$ and $\diag(\l_{1},\l_{2})_{*}Y=Y$ one has also $\diag(\l_{1},\l_{2})_{*}P_{n}(X,Y)=P_{n}(X,Y)$ and  $\diag(\l_{1},\l_{2})_{*}R_{n}(X,Y)=R_{n}(X,Y)$.
\end{cor}
\begin{proof}By induction on $n$: assuming this is true at step $n$, one writes 
$$\iota_{R_{n}(X,Y)}=\phi^1_{J\nabla R_{n}(X,Y)}\circ \iota_{O(\d^{2n})}$$
hence
\begin{align*}\phi^1_{X}\circ \phi^1_{Y}&=\phi^1_{P_{n}(X,Y)} \circ \phi^1_{J\nabla R_{n}(X,Y)}\circ \iota_{O(\d^{2n})}\\
&=\phi^1_{P_{n}(P_{n}(X,Y),J\nabla R_{n}(X,Y))+J\nabla R_{n}(X,Y)}\\ &\hskip 2cm \circ (id+\fO_{2}(P_{n}(X,Y),R_{n}(X,Y)))\circ \iota_{O(\d^{2n})}.\end{align*}
Let 
$$P_{n+1}(X,Y)=P_{n}(P_{n}(X,Y),J\nabla R_{n}(X,Y))+J\nabla R_{n}(X,Y)=O(\d).$$ 
By the induction assumption 
$${\rm div}\ P_{n+1}(X,Y)= {\rm div}\ P_{n}(X,Y)={\rm div}\ X$$
hence 
$$\det \phi^1_{P_{n+1}(X,Y)}=e^{{\rm div}\ X}=\det\phi^1_{X}=\det (\phi^1_{X}\circ \phi^1_{Y}).$$
Because $\det \iota_{O(\d^{2n})}=1$ one thus has in the above formula
$$\det (id+\fO_{2}(P_{n}(X,Y),R_{n}(X,Y)))=1.$$
There thus exists $R_{n+1}(X,Y)=O(\d^{n+1})$ such that 
$$\iota_{R_{n+1}(X,Y)}=(id+\fO_{2}(P_{n}(X,Y),R_{n}(X,Y)))\circ \iota_{O(\d^{2n})}.$$
This gives us the searched for decomposition
$$\phi^1_{X}\circ \phi^1_{Y}=\phi^1_{P_{n+1}(X,Y)}\circ \iota_{R_{n+1}(X,Y)}.$$

\bigskip
Furthermore, if $\diag(\l_{1},\l_{2})_{*}X=X$ and $\diag(\l_{1},\l_{2})_{*}Y=Y$ one has by the induction assumption  
$$\diag(\l_{1},\l_{2})_{*}P_{n}(X,Y)=P_{n}(X,Y)\  {\rm and}\  \diag(\l_{1},\l_{2})_{*}R_{n}(X,Y)=R_{n}(X,Y)$$ hence (by the induction assumption) $\diag(\l_{1},\l_{2})_{*}(P_{n}(P_{n}(X,Y),R_{n}(X,Y)))=P_{n}(P_{n}(X,Y),R_{n}(X,Y))$ and $\diag(\l_{1},\l_{2})_{*}P_{n+1}(X,Y)=P_{n+1}(X,Y)$. Because $\phi^1_{X},\phi^1_{Y},\phi^1_{P_{n+1}(X,Y)}$ commute with $\diag(\l_{1},\l_{2})$, we deduce that $\iota_{R_{n+1}(X,Y)}$ (hence $R_{n+1}(X,Y)$) commutes with $\diag(\l_{1},\l_{2})$.
\end{proof}

\subsection{Summary of the  notations used in the text}

\begin{itemize}
\item
We shall use the following notations: if $a\geq 0$ and $b>0$ are two  real numbers we write
$a\lesssim b$ for: ``there exists a constant $C>0$ independent of $a$ and $b$ such that $a\leq Cb$''. If we want to insist on the fact that this constant $C$ depends on a quantity $\beta$ we write $a\lesssim_{\b} b$.  We shall also write $a\ll b$ to say that $a/b$ is small enough and $a\ll_{\b} b$ to express the fact  that this smallness condition depends on $\b$. The notations $b\gtrsim a$, $b\gtrsim_{\b}a$, $b\gg a$ and $b\gg_{\b}a$ are defined in the same way. When one has $a\lesssim b$ and $b\lesssim a$ we write $a\asymp b$.
\item If $I,J$ are  interval of $\R$ we denote $I_{J}$ the set of complex numbers $x+iy$, $x\in I$, $y\in J$ and when $J=(-s,s)$ for some $s>0$ we just denote $I_{s}=I_{(-s,s)}$.

Similarly, we denote by $\T_{J}=\R_{J}/\Z$ and $\T_{s}=(\R+i(-s,s))/\Z$.
\item $\diag(\l_{1},\l_{2})$ is the linear map $(z,w)\mapsto (\l_{1}z,\l_{2}w)$.
\item For the notations $\iota_{F}$, $\phi^1_{X}$, $\Phi_{Y}$ see Subsection \ref{sec:exsymp}.
\item For the notations $\frak{O}$, $\frak{d}$ see Subsection \ref{sec:notationfrakO}.
\item If $\a$ is a complex number, we define its integer part $[\a]$ as the unique integer $q\in \Z$ for which $\Re(\a-m)\in [0,1)$ and we set $\{\a\}=\a-[\a]$.
\item If $\a$ and $\b$ are complex numbers, we set
$$\cT_{\a,\b}:\C^2\ni(z,w)\mapsto (z+\a,e^{2\pi i \b}w)\in \C^2.$$
\item If $v=(v_{1},v_{2})$ is a vector of $\C^2$ we denote $\|v\|$ its $l^2$-norm $\|v\|=(|v_{1}|^2+|v_{2}|^2)^{1/2}$. If $M=(m_{i,j})_{1l\leq i,j\leq 2}\in M(2,\C)$ is a matrix we denote $\|M\|$ or $\|M\|_{HS}$ its Hilbert-Schmidt norm $(\sum_{i=1}^2\sum_{j=1}^2|m_{i,j}|^2)^{1/2}$. It is a multiplicative norm ($\|M_{1}M_{2}\|_{HS}\leq \|M_{1}\|_{HS}\|M_{2}\|_{HS}$) and it controls the operator norm $\|M\|_{op}=\sup_{0\ne v\in\C^2}\|Mv\|/\|v\|$. In particular, $\|Mv\|\leq \|M\|_{HS}\|v\|$ hence $\|M\|_{op}\leq \|M\|_{HS}$.
\end{itemize}

\bigskip
\section{Birkhoff Normal Forms and Ushiki's resonance}\label{sec:Ushikisresonance}

\subsection{Modified Hénon maps}
Recall the Hénon map
\begin{align*}&h^{\textrm{Hénon}}_{\b,c}:\C^2\ni (x,y)\mapsto (e^{i\pi \b}(x^2+c)-e^{2\pi i \beta} y,x)\in \C^2,\qquad \b,c\in\C\end{align*}
has two fixed points of the form $(t,t)$ 
\begin{align}&t\ \textrm{satisfies}\ t^2-2t\cos(\pi\b)+c=0\notag\\
&\l_{1},\l_{2}\ \textrm{are\ the\  eigenvalues\  of}\  Dh^{\textrm{Hénon}}_{\b,c}(t,t)\notag\\
 &\l_{1}=e^{2\pi i (-\a+\b/2)},\qquad \l_{2}=e^{2\pi i (\a+\b/2)},\qquad \a\in\C.\label{alpha,beta}
 \end{align}
Assume  $\l_{1}\ne \l_{2}$ and define the translation
$$T_{-t}:\C^2\ni(x,y)\mapsto (x-t,y-t)\in\C^2$$
and the linear map
$$L:\C^2\ni (x,y)\mapsto \biggl(\frac{1}{\l_{1}-\l_{2}}(x-\l_{2}y),\frac{1}{\l_{1}-\l_{2}}(-x+\l_{1}y) \biggr)  $$
associated to the matrix $$L=\bm \l_{1}&\l_{2}\\ 1&1\em^{-1}=\frac{1}{\l_{1}-\l_{2}}\bm 1 & -\l_{2}\\ -1 & \l_{1}\em.$$
The {\it modified} Hénon  map
\be h^{\textrm{mod}}_{\a,\b}=(L\circ T_{-t})\circ h^{\textrm{Hénon}}_{\b,c}\circ (L\circ T_{-t})^{-1}\label{TtL-1}\ee
is still a quadratic polynomial automorphism of $\C^2$ of the form
\be h^{\textrm{mod}}_{\a,\b}:\C^2\ni \bm z\\ w\em\mapsto  \bm \l_{1}z\\\ \l_{2}w\em+\frac{q(\l_{1}z+\l_{2}w)}{\l_{1}-\l_2}\bm 1\\ -1\em\in \C^2\label{def:f}\ee
where 
$$q(z)=e^{i\pi \b}z^2.$$

\begin{rem}
The involution $\s^{\textrm{Hénon}}$ becomes  $\ti\s^{\textrm{Hénon}}:(x,y)\mapsto (\bar y+\bar t-t,\bar x+\bar t-t)$ after conjugation by the translation  $(x,y)\mapsto (x-t,y-t)$ and $\s^{\rm mod}=L\circ \ti\s^{\textrm{Hénon}}\circ L^{-1}$ after  conjugation by $L$  :
$$\s^{\rm mod}\bm x\\ y\em=\frac{1}{\l_{1}-\l_{2}}\bm (\bar x+\bar y)-\l_{2}(\bar \l_{1}\bar x+\bar \l_{2}\bar y)\\ -(\bar x+\bar y)+\l_{1}(\bar \l_{1}\bar x+\bar \l_{2}\bar y)\em+\bm\bar t-t\\ 0\em.$$

\begin{comm}
{\color{gray}We observe that
\begin{itemize}
\item if $|\l_{1}|=|\l_{2}|=1$ one has
$$\s^{}(x,y)=(\l_{1}^{-1}\bar x, \l_{2}^{-1}\bar y);$$
\item if $\l_{1}=\l e^{i\pi \b}$, $\l_{2}=\l^{-1}e^{i\pi \b}$ with $\l\in\R$ one has
$$\s(x,y)=(\bar \l_{2}\bar y,\bar \l_{1}\bar x).$$ 
\end{itemize}
}
\end{comm}
\end{rem}

\subsection{Exact symplectic setting}\label{sec:5.2} Let us introduce some notations.

\bigskip

With the preceding notations, the diffeomorphism 
$$ h^{\rm mod}_{\a,\b}:\bm z\\ w\em\mapsto  \bm \l_{1}z\\\ \l_{2}w\em+e^{i\pi\b}\frac{(\l_{1}z+\l_{2}w)^2}{\l_{1}-\l_2}\bm 1\\ -1\em$$
can be written 
$$h^{\rm mod}_{\a,\b}= \iota_{F}\circ \diag(\l_{1},\l_{2})$$ 
where 
$$\iota_{F}:\bm z\\w\em\mapsto\bm z\\ w\em+\frac{e^{i\pi \b}}{\l_{1}-\l_{2}} \bm z^2\\ -w^2\em=\bm z\\ w\em+i\mu_{\d} \bm z^2\\ -w^2\em$$
\be \biggl(i\mu_{\d}=\frac{e^{i\pi\b}}{\l_{1}-\l_{2}}\biggr)\label{recdefmu}\ee
is the canonical map\footnote{See Subsection \ref{sec:exsymp}.} (hence symplectic) associated to some $F\in \cO(\C^2,(0,0))$ of the form
$$F(z,w)=i\mu_{\d} \frac{(z+w)^3}{3}+O^4(z,w)$$
with
\be\mu_{\d}=\frac{1}{2\sin(2\pi\a)} .\label{def:mu}\ee

\bigskip
Note that,
\be
\left\{
\begin{aligned}
&\diag(\l_{1},\l_{2})^{-1}\circ \iota_{F}\circ \diag(\l_{1},\l_{2}) =\iota_{F'}\\
&\textrm{where}\quad  F'=(\diag(\l_{1},\l_{2})^{-1})_{*}F\quad   \textrm{is \ defined \ by}\\
&F'(z,w)=(\l_{1}\l_{2})^{-1}F(\l_{1}z,\l_{2}w).
\end{aligned}
\right.
\label{eq:4.11n}
\ee
Hence
$$h^{\rm mod}_{\a,\b}= \diag(\l_{1},\l_{2})\circ \iota_{F'}$$ 
where 
\begin{align*}F'(z,w)&=\frac{1}{\l_{1}\l_{2}}F(\l_{1}z,\l_{2}w)\\
&=\frac{i\mu_{\d}}{3\l_{1}\l_{2}}(\l_{1}z+\l_{2}w)^3+O^4(z,w)
\end{align*}
so,
\be F'(z,w)=i\frac{\mu_{\d}}{3\l_{1}\l_{2}} (\l_{1}^3z^3+3\l_{1}^2\l_{2}z^2w+3\l_{1}\l_{2}^2zw^2+\l_{2}^3w^3)+O^4(z,w).\label{defF'}\ee

\subsection{Ushiki's resonance}

As we shall soon see, an   important feature in S. Ushiki's example  described in Subsection \ref{sec:ushiki'sexample}  is the {\it resonance relation}
$$\begin{cases}&\a\approx \beta/2\\
&(4-1)\times \beta\approx 1
\end{cases}$$
(see (\ref{rescond})).
This suggests to construct examples with
\be \a=\a_{\d}=\frac{1}{6}+\d\malpha,\qquad \b=\b_{\d}=\frac{1}{3}+\d\mbeta\label{defalphabeta}\ee
where $\d$ is a small parameter and $(\malpha,\mbeta)$ is chosen carefully.

\begin{rem} When $\a=\b/2$ ($\d=0$), equations (\ref{eq:1.1}) and (\ref{eq:1.2}) show that the two fixed points of $h^{\rm mod}_{\b,c}$ coincide. When (\ref{defalphabeta}) is satisfied, they are at  distance  $O(\d)$ from each other.
\end{rem}

\begin{rem} 
Assume $\b\in\R$ and assume $\a=(1/6)+\d \malpha$ is such that 
(cf. (\ref{eq:1.2}))
$$c=-(\cos(2\pi\a))^2+2\cos(2\pi\a)\cos(\pi\b)\in\R.$$
Because one has
\begin{align*}t_{\a,\b}=\cos((\pi/3)+2\pi\d\malpha)&=(1/2)-\frac{\sqrt{3}}{2}2\pi \d \malpha+O(\d^2),
\end{align*}
the anti-holomorphic involution $\s^{\rm mod}_{\a,\b}$ for which $h^{\rm mod}_{\a,\b}$ is reversible takes the form
\be \s^{\rm mod}_{\a,\b}:(z,w)\mapsto (2\pi (\sqrt{3}/2)\d (\malpha-\bar \malpha,0)+(\bar z,j^2\bar w)+\d g(\bar z,\bar w)\label{eq:n4.13}\ee
for some $g\in \cO((0,0))$ such that $g(0,0)=0$.
\end{rem}

\subsection{Resonant Birkhoff normal forms}

We now perform a {\it Birkhoff normal form} on $h^{\rm mod}_{\a,\b}$ which means that we
 try to conjugate $h^{\rm mod}_{\a,\b}$ to some simpler diffeomorphism by using  a symplectic change of coordinates\footnote{Recall $\Phi_{Y}$ is the time-1 map of the Hamiltonian vector field $J\nabla Y$.} $(z,w)\mapsto \Phi_{Y}(z,w)$ 
 with $Y=O^3(z,w)$, $Y:(\C^2,(0,0))\to \C^2$:
$$
\begin{aligned}
&\Phi_{Y}^{-1}\circ h^{\rm mod}_{\a,\b}\circ \Phi_{Y}=\diag(\l_{1},\l_{2})\circ \iota_{ \ti F}\\
\end{aligned}
$$
where $\ti F$ has the simplest  possible form. 

A computation shows that
\begin{align*}\Phi_{Y}^{-1}\circ h^{\rm mod}_{\a,\b}\circ \Phi_{Y}&=\Phi_{Y}^{-1} \circ (\iota_{F}\circ \diag(\l_{1},\l_{2})) \circ \Phi_{Y}\\
&=\Phi_{Y}^{-1} \circ ( \diag(\l_{1},\l_{2}) \circ \iota_{F'}) \circ \Phi_{Y}\qquad(F'\ \textrm{as\ in}\ (\ref{defF'})) \\
&=\diag(\l_{1},\l_{2})\circ \iota_{\ti F}\circ \iota_{O^4(z,w)}\end{align*}

where 
$$\ti F=F'-(e^{-2\pi i \b}Y\circ \diag(\l_{1},\l_{2})-Y).$$
In particular, if one can solve
\begin{multline} e^{-2\pi i \b}Y_{}\circ \diag(\l_{1},\l_{2})(z,w)-Y_{}(z,w)=\\ i\frac{\mu_{\d}}{3\l_{1}\l_{2}} (\l_{1}^3z^3+3\l_{1}^2\l_{2}z^2w+3\l_{1}\l_{2}^2zw^2+\l_{2}^3w^3)\label{conjeq}\end{multline}
one gets
$$\Phi_{Y}^{-1}\circ h^{\rm mod}_{\a,\b}\circ \Phi_{Y}=\diag(\l_{1},\l_{2})\circ \iota_{O^4(z,w)}.$$

\medskip An equation of the form 
\be e^{-2\pi i \b}Y_{}\circ \diag(\l_{1},\l_{2})(z,w)-Y_{}(z,w)=G(z,w)\label{coho:eq:n}\ee
is called a {\it cohomological equation}. If
$$G(z,w)=\sum_{(k,l)\in \N^2}\hat G(k,l)z^kw^l$$
is given, finding $Y$
$$Y(z,w)=\sum_{(k,l)\in \N^2}\hat Y(k,l)z^kw^l$$
satisfying (\ref{coho:eq:n}) is equivalent 
to solving for all $(k,l)\in\N^2$ 
\be (e^{-2\pi i\b}\l_{1}^k\l_{2}^l-1)\hat Y_{}(k,l)=\hat G(k,l). \label{eq:coeffkl}\ee
Equation (\ref{eq:coeffkl})   has a solution $\hat Y (k,l)$ provided the following {\it non resonance condition} holds:
$$\biggl(\frac{k+l}{2}-1\biggr)\b+(l-k)\a\notin\Z.$$

If 
\be \a=\frac{1}{6}+\d\malpha,\qquad \b=\frac{1}{3}+\d\mbeta\label{1.4}\ee
one has 
\be \biggl(\frac{k+l}{2}-1\biggr)\b+(l-k)\a\approx  (l-1)/3;\label{rescond}\ee
hence, if $G=F'$ we see that the we can eliminate in $F'$ all the terms $z^kw^l$, $k+l=3$, except  the term $-i\mu_{\d} z^2w$. With  $Y_{1}=Y$ defined by (\ref{eq:coeffkl}) for $(k,l)\in\{(3,0),(1,2),(0,3)\}$ (the other coefficients are set to zero), we thus get
$$\Phi_{Y}^{-1}\circ h^{\rm mod}_{\a,\b}\circ \Phi_{Y}=\diag(\l_{1},\l_{2})\circ \iota_{b_{2,1} z^2w+O^4(z,w)}.$$
Observe that 
\begin{align*}b_{2,1}&=i\frac{\mu_{\d}}{3\l_{1}\l_{2}}\times (3\l_{1}^2\l_{2})\\
&=i\mu+O(\d)
\end{align*}
where 
\be \mu=\mu_{0}=\frac{1}{2\sin(2\pi /6)}=\frac{1}{\sqrt{3}}.\label{introdmu}\ee

\begin{rem}\label{Y-5.43}
We find with the notation $Y_{k,l}=\hat Y(k,l)$
\begin{align*}
&Y_{3,0}=\frac{i\mu/3}{1-j}+O(\d)\\
&Y_{1,2}=\frac{i\mu j}{j-1}+O(\d)\\
&Y_{0,3}=\frac{\mu j^2/3}{j^2-1}+O(\d).
\end{align*}
\end{rem}

One can push the normal form to the next order: by the same procedure we try to eliminate in $b_{2,1} z^2w+O^4(z,w)$ as many   $z^kw^l$, $k+l=4$ terms  as possible. There are now two more terms that cannot be eliminated, $(k,l)=(3,1)$ and $(k,l)=(0,4)$. We thus get for some $Y_{2}=O^4(z,w)$ homogeneous of degree 4,
\be \Phi_{Y_{2}}^{-1}\circ \Phi_{Y_{1}}^{-1}\circ h^{\rm mod}_{\a,\b}\circ \Phi_{Y_{1}}\circ \Phi_{Y_{2}}=\diag(\l_{1},\l_{2})\circ \iota_{F_{4}}\label{resBNFo2}\ee
with
$$F_{4}(z,w)=b_{2,1} z^2w+b_{3,1}z^3w+b_{0,4}w^4+O^5(z,w).$$
One can show that (see the Appendix \ref{appendix:compnu})
\be -4i b_{0,4}=\nu+O(\d)\quad\textrm{with}\quad \nu:=-(2/3)\frac{1}{\sqrt{3}}+O(\d).\label{introdnu}\ee
Because of (\ref{1.4}) and (\ref{alpha,beta}) one can write
\begin{align*}\diag(\l_{1},\l_{2})&=\diag(1,j)\circ \diag(e^{2\pi i \d(\mbeta/2-\malpha)},e^{2\pi i \d(\mbeta/2+\malpha)} )\\
&=\diag(1,j)\circ\diag(e^{i\pi\d\mbeta},e^{i\pi\d\mbeta})\circ \iota_{-2\pi \d\malpha zw}
\end{align*}
hence
\begin{align*}\Phi_{Y_{2}}^{-1}\circ \Phi_{Y_{1}}^{-1}\circ h^{\rm mod}_{\a,\b}\circ \Phi_{Y_{1}}\circ \Phi_{Y_{2}}&=\diag(1,j)\circ\diag(e^{i\pi\d\mbeta},e^{i\pi\d\mbeta})\circ \iota_{-2\pi \d\malpha zw}\circ   \iota_{F_{4}}\\
&=\diag(1,j)\circ\diag(e^{i\pi\d\mbeta},e^{i\pi\d\mbeta})\circ \iota_{\ti F_{4}}
\end{align*}
where
\begin{align*}&\ti F_{4}(z,w)=-2\pi i \malpha \d zw+\ti b_{2,1}z^2w+\ti b_{3,1}z^3w+\ti b_{0,4}w^4+O^5(z,w)\\
&\ti b_{2,1}=b_{2,1}+O(\d),\quad \ti b_{3,1}=b_{3,1}+O(\d),\quad \ti b_{0,4}=b_{0,4}+O(\d).
\end{align*}
By the same token, we can also kill all the terms $z^kw^l$, $k+l=5$ and $k+l=6$ except $z^4w$ and $z^5w$ and all the terms $z^kw^l$, $k+l=7$ except $z^6w$ and $w^7$. 

\medskip This procedure can be done to any order.  
We have thus proved
\begin{prop}[Resonant BNF]\label{prop:BNF}Let $m\in\N$, $m\geq 2$. There exists a polydisk $\cV_{BNF}:=\bD(0,\rho)\times \bD(0,\rho)$ such that for any $(\a,\b)$ of the form (\ref{1.4}), there exist  $Y,F_{BNF}\in \cO(\cV_{BNF})$ such that 
$$\iota_{Y_{}}^{-1}\circ h^{\rm mod}_{\a,\b}\circ \iota_{Y_{}}=\diag(\l_{1},\l_{2})\circ \iota_{F_{BNF}}$$
where $F_{BNF}$ is of the form
\begin{multline*}F_{BNF}(z,w)=-2\pi i \malpha \d zw+b_{2,1}^{BNF}z^2w+b_{0,4}^{BNF}w^4\\ +\sum_{k=3}^{3m}b^{BNF}_{k,1}z^kw+\sum_{n=2}^{m}b^{BNF}_{0,3n+1}w^{3 n+1}+O^{3m+2}(z,w)
\end{multline*}

\begin{align*}&b_{2,1}^{BNF}=(i\mu+O(\d)),\quad b_{0,4}^{BNF}=(1/4)(i\nu+O(\d))\\
&\forall k\in \N\cap [3,2m], \ b_{k,1}^{BNF}=O_{\d}(1),\qquad \forall n\in\N\cap [2,m],\ b_{0,3n+1}^{BNF}=O_{\d}(1)
\end{align*}
$\mu,\nu$ being  defined by (\ref{introdmu}).
\end{prop}
\begin{rem}\label{rem:4.4}The anti-holomorphic involution $\s^{\rm mod}_{\a,\b}$ (cf. \ref{eq:n4.13})) becomes after this change of coordinates
$$\s^{BNF}_{\a,\b}:(z,w)\mapsto (2\pi \d (\malpha-\bar \malpha,0)+(\bar z,j^2\bar w)+\d^{1/3}g_{\d}(\bar z,\bar w)$$
where $g\in \cO((0,0))$, $g(0,0)=0$, is some holomorphic function. 
\end{rem}

\medskip
\section{Vector field approximation}\label{sec:vfapprox}
\subsection{Dilation}
We now perform a dilation (zoom at the origin) that has the peculiarity of not being symmetric in the $(z,w)$-variables.

If 
$$\L_{\d}:(z,w)\mapsto (\d^{-1}z,\d^{-2/3}w)$$ one gets
\begin{align}\L_{\d}\circ\iota_{Y_{}}^{-1}\circ h^{\rm mod}_{\a,\b}\circ \iota_{Y_{}}\circ \L_{\d}^{-1}
&=\diag(1,j)\circ \diag(e^{i\pi\d\mbeta},e^{i\pi\d\mbeta})\circ\L_{\d}\circ  \iota_{\mathring{F}}\circ \L_{\d}^{-1}\notag\\
&=\diag(1,j)\circ\diag(e^{i\pi\d\mbeta},e^{i\pi\d\mbeta})\circ \iota_{\mathring{F}_{BNF}}\label{eq:5.21}
\end{align}
with 
\be \mathring{F}_{BNF}(z,w)=\d^{-5/3}F_{BNF}(\d z,\d^{2/3}w)\label{eq:5.22}\ee 
hence
\begin{multline*}F_{BNF}(z,w)=-2\pi i \malpha \d zw+\d^{-5/3}\biggl(b_{2,1}^{BNF}\d^{8/3}z^2w+b_{0,4}^{BNF}\d^{8/3}w^4\\ +\sum_{k=3}^{3m}b^{BNF}_{k,1}\d^{k+(2/3)}z^kw+\sum_{n=2}^{m}b^{BNF}_{0,3n+1}\d^{2n+(2/3)}w^{3 n+1}+\d^{2m+(4/3)}O^{3m+2}(z,w)\biggr)
\end{multline*}
or 
\begin{multline*}F_{BNF}(z,w)=-2\pi i \malpha \d zw+\d^{}\biggl(b_{2,1}^{BNF}\z^2w+b_{0,4}^{BNF}w^4\\ +\sum_{k=3}^{3m}b^{BNF}_{k,1}\d^{k-1}z^kw+\sum_{n=2}^{m}b^{BNF}_{0,3n+1}\d^{2n-1}w^{3 n+1}\biggr)+\d^{2m-(1/3)}O^{3m+2}(z,w).
\end{multline*}
Note that $\mathring{F}_{BNF}$ is defined on a polydisk $\bD(0,\d^{-1}\rho)\times \bD(0,\d^{-2/3}\rho)$ and bounded there.
\subsection{Approximation by a vector field}
Let 
\begin{multline*}A_{\d}(z,w)=-2\pi i \malpha  zw+\biggl(b_{2,1}^{BNF}z^2w+b_{0,4}^{BNF}w^4\\ +\sum_{k=3}^{3m}b^{BNF}_{k,1}\d^{k-1}z^kw+\sum_{n=2}^{m}b^{BNF}_{0,3n+1}\d^{2n-1}w^{3 n+1}\biggr)
\end{multline*}
so that
\be \mathring{F}_{BNF}=\d A_{\d}+O(\d^{2m-(1/3)})\label{eq:m5.20}\ee
and (cf. (\ref{eq:5.21}), (\ref{eq:5.22}))
\begin{multline} \L_{\d}\circ\iota_{Y_{}}^{-1}\circ h^{\rm mod}_{\a,\b}\circ \iota_{Y_{}}\circ \L_{\d}^{-1}
=\\ \diag(1,j)\circ\diag(e^{i\pi\d\mbeta},e^{i\pi\d\mbeta})\circ \iota_{\d A_{\d}}\circ \iota_{O(\d^{2m-(1/3)})}.
\end{multline}

One can check that 
$$\diag(1,j)_{*}A_{\d}=A_{\d}
$$
(cf.  (\ref{notationdiag}))
hence (cf. Lemma \ref{lemma:prel4.3})
\begin{align*}&\iota_{\d A_{\d}}\circ \diag(1,j)=\diag(1,j)\circ \iota_{\d A_{\d}}\\
&\Phi_{\d A_{\d}}\circ \diag(1,j)=\diag(1,j)\circ \Phi_{\d A_{\d}}.
\end{align*}

\begin{prop}Let $n\in\N$ and  $A\in \cO(U)$ with $A=O(\d)$   small enough be such that $\diag(1,j)_{*}A=A$. For any $\mbeta\in\C$,  there exists a holomorphic vector field $X_{n}:\cU_{\nu}(U)\to \C^2$  ($\nu=\frak{d}(A)$) such that  $\diag(1,j)_{*}X_{n}=X_{n}$, ${\rm div} X_{n}=2\pi i\mbeta$
and
$$\diag(1,j)\circ \diag(e^{i\pi\d \mbeta},e^{i\pi \d\mbeta})\circ \iota_{A}=\diag(1,j)\circ \phi^1_{X_{n}}\circ \iota_{O(\d^n)}.$$
\end{prop}
\begin{proof}We use Corollaries \ref{cor:avf} and \ref{cor:BCH}.
One can write 
\begin{align*}&\iota_{A}=\phi^1_{J\nabla A_{n}}\circ \iota_{O(\d^n)}\\
&\diag(e^{i\d \mbeta},e^{i\pi \d\mbeta})=\phi^1_{i\pi \d\mbeta (z\pa_{z} +w\pa_{w})}
\end{align*}
hence, using Corollary \ref{cor:BCH},
\begin{align*}\diag(e^{i\pi \d \mbeta},e^{i\pi \d\mbeta})\circ \iota_{A}&=\phi^1_{i\pi \d\mbeta (z\pa_{z} +w\pa_{w})}\circ \phi^1_{J\nabla A_{n}}\circ \iota_{O(\d^n)}\\
&=\phi^1_{P_{n}(i\pi \d\mbeta (z\pa_{z} +w\pa_{w}),J\nabla A_{n})}\circ\iota_{R_{n}(i\pi \d\mbeta (z\pa_{z} +w\pa_{w}),J\nabla A_{n})} \circ \iota_{O(\d^n)}\\
&=\phi^1_{P_{n}(i\pi \d\mbeta (z\pa_{z} +w\pa_{w}),J\nabla A_{n})}\circ\iota_{O(\d^n)}.
\end{align*}
The vector field 
\begin{align*}X_{n}&=P_{n}(i\pi \d\mbeta (z\pa_{z} +w\pa_{w}),J\nabla A_{n})\\
&=i\pi \d\mbeta (z\pa_{z} +w\pa_{w})+J\nabla A_{n}+\frak{O}_{2}(i\pi \d\mbeta (z\pa_{z} +w\pa_{w}, J\nabla A_{n})
\end{align*}
commutes with $\diag(1,j)$ since $i\pi \d\mbeta (z\pa_{z} +w\pa_{w})$ and $J\nabla A_{n}$ do.

\end{proof}

Applying the previous Proposition to $A_{\d}$ we thus get
\begin{cor}There exists a holomorphic vector field $X^{BNF}_{\d}$ defined on $\bD(0,M/4)^2$ such 

\be
\left\{ \begin{aligned}&\diag(1,j)_{*}X^{BNF}_{\d}=X^{BNF}_{\d}\\
&{\rm div}X^{BNF}_{\d}=2i\pi\mbeta
\end{aligned}
\right.
\ee and 
 on $\bD(0,M/4)^2$ one has
$$X^{BNF}_{\d}(z,w)=i\pi\mbeta\bm z\\ w\em+J\nabla\biggl( -2\pi i \malpha  zw+i\mu z^2w+i(\nu/4) w^4+O_{}(\d)\biggr)$$
and
$$\L_{\d}\circ\biggl( \diag(\l_{1},\l_{2})\circ \iota_{F_{BNF}}\biggr)\circ \L_{\d}^{-1}=\diag(1,j)\circ \Phi_{X^{BNF}_{\d}}\circ \iota_{O(\d^{2m-1/3})}.$$
\end{cor}

More explicitly
$$X_{\d}^{BNF}(z,w)=  i\bm 2\pi \mbeta_{1}z+\mu z^2+\nu w^3\\ 2\pi \mbeta_{2}w-2\mu zw\em+O(\d^{})
$$
where
$$\mbeta_{1}=\mbeta/2-\malpha,\qquad \mbeta_{2}=\mbeta/2+\malpha.$$

\begin{notation} Let  $\tau$ be defined by 
\be \malpha=(\tau-(1/2))\mbeta\ee
so that
\begin{align*}
&\mbeta_{1}= (1-\tau)\times \mbeta\\
&\mbeta_{2}=\tau\times \mbeta.
\end{align*}
We set
$$\tau'=(\tau,\mbeta).$$
\end{notation}

Note that after a further change of variables\footnote{$\mbeta^{2/3}$ is a complex number the square of which is $\mbeta^2$.}
$$\L_{(2\pi\mbeta)^{}}:(z,w)\mapsto  ((2\pi \mbeta)^{-1} z , (2\pi\mbeta)^{-2/3}w)$$
the vector field  $X_{\d}^{BNF}$ is transformed into $\mbeta X_{\tau',\mu,\nu}$ where 
\be X_{\tau',\mu,\nu}(z,w)=2\pi i \bm (1-\tau)z+\mu z^2+\nu w^3\\ \tau w-2\mu zw\em+O(\d^{}).\label{vfmod}
\ee
The coefficient $\mu$ is  still given by (\ref{introdmu}) and $\nu$ by (\ref{introdnu}) (obtained at the second step of the BNF).
Numerical values of $\mu$ and $\nu $ are
\be \mu=\frac{1}{\sqrt{3}}\approx 0.577,\qquad \nu=-(2/3)\frac{1}{\sqrt{3}} \approx -0.3849.\label{munu}\ee
Numerical experiments confirm the fact that this vector field is a good model for the dynamics: when one varies $\d$ (remaining small) the phase portrait of the diffeomorphism $\L_{\d}\circ Z^{-1}\circ f\circ Z\circ \L_{\d}^{-1}$ is very similar to the one of $\diag(1,j)\circ \phi^1_{X}$. In some sense, this vector field provides  a universal model for the dynamics of the Hénon map (in the regime we are considering). Compare Figures \ref{fig:2} and \ref{fig:3}.

If one makes a further (linear) change of coordinates
$$z\lra \frac{\sqrt{3}}{2} z=(\mu^{-1}/2) z ,\qquad w\lra \biggl(\frac{\sqrt{3}}{2}\biggr)^{2/3} w$$
the vector field $X_{\tau,\mu,\nu}$ becomes
\be X_{\d,\tau'}: (z,w)\mapsto 2\pi i \bm (1-\tau)z+z^2/2-w^3/3\\ \tau w-zw\em+O(\d^{}).
\ee

\bigskip 

We can deduce from the preceding discussion the main result of this Section.

Recall 
\begin{align*}&h^{\textrm{mod}}_{\a,\b}:\C^2\ni \bm z\\ w\em\mapsto  \bm \l_{1}z\\\ \l_{2}w\em+\frac{q(\l_{1}z+\l_{2}w)}{\l_{1}-\l_2}\bm 1\\ -1\em\in \C^2\label{def:f}\\
&q(z)=e^{i\pi \b}z^2
\end{align*}
and $\mbeta$, $\tau$  and $\d$ are defined by
\begin{align*}
&\a=\frac{1}{6}+\d\malpha,\qquad \b=\frac{1}{3}+\d\mbeta\\
&\malpha=(\tau-(1/2))\mbeta.
\end{align*}
With these notations we set 
$$c_{\d}(\tau,\mbeta)=-(\cos(2\pi\a))^2+2\cos(2\pi\a)\cos(\pi\b).$$

\begin{theo}[Approximation by a vector field]\label{theo:approxbyvf}Let $M>0$ and $m\in\N^*$. Assume $\mbeta\in \C\setminus\{0\}$ and $\tau\in \C$ and recall our notation $\tau'=(\tau,\mbeta)$. There exists $\d_{0}>0$ such that for any $\d\in (-\d_{0},\d_{0})$, there exists a holomorphic  diffeomorphism $Z_{\d,\tau'}$ of the form 
$$Z_{\d,\tau'}=\diag((2\pi   (\sqrt{3}/2) \mbeta\d)^{-1},(2\pi (\sqrt{3}/2) \mbeta\d)^{-2/3})\circ \iota_{G_{\d,\tau'}}$$
with $G_{\d,\tau'}\in \cO(\bD(0,M)^2)$
and a holomorphic vector field $X^{}_{\d,\tau'}:\bD(0,M)^2\to \C^2$ with divergence equal to $2\pi i=2\pi \sqrt{-1}$ that commutes with $\diag(1,e^{2\pi i/3})$: 
\be \diag(1,e^{2\pi i/3})_{*}X^{}_{\d,\tau'}=X^{}_{\d,\tau'}\label{eq:theo:approxbyvf}\ee
and which is of the form
\be 
\begin{aligned}
&X^{}_{\d,\tau'}=2\pi i \bm (1-\tau)z+z^2/2-w^3/3\\ \tau w- zw\em+O(\d)\\
\end{aligned}
\label{eq:defin:Xdeltatauprime}
\ee
such that on $\bD(0,M)^2$ one has 
\be Z_{\d,\tau'}\circ h^{\rm mod}_{\a,\b}\circ Z_{\d,\tau'}^{-1}=\diag(1,e^{2\pi i/3})\circ \phi^1_{\d \mbeta X^{}_{\d,\tau'}}\circ \iota_{O(\d^{2m-(1/3)})}.\label{eq:n1.3}\ee

\end{theo}

We shall set
\be h^{\rm bnf}_{\d,\tau'}=Z_{\d,\tau'}\circ h^{\rm mod}_{\a,\b}\circ Z_{\d,\tau'}^{-1}=\diag(1,e^{2\pi i/3})\circ \phi^1_{\d \mbeta X^{}_{\d,\tau'}}\circ \iota_{O(\d^{2m-(1/3)})}.\label{themaphetc}\ee

Let us mention the following consequence of Remark \ref{rem:4.4}:
\begin{prop}\label{prop:involution}When $\mbeta\in\R$ and $c_{\d}(\tau,\mbeta)\in\R$, the diffeomorphism $h^{\rm bnf}_{\delta,\tau'}$ is reversible with respect to an anti-holomorphic involution
$$\s_{\d,\tau'}:(z,w)\mapsto (\tau-\bar\tau,0)+(\bar z,j^2\bar w)+O(\d^{1/3}).$$
Besides, the vector field 
\be X_{\tau}:=X_{\d=0,\tau}=2\pi i \bm (1-\tau)z+z^2/2-w^3/3\\ \tau w- zw\em\label{def:Xtau}\ee is reversible with respect to the anti-holomorphic involution
$$(z,w)\mapsto (\tau-\bar\tau,0)+(\bar z,j^2\bar w).$$
\end{prop}

We can perform a last change of variables on $X_{\d,\tau'}$: replacing $z$ by $z-\tau$ yields the vector field
\be \hat X_{\d,\hat \tau,\mbeta}(z,w)=2\pi i \bm \hat \tau+z+(1/2)z^2-(1/3)w^3\\ - zw\em\label{vfmod2}+O(\d)
\ee
with 
$$\hat \tau=\tau-(1/2)\tau^2.$$
One can check directly that 
\begin{prop}\label{prop:Xisreversible}For $\hat \tau\in \R$, the vector field $\hat X_{\hat \tau,0}$ is  reversible w.r.t. the anti-holomoprhic involution
$$\s:(z,w)\mapsto (\bar z, j^2\bar  w).$$
\end{prop}
\begin{rem}\label{rem:rev6.1}Because $\diag(1,j)_{*}\hat X_{\hat\tau}=\hat X_{\hat\tau}$ the vector field $\hat X_{\hat\tau}$ is also reversible w.r.t. to the involution $\hat \s=\diag(1,j)\circ\s\circ\diag(1,j)^{-1}:(z,w)\mapsto (\bar z,\bar w) $  (for $\hat \tau\in \R$).
\end{rem}

\begin{notation}
We shall write
$$X_{\tau}=X_{0,\tau},\qquad \hat X_{\hat \tau}=\hat X_{0,\hat \tau}.$$
\end{notation}

\bigskip 
\section{The invariant annulus theorem for vector fields}\label{sec:invannulusthm}
\subsection{Invariant annulus and exotic periodic orbits}
Let $X:U\to \C^2$ be a nonconstant holomorphic vector field defined on an open set $U\subset\C^2$ and assume it has a $T$-periodic orbit 
$(\phi_{X}^{t}(\zeta))_{t\in\R}$ ($\zeta\in U$, $T>0$)  inside $U$. Then, there exists $s>0$ such that the flow $\phi^\th_{X}(\zeta)$ is defined for any $\th\in \R_{s}=\R+i(-s,s)$
and for any $y\in (-s,s)$, the orbit $(\phi_{X}^{t+iy}(\zeta))_{t\in\R}$ is also $T$-periodic and included in $U$.
 The map $\R_{s}\ni\th\mapsto \phi_{X}^\th(\zeta)$ is $T\Z$-periodic hence defines a holomorphic injective\footnote{Otherwise,  the orbit $(\phi_{X}^\th(\zeta))_{\th}$ would admit two periods and the map $\psi$ would be defined on a 1-dimensional complex torus and would be constant. } map $\psi:\T_{s}=\R/\Z+i(-s,s)\ni \th\mapsto \phi_{X}^{T\th }(\zeta)\to U$. Let $\cA_{s}$ be the 1-dimensional complex submanifold of $U\subset \C^2$ defined by 
$$\cA^{}_{s}=\psi(\T_{s}).$$ 
The diffeomorphism 
$\psi$ sends the constant vector field $\pa_{\th}$ defined on $\T_{s}$ to the restriction of the vector field $(1/T)X\mid \cA_{s}$:
$$\psi_{*}\pa_{\th}=(1/T)X.$$

Let $(s_{-},s_{+})\subset\R$ be the maximal interval for which the map $\psi:\T_{(s_{-},s_{+})}:=\R/\Z+i(s_{-},s_{+})\ni \th\mapsto \phi_{X}^{\th/T}(\zeta)\in U\subset \C^2$ is defined. We then define the annulus
\be \cA_{\rm max}=\psi(\T_{(s_{-},s_{+})}).\label{eq:maxmodannulus}\ee
\begin{prop}\label{exoticperorb}If the closure of $\cA_{\rm max}$ contains a fixed point of $X$, then one of the two boundaries $s_{-}$ or $s_{+}$ is infinite and there exists a 1-disk $D$ containing this fixed point, invariant by the flow of $X$ and  on which the dynamics of the flow of $X$ is a $T$-periodic rotation flow.
\end{prop}
\begin{proof}
Let $p_{*}$ be this fixed point. For any $\e>0$, there exists $\d>0$ such that for any $\zeta\in \bD_{\C^2}(p_{*},\d)$ and any $t\in [0,T]$ one has $\phi_{X}^t(\zeta)\in \bD_{\C^2}(p_{*},\e)$. By assumption $p_{*}$ is in the closure of $\cA_{\rm max}$;  there hence  exists $y_{\e}\in (s_{-},s_{+})$ such that $\psi(\T+iy_{\e})\in \bD_{\C^2}(p_{*},\e)$. As $\e$ goes to zero, $y_{\e}$ must accumulate $s_{-}$ or $s_{+}$. If both $s_{-}$ and $s_{+}$ are finite, the holomorphic function $\psi$ extends as a continuous function of $\T+i(s_{-},s_{+}]$ or $\T+i[s_{-},s_{+})$ that must be equal to $p_{*}$ on $\T+is_{-}$ or $\T+is_{+}$. It must thus be constant, which is a contradiction.

\begin{figure}[h]
\includegraphics[scale=0.5]{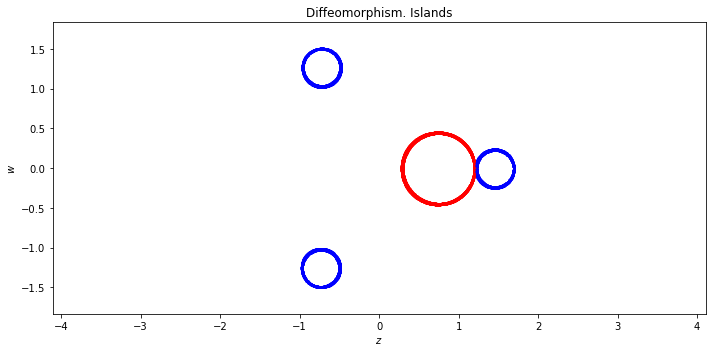}
\caption{``Elliptic Islands''. $\mbeta=0.311841$, $\malpha=-0.0535$.  $\delta=0.01$; initial condition 
$(z_{*},w_{*})$, $z_{*}=1.2$, $w_{*}=1.22$. The red (resp. blue) curve is the projection of the orbit on the $z$-coordinate (resp. $w$-coordinate). 10000 iterations.  }\label{fig:4}
\end{figure}

\begin{figure}[h]
\includegraphics[scale=0.5]{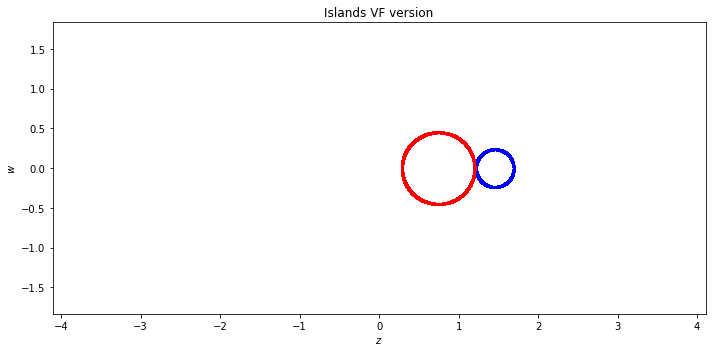}
\caption{Vector field version with the same parameters $\mbeta=0.311841$, $\malpha=-0.0535$ and the same initial condition 
$(z_{*},w_{*})$, $z_{*}=1.2$, $w_{*}=1.22$. The red (resp. blue) curve is the projection of the orbit on the $z$-coordinate (resp. $w$-coordinate). (Scaling 1).}  \label{fig:5}
\end{figure}

\medskip
Now, if  for example  $s_{+}$ is infinite, the previous discussion shows that the holomorphic function $\th\mapsto \psi(\th)-p_{*}$ vanishes when $\Im\th\to\infty$. Setting $r=e^{2\pi i\th}$, the function $\ti\psi(r)=\psi(\th)$ defines a holomorphic function on some disk $\bD(0,\rho)$. In these coordinates, the flow of $X$ becomes $r\mapsto e^{2\pi i(t/T)}r$. 
\end{proof}

We say that a periodic orbit is {\it exotic} if  its maximal invariant annulus $ \cA_{\rm max}$  has finite modulus or equivalently if its  closure 
 does not contain any fixed point of $X$.

\subsection{The periodic orbit theorem}

\medskip In what follows $X_{\d,\tau'}$ is the vector field (\ref{eq:defin:Xdeltatauprime}) defined in Theorem \ref{theo:approxbyvf}
$$X^{}_{\d,\tau'}=2\pi i \bm (1-\tau)z+z^2/2-w^3/3\\ \tau w- zw\em+O(\d)$$
and $X_{\tau}$  (cf. (\ref{def:Xtau}))
$$X_{\tau}:=X_{\d=0,\tau}=2\pi i \bm (1-\tau)z+z^2/2-w^3/3\\ \tau w- zw\em.$$
We recall the notation $\tau'=(\tau,\mbeta)$.

\medskip 
Let $\nu>0$ and define
$$\frak{T}_{\nu}=\{\tau \in \C\mid \tau-(1/2)\tau^2\in \bD(1/2,\nu)\}.$$

The following result shall be proved in Section \ref{sec:perorbthm}, Theorem \ref{theo:expero}.
\begin{theo}[Exotic periodic orbit Theorem]\label{theo:periodicorbit}The vector field $X_{\tau=1}$ (cf. (\ref{def:Xtau})) admits an exotic $T_{*}=1/g_{*}$-periodic orbit  $(\phi^t_{X_{1}}(p_{*}) )_{t\in\R}$ with $g_{*}\in \R$ equal to $-0.834\pm 10^{-3}$. This orbit is invariant by $\diag(1,j)$ and more precisely for any $t\in\R$,
\be \diag(1,j)(\phi^t_{X_{1}}(p_{*}))=\phi^{t+T_{*}/3}_{X_{1}}(p_{*}).\label{eq:diag1jinvperorb}
\ee
Furthermore, 
it is reverisble w.r.t.  the anti-holomorphic involution $\s:(z,w)\mapsto (\bar z, j^2\bar w)$ and there exists $t_{*}\in \R$ such that 
$$\s(p_{*})=\phi^{t_{*}}_{X_{1}}(p_{*}).$$

\end{theo}

\subsection{Perturbations of $X_{\tau=1}$} 
Let $M$ be large enough so that 
$$\{\phi^t_{X_{1}}(p_{*})\mid t\in\R\}\subset\bD(0,M/2)\times \bD(0,M/2).$$
For $\mbeta\in \bD(0,1)$ we set
$$\tau'_{1,\mbeta}=(1,\mbeta).$$
For  $\nu_{0}>0$,  $\cV:=\bD_{\C^2}(\tau'_{1,\mbeta},\nu_{0})\times \bD_{}(0,\nu_{0})\times  \bD(0,M)\times \bD(0,M)\ni(\tau',\d,z,w)\mapsto X_{\d,\tau'}(z,w)\in \C^2$ is a holomorphic map such that $X_{0,1}=X_{1}$ and ${\rm div}\ X_{\d,\tau'}=2\pi i$.

We denote by $(\phi_{X_{1}}^t(p_{*}))_{t\in\R}$ the periodic orbit of Theorem \ref{theo:periodicorbit} and introduce a vector $e\in \C^2$ such that 
$$\C^2=\C e\oplus \C X_{1}(p_{*}).$$

\begin{theo}[Periodic orbit theorem] \label{theo:perobrit}There exists $\nu_{1},\nu'_{1}>0$  and holomorphic functions 
\begin{align*} 
&\bD_{\C^2}(\tau'_{1,\mbeta},\nu_{1})\times \bD(0,\nu_{1})\ni (\tau',\d)\mapsto  g_{\d}(\tau')\in \C\\
&\bD_{\C^2}(\tau'_{1,\mbeta},\nu_{1})\times \bD(0,\nu_{1})\ni (\tau',\d)\mapsto  \zeta_{\d}(\tau')\in\C
\end{align*}
for which the following holds. 
\begin{enumerate}
\item For any $\mbeta$, one has $g_{*}=g_{\d=0}(1,\mbeta)$, $0=\zeta_{\d=0}(1,\mbeta)$ ($g_{*}$ is defined in Theorem \ref{theo:periodicorbit}).
\item For any 
 $(\tau',\d)\in \bD_{\C^2}(\tau'_{1,\mbeta},\nu_{1})\times \bD_{}(0,\d_{1})$, the couple $( g_{\d}(\tau'),\zeta_{\d}(\tau'))$ is the unique $(g,\zeta)\in\bD_{\C^2}(g_{*},\nu_{1}')\times \bD(0,\nu_{1}')$ satisfying 
 \be \phi^{1 /g}_{X_{\d,\tau'}}(p_{*}+\zeta  e)=p_{*}+\zeta e.\label{7.41}\ee
 \end{enumerate}
\end{theo}
\begin{proof}Let $T_{*}=1/g_{*}$ and 
 consider the map ($\mbeta$ is fixed)
$$\Xi_{}:(\C^4,(1,0,0,0))\ni (\tau,\d,t,\zeta)\mapsto \phi^{T_{*}+t}_{X_{\d,\tau,\mbeta}}(p_{*}+\zeta e)\in \C^2.$$
The map $\Xi_{}$ is holomorphic on some neighborhood of $(1,0,0,0)$ and, by the linearization theorem for ODE's,  its  derivative  $(\Delta\tau,\Delta\d,\Delta t,\Delta \zeta)\mapsto D\Xi_{}(1,0,0,0)\cdot(\Delta\tau,\Delta\d,\Delta t,\Delta r)$ is equal to 
\begin{multline*}(\Delta t)X_{0,1}(p_{*})+(\Delta \zeta)R_{}(T_{*},0)\cdot e\\ +\int_{0}^{T_{*}}R_{}(T_{*},s)\cdot (\pa_{\d,\tau} X_{\d,\tau,\mbeta})_{\mid \d=0, \tau=1}(\phi^s_{X_{1}}(p_{*}))\cdot(\Delta\tau,\Delta\d)ds\end{multline*}
where $R_{}(t,s)$ is the resolvent of the linearized equation 
$$\frac{d}{dt}Y(t)=D X_{1}(\phi^t_{X_{1}}(p_{*}))\cdot Y(t).$$
Since ${\rm div} X_{\d,\tau,\mbeta}=2\pi i$ and $(\phi^t_{X_{1}}(p_{*}))_{t\in\R}$ is a $T_{*}$-periodic orbit of the autonomous vector field $X_{\tau=1}$, the endomorphism $R_{}(T_{*},0)$ written in  the base $(X_{1}(p_{*}),e)$  takes the form
$$\ti R_{T_{*}}:=\bm 1&a\\ 0&e^{2\pi iT_{*}} \em \qquad (a\in\C).$$

In this base  the linear map $$(\Delta t,\Delta \zeta)\mapsto D\Xi_{}(1,0,0,0)\cdot(0,0,\Delta t, \Delta \zeta)-(\Delta t, \Delta \zeta)$$ reads
 $$\bm\Delta t\\ \Delta\zeta\em \mapsto \Delta t\bm 1\\ 0\em+ \Delta\zeta\bm a\\ e^{2\pi iT_{*}}-1 \em =\bm 1&a\\ 0& e^{2\pi iT_{*}}-1 \em\bm\Delta t\\ \Delta \zeta\em$$
 which is invertible because $T_{*}\notin  \Z$;  by the implicit function theorem, for $(\tau,\d)$ in a neighborhood of $(1,0)$, there exist $\zeta_{\d,\tau'}\in \bD(0,M)^2$ and $T_{\d,\tau'}=T_{*}+O(\d)$ for which  the fixed point equation
\be\phi_{X_{\d,\tau'}}^{T_{\d,\tau'}}(p_{*}+\zeta_{\d,\tau'}e)=p_{*}+\zeta_{\d,\tau'}e\label{eq:n6.31}\ee 
 is satisfied. Moreover, the couple $(\zeta_{\d,\tau'},T_{\d,\tau'})$ is the unique solution of this equation  in a neighborhood $\bD(0,\nu')\times \bD(T_{*},\nu')$  and the map $(\tau,\d)\mapsto (\zeta_{\d,\tau,\mbeta},T_{\d,\tau,\mbeta})$ is holomorphic in some open neighborhood of $(\tau,\d)=(0,0)$. To get the conclusion we set $g_{\d,\tau'}=1/T_{\d,\tau'}$.

\end{proof}
\begin{rem}\label{rem:7.1}The proof also shows that 
if $(T,\zeta)\in\bD(T_{*},\nu')\times \bD(0,\nu')$ satisfies
 \be \phi^{T}_{X_{\d,\tau'}}(p_{*}+\zeta  e)-(p_{*}+\zeta e)=\eta\label{7.41n}\ee
 then
 $\max(|T-T_{\d,\tau'}|,|\zeta-\zeta_{\d,\tau'}|)=O(\eta)$.
\end{rem}

We shall prove in Section  \ref{sec:perorbthm}, Theorem \ref{theo:expero},  the following result: 
\begin{theo}[The frequency map $\hat g$]\label{theogreal} There exists a holomorphic function $\hat g:\bD(1/2,\nu'')\to \C$ such that 
\begin{itemize}
\item For any $\hat\tau\in \bD(1/2,\nu'')\cap \R$ ($\nu''\in (0,\nu')$) one has $\hat g(\hat \tau)\in\R$.
\item The derivative of the function $\hat g$ at the point $1/2$ is a negative number.
\item For any $\tau$ such that $\hat\tau:=\tau-(1/2)\tau^2\in \bD(0,\nu'')$ one has
$$g_{0}(\tau)=\hat g(\tau-(1/2)\tau^2).$$
In particular $\pa g_{0}(\tau)=(1-\tau)\pa\hat g(\tau-(1/2)\tau^2)$.
\item $g_{0}(1)=g_{*}=-0.834\pm 10^{-3}$.
\end{itemize}
\end{theo}

\subsection{The invariant annulus theorem}
We now give a more geometric interpretation of Theorem \ref{theo:perobrit}.  For $(\tau,\d)\in \bD(1,\nu)\times \bD(0,\nu)$ let $\ph_{\d,\tau'}\in\R$, $\ph_{\d,\tau'}=O(|\tau-1|+|\d|)$, be such that 
$$e^{-i\ph_{\d,\tau'}}T_{\d,\tau'}:=1/(e^{i\ph_{\d,\tau'}}g_{\d,\tau'})\in \R^*.$$

Equation (\ref{7.41}) shows that
$$\phi^{e^{-i\ph_{\d,\tau'}}T_{\d,\tau'}}_{e^{i\ph_{\d,\tau'}}X_{\d,\tau'}}(p_{*}+\zeta_{\d,\tau'}  e)=p_{*}+\zeta_{\d,\tau'} e$$
hence
$(\phi^t_{e^{i\ph_{\d,\tau'}}X_{\d,\tau'}}(p_{*} +\zeta_{\d,\tau'}e))_{t\in\R}$ is a $e^{-i\ph_{\d,\tau'}}T_{\d,\tau'}$-periodic orbit of the vector field $e^{i\ph_{\d,\tau'}}X_{\d,\tau'}$. There thus exists $s_{\d,\tau'}>0$ and a holomorphic injective mapping $$\psi_{\d,\tau'}:\T_{s_{\d,\tau'}}\ni\th\mapsto \phi_{e^{i\ph_{\d,\tau'}}X_{\d,\tau'}}^{(e^{-i\ph_{\d,\tau'}}T_{\d,\tau'})\th}(p_{*}+\zeta_{\d,\tau'}e)=\phi_{X_{\d,\tau'}}^{T_{\tau,\d}\th}(p_{*}+\zeta_{\tau,\d}e)\in \bD_{\C^2}(0,M),$$
which depends holomorphically on $(\tau,\d)$, 
such that
\be (\psi_{\d,\tau'})_{*}\pa_{\th}=(1/T_{\d,\tau'})X_{\d,\tau'}.\label{eq:conjpathetaX}\ee
One has $s_{*}:=\inf_{(\d,\tau')\in \bD(0,\nu)\times \bD(\tau'_{1,\mbeta},\nu) }s_{\d,\tau'}>0$. We then define the embedded annulus
\be \cA^{\rm vf}_{\d,\tau'}=\cA^{\rm vf,s_{*}}_{\d,\tau'}=\psi_{\d,\tau'}(\T_{s_{*}}).\label{theinvannulus}\ee
We have thus proved
\begin{theo}[Invariant annulus theorem]\label{theo:invannthm}The restriction of the vector field $X_{\d,\tau'}$ to the $X_{\d,\tau'}$-invariant embedded annulus $\cA^{\rm vf,s_{*}}_{\d,\tau'}$ is conjugate to the vector field $g_{\d}(\tau')\pa_{\th}$ on $\T_{s_{*}}$.

\end{theo}

\subsection{$\diag(1,j)$-symmetry}
Let's make some preliminary remarks. Recall when  $\d=0$, one has $X_{\d=0;\tau'}=X_{1}$ and $T_{\d=0,\tau'}=T_{*}=1/g_{*}$. 
From the Exotic periodic orbit theorem \ref{theo:periodicorbit} we know that the periodic  orbit of $X_{1}$ is $\diag(1,j)$-invariant and satisfies (see
(\ref{eq:diag1jinvperorb}))
\be  \diag(1,j)(\phi^t_{X_{1}}(p_{*}))=\phi^{t+T_{*}/3}_{X_{1}}(p_{*}). \label{eq:diag1jinvperorbbis} \ee

\begin{theo}[$\diag(1,j)$-symmetry] \label{diag1jsymmetry}There exist $\nu''>0$ and $s_{*}'>0$ such that  for $(\d,\tau')\in \bD(0,\nu'')\times \bD_{\C^2}(\tau'_{1,\mbeta},\nu'')$ one has 
$$\diag(1,j)(\cA^{{\rm vf},s'_{*}}_{\d,\tau'})\subset \cA^{{\rm vf},s_{*}}_{\d,\tau'}.$$
\end{theo}
\begin{proof}The relation $\diag(1,j)_{*}X_{\d,\tau'}=X_{\d,\tau'}$ yields 
$$ \phi^{T_{\d,\tau'}}_{X_{\d,\tau'}}(\diag(1,j)(p_{*}+\zeta_{\d,\tau'}  e))=\diag(1,j)(p_{*}+\zeta_{\d,\tau'} e)$$
hence
\begin{multline} \phi^{T_{\d,\tau'}}_{X_{\d,\tau'}}(\phi_{X_{\d,\tau'}}^{-T_{\d,\tau'}/3}\circ \diag(1,j)(p_{*}+\zeta_{\d,\tau'}  e))= \\(\phi_{X_{\d,\tau'}}^{-T_{\tau,\d}/3} \circ \phi^{T_{\d,\tau'}}_{X_{\d,\tau'}}\circ \diag(1,j))(p_{*}+\zeta_{\d,\tau'}  e)=
\\ \phi_{X_{\d,\tau'}}^{-T_{\d,\tau'}/3}\circ \diag(1,j)(p_{*}+\zeta_{\d,\tau'} e).\label{a7.43}
\end{multline}
Besides, from (\ref{eq:diag1jinvperorbbis})
$$\phi_{X_{1}}^{-T_{*}/3}\circ \diag(1,j)(p_{*})=p_{*}$$
and   for $(\d,\tau')\in \bD(0,\nu'')\times \bD_{\C^2}(\tau'_{1,\mbeta},\nu'')$
$$
\phi_{X_{\d,\tau'}}^{-T_{\d,\tau'}/3}\circ \diag(1,j)=\phi_{X_{1}}^{-T_{*}/3}\circ \diag(1,j)\circ (id+O(\nu''));
$$
as a consequence,
one has 
\begin{align}\phi_{X_{\d,\tau'}}^{-T_{\d,\tau'}/3}\circ \diag(1,j)(p_{*}+\zeta_{\d,\tau'} e)&=p_{*}+\zeta_{\d,\tau'}e+\eta_{\d,\tau'}\notag\\
&=p_{*}+\zeta_{}'e+t'X_{\d,\tau'}(p_{*})\notag\\
&=\phi^{t''}_{X_{\d,\tau'}}(p_{*}+\zeta'' e)\label{a7.44}
\end{align}
with $t'',\zeta''=O(\nu'')$.
This and (\ref{a7.43}) show that 
$$\phi_{X_{\d,\tau'}}^{T_{\d,\tau'}}(\phi^{t''}_{X_{\d,\tau'}}(p_{*}+\zeta''e))=\phi^{t''}_{X_{\d,\tau'}}(p_{*}+\zeta''e)$$
hence
$$\phi_{X_{\d,\tau'}}^{T_{\d,\tau'}}(p_{*}+\zeta''e)=p_{*}+\zeta''e.$$
If $\nu''$ is small enough, the uniqueness result of Theorem \ref{theo:perobrit} shows that $\zeta''=\zeta_{\d,\tau'}$ and consequently (see (\ref{a7.44}))
$$\phi_{X_{\d,\tau'}}^{-(t''+T_{\d,\tau'}/3)}\circ \diag(1,j)(p_{*}+\zeta_{\d,\tau'} e)=p_{*}+\zeta_{\d,\tau'} e.$$
Let $s'_{*}$ be such that $\T_{s'_{*}}+t''+T_{\d,\tau'}/3\in \T_{s_{*}}$; for any $\th\in \T_{s_{*}'}$ we thus have
$$\diag(1,j)(\phi_{X_{\d,\tau'}}^\th(p_{*}+\zeta_{\d,\tau'} e))=\phi_{X_{\d,\tau'}}^{\th+t''+T_{\d,\tau'}/3}(p_{*}+\zeta_{\d,\tau'} e)\subset\cA^{{\rm vf},s_{*}}_{\d,\tau'}
$$
which is the conclusion we are looking for.
\end{proof}

\begin{cor}\label{cor:7.7}Assume $g_{\d}(\tau')\in \R$. There exists $s''_{*}\in (0,s'_{*})$ independent of $\d,\tau'$, such that for any  $\xi \in \cA_{\d,\tau'}^{{\rm vf}, s''_{*}}$ one has
$$\diag(1,j)\biggl(\{\phi^t_{X_{\d,\tau'}}(\xi) \mid t\in \R \}\biggr)=\{\phi^t_{X_{\d,\tau'}}(\xi) \mid t\in \R \}.$$
\end{cor}
\begin{proof}Let $\psi_{\d,\tau'}$ be the diffeomorphism of (\ref{eq:conjpathetaX}) and 
$$f_{\d,\tau'}=\psi_{\d,\tau'}^{-1}\circ \diag(1,j)\circ\psi_{\d,\tau'}:\T_{s''_{*}}\to \T_{s_{*}}$$
(for some $s''_{*}\in (0,s'_{*})$ independent of $\d,\tau'$).
Because $X_{\d,\tau'}$ and $\diag(1,j)$ commute, the real orbits $\{\phi^t_{X_{\d,\tau'}}(\xi) \mid t\in \R \}$ of $X_{\d,\tau'}$ are sent to real orbits of $X_{\d,\tau'}$. The fact that $g_{\d}(\tau')$ is real implies that the images of these orbits by $\psi_{\d,\tau'}^{-1}$ are horizontal circles on $\T_{s''_{*}}$. In particular, the holomorphic diffeomorphism $f_{\d,\tau'}$ sends horizontal circles to horizontal circles and it thus must be a translation $\T_{s''_{*}}\ni \th\mapsto \th+a_{\tau,\d'}\in \T_{s_{*}}$. Because the third iterate of $f_{\d,\tau'}$ is the identity ($\diag(1,j)^3=id$) we must have $a_{\d,\tau'}=0\mod \Z$. Conjugating back by $\psi_{\d,\tau'}$ this yields the conclusion.
\end{proof}

\subsection{Reversibility}

Recall a holomorphic diffeomorphism $h$ is reversible if there exists an involution $\s$ (i.e. $\s\circ \s=id$) such that 
$$\s\circ h\circ \s=h^{-1}.$$
We shall require in addition  that the involution $\s$ is {\it anti-holomorphic} which  means that $(z,w)\mapsto \s(\bar z,\bar w)$ is holomorphic.

Similarly, a holomorphic vector field $X$ is reversible if for some anti-holomorphic involution $\s$ one has
$$\s_{*} X=-X.$$
In terms of flow, this is equivalent to 
$$\forall t\ (\in \C)\quad\s \circ \phi^t_{X}\circ\s=\phi_{X}^{-\bar t}$$
(whenever it makes sense).

\bigskip

 As we have seen in  Proposition \ref{prop:involution}, when $\b$ and $c$ are real,  the map $h^{\rm bnf}_{\d,\tau,\mbeta}=h_{\d,\tau'}$ defined by (\ref{themaphetc})
 $$h^{\rm bnf}_{\d,\tau'}=Z_{\d,\tau'}\circ h^{\rm mod}_{\a,\b}\circ Z_{\d,\tau'}^{-1}=\diag(1,e^{2\pi i/3})\circ \phi^1_{\d \mbeta X^{}_{\d,\tau'}}\circ \iota_{O(\d^{2m-(1/3)})}$$
 ($m$ is the positive integer fixed in Proposition \ref{theo:approxbyvf} that we can assume large enough)
  is reversible w.r.t. the anti-holomorphic $\s_{\d,\tau'}=\s_{\d,\tau,\mbeta}$, which satisfies
 \be\begin{cases}
 &\s_{\d,\tau'}(z,w)=\s_{0,\tau}(z,w)+O(\d^{1/3})\\
 &\s_{0,\tau}:(z,w)\mapsto (\tau-\bar\tau,0)+(\bar z,j^2\bar w).
 \end{cases}\label{eq:7.67}\ee
 Moreover, the vector field $X_{\d=0,\tau=1}$ is reversible w.r.t. to $\s:=\s_{0,1}$:
 $$(\s_{0,1})_{*}X_{0,1}=-X_{0,1}.$$
 
 \medskip
 As we shall now see,
  the frequency $g_{\d}(\tau')$ of the vector field $X_{\d,\tau'}$ is very close to a real number, at least when the diffeomorphism $h_{\d,\tau'}$ is reversible, i.e. when $c_{\d}(\tau')=c_{\d}(\tau,\mbeta)$ and $\mbeta$ are real numbers (see (\ref{defctaumbeta})).

Before proceeding to the proof of this fact let us observe that  because $\diag(1,j)$ and $X_{\d,\tau'}$ commute, the diffeomorphism 
$$h_{\d,\tau'}:=(h_{\d,\tau'}^{\rm bnf})^{\circ 3}$$
satisfies 
$$h_{\d,\tau'}=\phi^1_{3\d \mbeta X^{}_{\d,\tau'}}\circ \iota_{O(\d^{2m-(1/3)})}.
$$

\begin{prop}\label{prop:7.7}If  $\mbeta\in\R$ and $c_{\d}(\tau,\mbeta)\in\R$, one has 
\begin{align}
&\Im T_{\d,\tau'}=O(\d^{2m-(5/3)})\label{Ttaureal}\\
&\dist\biggl(\s_{\d,\tau'}(p_{*}+\zeta_{\d,\tau'}e),\cA^{{\rm vf},s'_{*}}_{\d,\tau'} \biggr)=O(\d^{2m-(5/3)}).
\end{align}

\end{prop}
\begin{proof} 
Let $\cA^{\rm vf}_{*}=\cA^{{\rm vf},s_{*}}_{0,\tau'_{1,\mbeta}}$ be the invariant annulus associated to the vector field $X_{1}$ (i.e. $\d=0$). Because for $(\d,\tau)\in \bD(0,\nu')\times \bD(1,\nu')$, $\dist(\cA^{{\rm vf},s_{*}}_{\d,\tau'},\cA^{\rm vf}_{*})=O(\nu')$ (see (\ref{theinvannulus})), there exists a neighborhood $\cV(\cA^{\rm vf}_{*})$ of $\cA^{\rm vf}_{*}$ such that  for any $n\in\N\cap[0,4T_{*}\mbeta^{-1}\d^{-1}]$ ($n\d \leq 4T_{*}\mbeta^{-1}$) one has on $\cV(\cA^{\rm vf}_{*})$
$$h_{\d,\tau'}^{\pm n}=\phi_{3\d \mbeta X_{\d,\tau'}}^{\pm n}\circ (id+O(n\d^{(2m-1/3)}))
$$
hence
$$h_{\d,\tau'}^{\pm n}\circ (id+O(\d^{(2m-(4/3))}))=\phi_{ 3\mbeta X_{\d,\tau'}}^{\pm n\d}
$$
and in particular for any $p\in \cV(\cA^{\rm vf}_{*})$
$$\phi_{ 3 \mbeta X_{\d,\tau'}}^{\pm n\d}(p)=h_{\d,\tau'}^{\pm n}(p)+O(\d^{(2m-(4/3))})
.$$
Because $\s_{\d,\tau'}\circ h_{\d,\tau'}^n\circ \s_{\d,\tau'}=h_{\d,\tau'}^{-n}$
we get
$$\s_{\d,\tau'}\circ \phi^{n\d}_{3\mbeta X_{\d,\tau'}}\circ \s_{\d,\tau'}(p)=\phi^{-n\d}_{3\mbeta X_{\d,\tau'}}(p)+O(\d^{(2m-(4/3))})$$
or equivalently
$$\s_{\d,\tau'}\circ \phi^{ 3\mbeta n\d}_{X_{\d,\tau'}}\circ \s_{\d,\tau'}(p)=\phi^{-3\mbeta n\d}_{X_{\d,\tau'}}(p)+O(\d^{(2m-(4/3))}).$$
We shall need the following lemma. 
\begin{lemma}\label{lemma:harmonic7.2} Let $\e>0$ and $R_{A,s}=[-A,A]+i(-s,s)\subset \C$ a rectangle. Then, there exists $s_{\e}>0$, $c_{\e,A,s}>0$ such that for any holomorphic function  $f:R_{A,s}\to\C$ satisfying $\sup_{R_{A,s}}|f|\leq 1$, the following holds. If 
$$
\begin{cases}&\sup\{|f(z)|\mid z\in R_{A,s}\cap( \d\Z)\}\leq \nu\\
&\d\ln(1/\nu)\leq c_{\e,A,s}
\end{cases}$$
then 
$$\sup_{R_{A/4,s_{\e}}}|f|\leq \nu^{1-\e}.$$
\end{lemma}
\begin{proof}
The proof uses three ingredients:
\begin{enumerate}
\item Harnack's inequality: for any rectangle $R$ centered at 0 and  $\e>0$ there exists $\rho_{\e}\in (0,1)$ such that  for any holomorphic function  $f:R\to \C$  of maximum module less than 1, which does not vanish on $R$, one has
$$\biggl(\sup_{\rho_{\e}R}|f|\biggr)^{1/(1-\e)} \leq |f(0)|\leq \biggl(\inf_{\rho_{\e}R} |f|\biggr)^{1/(1+\e)}.$$
($\rho_{\e}R$ is the rectangle homothetic to $R$ with diameter $\rho_{\e}$ times the diameter of $R$).
\item Jensen's inequality (for a rectangle): Let $R$ be a rectangle centered at 0;  there exists a constant $C_{R}>0$ such that for any holomorphic function $f:R_{}\to \C$  of maximum module less than 1 
one has
$$\sup_{(1/4) R}\ln |f|\leq  -C_{R}\times \#\{\zeta\in (1/2) R \mid f(\zeta)=0\}.$$
\item \label{item:3}Poisson like subharmonic  inequality: Let $R_{A,s}\subset \C$ be a rectangle. 
For any $\e>0$, there exists $s_{\e}>0$ such that for any holomorphic function $f:R_{A,s}\to \C$  of maximum module less than 1 and any $\nu\in ]0,1[$
$$\frac{\biggl|\{z\in [-A/2,A/2]\mid |f(z)|\leq \nu\} \biggr| }{|[-A/2,A/2]|}\geq 1-\e\implies \sup_{R_{A/4,s_{\e}}}\ln |f|\leq (1-\e)\times \ln \nu .$$
\end{enumerate}

\medskip
Let $(D_{k})_{k\in I}$ be the finite collection of rectangles with centers located on  $[-(A/2-10\d),(A/2-10\d)]\cap \d\Z$ and with diameter $10\d$. Let $\rho_{\e}^{-1}D_{k}$ the rectangles with the same centers but diameter $10\rho_{\e}^{-1}\d$. Let $p$ be the proportion of $k\in I$ such that $f$ has a zero in ${\rho_{\e}^{-1}}D_{k}$. Because the overlap of the rectangles ${\rho_{\e}^{-1}}D_{k}$ is $\asymp \rho_{\e}^{-1}$, the number of zeros of $f$ in  $R_{A,s}$ is at least ${\rm cst}\times p\times \#I\times \rho_{\e}$ and  by Jensen's formula
\begin{align*}\sup_{(1/4)R_{A,s}}\ln |f|&\leq -C_{R_{A,s}}\times p\times \#I\times \rho_{\e}\\
&\leq -C'_{R_{A,s}}\times p\times \d^{-1}\times \rho_{\e}.
\end{align*}
If $p\geq \e$ this yields for $\zeta\in R_{A/4,s/4}$
\be \ln |f(\zeta)|\leq -C'_{R_{A,s}} \times \e\times \d^{-1}\times \rho_{\e}.\label{b7.46}\ee
If we  assume 
\be \d\ln (1/\nu)\leq c_{\e,A,s}:= (1-\e)^{-1} C'_{R_{A,s}}\times \e\times \rho_{\e}.\label{c7.46}\ee
one has
\be \ln |f(\zeta)|\leq (1-\e)\times\ln \nu. \ee

Otherwise, if $p<\e$, by Harnack's principle, the  Lebesgue measure of the set $H=\{z\in [-A/2,A/2]\mid |f(z)|\leq \nu^{1-\e}\}$ is $$\geq (1-p) |[-A/2,A/2]|\geq (1-\e) |[-A/2,A/2]| .$$ 
The  subharmonic estimate of item  (\ref{item:3}) gives  the existence of $s_{\e}>0$ for which
\be \sup_{R_{A/4,s_{\e}}}\ln |f|\leq (1-\e)\times\ln \nu.\label{b7.47}\ee

In any case,  if (\ref{c7.46}) holds, one has
$$\sup_{R_{A/4,\min(s_{\e},s/4)}}\ln |f|\leq (1-\e)\times\ln \nu.$$

\end{proof}

Define for $\th=t+is\in [-4T,4T]+i(-s_{\e},s_{\e})$,  $p\in \cV(\cA^{\rm vf}_{*})$ the holomorphic function (recall $\s_{\d,\tau'}$ is anti-holomorphic)
$$f_{p}(\th)=\s_{\d,\tau'}\circ \phi^{\bar\th}_{X_{\d,\tau'}}\circ \s_{\d,\tau'}(p) -\phi^{-\th}_{X_{\d,\tau'}}(p).$$
One has  with $\d'=3\mbeta\d$
$$\forall n\in\N\cap[0,4T(\d')^{-1}] ,\quad  f_{p}(n\d')=O(\d^{2m-(4/3)})$$
and by the previous Lemma applied to the components of $f_{p}$ (with $\e$ such that $(1-\e)(2m-(4/3))=(2m-(5/3))$) there exists $s_{*}'$ (independent of $\d$) for which 
$$\forall\th\in \T_{s_{*}'},\quad f_{p}(\th)=O(\d^{2m-(5/3)}).$$
In particular,
$$\phi^{-\th}_{-(\s_{\d,\tau'})_{*}X_{\d,\tau'}}(p) -\phi^{-\th}_{X_{\d,\tau'}}(p)=O(\d^{2m-(5/3)})$$
and taking the derivative at $\th=0$
$$(\s_{\d,\tau'})_{*}X_{\d,\tau'}(p)=-X_{\d,\tau'}(p)+O(\d^{2m-(5/3)}).$$
This gives with $p=p_{*}+\zeta_{\d,\tau'}e$
$$\s_{\d,\tau'}\circ \phi^{\bar T_{\d,\tau'}}_{X_{\d,\tau'}}\circ \s_{\d,\tau'}(p_{*}+\zeta_{\d,\tau'}e) -\phi^{-T_{\d,\tau'}}_{X_{\d,\tau'}}(p_{*}+\zeta_{\d,\tau'}e)=O(\d^{2m-(5/3)})$$
hence (remember $\phi^{T_{\d,\tau'}}_{X_{\d,\tau'}}(p_{*}+\zeta_{\d,\tau'}e)=p_{*}+\zeta_{\d,\tau'}e)$)
\be  \phi^{\bar T_{\d,\tau'}}_{X_{\d,\tau'}}\circ \s_{\d,\tau'}(p_{*}+\zeta_{\d,\tau'}e) -\s_{\d,\tau'}(p_{*}+\zeta_{\d,\tau'}e)=O(\d^{2m-(5/3)}).\label{eq:7.74n}\ee
We now observe that for some $t_{*}\in \R$ (see  Theorem \ref{theo:periodicorbit})
$$\s_{0,1}(p_{*})=\phi^{t_{*}}_{X_{0,1}}(p_{*})$$
Hence from (\ref{eq:7.67}) and the fact that $X_{\d,\tau'}=X_{0,1}+o(\nu')$ if $\d$ is small enough and $\tau$ close enough to 1 one has
$$\phi^{-t_{*}}_{X_{\d,\tau'}}(\s_{\d,\tau'}(p_{*}))=p_{*}+o(\nu)$$
and we can thus write
\be \phi^{-t_{*}}_{X_{\d,\tau'}}( \s_{\d,\tau'}(p_{*}+\zeta_{\d,\tau'}e))=\phi^{-t_{\d,\tau'}}_{X_{\d,\tau'}}(p_{*}+\ti \zeta_{\d,\tau'}e)\label{eq:7.75}\ee
for some $t_{\d,\tau'}\in \C$, $\ti \zeta_{\d,\tau'}\in \C$ in a $\nu'$-neighborhood of $(0,0)$. Equation (\ref{eq:7.74n}) can be written
$$  \phi^{\bar T_{\d,\tau'}}_{X_{\d,\tau'}}\circ\phi^{t_{*}-t_{\d,\tau'}}_{X_{\d,\tau'}}(p_{*}+\ti \zeta_{\d,\tau'}e) -\phi^{t_{*}-t_{\d,\tau'}}_{X_{\d,\tau'}}(p_{*}+\ti \zeta_{\d,\tau'}e)=O(\d^{2m-(5/3)})$$
whence
\be  \phi^{\bar T_{\d,\tau'}}_{X_{\d,\tau'}}(p_{*}+\ti \zeta_{\d,\tau'}e) -(p_{*}+\ti \zeta_{\d,\tau'}e)=O(\d^{2m-(5/3)}) \label{eq:7.75bis}\ee
and
by  the Remark \ref{rem:7.1} 
$$\begin{cases}
&\bar T_{\d,\tau'}=T_{\d,\tau'}+O(\d^{(2m-(5/3))})\\
&\ti\zeta_{\d,\tau'}-\zeta_{\d,\tau'}=O(\d^{2m-(5/3)}).
\end{cases}
$$
This and (\ref{eq:7.75}) yield
$$\begin{cases}
&\Im T_{\d,\tau'}=O(\d^{(2m-(5/3))})\\
&\s_{\d,\tau'}(p_{*}+\zeta_{\d,\tau'}e) -\phi^{t_{*}-t_{\d,\tau'}}_{X_{\d,\tau'}}(p_{*}+\zeta_{\d,\tau'}e)=O(\d^{2m-(5/3)}).
\end{cases}
$$
This  is the searched for conclusion (note that $\phi^{t_{*}-t_{\d,\tau'}}_{X_{\d,\tau'}}(p_{*}+\zeta_{\d,\tau'}e)\in \cA_{\d,\tau'}^{{\rm vf},s'_{*}}$).

\end{proof}

\section{First return maps, renormalization and commuting pairs}\label{sec:renormandcommpairs}

We define in this section the renormalization of certain holomorphic diffeomorphisms close to the identity, more precisely close  to the time-$\d$ map ($\d$ small)  of some holomorphic vector fields.  These results will be applied  in Section \ref{sec:proofmainA} and \ref{sec:proofmainA:B} to the third iterate of the  diffeomorphism $h_{\d,\tau'}^{\rm bnf}$ (see (\ref{themaphetc})) defined in  Proposition \ref{theo:approxbyvf}.

\bigskip 
Let $X_{}$ be a holomorphic vector field defined  in a bounded open set $V$ of $\C^2$ with
$$\|X\|_{V}\leq A.$$
We assume that 
\begin{assumption}\label{assump:8.1}
\begin{enumerate}
\item The vector field  $X_{}$ has an invariant annulus 
$$\cA^{\rm vf}_{}=\{\phi^{t+is}_{X_{} }(\zeta)\mid t\in\R, \ s\in (-s_{*},s_{*})\}$$
 on which $X$ is conjugate to the vector field $g\pa_{\th}$ defined on $\T_{s_{*}}$ with
 $$g\in\R^*.$$
 \item The invariant annulus $\cA^{\rm vf}_{}$ intersects and is transverse to some  $\zeta+\C e_{}$ and we can assume $e=e_{2}=\bm 0\\ 1\em$.   \end{enumerate}
 We denote 
 \be T=\frac{1}{g}\in\R.\label{periodtofX}\ee
 The vector field $X_{}^{}$ has thus a  $T^{}_{}$-periodic orbit $(\phi^t(\zeta_{} )_{t\in\R}\subset \cA^{\rm vf}_{}$.
 \end{assumption}
 \bigskip
 
 \begin{assumption}\label{assump:8.2}
 We also assume  we are given $\eta\in \cO(V,\C^2)$  such that for $\d>0$ small enough
\be \begin{cases}&\|\eta\|_{V}\leq A\d^p\\
&p>2\\
\end{cases}\label{eq:7.39hedobis}\ee
and we  set
\be h_{\d,\eta}=\phi^1_{\d X_{}}\circ (id+\eta).\label{def:hdeltaeta}\ee
In many cases $X$ shall have constant divergence and  $\eta$ will be of the form
\be \begin{cases}&\eta=\iota_{F}\\
&F\in \cO(V),\qquad \|F\|_{V}\leq A\d^p.\end{cases}\label{formetat}\ee
\end{assumption}

\subsection{Boxes} We associate to the vector field $X$ and the diffeomorphism $id+\eta$ various domains that we call {\it boxes}.

We define first  the 3-dimensional {\it real} manifold
$$\Sigma^X_{\d,s,\rho}=\{\phi_{X^{}_{}}^{\th}(\zeta_{})+re_{2}\mid \th\in i\d\times (-s,s),\ r\in\C,\  |r|<\d \rho\}.$$
For $\d>0$ and $\nu\in [0,2]$ we then define the open set of $\C^2$
\be \mathcal{W}_{\d,s,\rho,\nu}^{X,0}=\bigcup_{t\in (-\nu,1+\nu)}\phi^t_{\d X^{}_{}}(\Sigma^X_{\d,s,\rho_{}}).\label{def:Wsigma}\ee

For $t\in (-2,3)$ we define
\be h_{\d,\eta}^t=\phi^t_{\d X_{}}\circ (id+t\eta).\label{def:hdeltaetat}\ee
and we observe that if $\d$ is small enough the map 
$$\Sigma^X_{\d,s_{},\rho}\times (-\nu,1+\nu)\ni (\xi,t)\mapsto h_{\d,\eta}^t(\xi)$$
is a diffeomorphism onto its image (this follows from the case $\eta=0$). We then introduce the {\it box}
\be\cW^{X,\eta}_{\d,s_{},\rho,\nu}=\bigcup_{t\in (-\nu,1+\nu)}h_{\d,\eta}^t(\Sigma^X_{\d,s_{},\rho}).\label{def:Wsigmadelta}\ee
Note that for any $\nu\in (0,1/3)$, if  $\d$ is small enough, the domains $W^{X,\eta}_{\d,s_{},\nu}$ are included in a domain $\zeta+U$, $U=U'\times U''$, inside which  $\cA^{\rm vf}_{}$ can be described as  a graph $\zeta+\{(z,E(z))\mid z\in U'\}$ ($E:\C\supset U'\to U'' \subset\C$ holomorphic, $0\in U'$, $E(0)=0$). We denote
\be \Gamma^X: \zeta+U\ni(z,w)\mapsto (z,w-E(z))-\zeta\label{eq:definGammaX}\ee
which satisfies $\Gamma^X(\zeta)=(0,0)$.

Inequality (\ref{eq:7.39hedobis}) implies that for $\d$ small enough
\be \cW^{X,\eta=0}_{\d,s=0,\rho=0,(9/10)\nu}\subset \cW^{X,\eta}_{\d,s, \rho=\d^{p-1},\nu}.\ee

We set 
\be {\bar \cW}^{X,\eta}_{\d, s_{},\rho}=\Sigma^X_{\d,s_{},\rho}\cup \cW^{X,\eta}_{\d,s_{},\rho,\nu=0}.\label{defbarcW}\ee

\begin{notation}We shall remove the dependence on $X$ in this section and denote for example $\Sigma_{\d,s,\rho}$, $\cW^\eta_{\d,s,\rho,\nu}$ {\it etc.} in place of $\Sigma^X_{\d,s,\rho}$, $\cW^{X,\eta}_{\d,s,\rho,\nu}$.

Also,  if $s=\rho$ we remove the dependence on $\rho$ in the above formulas: for example we denote 
\be \Sigma_{\d,s}=\Sigma_{\d,s,s},\quad\cW^\eta_{\d,s,\nu}=\cW^\eta_{\d,s,s,\nu}\quad\textrm{and}\quad {\bar \cW}^{\eta}_{\d, s_{}}={\bar \cW}^{\eta}_{\d, s_{},s}.\label{new-notation}\ee

\end{notation}

\subsection{First return maps}\label{sec:7.1}
\begin{defin} [First return map]\label{def:firstreturnmap} If there exists $0<s_{}'\leq s_{}$, $0<\rho'\leq \rho$ such that    $$\forall \xi\in  {\bar \cW}^\eta_{\d,s_{}',\rho'}\quad \exists n\in \N^* \quad h_{\d,\eta}^n(\xi)\in {\bar \cW}^\eta_{\d,s_{},\rho}$$
we say that ${\bar \cW}^\eta_{\d,s,\rho}$ is a {\it first return domain} of  $(h_{\d,\eta},{\bar \cW}^\eta_{\d,s',\rho'})$ and that ${\bar \cW}^\eta_{\d,s',\rho'}$ is a {\it renormalization box} for $h_{\d,\eta}$ . The maps
$$n:{\bar \cW}^\eta_{\d,s',\rho'}\ni \xi \mapsto n(\xi)=\min\{n\in\N^*\mid h_{\d,\eta}^n(\xi)\in {\bar \cW}^\eta_{\d,s,\rho}\}\in\N$$
and 
$$\hat h_{\d,\eta}:{\bar \cW}^\eta_{\d,s',\rho'}\ni \xi \mapsto h_{\d,\eta}^{n(\xi)}(\xi)\in {\bar \cW}^\eta_{\d,s,\rho}$$
are called respectively  the associated {\it first return time map} and the {\it first return map}.
\end{defin}
Note that $\hat h_{\d,\eta}$  is in general not continuous but locally holomorphic   on an  open set.
The map
$\hat h_{\d,\eta}: {\bar \cW}^\eta_{\d,s',\rho'}\to {\bar \cW}^\eta_{\d,s,\rho}$ is injective. It is holomorphic  on $\cW_{\d,s',\rho'}^\eta\setminus \hat h_{\d,\eta}^{-1}(\Sigma_{\d,s,\rho})$.

\medskip

The main result of this Section is the following proposition.

\begin{prop}\label{lemma:7.1} There exists $\d_{*}>0$ and $0<s'\leq s$  such that, for any $\d\in (0,\d_{*})$ for which\footnote{In what follows $\{\frac{T^{}_{}}{\d}\}$ is the fractional part of $T/\d$ and $[\frac{T^{}_{}}{\d}]$ its integer part.} 
$$\biggl\{\frac{T^{}_{}}{\d}\biggr\}\in ((1/10),(9/10))$$
the set ${\bar \cW}^\eta_{\d,s}$ is a first return domain of  $(h_{\d,\eta},{\bar \cW}^\eta_{\d,s'})$. Moreover, the first return time map $n$ takes two values $q_{\d}$ and $q_{\d}+1$ where 
$$q_{\d}=\biggl[\frac{T^{}_{}}{\d}\biggr],$$
i.e.  $n:{\bar \cW}^\eta_{\d,s'}\ni \xi \mapsto n(\xi)\in\{q_{\d},q_{\d}+1\}\in\N$.
\end{prop}
\begin{proof}

\bn{ 1) \it First return map of $\phi^1_{\d X^{}_{}}$ in $\bar{\mathcal{W}}_{\d,s}^{0}$.}
The dynamics of $\phi^1_{\d X^{}_{}}$ on $\cA^{\rm vf}_{}\simeq \T_{s_{*}}$ is   conjugate
by the map
$$\ph=\psi^{-1}:\cA^{\rm vf}_{}\ni\phi_{X_{}}^t(\zeta)\mapsto (t/T_{}^{})+\Z\in \T_{s_{*}}$$ to a rigid rotation
$$R_{\a_{\d}}:\T_{s_{*}}\ni\th\mapsto \th+\a_{\d}\in \T_{s_{*}}$$
with  rotation number $$\a_{\d}:=\d /T^{}_{}>0.$$
By assumption $\a_{\d}\notin\Z$.

We now consider the restriction of $R_{\a_{\d}}$ to the circle $\R/\Z$.  The non-vanishing vector field $\ph_{*}((1/T_{}^{})X_{}^{})$ defines an orientation on the circle $\R/\Z$ and one can define for any two points   $p_{1},p_{2}\in \R/\Z$ the arc segment  $[p_{1},p_{2}]\subset\R/\Z$.
 
  Let 
 $$q_{\d}=[1/\a^{}_{\d}]\quad \textrm{and}\quad \ti\a^{}_{\d}=\{1/\a^{}_{\d}\}$$
  so 
  $$1=q_{\d}\a_{\d}+\a_{\d}\ti\a_{\d}.$$
  
 In what follows we use the shorthand notations $$\a=\a_{\d},\qquad q=q_{\d},\qquad \textrm{and }\ \ti \a=\ti \a_{\d}.$$

 The first return map $\hat R_{\a}$ in  $[0,\a]+\Z$ is then 
 \begin{itemize}
 \item $\hat R_{\a}(x)=R_{\a}^q(x)$ if $x\in [\ti\a\a,\a]+\Z$;
 \item $\hat R_{\a}(x)=R_{\a}^{q+1}(x)$ if $x\in [0,\ti\a\a]+\Z$.
 \end{itemize}
 Note that the points $\ti\a \a+\Z=R_{\a}^{-q}(0+\Z)$ and $\a-\ti\a\a+\Z=R_{\a}^{q+1}(0)$  lie  in the arc segment  $[0,\a]+\Z$ and  we can write
  \begin{itemize}
 \item $\hat R_{\a}(x)=R_{\a}^q(x)$ if $x\in [R_{\a}^{-q}(0),R_{\a}(0)]$; 
 \item $\hat R_{\a}(x)=R_{\a}^{q+1}(x)$ if $x\in [0,R_{\a}^{-q}(0)]$.
 \end{itemize}
 One then has
 $$\begin{cases}&\hat R_{\a}([R_{\a}^{-q}(0),R_{\a}(0)])=[0,R_{\a}^{q+1}(0)]\\
 &\hat R_{\a}([0,R_{\a}^{-q}(0)])=[R_{\a}^{q+1}(0),R_{\a}(0)].
 \end{cases}
 $$
 
 As a consequence, on the circle $\{\phi_{\d X^{}}^t(\zeta)\mid t\in\R\}$,  the first return map $\hat {\phi^1_{\d X^{}}}$ of $\phi^1_{\d X^{}}$ in the segment $\{\phi_{\d X^{}}^t(\zeta)\mid t\in [0,1] \}$ satisfies 
  $$\begin{cases}&\hat {\phi^1_{\d X_{}^{}}}([\phi_{\d X_{}^{}}^{-q}(\zeta),\phi^1_{\d X_{}^{}}(\zeta) ])=[\zeta,\phi_{\d X_{}^{}}^{q+1}(\zeta)]\\
&\hat {\phi^1_{\d X_{}^{}}}([\zeta,\phi_{\d X_{}^{}}^{-q}(\zeta)])=[\phi_{\d X_{}^{}}^{q+1}(\zeta),\phi_{\d X_{}^{}}^1(\zeta)];
 \end{cases}
 $$
 see Figure \ref{fig:renorm-flow-im}.

Recall $\cW_{\d,s}^{0}$ (cf. (\ref{def:Wsigma} and (\ref{new-notation})) is the domain between the hypersurfaces (in $\R^4$) $\Sigma_{\d,s}$ and $\phi^1_{\d X_{}^{}}(\Sigma_{\d,s})$. 
For $s'$ small enough, points of $\cW_{\d,s'}^{0}$ which are at the left of the hypersurface $\phi_{\d X^{}_{}}^{-q}(\Sigma_{\d,s})$ return in  $\cW_{\d,s}^{0}$ after $q+1$ iterations, while points of $\cW_{\d,s'}^{0}$ which are at the right\footnote{If $s'$ is small enough, these notions of ``left'' and ``right'' are well defined in some neighborhood of the periodic orbit.} of the hypersurface $\phi_{\d X_{}^{}}^{-q}(\Sigma_{\d,s})$ return in  $\cW_{\d,s}^{0}$ after $q$ iterations. One may wonder whether these domains are empty. These is indeed not the case if $s'$ is small enough, this smallness being  {\it independent of} $\d$. Indeed, it is enough to observe that 
$$\phi_{\d X_{}^{}}^{-q}(\Sigma_{\d,s})=\phi_{ X_{}^{}}^{-q\d}(\Sigma_{\d,s}),\qquad \phi_{\d X_{}^{}}^{q+1}(\Sigma_{\d,s})=\phi_{ X_{}^{}}^{(q+1)\d}(\Sigma_{\d,s})$$
and that $|-q \d|\asymp 1$, $ (q+1)\d\asymp 1$. In particular, if $s'$ is small enough, independent of $\d$, the hypersurfaces $\phi_{\d X_{}^{}}^{-q}(\Sigma_{\d,s})$ and $\phi_{\d X_{}^{}}^{q+1}(\Sigma_{\d,s})$, which are transverse to the periodic orbit $\{\phi_{\d X_{}^{}}^t(\zeta)\mid t\in\R\}$, will (possibly) cut $\Sigma_{\d,s}$ or $\phi^1_{\d X_{}^{}}(\Sigma_{\d,s})$ at  points which are at a distance from $\zeta$ bounded below by a number independent of $\d$; see Figure \ref{fig:renorm-flow-im}.

Let us denote by 
$[\Sigma_{\d,s},\phi_{\d X_{}^{}}^{-q}(\Sigma_{\d,s})\mid \cW^{0}_{\d,s'}]$ (resp. $[\phi_{\d X_{}^{}}^{-q}(\Sigma_{\d,s}),\phi^1_{\d X_{}^{}}(\Sigma_{\d,s})\mid \cW^{0}_{\d,s'}]$)  the set of points of $\cW_{\d,s'}^{0}$ that are between\footnote{If $s'$ is small enough, this notion is well defined.} the hyper-surfaces $\Sigma_{\d,s}$ and $\phi_{\d X_{}^{}}^{-q}(\Sigma_{\d,s})$  (resp.  $\phi_{\d X_{}^{}}^{-q}(\Sigma_{\d,s})$ and $\phi_{\d X_{}^{}}^1(\Sigma_{\d,s})$). One has for $s''<s'<s$ ($s''$ small enough, independent of $\d$)
 \be \begin{cases}&\hat {\phi^1_{\d X_{}^{}}}([\phi_{\d X_{}^{}}^{-q}(\Sigma_{\d,s}),\phi^1_{\d X_{}^{}}(\Sigma_{\d,s})\mid \cW^{0}_{\d,s''} ])\subset[\Sigma_{\d,s},\phi_{\d X_{}^{}}^{q+1}(\Sigma_{\d,s})\mid \cW^{0}_{\d,s'}]\\
&\hat {\phi^1_{\d X_{}^{}}}( [\Sigma_{\d,s},\phi_{\d X_{}^{}}^{-q}(\Sigma_{\d,s}) \mid \cW^{0}_{\d,s''}])\subset[\phi_{\d X_{}^{}}^{q+1}(\Sigma_{\d,s}),\phi_{\d X_{}^{}}^1(\Sigma_{\d,s}) \mid \cW^{0}_{\d,s'}]\\
&\hat {\phi^1_{\d X_{}^{}}}\mid [\phi_{\d X_{}^{}}^{-q}(\Sigma_{\d,s}),\phi^1_{\d X_{}^{}}(\Sigma_{\d,s}) \mid \cW^{0}_{\d,s''}]=\phi_{\d X_{}^{}}^{q+1}\\
&\hat {\phi^1_{\d X_{}^{}}}\mid  [\Sigma_{\d,s},\phi_{\d X_{}^{}}^{-q}(\Sigma_{\d,s})\mid \cW^{0}_{\d,s''}]=\phi_{\d X_{}^{}}^q.
 \end{cases}
 \label{eq:renorm:7.30}
 \ee
\bn {\it 2) General case.}\quad  We first  observe: 
\begin{lemma}\label{lemma:est7.2}One has 
$$h_{\d,\eta}^{\pm q}=\phi_{\d X_{}}^{\pm q}\circ (id+O_{A}(\d^{p-1})),\qquad  h_{\d,\eta}^{q+1}=\phi_{\d X_{}}^{q+1}\circ (id+O_{A}(\d^{p-1})).$$
When $X$ has constant divergence and $\eta$ is of the form (\ref{formetat}) one has
$$h_{\d,\eta}^{\pm q}=\phi_{\d X_{\d}}^{\pm q}\circ \iota_{O_{A}(\d^{p-1})},\qquad  h_{\d,\eta}^{q+1}=\phi_{\d X_{\d}}^{q+1}\circ \iota_{O_{A}(\d^{p-1})}.$$
\end{lemma}
\begin{proof}
Let's prove the second set of equations (the other one is treated similarly).
Let $n\in\N$ be such that $n\d\asymp 1$. One has for some $F=O(\d^p)$
\begin{align*}h_{\d,\eta}^n&=(\phi^1_{\d X_{}}\circ \iota_{F})\circ \cdots \circ (\phi^1_{\d X_{}}\circ \iota_{F})\\
&=\phi^n_{\d X_{}}\circ g_{n}
\end{align*}
where 
$$g_{n}=(\phi^{-(n-1)}_{\d X_{}}\circ\iota_{F}\circ \phi^{(n-1)}_{\d X})\circ\cdots\circ (\phi^{-1}_{\d X_{}}\circ\iota_{F}\circ \phi^{1}_{\d X_{}}) \circ \iota_{F}.$$
Because  $n\d\asymp 1$ and $\phi_{\d X_{}}^1$ is conformal symplectic, one has for $0\leq k\leq n-1$
$$\phi^{-k}_{\d X_{}}\circ\iota_{F}\circ \phi^{k}_{\d X_{}}=\iota_{G_{k}}$$
with $G_{k}=O_{A}(\d^p)$. This implies that 
$$\iota_{G_{n-1}}\circ\cdots\circ  \iota_{G_{0}}=\iota_{G}$$
with $G=O_{A}(n\d^p)=O_{A}(\d^{p-1})$.

\end{proof}
 The preceding lemma shows 
 the geometric picture depicted  in Figure \ref{fig:renorm-flow-im}, describing the first return map of $\phi^1_{\d X_{}^{}}$ in $\bar {\mathcal{W}}_{\d,s}^{0}$,   remains essentially  the same if one wants to  describe the first return map of $h_{\d ,\eta}$ in $	\bar{\mathcal{W}}_{\d,s}^{\eta}$, except  that there is no more an obvious circle left invariant  by $h_{\d,\eta}$; see Figure \ref{fig:renorm-diff-im}.

One then has for some $s''<s'<s$ ($s'',s'$ independent of $\d$)
 \be \begin{cases}&\hat {h_{\d,\eta}}([h_{\d ,\eta}^{-q}(\Sigma_{\d,s}),h_{\d,\eta}(\Sigma_{\d,s})\mid \cW^\eta_{\d,s''} ])\subset[\Sigma_{\d,s},h_{\d ,\eta}^{q+1}(\Sigma_{\d,s})\mid \cW^\eta_{\d,s'}]\\
&\hat {h_{\d ,\eta}}( [\Sigma_{\d,s},h_{\d,\eta}^{-q}(\Sigma_{\d,s}) \mid \cW^\eta_{\d,s''} ])\subset[h_{\d ,\eta}^{q+1}(\Sigma_{\d,s}),h_{\d ,\eta}(\Sigma_{\d,s})\mid \cW^\eta_{\d,s'}]\\
&\hat {h_{\d \eta}}\mid ([h_{\d,\eta}^{-q}(\Sigma_{\d,s}),h_{\d ,\eta}(\Sigma_{\d,s}) \mid \cW^\eta_{\d,s''}]=h_{\d ,\eta}^{q+1}\\
&\hat {h_{\d ,\eta}}\mid  [\Sigma_{\d,s},h_{\d ,\eta}^{-q}(\Sigma_{\d,s}) \mid \cW^\eta_{\d,s''}]=h_{\d ,\eta}^q
 \end{cases}
 \label{eq:renorm:7.31}
 \ee
where we have denoted for example  $[\Sigma_{\d,s},h_{\d,\eta}^{-q}(\Sigma_{\d,s}) \mid \cW^\eta_{\d,s'}]$ the set of points of $\cW_{\d,s'}^\eta$ that are between\footnote{Like in the flow case,  if $s'$ is small enough, this notion is well defined by using the isotopy (\ref{def:hdeltaetat}).} the hyper-surfaces $\Sigma_{\d,s}$ and $h_{\d,\eta}^{-q}(\Sigma_{\d,s})$.

This last set of inclusions concludes the proof of  Proposition \ref{lemma:7.1}.
\end{proof}

\begin{figure}[h]
\hspace{-4cm}
\includegraphics[scale=.25, left]{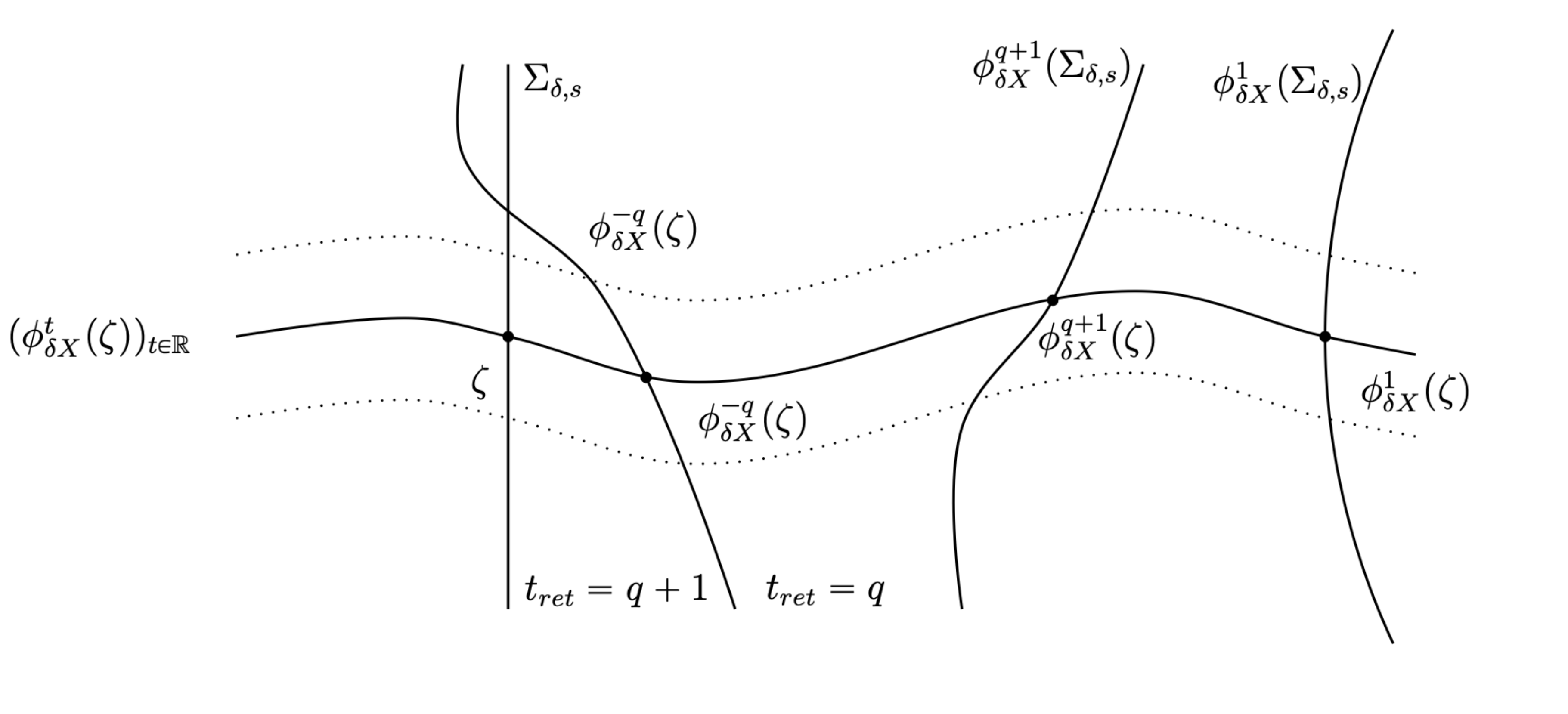}
\caption{Renormalization box for the flow. The size of the depicted domain is of order $\d$. }\label{fig:renorm-flow-im}
\end{figure}

\begin{figure}[h]
\hspace{-4cm}
\vskip.7cm \includegraphics[scale=.25, left]{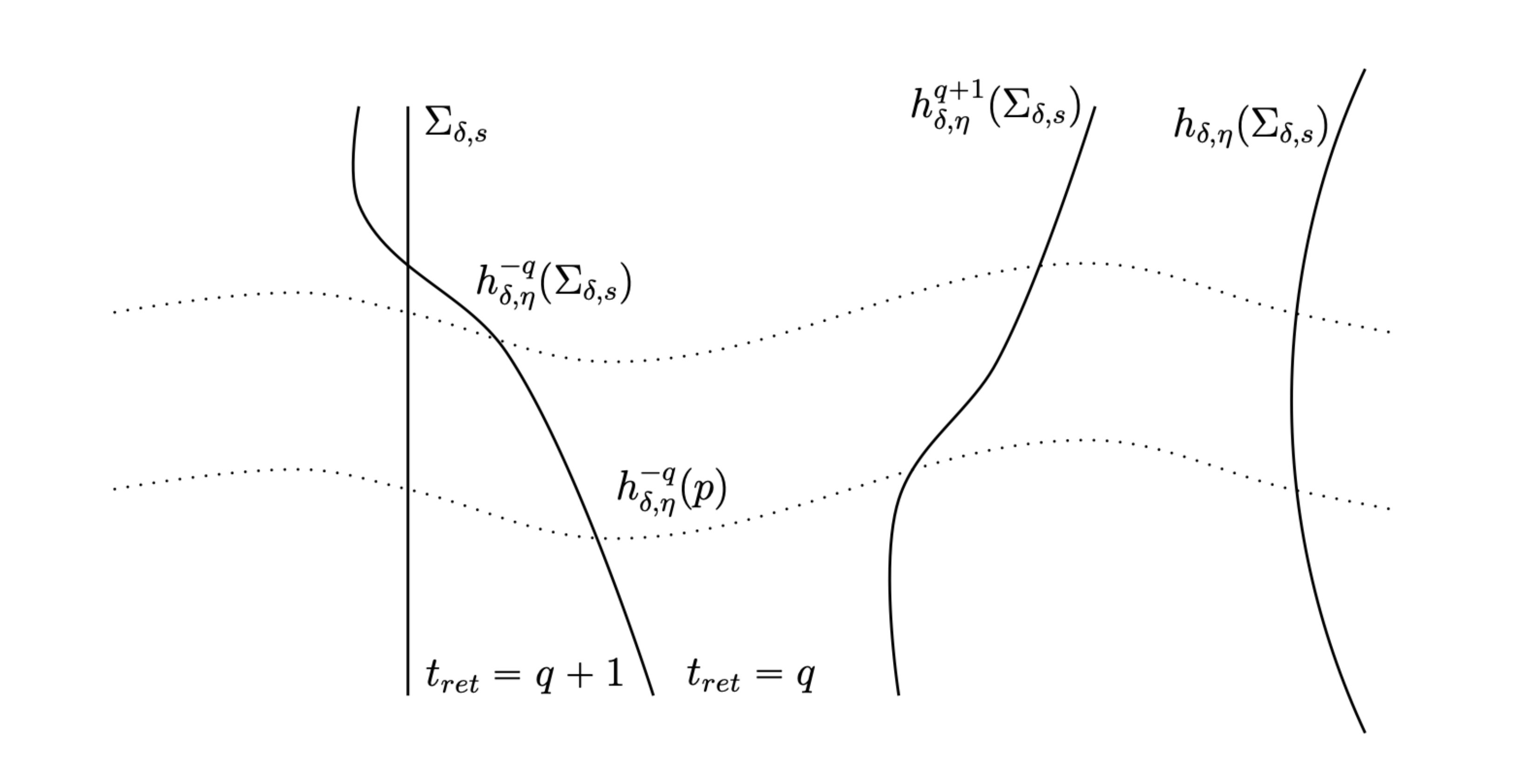}
\caption{Renormalization box for the diffeomorphism }\label{fig:renorm-diff-im}
\end{figure}

\subsection{Backward iterates of first return domains}\label{sec:8.2}
\begin{lemma}\label{lemma:connectedness}Assume $\nu\in (0,1/3)$, $s\in (\d^{p-(5/2)},1)$ and $\d>0$ small enough. 
\begin{enumerate}
\item For any $l\in\{0,\ldots,q_{\d}\}$, $\{\phi^t_{X}(\zeta)\mid t\in\R\}\cap h_{\d,\eta}^{-l}(\cW_{\d,s,\nu}^\eta)\ne \emptyset$.
\item One has $\{\phi^t_{X}(\zeta)\mid t\in\R\}\subset\bigcup_{l=0}^{q_{\d}}h_{\d,\eta}^{-l}(\cW_{\d,s,\nu}^\eta)$.
\end{enumerate}
\end{lemma}
\begin{proof}
This is a consequence: of the fact that the corresponding statements are true for $\eta=0$, of the estimate
\begin{align*}\forall\ 0\leq l\leq q_{\d},\quad h_{\d,\eta}^{-l}&=\phi_{\d X_{}}^{-l}\circ (id+O_{A}(\d^{p-1}))\\
&=\phi_{X_{}}^{-l\d}\circ (id+O_{A}(\d^{p-1}))
\end{align*}
and of $l\d\asymp 1$ (for item 2 note that $\bigcup_{l=0}^{q_{\d}} (-l\a+[0,\a]+\Z)=[0,1]+\Z$).

\end{proof}
\begin{rem}\label{rem:81}By the same token one can prove that if $\d$ is small enough,
for all $k\geq l$, $k,l\in [0,q_{\d}]\cap\N$,  $$\g(k,l)=0\iff h_{\d,\eta}^{-k}({\cW}^\eta_{\d, s,\nu})\cap h_{\d,\eta}^{-l}({\cW}^\eta_{\d,s,\nu})= \emptyset$$
where we've set $\g(k,l)=0$ if $ k-l\notin \{0, 1,q_{\d}\}$ and 1 otherwise.
\end{rem}

The previous Lemma has the following immediate Corollary:
\begin{cor}\label{prop:iteratesfrd}For $s\in (\d^{p-(5/2)},1)$ and $\d$ small enough, the set 
\be  \cC^\eta_{\d,s,\nu}=\bigcup_{l=0}^{q_{\d}}h_{\d,\eta}^{-l}(\cW^\eta_{\d,s\,\nu})\label{defhatcA}\ee
is an open connected set that contains  the orbit $\{\phi^t_{X}(\zeta)\mid t\in\R\}$.
\end{cor}

\begin{rem}\label{rem:82} The set \be  \cC^\eta_{\d,s}:= \bigcap_{\nu>0} \cC^\eta_{\d,s,\nu}=\bigcup_{l=0}^{q_{\d}}h_{\d,\eta}^{-l}(\bar{\cW}^\eta_{\d,s})\label{defhatcAbis}\ee
is thus also an open set of $\C^2$.
\end{rem}

\subsection{Glueing}\label{sec:glueing}
\begin{prop}[Glueing]\label{lemma:glueing}The manifold ${\ti \cW}^\eta_{\d,s}$ obtained from ${\bar \cW}^\eta_{\d,s}$ by glueing $\Sigma_{\d,s}$ and $h_{\d,\eta}(\Sigma_{\d,s})$ using $h_{\d,\eta}$ has a natural complex structure and the canonical injection of ${\bar \cW}^\eta_{\d,s'}$ in ${\bar \cW}^\eta_{\d,s}$ yields a canonical  injection of  complex manifolds of ${\ti \cW}^\eta_{\d,s'}$ in ${\ti  \cW}^\eta_{\d,s}$. Moreover, the first return map $\hat h_{\d,\eta}$ induces a holomorphic injective  map $\cR_{\rm fr}(h_{\d,\eta}):=\ti h_{\d,\eta}:{\ti \cW}^\eta_{\d,s'}\to {\ti  \cW}^\eta_{\d,s}$ which is called a (first-return) renormalization of $h_{\d,\eta}$.
\end{prop}
\begin{proof}

To define the manifold ${\ti \cW}^\eta_{\d,s}$ we first have
 to define an atlas $\{(U_{\a},\psi_{\a})_{\a}\}$  on ${\bar \cW}^\eta_{\d,s}$ i.e. a base of neighborhoods $U_{\a}$ that defines a topology on ${\bar \cW}^\eta_{\d,s}$ together with bijective maps $U_{\a}\to  \psi_{\a}(U_{\a})\subset \R^4$ (where the $\psi_{\a}(U_{\a})$ are open set of $\R^4$) verifying the fact that  $\psi_{\a}\circ \psi_{\b}^{-1}:\psi_{\b}(U_{\b}\cap U_{\a})\to \psi_{\a}(U_{\b}\cap U_{\a})$ is a diffeomorphism between two open sets of $\R^4$.
 
For the collection $\{U_{\a}\}_{\a}$ we choose 
\begin{enumerate}
\item The open balls included in the interior of ${\bar \cW}^\eta_{\d,s}$.
\item For each open ball $B\subset\R^4$ centered at a point $p\in \Sigma_{\d,s}$,  the union $B_{in} \cup h_{\d,\eta}(B_{out})$ where $B_{in}=B\cap {\bar \cW}^\eta_{\d,s}\subset {\bar \cW}^\eta_{\d,s} $ (if the radius of $B$ is small enough)  and $B_{out}=B\setminus B_{in}$.
\end{enumerate}
For the maps $\psi_{\a}$ we choose in case (1), the identity and in case (2) the map $\psi:B_{in} \cup h_{\d,\eta}(B_{out})\ni \xi \mapsto \psi(\xi)\in\R^4$ defined by $\psi(\xi)=\xi$ if $\xi \in B_{in}$ and $\psi(\xi)=h_{\d,\eta}^{-1}(\xi)$ if $\xi \in h_{\d,\eta}(B_{out})$.

It is not difficult to check that these data define a differentiable structure on ${\ti \cW}^\eta_{\d,s}$ (which by definition is ${\bar \cW}^\eta_{\d,s}$ endowed with this topology and  differentiable structure). 

Besides, one can define a canonical  almost complex structure on ${\ti \cW}^\eta_{\d,s}$: in the preceding coordinate charts it is equal to the multiplication by $J_{0}=\bm 0& -I_{2}\\ I_{2}&0\em$ in the tangent space $T\psi_{\a}(U_{\a})$. Because $h_{\d,\eta}$ is holomorphic, the changes of coordinates $\psi_{\a}\circ \psi_{\b}^{-1}$ preserve this almost complex structure. 

Furthermore, the preceding almost complex structure is Frobenius-integrable\footnote{Equivalently, its Nijenhuis tensor vanishes.} hence, thanks to the Newlander-Nirenberg Theorem \cite{NeNi}, integrable: it defines a genuine complex structure.

\medskip The fact that the first return map $\hat h_{\d,\eta}$ induces a holomorphic injective  map $\cR_{\rm fr}(h_{\d,\eta}):=\ti h_{\d,\eta}:{\ti \cW}^\eta_{\d,s'}\to {\ti  \cW}^\eta_{\d,s}$ is then tautological.

\end{proof}

The interest of the glueing construction comes from the following simple result.
\begin{prop} If $O\subset \ti \cW^\eta_{\d,s''}$ ($0<s''<s'/10$) is a forward  invariant set for $\ti h_{\d,\eta}$ then $O\cap \bar \cW^{\eta}_{\d,s''}$ is a forward invariant set for the first return map $\hat h_{\d,\eta}$. If $O$ is the basin of an attracting set $O'\subset O$ for $\ti h_{\d,\eta}$ then $O\cap \bar \cW^{\eta}_{\d,s''}$ is the basin of the  attracting set $O'\cap \bar \cW^{\eta}_{\d,s''}$ for $\hat h_{\d,\eta}$.
\end{prop}

\subsection{Commuting pairs and normalization}\label{sec:commpair}
Another convenient way to describe the preceding glueing construction is to use the language of commuting pairs.
\begin{defin}[Commuting pairs]Let $W$ be an   open set of $\C^2$. We say that a   a couple of holomorphic diffeomorphisms $(h_{1},h_{2})$, 
$h_{1},h_{2}: W\cup h_{1}(W)\cup h_{2}(W)\to \C^2$ is a commuting pair on $W$ if  
$$\forall x\in W, \quad h_{1}(h_{2}(x))=h_{2}(h_{1}(x)).$$
We denote these data $(h_{1},h_{2})_{W}$. 
\end{defin}

Let us make some simple remarks. 

If $(h_{1},h_{2})_{W}$ is a commuting pair on $W$ and if $W'\subset W$ is an open set, one can consider its restriction $(h_{1},h_{2})_{W'}$ to $W'$. 

Commuting pairs can be conjugated: if $(h_{1},h_{2})$ is a commuting pair on $W$ and $N:W\cup h_{1}(W)\cup h_{2}(W)\to \C^2$ is an injective  holomorphic map   then 
$$N\circ(h_{1},h_{2})\circ N^{-1}:=(N\circ h_{1}\circ N^{-1},N\circ h_{2},N^{-1})$$ is a commuting pair on $N(W)$. 
In this case 
we say that the  commuting pair $(h_{1},h_{2})_{W}$  is {\it conjugate on $W$}  to the  commuting pair $(N\circ h_{1}\circ N^{-1},N\circ h_{2}\circ N^{-1})_{N(W)}$. We shall sometimes use the notations
$${\rm Ad}(N\mid W)\cdot(h_{1},h_{2})=(N\circ h_{1}\circ N^{-1},N\circ h_{2}\circ N^{-1})$$
or
$${\rm Ad}(N\mid W)\cdot\bm h_{1}\\ h_{2}\em=\bm N\circ h_{1}\circ N^{-1}\\ N\circ h_{2}\circ N^{-1}\em=N\circ \bm h_{1}\\ h_{2}\em\circ N^{-1}.$$

Let $$\cT_{1,0}:\C^2\ni (z,w)\mapsto (z+1,w)\in\C^2.$$

\begin{defin}A commuting pair  $(h_{1},h_{2})$ on $W$ is said to be normalized if for some $s,\rho,\nu>0$
$$\begin{cases}
&W=W_{s,\nu,\rho}:=(-\nu,1+\nu)_{s}\times \bD(0,\rho)\\
&h_{1}=\cT_{1,0}.
\end{cases}
$$
\end{defin}

 \begin{lemma}If $(\cT_{1,0},\ti h_{})$ is a normalized pair on $W_{s,\nu,\rho}:=(-\nu,1+\nu)_{s}\times \bD(0,\rho)$,
   the diffeomorphism $\ti h_{}$ defines a holomorphic injective map $\T_{s}\times \bD(0, \rho)\to \T_{\infty}\times \C$.
\end{lemma}
\begin{proof}
 Indeed, by definition,
 $$\forall (z,w)\in (-\nu,1+\nu)_{s}\times \bD(0,\rho), \quad \ti h_{}(z,w)+(1,0)=\ti h_{}(z+1,w)$$
 hence the map $(z,w)\mapsto \ti h(z+1,w)-\ti h(z,w)$ is constant on  $(-\nu,1+\nu)_{s}\times \bD(0,\rho)$. In this situation it's easy to prove that the map $\ph:(z,w)\mapsto \ti h_{}(z,w)-(z,0)$ extends as a holomorphic map on $\R_{s}\times \bD(0,\rho)$ which is 1-periodic in the $z$-variable; one can thus consider $\ph$ as a holomorphic function defined on $\T_{s}\times \bD(0,\rho)$. The holomorphic diffeomorphism $(z,w)\mapsto \ti h_{}(z,w)=(z,0)+\ph(z,w)$ defines a holomorphic injective map $\T_{s}\times \bD(0,\rho)\to \T_{\infty}\times \C$.
\end{proof}

\begin{defin}[Normalization]\label{def:normalization} We say a commuting pair $(h_{1},h_{2})_{W}$ on $W$ can  be normalized if it is conjugate to a normalized pair $(\cT_{1,0},\ti h)$ on $(-\nu,1+\nu)_{s}\times \bD(0,\rho)$. If $N$ is the conjugating diffeomorphism we denote $\ti h=\cR_{N}(h_{1},h_{2})$ the  holomorphic injective map $\T_{s}\times \bD(0,\rho)\to \T_{\infty}\times \C$ thus obtained.
 \end{defin}
 
 \begin{rem}The conjugating diffeomorphism $N$ is by definition a diffeomorphism $N:W\cup h_{1}(W)\cup h_{2}(W)\to W_{s,\rho,\nu}\cup \cT_{1,0}(W_{s,\rho,\nu})\cup \ti h(W_{s,\rho,\nu})$. We call $W\cup h_{1}(W)\cup h_{2}(W)$ a {\it normalization box}.
 \end{rem}

\subsection{Link with the glueing construction}
The proof of Proposition \ref{lemma:7.1} of subsection \ref{sec:7.1} yields
the following result on commuting pairs.
\begin{cor}Let $\nu\in (0,1)$. There exist $0<s'<s$  such that, for any  $\d$ small enough,  $(h_{\d,\eta},h^{q_{\d}}_{\d,\eta})$ is a commuting pair on $\cW^\eta_{\d,s',\nu}$. 
\end{cor}

Remember the definition of the  manifold   ${\ti \cW}^\eta_{\d,s}$  introduced in Proposition \ref{lemma:glueing} and obtained from ${\bar \cW}^\eta_{\d,s}$ by glueing $\Sigma_{\d,s}$ and $h_{\d,\eta}(\Sigma_{\d,s})$ using $h_{\d,\eta}$. 

Assume there exists a holomorphic diffeomorphism 
\be N:\cW^\eta_{\d,s,\nu}\to N(\cW^\eta_{\d,s,\nu})\subset \C^2\label{e8.64}\ee
and denote by ${\rm Conj}_{N}({\ti \cW}^\eta_{\d,s})$ the manifold obtained  from $N({\bar \cW}^\eta_{\d,s})$ by glueing $N(\Sigma_{\d,s})$ and $N(h_{\d,\eta}(\Sigma_{\d,s}))$ using the map $N\circ h_{\d,\eta}\circ N^{-1}$ (which is defined from a neighborhood of  $N({\bar \cW}^\eta_{\d,s})$ to a neighborhood of $N(h_{\d,\eta}(\Sigma_{\d,s}))$). One can then define  a tautological  holomorphic diffeomorphism $\ti N:{\ti \cW}^\eta_{\d,s}\to  {\rm Conj}_{N}({\ti \cW}^\eta_{\d,s})$ and  a diffeomorphism 
\be {\rm Conj}_{N}(\ti h_{\d,\eta}):=\ti N\circ \ti h_{\d,\eta}\circ {\ti N}^{-1}:{\rm Conj}_{N}({\ti \cW}^\eta_{\d,s'})\to {\rm Conj}_{N}({\ti \cW}^\eta_{\d,s}).\label{e8.65}\ee

\begin{rem}Since $N$ is not defined globally (it is only defined on $\cW^\eta_{\d,s,\nu}$), the diffeomorphism ${\rm Conj}_{N}(\ti h_{\d,\eta})$ is not  associated  to a first return map in a direct way.
\end{rem}

If the map $N$ in  (\ref{e8.64}) satisfies on a open neighborhood of $\Sigma_{\d,s}$ 
$$N\circ h_{\d,\eta}(z,w)=N(z,w)+(1,0),$$
the manifold ${\rm Conj}_{N}({\ti \cW}^\eta_{\d,s})$ is obtained by glueing $N(\Sigma_{\d,s})$ and $(1,0)+N(\Sigma_{\d,s})$ by the map $(z,w)\mapsto (z+1,w)$. Furthermore, if $\d$ is small enough, one can find a $C^\infty$-diffeomorphism commuting with $(z,w)\mapsto (z+1,w)$ and  sending $N(\Sigma_{\d,s})$ to some  $(\{0\}+i(-\check s,\check s))\times \bD(0,\check \rho)$, henceforth  $N(h_{\d,\eta}(\Sigma_{\d,s}))$ to $(\{1\}+i(-\check s,\check s))\times \bD(0,\check \rho)$. The manifold  ${\rm Conj}_{N}({\ti \cW}^\eta_{\d,s})$   is thus $C^\infty$-diffeomorphic to  an open cylinder i.e. the product of an open annulus by an open disk\footnote{It is a little bit more complicate to prove this in the  holomorphic category. This can be done by using the following version of the Newlander-Nirenberg theorem on cylinders: an integrable almost complex structure which is $C^k$-close ($k$ large enough) to the standard complex structure $J_{0}$ is conjugate to $J_{0}$.}. 

\medskip
An  examination of the glueing construction shows the following result.
\begin{prop}Assume that $N_{\d,\eta}$ is a normalization map for the commuting pair $(h_{\d,\eta},h^{q_{\d}}_{\d,\eta})$ on $\cW^\eta_{\d,s,\nu}$. Then, on the complex manifold ${\rm Conj}_{N}(\ti W^\eta_{\d,s''})$ ($0<s''<s'$) one has 
$${\rm Conj}_{N}(\cR_{\rm fr}(h_{\d,\eta}))=\cR_{N}(h_{\d,\eta},h^{q_{\d}}_{\d,\eta}).$$
\end{prop}
\begin{proof}Because the first return map $\hat h_{\d,\eta}$ to $\bar \cW^\eta_{\d,s,\nu}$ is either $h_{\d,\eta}^{q_{\d}}$ or $h_{\d,\eta}^{q_{\d}+1}$, 
the map ${\rm Conj}_{N}(\ti h_{\d,\eta})$  (cf. (\ref{e8.65}))  takes the form
$${\rm Conj}_{N}(\ti h_{\d,\eta}):(z,w)\mapsto \cR_{N}(h_{\d,\eta},h^{q_{\d}}_{\d,\eta}) \mod (\Z,0).$$
\end{proof}

\subsection{Existence of normalization maps}

A normalization maps  $N$ can be seen as a {\it uniformization} map i.e. a diffeomorphism achieving  the uniformization of the complex manifold ${\ti \cW}^\eta_{\d,s}$ (to the product of a annulus by a disk). In some important cases one can prove they exist.
\begin{theo}\label{theo:normalization}If $X$ has constant divergence, $\eta$ is of the form (\ref{formetat}) and $\d$ is small enough,  the commuting pair $(h_{\d,\eta},h^{q_{\d}}_{\d,\eta})$ can be normalized on $\cW^\eta_{\d,s',\nu}$.
\end{theo}
As we shall soon see, Theorem \ref{theo:normalization} is  a consequence of  Theorem \ref{cor:2.10:commpair} of  Section \ref{sec:panormpair}  on {\it partial normalization}.

\medskip
Recall the notation
$$\cT_{1,\b}:\C^2\ni (z,w)\mapsto (z+1,e^{2\pi i \b}w)\in\C^2.$$

\begin{defin}[Partial normalization]A commuting pair  $(h_{1},h_{2})$ on $W$ is said to be partially normalized if for some $s,\rho,\nu>0$, $\b\in\C$
$$\begin{cases}
&W=W_{s,\nu,\rho}:=(-\nu,1+\nu)_{s}\times \bD(0,\rho)\\
&h_{1}=\cT_{1,\b}.
\end{cases}
$$
A commuting pair $(h_{1},h_{2})$ on $W$ can be partially normalized if it conjugate to a partially normalized pair.
\end{defin}
\begin{lemma}Partially normalized pair can be normalized.
\end{lemma}
\begin{proof}Indeed, the map 
\be \Psi_{\b_{}}:\C^2\ni (z,w)\mapsto (z,e^{- 2\pi i\b_{} z }w)\in \C^2\label{def:Psinew}\ee
satisfies
$$\Psi_{\b}\circ \cT_{1,\b}\circ \Psi_{\b}^{-1}=\cT_{1,0}:(z,w)\mapsto (z+1,w).$$
\end{proof}
As we mentioned, the existence of such partial normalization maps is the content of Theorem \ref{cor:2.10:commpair}.  The preceding discussion allows us to reformulate Theorem \ref{theo:normalization} as follows: 
\begin{theo}If $\d$ is small enough, the commuting pair $(h_{\d,\eta},h^{q_{\d}}_{\d,\eta})$ can be partially normalized on $\cW_{\d,s',\nu'}$. It can hence be normalized.
\end{theo}

\bigskip
\section{A criterion for the existence of  rotation domains or Herman rings}\label{sec:criterion}

The aim of this Section is essentially  to prove that if the renormalization $\ti h_{\d,\eta}$ associated to the diffeomorphism $h_{\d,\eta}$ defined in Section \ref{sec:renormandcommpairs} (see Proposition \ref{lemma:glueing})  has a rotation domain resp. an attracting invariant annulus, then the same property holds for $h_{\d,\eta}$; see Propositions \ref{invariantannulus}, \ref{invariantannulusbis}, \ref{tho:rank2rotdomain}. To make the statements more precise we use the language of commuting pairs. 

\medskip

We assume the assumptions of  Proposition \ref{lemma:7.1}  are satisfied. In particular  the set ${\bar \cW}^\eta_{\d,s}$ is a first return domain of  $(h_{\d,\eta},{\bar \cW}^\eta_{\d,s'})$.
We also assume that for some $a>0$ large enough one has 
\be p>20(a+1).\label{pvsabisante}\ee

Recall the notation for $\a,\b\in\C$
$$\cT_{\a,\b}:(z,w)\mapsto(z+\a,e^{2\pi i\b}w).$$

\bigskip In addition to Assumptions \ref{assump:8.1} and \ref{assump:8.2}, we make in this section the following hypothesis:

\begin{assumption}[Linearization assumption]\label{assump:8.3}There exist $\check \nu,\check s,\check \rho$ (which are positive and  $\asymp 1$), 
$$\check \a \in (-1,0),\qquad \check \b\in \C,\ \Im\check\b\geq 0$$
 such that: the commuting pair $(h_{\d,\eta},h^{q_{\d}}_{\d,\eta})$ is defined on some open set $\check  \cW^{\d,\eta}_{\check s,\check \nu}$ and conjugate  to the normalized commuting pair $(\cT_{1,0},\cT_{\check \a,\check \b})$ defined on $(-\check \nu,1+\check \nu)_{\check s}\times \bD(0,\check s)$  by    a holomorphic diffeomorphism $N_{\d,\eta}$ (hence 
$\check  \cW^{\d,\eta}_{\check s,\check \nu}=N_{\d,\eta}^{-1}((-\check \nu,1+\check \nu)_{\check s}\times \bD(0,\check s))$)
which  satisfies (we here refer to notations (\ref{new-notation}))
\be
\begin{cases}
& \cW_{\d,\d^{p/2+2},\nu}^\eta\subset \check \cW^{\d,\eta}_{\check s,\check \nu}\subset \cW_{\d,s',\nu}^\eta\\
&N_{\d,\eta}^{-1}(0,0)\in\bD_{\C^2}(\zeta,\d^{p-a}),\\
&(N_{\d,\eta}^{-1})_{*}\pa_{z}=\d X+O(\d^{p/2-a}).
\end{cases}
\label{cond9.62}
\ee
\end{assumption}

\bigskip 
We then define 
\be W:=W_{\check s,\check\nu}=\biggl((-\check \nu,1+\check \nu)+i(-\check s,\check s)\biggr)\times \bD(0,\check s).\label{Wchecksnu}\ee
$$\check{\cW}^{\d,\eta}_{\check s,\check \nu}=N_{\d,\eta}^{-1}(W_{\check s,\check\nu})$$
$$\check \cC^{\d,\eta}_{\check s,\check \nu}=\bigcup_{l=0}^{q_{\d}}h_{\d,\eta}^{-l}(\check{\cW}^{\d,\eta}_{\check s, \check \nu}).$$
To keep simple notations, we set
$$\check\cW_{\check s,\check \nu}=\check{\cW}^{\d,\eta}_{\check s,\check \nu}\quad \textrm{and}\quad \check \cC_{\check s,\check \nu}=\check \cC^{\d,\eta}_{\check s,\check \nu}.$$

\begin{rem}\label{Ccontains}The first condition of (\ref{cond9.62}) shows that 
\be  \cC^\eta_{\d,\d^{p/2+2},\nu}\subset \check \cC_{\check s,\check \nu}\label{eq:Cccontains}\ee
(see the definition (\ref{defhatcA}) of $ \cC^\eta_{\d,\d^{p/2+2},\nu}$)
hence by Corollary \ref{prop:iteratesfrd} of Section \ref{sec:8.2} $\check \cC_{\check s,\check \nu}$ contains $ \cC^\eta_{\d,\d^{(2/3)p},\nu}$ which is  a $\d^{(2/3)p-1}$-neighborhood of the $T$-periodic orbit $(\phi^t_{X}(\zeta))_{t\in\R}$ (cf. condition (\ref{pvsabisante}) on $p$).
\end{rem}

The fact that $p-3\geq p/2-1$ ($p\geq 10$), the  first condition of (\ref{cond9.62}) and Lemma \ref{lemma:connectedness} (see also  Corollary \ref{prop:iteratesfrd}) show that  $\check C_{\check s, \check \nu}$ is connected as well as all the intersections 
$$k,l\in [0,q_{\d}]\cap\N,\quad  h_{\d,\eta}^{-k}(\check{\cW}_{\check s,\check \nu})\cap h_{\d,\eta}^{-l}(\check{\cW}_{\check s,\check \nu}).$$
Also, it holds that (see Remark \ref{rem:81})
\be \forall k, l\in [0,q_{\d}]\cap\N,\  h_{\d,\eta}^{-k}(\check{\cW}_{\check s,\check\nu})\cap h_{\d,\eta}^{-l}(\check{\cW}_{\check s,\check\nu})\ne  \emptyset\implies \g(k,l)=1\label{9.64}\ee
where we've set $\g(k,l)=0$ if $ |k-l|\notin \{0,1,q_{\d}\}$ and 1 otherwise.

\medskip
We can define the atlas $\{(h^{-l}_{\d,\eta}(\check \cW_{\check s, \check \nu}),\psi_{l})\mid l\in\{0,\ldots q_{\d}\}$ of $\check\cC_{\check s , \check \nu}$ where 
$$\psi_{l}:h^{-l}_{\d,\eta}(\check \cW_{\check s, \check \nu})\ni \xi\mapsto \psi_{l}(\xi)=N\circ h^l_{\d,\eta}(\xi)\in \C^2.$$
\begin{lemma}\label{lemma:matching}For any $(k,l)\in\{0,\ldots,q_{\d}\}$ such that $\g(k,l)=1$ (i.e. $h_{\d,\eta}^{-k}(\check{\cW}_{\check s,\check \nu})\cap h_{\d,\eta}^{-l}(\check{\cW}_{\check s,\check\nu})\ne  \emptyset$, see (\ref{9.64})), the  transition  maps  $\psi_{k}\circ \psi_{l}^{-1}$ are of the form $\cT_{\a_{k,l},\b_{k,l}}$ for $\a_{k,l}\in\R$, $\b_{k,l}\in\C$. If $|k-l|=1$ then $\b_{k,l}=0$ and  $\b_{q_{\d},0}=\check \b$, $\b_{0,q_{\d}}=-\check \b$.  (If $\check \a$ and $\check \b$ are real then $\b_{k,l}\in\R$.)
\end{lemma}
\begin{proof}One has when $\gamma(k,l)=1$ (wherever it makes sense)
\begin{align*}
\psi_{k}\circ \psi_{l}^{-1}=N_{\d,\eta}\circ h_{\d,\eta}^{k-l}\circ N_{\d,\eta}^{-1}.
\end{align*}
If we assume $k\geq l$, this is clear when $0\leq l\leq q_{\d}-1$ and $k=l+1$ (then $\b_{k,l}=0$) because $N_{\d,\eta}\circ h_{\d,\eta}\circ N_{\d,\eta}^{-1}=\cT_{1,0}$. It is also true in the case  $k=q_{\d}$, $l=0$ because $N_{\d,\eta}\circ h_{\d,\eta}^{q_{\d}}\circ N_{\d,\eta}^{-1}=\cT_{\check \a,\check \beta}$. The case $k<l$ is treated similarly.
\end{proof}

\begin{rem}\label{rem:flow}Note that if $\Im\check \b\geq 0$ and if  $\xi \in h_{\d,\eta}^{-k}(\check\cW_{\check s,\check\nu})$ is a point such that $\psi_{k}(\xi)=(z,w)\in W_{\check s, \check\nu}:=(-\check \nu,1+\check\nu)_{\check s}\times \bD(0,\check s)$ with $\Re z\geq 1+\check \nu/2$,  then:
\begin{enumerate}
\item in the case $k\in \{1,\ldots,q_{\d}\}$,   it also belongs to  $ h_{\d,\eta}^{-(k-1)}(\check\cW_{\check s, \check\nu})$ and  $\psi_{k-1}(\xi)=(z+\a_{k-1,k},e^{2\pi i \b_{k-1,k}}w)$ with $\a_{k-1,k}=-1$, $\check \b_{k-1,k}=0$; 
\item in the case  $k=0$,  it also belongs to $ h_{\d,\eta}^{-q_{\d}}(\check\cW_{\check s,\check\nu})$ and  $\psi_{q_{\d}}(\xi)=(z+\a_{q_{\d},0},e^{2\pi i \b_{q_{\d},0}}w)$ where  $\a_{q_{\d},0}=\check \a\in (-1,0)$, $\check \b_{q_{\d},0}=\check \b$; it  is in 
$ W_{\check s,\check \nu}$ because $\Im \check \b\geq 0$ and $\check \a\in (-1,0)$.
 \end{enumerate}
 On the other hand, if  $\xi \in h_{\d,\eta}^{-k}(\check\cW_{\check s,\check\nu})$ is a point such that $\psi_{k}(\xi)=(z,w)\in W_{\check s, \check\nu}:=(-\check \nu,1+\check\nu)_{\check s}\times \bD(0,\check s)$ with $\Re z\leq -\check \nu/2$,  then when $k\in \{0,\ldots,q_{\d}-1\}$,   it also belongs to  $ h_{\d,\eta}^{-(k+1)}(\check\cW_{\check s,\check\nu})$; but if $k=q_{\d}$, it does not necessarily belong to $\check\cW_{\check s,\check\nu}$ {\it except} in the case  $\Im\check \b=0$.

\end{rem}

\subsection{Invariant annulus and rotation domains}
\bigskip The main result of this section is the proof of the following theorems.

\begin{theo}[Normal family]\label{lemma:9.6} If $\Im\check\b\geq 0$, the  bounded open set $\check  \cC_{\check s,\check \nu}$  is invariant by $h_{\d,\eta}$  and the family $(h^n_{\d,\eta}\mid \check \cC_{\check s,\check \nu})_{n\in\N}$ is  thus normal. Furthermore, $\check  \cC_{\check s,\check \nu}$ is connected and  contains a $\d^{(2/3)p}$ neighborhood of the $T$-periodic orbit $(\phi^t_{X}(\zeta))_{t\in\R}$ (see Remark \ref{Ccontains}).
\end{theo}

\begin{theo}[Invariant annulus]\label{invariantannulus}If 
 $(\check\a,\check\b)$ is non resonant, there exists an $h_{\d,\eta}$-invariant (relatively compact) annulus $\cA_{\d,\eta}$ ($\ne\emptyset$) included in   $\check \cC_{\check s,\check \nu}$  on which the diffeomorphism $h_{\d,\eta}$ is conjugate to a translation the rotation number of which   satisfies 
\be \a=\frac{\d}{T}+O(\d^2)\label{c8.62}\ee
where $T$ is the period of the  orbit $(\phi^t_{X}(\zeta))_{t\in\R}$ associated to the vector field $X$, cf. (\ref{periodtofX}). Moreover, one can choose $\cA_{\d,\eta}$ such that it is included in a $\d^{(2/3)p+1}$-neighborhood of the periodic orbit $\{\phi_{X}^\th(\zeta)\mid \th\in \R\})$.
\end{theo}
\begin{theo}[Dissipative case]\label{invariantannulusbis}
 If furthermore $\Im\check\b>0$, this annulus is attracting and has a non empty (open) basin of attraction in  $\check \cC_{\check s,\check \nu}$. Moreover, this invariant annulus is $\d^{(2/3)p}$-isolated in the following sense: if $\cA'$ is any other $h_{\d,\eta}$-invariant annulus (on which the dynamics of $h_{\d,\eta}$ is conjugate to a rotation) such that ${\rm dist}(\cA_{\d,\eta},\cA')\leq \d^{(2/3)p}$ then their intersection contains a nonempty $h_{\d,\eta}$-invariant annulus.
\end{theo}

\begin{theo}[Conservative case]\label{tho:rank2rotdomain}If  $(\check \a,\check\b)$ is non resonant and $\check \b\in\R$, then, $\check  \cC_{\check s,\check \nu}$ is a rank-2 rotation domain for $h_{\d,\eta}$: there exist a holomorphic diffeomorphism   map $\Phi^{-1}:\check  \cC_{\check s,\check \nu}\to \T_{\check s}\times \bD(0,\check s)$ that conjugates $(h_{\d,\eta}\mid \check  \cC_{\check s,\check \nu})$ to  the map
$$\T_{\check s}\times \bD(0,\check s)\ni (\th,r)\mapsto (\th+\a_{}, e^{2\pi i\b_{}}r)\in  \T_{\check s}\times \bD(0,\check s)$$
($\a$ from Theorem \ref{invariantannulus}) and $(\a,\b)\in\R^2$ is non resonant. 
\end{theo}

\subsection{Two commuting vector fields and the proof of Theorem \ref{lemma:9.6}}

We define the following two commuting  vector fields  on $W_{\check s,\check \nu}$
$$\begin{cases}&\Theta_{\d,\eta}=(N_{\d,\eta}^{-1})_{*}\pa_{z}\\
&R_{\d,\eta}=(N_{\d,\eta}^{-1})_{*}(2\pi iw\pa_{w})
\end{cases}$$
 ($[\Theta_{\d,\eta},R_{\d,\eta}]=0$).
Because the vector fields $\pa_{z}$ and $w\pa_{w}$ are equivariant w.r.t. any map of the form $\cT_{\a,\b}$, $\a,\b\in\C$ (i.e. $(\cT_{\a,\b})_{*}\pa_{z}=\pa_{z}$, $(\cT_{\a,\b})_{*}(iw\pa_{w})=iw\pa_{w}$ whenever it makes sense), we can by using Lemma  \ref{lemma:matching}, extend these vector fields to  the open set  $\check{\cC}_{\check s,\check \nu}$ as commuting holomorphic vector fields by setting
$$\Theta_{\d,\eta}\mid h_{\d,\eta}^{-l}(\check{\cW}_{\check s,\check \nu})= (\psi_{l}^{-1})_{*}\pa_{z},\qquad R_{\d,\eta}\mid h_{\d,\eta}^{-l}(\check{\cW}_{\check s,\check \nu})= (\psi_{l}^{-1})_{*}(2\pi iw\pa_{w}).$$

In any coordinate chart $(h_{\d,\eta}^{-l}(\check{\cW}_{\check s,\check \nu}),\psi_{l})$ the vector fields $\Theta_{\d,\eta}$ and $R_{\d,\eta}$ take respectively  the form $\pa_{z}$ and $2\pi iw\pa_{w}$.

As a consequence,
for any $\zeta_{1},\zeta_{2}\in \C$ small enough  and $k,l\in \{0,\ldots ,q_{\d}\}$,  the flow, when it is defined,    
\begin{align*}\psi_{k}\circ (\phi^{\zeta_{1}}_{\Theta_{\d,\eta}}\circ \phi^{\zeta_{2}}_{R_{\d,\eta}})\circ \psi_{l}^{-1}&=(\psi_{k}\circ \psi_{l}^{-1})\circ \psi_{l}\circ (\phi^1_{\zeta_{1}\Theta_{\d,\eta}+\zeta_{2}R_{\d,\eta}})\circ \psi_{l}^{-1}\\
&=(\psi_{k}\circ \psi_{l}^{-1})\circ  (\phi^1_{\zeta_{1}(\psi_{l})_{*}\Theta_{\d,\eta}+\zeta_{2}(\psi_{l})_{*}R_{\d,\eta}})\\
&=(\psi_{k}\circ \psi_{l}^{-1})\circ\phi^1_{\zeta_{1}\pa_{z}+2\pi i\zeta_{2}w\pa_{w}}
\end{align*}
 takes the form 
 \be (z,w)\mapsto (z+\a_{k,l}+\zeta_{1},e^{2\pi i(\zeta_{2}+\b_{k,l}) }w)\label{rem:formvf} \ee for some  $\a_{k,l}\in \R$, $\b_{k,l}\in \C$.

\begin{lemma}\label{lemma:flowisdefined}\begin{enumerate}
\item If $\Im \check \b\geq 0$ the flow $\phi_{\Theta_{\d,\eta}}^t$ is defined on $\check\cC^{}_{\check s,\check \nu}$ for all $t\geq 0$  and the flow of $\phi_{R_{\d,\eta}}^t$ for any $t\in\R$.
\item If $\Im \check \b=0$ both  flows $\phi_{\Theta_{\d,\eta}}^t$ and $\phi_{R_{\d,\eta}}^t$  are defined on $\check\cC^{}_{\check s,\check \nu}$ for any $t\in\R$.
\end{enumerate}
\end{lemma}
\begin{proof}

\mn (1) From (\ref{rem:formvf} )  
$$\psi_{k}\circ \phi^{t}_{\Theta_{\d,\eta}}\circ \psi_{k}^{-1}:(z,w)\mapsto (z+t,w).$$
Thus, the only way the flow $\psi_{k}\circ \phi^t_{\Theta_{\d,\eta}}\psi_{k}^{-1}(\xi)$ stops to be defined is when it reaches from the left the right  boundary $\{(z,w)\mid \Re z=1+\check \nu,\ |\Im z|\leq \check s, |w|\leq \check \rho \}$. But in this situation, Remark \ref{rem:flow} tells us that  it belongs to some $\psi_{l}(W_{\check s,\check \nu})$ where the flow $\psi_{l}\circ \phi^t_{\Theta_{\d,\eta}}\psi_{l}^{-1}(\xi)$ can be continued (on the right of $t$).

\mn (2) The same Remark \ref{rem:flow}  shows that when $\Im \check \b=0$ the flow can be defined for all $t\in \R$.

\medskip The fact that the flow $\phi^t_{R_{\d,\eta}}$  is defined for all $t\in \R$ (when $\Im\check \b\geq 0$ ) is done in a similar and simpler way.

\end{proof}
\begin{lemma}On $ \check \cC_{\check s,\check \nu}$ one has
\be  h_{\d,\eta}=\phi^1_{\Theta_{\d,\eta}},\qquad h_{\d,\eta}^{q_{\d}}=\phi^1_{\check \a\Theta_{\d,\eta}}\circ \phi^1_{\check \b R_{\d,\eta}} \label{eqhthtatr}\ee
\end{lemma}
\begin{proof}
Both $h_{\d,\eta}$ and  $\phi^1_{\Theta_{\d,\eta}}$ act as $$h^{-1}_{\d,\eta}(W_{\check s,\check \nu})\cap W_{\check s, \check \nu}\ni (z,w)\mapsto (z+1,w)\in  W_{\check s,\check \nu}\cap h_{\d,\eta}(W_{\check s, \check \nu}) $$ 
and both $h^{q_{\d}}_{\d,\eta}$ and  $\phi^1_{\check\a \Theta_{\d,\eta}}\circ \phi^1_{\check\b R_{\d,\eta}}$ act as 
$$h^{-q_{\d}}_{\d,\eta}(W_{\check s, \check \nu})\cap W_{\check s, \check \nu}\ni (z,w)\mapsto (z+\check \a,e^{2\pi i\check \b}w)\in  W_{\check s, \check \nu}\cap h^{q_{\d}}_{\d,\eta}(W_{\check s, \check \nu}).$$ 
 In particular on these open sets
$$  h_{\d,\eta}=\phi^1_{\Theta_{\d,\eta}},\qquad h_{\d,\eta}^{q_{\d}}=\phi^1_{\check \a\Theta_{\d,\eta}}\circ \phi^1_{\check \b R_{\d,\eta}} $$
which implies that these  relations hold on the whole open connected set $ \check \cC_{\check s,\check \nu}$.
\end{proof}

As a Corollary of the two previous Lemmas we can state:

\begin{cor}Theorem \ref{lemma:9.6} holds true.
\end{cor}

\subsection{Invariant circles, invariant tori}\label{Invariant circles, invariant tori}

 Let $r\in (-\check s,\check s)$, $\rho_{}\in [0,\check s)$. 
 We define the following subsets of $\check C_{\check s,\check \nu}$:
 \begin{enumerate}
 \item The set $\hat B_{r,\rho}$ of $\xi \in \check \cC_{\check s,\check \nu}$ such that in some coordinate chart $(h_{\d,\eta}^{-l}(\check{\cW}_{\check s,\nu}),\psi_{l})$, the point $(z_{l},w_{l}):=\psi_{l}(\xi)\in W_{\check s,\check\nu}$ satisfies $$|\Im z_{l}|< |r|,\qquad |w_{l}|<\rho.$$
 \item The set $B_{r,\rho}$ of $\xi \in \check \cC^\eta_{\check s,\check \nu}$ such that in some coordinate chart $(h_{\d,\eta}^{-l}(\check{\cW}_{\check s,\check \nu}),\psi_{l})$, the point $(z_{l},w_{l}):=\psi_{l}(\xi)\in W_{\check s,\check\nu}$ satisfies $$\Im z_{l}=r,\qquad |w_{l}|= \rho.$$
 In particular, $B_{r,0}$ is the set of points such that $\Im z_{l}=r, w_{l}=0.$
 \end{enumerate}
 Note that
 $$\check \cC_{\check s,\check \nu}=\hat B_{\check s,\check s}.$$
 \begin{lemma} \label{posinvsets}\begin{enumerate}
 \item \label{posinvsets-i1} If $\Im\check \b\geq 0$, the set $\hat B_{r,\rho}$ is open, connected,  invariant by the positive flow $(\phi^t_{\Theta_{\d,\eta}})_{t\in\R_{+}}$ (hence forward invariant by $h_{\d,\eta}$).
 \item If $\Im\check \b\geq 0$, the set $B_{r,0}$ is compact, connected and  invariant by the  flow $(\phi^t_{\Theta_{\d,\eta}})_{t\in\R_{}}$ (hence  forward and backward invariant by  $h_{\d,\eta}$).
\item\label{posinvsets-i2} If $\Im \check \b=0$, the set $B_{r,\rho}$ is connected, compact and invariant by $\phi^{t_{1}}_{\Theta_{\d,\eta}}\circ \phi^{t_{2}}_{R_{\d,\eta}}$ for any $t_{1},t_{2}\in\R$ (hence also by $h_{\d,\eta}$).
 \end{enumerate}
 \end{lemma}
 \begin{proof}
 
 \mn (1) The fact that $\hat B_{r,\rho}$ is open is clear and its connectedness comes from the following  chain condition: for any $k,l\in \{0,\ldots,q_{\d}\}$, $l\leq k$,  there exist $l_{0}=l,\ldots, l_{m}=k$ in $\{0,\ldots,q_{\d}\}$ such that $\g(l_{n},l_{n+1})=1$ ($0\leq n\leq m-1$). 
 
 To prove it is invariant by the positive flow of $\Theta_{\d,\eta}$ one proceeds like in the proof of Lemma \ref{lemma:flowisdefined}. 
  
 \mn (2) The connectedness of $B_{r,0}$ and its  invariance by the flow  is proved like in (1). 
 
 To prove it is compact we just need to check it  is a closed subset of $\C^2$ (because it is bounded) a fact which is not difficult to establish if one has in mind Remark \ref{rem:flow}.
 
 \mn (3) Done the same way as in (1) and (2).

 \end{proof}

\begin{lemma}\label{lemma:periodsvf}
\begin{enumerate}
\item\label{it1} If $\Im \check \b\geq 0$, the set $B_{r,0}$ is a circle invariant by the flow of $\Theta_{\d,\eta}$. There exists $T_{1}\in\R_{+}^*$ (we choose it minimal) such that $\phi_{\Theta_{\d,\eta}}^{T_{1}}=id$.
\item\label{it2} Assume $\check\b\in \R$.
There exist a matrix $A=\bm a& b\\ c& d\em\in SL(2,\Z)$ and  two non zero real numbers $T_{1},T_{2}$ (depending on $\d,\eta$) such that if one sets 
\be
\bm\ti \Theta_{\d,\eta}\\ \ti R_{\d,\eta}\em=\bm a& b\\ c& d\em\bm\Theta_{\d,\eta}\\  R_{\d,\eta}\em\label{abcd}
\ee
one has 
\be 
\begin{cases}
&\phi^{T_{1}}_{\ti \Theta_{\d,\eta}}=id\\
&\phi_{\ti R_{\d,\eta}}^{T_{2}}=id.
\end{cases}
\label{eq:T1T2}
\ee
Furthermore, if $\rho_{}\ne 0$, the set $B_{r,\rho_{}}$ is a real 2-torus and if $\rho_{}=0$ it is a circle.
\end{enumerate}
\end{lemma}
\begin{proof}

\mn (1) By Lemma \ref{posinvsets}, one can define for any $\xi \in B_{r,0}$ the action 
$$(\R,+)\ni t_{1}\mapsto \phi^{t_{1}}_{\Theta_{\d,\eta}}(\xi)\in B_{r,0}.$$
One can check this action   is locally  transitive and closed\footnote{The image of a closed set is closed.};  its image is thus  a compact connected  subset of $ B_{r,0}$ and is hence equal to  $B_{r,0}$. Because $ B_{r,0}$ is compact and $\R$ is not, the set of $t\in\R^*$  such that  $\phi^{t}_{\Theta_{\d,\eta}}(\xi)=\xi$ is an abelian subgroup $T_{1}\Z$ of $\R$. The quotient map 
$$(\R/\Z,+)\ni t\mapsto \phi^{tT_{1}}_{\Theta_{\d,\eta}}(\xi)\in B_{r,0}$$
is then a diffeomorphism.

\mn (2) Similarly, for any point $\xi\in B_{r,\rho}$ one can define  the action
$$(\R^2,+)\ni (t_{1},t_{2})\mapsto \phi^{t_{1}}_{\Theta_{\d,\eta}}\circ \phi_{R_{\d,\eta}}^{t_{2}}(\xi)\in B_{r,\rho_{}}$$
which is locally transitive and closed and the image of which coincides with $B_{r,\rho_{}}$. 

The  set of $(t_{1},t_{2})\in \R^2$  such that  $\phi^{t_{1}}_{\Theta_{\d,\eta}}\circ \phi_{R_{\d,\eta}}^{t_{2}}(\xi)=\xi$ is a cocompact abelian subgroup $\Gamma$ of $\R^2$ and the quotient map 
$$(\R^2/\Gamma,+)\ni (t_{1},t_{2})\mapsto \phi^{t_{1}}_{\Theta_{\d,\eta}}\circ \phi_{R_{\d,\eta}}^{t_{2}}(\xi)\in B_{r,\rho_{}}$$
is a diffeomorphism.

Furthermore, there exist a matrix $A=\bm a& b\\ c& d\em\in GL(2,\Z)$ and $T_{1},T_{2}>0$ such that 
 $$\Gamma=\biggl\{\bm a&c\\ b& d\em \bm n_{1}T_{1}\\ n_{2}T_{2}\em\mid (n_{1},n_{2})\in\Z^2\biggr\}.$$
 Setting $$
\bm\ti \Theta_{\d,\eta}\\ \ti R_{\d,\eta}\em=\bm a& b\\ c& d\em\bm\Theta_{\d,\eta}\\  R_{\d,\eta}\em
$$
we see that
$$ \phi^t_{\ti \Theta_{\d,\nu}}\circ \phi^s_{\ti R_{\d,\nu}}=\phi^{at+cs}_{\Theta_{\d,\nu}}\circ \phi^{bt+ds}_{R_{\d,\nu}}$$
hence
 $\phi^t_{\ti \Theta_{\d,\eta}}\circ \phi^s_{\ti R_{\d,\eta}}(\xi)=\xi$ if and only if $(t,s)\in T_{1}(\Z,0)\oplus T_{2}(0,\Z)$.
 The fact that  (\ref{rem:formvf}) holds for any $\zeta_{1},\zeta_{2}\in \C$ small enough and the relation $(\phi^t_{\ti \Theta_{\d,\eta}}\circ \phi^s_{\ti R_{\d,\eta}})(\phi^{\zeta_{1}}_{\ti \Theta_{\d,\eta}}\circ \phi^{\zeta_{2}}_{\ti R_{\d,\eta}}(\xi))=(\phi^{\zeta_{1}}_{\ti \Theta_{\d,\eta}}\circ \phi^{\zeta_{2}}_{\ti R_{\d,\eta}})( \phi^t_{\ti \Theta_{\d,\eta}}\circ \phi^s_{\ti R_{\d,\eta}}(\xi))=\phi^{\zeta_{1}}_{\ti \Theta_{\d,\eta}}\circ \phi^{\zeta_{2}}_{\ti R_{\d,\eta}}(\xi)$   show that (\ref{eq:T1T2}) must hold everywhere.

As a consequence,  the quotient map
$$(\R^2/\Z^2,+)\ni (t_{1},t_{2})\mapsto \phi^{t_{1}T_{1}}_{\ti \Theta_{\d,\eta}}\circ \phi_{\ti R_{\d,\eta}}^{t_{2}T_{2}}(\xi)\in B_{r,\rho}$$
is a diffeomorphism.

\medskip
When $\rho=0$, the orbit $\phi_{\ti \Theta_{\d,\eta}}^t(\xi)$ coincides with $\phi_{\Theta_{\d,\eta}}^t(\xi)$.

\end{proof}

\begin{rem}\label{rem:8.6}When $\xi\in B_{r,0}$ the $T_{1}$-periodic orbits $(\phi^t_{\ti \Theta_{\d,\eta}}(\xi))_{t\in\R}$ and $(\phi^t_{\Theta_{\d,\eta}}(\xi))_{t\in\R}$ coincide.
Besides, the construction of the vector field $\Theta_{\d,\eta}$ shows that 
\be |T_{1}-q_{\d}|\leq 1.\label{9.70}\ee
Indeed, if $\xi \in B_{r,0}\cap h_{\d,\eta}^{-q_{\d}}(B_{r,0})$ (a nonempty set which is included  in $\cW_{\check s,\check\nu}$),  the $T_{1}$-orbit $(\phi^{t}_{\Theta_{\d,\eta}}(\xi))_{t\geq 0}$ visits the sets $h_{\d,\eta}^{-l}(\cW_{\check s,\check\nu})$, $l=q_{\d}-1,\ldots, 1$ before coming back to $\cW_{\check s,\check\nu}$.
\end{rem}

The following lemma gives a better estimate on $T_{1}$.
\begin{lemma}\label{lemma:8.12}The rotation number ${\rm rot}(h_{\d,\eta}\mid B_{0,0})$ satisfies 
$${\rm rot}(h_{\d,\eta}\mid B_{0,0})=\frac{\d}{T}+O(\d^2)$$
($T$ given by (\ref{periodtofX})).
\end{lemma}
\begin{proof}

\mn (1) We first observe that  on $\check \cC_{\check s,\check \nu}$  one has
\be \sup_{\check \cC_{\check s,\check \nu}}\|\Theta_{\d,\eta}-\d X\|\lesssim \d^{p/2-a-1}.\label{9.71}\ee
 Indeed,  the third estimate of (\ref{cond9.62}) yields 
$$\sup_{\check \cW_{\check s,\check \nu}}\| \Theta_{\d,\eta}-\d X\|\lesssim \d^{p/2-a}.$$
Besides, since  $$h_{\d,\eta}=\phi^1_{\d X_{}}\circ (id+\eta),\qquad \eta =O(\d^p),$$
we have for $l\in\{0,\ldots,q_{\d}\}$
$$h_{\d,\eta}^{-l}=\phi^{-l}_{\d X_{}}\circ (id+O(\d^{p-1}))=\phi^{1}_{-l\d X}\circ (id+O(\d^{p-1}))$$
and since $q_{\d}\asymp \d^{-1}$, 
we see that 
 one has on $h_{\d,\eta}^{-l}(\check \cW_{\check s,\check \nu})$
$$\sup_{h_{\d,\eta}^{-l}(\check \cW_{\check s,\check \nu})}\|\Theta_{\d,\eta}-(h_{\d,\eta}^{-l})_{*}(\d X)\|\lesssim \d^{p/2-a-1}.$$
This implies (because $(\phi^1_{-l\d X })_{*}X=X$)
$$ \sup_{h_{\d,\eta}^{-l}(\check \cW_{\check s,\check \nu})}\|\Theta_{\d,\eta}-\d X\|\lesssim \d^{p/2-a-1}$$
whence (\ref{9.71}).

\mn (2) The second estimate of (\ref{cond9.62}) shows that 
\be d(\zeta , N_{\d,\eta}^{-1}(0,0))\lesssim \d^{p-a}.\label{zetaclosetoB00}\ee

Let 
\be \xi_{0}=N_{\d,\eta}^{-1}(0,0)=\psi_{0}^{-1}(0,0) \in B_{0,0}\ee 
(so $d(\zeta ,\xi_{0})\lesssim \d^{p-a}$). 
Estimate (\ref{9.71}) and the fact that $T_{1}\asymp \d^{-1}$ (see (\ref{9.70})) give
$$ \xi=\phi^{T_{1}}_{\Theta_{\d,\eta}}(\xi)=\phi^{T_{1}}_{\d X+O(\d^{p/2-a-1})}(\xi)=\phi^{\d T_{1}}_{X+O(\d^{p/2-a-2})}(\xi)=\phi^{\d T_{1}}_{X}(\xi)+O( \d^{p/2-a-2})$$
hence
\be \zeta=\phi^{\d T_{1}}_{X}(\zeta)+O( \d^{p/2-a-2}).\label{9.72}
\ee
Besides, from  (\ref{9.70}) we have 
$|T_{1}-(T/\d)|\leq 2$, hence from  (\ref{9.72}) 
$|\d T_{1}-T|\lesssim \d^{p/2-a-2}$
that we can write (cf. (\ref{pvsabisante}))
$$T_{1}=(T/\d)+O(1).$$
To conclude, we note that the rotation number of $h_{\d,\eta}$ restricted to $B_{0,0}$ is the rotation number of $\phi^1_{\Theta_{\d,\eta}}$ restricted to  $B_{0,0} $ a number  which is equal to $1/T_{1}$.
As a consequence
\begin{align*}{\rm rot}(h_{\d,\eta}\mid B_{0,0})&=\frac{1}{(T/\d)+O(1)}\\
&=\frac{\d}{T}+O(\d^2).
\end{align*}
\end{proof}

\subsection{Proof of Theorem \ref{invariantannulus}}\label{sec:proofth9.3}

Item (\ref{it1}) of Lemma \ref{lemma:periodsvf} shows that for $\xi\in B_{0,0}$ the orbit $(\phi^t_{\Theta_{\d,\eta}}(\xi))_{t\in \R}$ is $T_{1}$-periodic. Because the vector field $\Theta_{\d,\eta}$ is holomorphic, there exists some $s_{0}>0$ such that for any $s\in(-s_{0},s_{0})$ the orbit $(\phi^{t+is}_{\Theta_{\d,\eta}}(\xi))_{t\in \R}$ is $T_{1}$-periodic. The image of the  map 
$$\T_{s_{0}}\ni \th\mapsto \phi^{\th T_{1}}_{\Theta_{\d,\eta}}(\xi)$$
is the searched for invariant annulus since (cf. (\ref{eqhthtatr})) $h_{\d,\eta}=\phi^1_{\Theta_{\d,\eta}}$ commutes with the flow of $\Theta_{\d,\eta}$. Also, because
$$h_{\d,\eta}(\phi^{\th T_{1}}_{\Theta_{\d,\eta}}(\xi))=\phi^{(\th+1/T_{1}) T_{1}}_{\Theta_{\d,\eta}}(\xi),$$
the restriction of $h_{\d,\eta}$ on this annulus is conjugate to the map $\th\mapsto \th+\a$ with 
$$\a=1/T_{1}.$$
The estimate (\ref{c8.62}) then comes from Lemma \ref{lemma:8.12}.

We then set for $\xi_{0}=\psi_{0}^{-1}(0,0)=N_{\d,\eta}^{-1}(0,0)\in B_{0,0}$
$$\cA_{\d,\eta}=\{\phi_{\Theta_{\d,\eta}}^{\th T_{1}}(\xi_{0})\mid \th\in \T_{s_{0}}\}$$
which is the $h_{\d,\eta}$-invariant annulus we are looking for.

\medskip By (\ref{9.71}) and (\ref{zetaclosetoB00}) one has \footnote{If $U$ is a set we define $\cV_{\d}(U)$ a $\d$-neighborhood of this set.}
$$\{\phi_{\Theta_{\d,\eta}}^t(\xi_{0})\mid t\in\R\}\cap \check\cW_{\check s,\check \nu}\subset  \cV_{\d^{p-a-1}}\biggl(\{\phi_{X}^t(\zeta)\mid t\in\R\}\cap \check\cW_{\check s,\check \nu}\biggr).$$
Using the fact that  $h_{\d,\eta}=\phi^1_{\d X}\circ (id+O(\d^p))$ we get, by definition of $\check\cC_{\check s, \check \nu}$ and the $\phi^t_{X}$-invariance of the orbit $\{\phi_{X}^t(\zeta)\mid t\in\R\}$,
$$\{\phi_{\Theta_{\d,\eta}}^t(\xi_{0})\mid t\in\R\}\cap \check\cC_{\check s,\check \nu}\subset  \cV_{\d^{p-a-2}}\biggl(\{\phi_{X}^t(\zeta)\mid t\in\R\}\cap \check\cC_{\check s,\check \nu}\biggr).$$
We thus get
$$\{\phi_{\Theta_{\d,\eta}}^t(\xi_{0})\mid t\in\R\}\subset  \cV_{\d^{(2/3)p+1}}\biggl(\{\phi_{X}^t(\zeta)\mid t\in\R\}\biggr)$$
  since $p-a-2>(2/3)p+1$. One can  take $t\in \R_{s}=\R+i(-s,s)$, $s$ small enough, in the left hand-side of the preceding inclusion.

\hfill $\Box$

\subsection{Proof of Theorem \ref{invariantannulusbis}}

To prove the existence of a basin of attraction we  use the proof of Proposition \ref{lemma:9.6}: because $\Im\check\beta>0$, the cylinder 
$\T_{\check s}\times \bD(0,\check s)$ is a basin of attraction of the annulus $\T_{\check s}\times\{0\}$ for the map $(\th,r)\mapsto (\th+\check \a,e^{2\pi i\check \b}r)$ with $\check \a\in\R$; now,  the fact that $N_{\d,\eta}\circ h_{\d,\eta}^{q_{\d}}\circ N_{\d,\eta}^{-1}$ is conjugate on $\T_{\check s}\times \bD(0,\check s)$ to $(\th,r)\mapsto (\th+\check \a,e^{2\pi i\check \b}r)$  shows that the forward iterates under the {\it first return map} $\hat h_{\d,\eta}$ of any point    $\xi\in {\bar \cW}^\eta_{\d,s_{*}}\cap \check\cW_{\check s, \check \nu}$ accumulate to some$N_{\d,\eta}^{-1}((-\nu,1+\nu)_{ s}\times\{0\})$ which is a piece of orbit lying in the compact annulus $F_{s}=\{\phi^t_{\Theta_{\d,\eta}}(\xi)\mid t\in \R+i[-s,s]\}$ ($s$ depends on $\xi$).  The family $\{(h_{\d,\eta}^n\mid \check\cC_{\check s,\check \nu} )\}_{n\in\N}$ being normal (see Theorem \ref{lemma:9.6})  this implies that the iterates of any point $\xi\in {\bar \cW}^\eta_{\d,s_{*}}\cap \check\cW_{\check s, \check \nu}$ under $h_{\d,\eta}$ accumulate some $F_{s}=\{\phi^t_{\Theta_{\d,\eta}}(\xi)\mid t\in \R+i[-s,s]\}$: indeed if this were not the case there would exist  sequences of positive times $n_{k}$ and $m_{k}>n_{k}$ such that $d(h_{\d,\eta}^{n_{k}}(\xi),F_{s})\to 0$ and $\inf d(h_{\d,\eta}^{m_{k}}(\xi),F_{s})>0$ so that $\inf d(h_{\d,\eta}^{m_{k}-n_{k}}(\xi_{k}),F_{s})>0$ with $\xi_{k}=h_{\d,\eta}^{n_{k}}(\xi)$ accumulating $F_{s}$; by normality of $\{(h_{\d,\eta}^n\mid \check\cC_{\check s,\check \nu} )\}_{n\in\N}$ this would yield the existence of $\xi_{*}\in F_{s}$ and of a holomorphic map $h_{*}:\check \cC_{\check s,\check \nu} \to\check \cC_{\check s,\check \nu}$ such that $h_{*}(\xi_{*})\in F_{s}$ and $d(h_{*}(\xi_{*}),F_{s})>0$, a contradiction. 

Because $ {\bar \cW}^\eta_{\d,s_{*}}\cap \check\cW_{\check s, \check \nu}$ is a return domain for any points of $\check\cW_{\check s, \check \nu}$, this concludes the proof of the fact that the orbit of  any $\xi \in \check\cW_{\check s, \check \nu}$ accumulates some $F_{s}$.

\medskip We now prove   that the annulus $\cA_{\d,\eta}$ is $\d^{(2/3)p}$-isolated by proving that any    $h_{\d,\eta}$-invariant annulus  $\cA'$ with small enough module and such that ${\rm dist}(\cA_{\d,\eta},\cA')\leq \d^{(2/3)p}$ is included in $\cA_{\d,\eta}$.
The last conclusion of Theorem \ref{invariantannulus} gives us the inclusion 
$$\{\phi_{\Theta_{\d,\eta}}^t(\xi_{0})\mid t\in\T_{s}\}\subset  \cV_{\d^{(2/3)p}}\biggl(\{\phi_{X}^t(\zeta)\mid t\in\R\}\biggr).$$
Taking $\cA'$,  with smaller modulus we thus have
$$\cA'\subset  \cV_{2\d^{(2/3)p}}\biggl(\{\phi_{X}^t(\zeta)\mid t\in\R\}\biggr).$$
Making use of Remark \ref{Ccontains} yields 
$$\cA'\subset \cV_{\d^{(2/3)p-1}}\biggl(\{\phi_{X}^t(\zeta)\mid t\in\R\}\biggr)\subset \cC^\eta_{\d,\d^{p/2+2},\nu}\subset \check \cC_{\check s,\check \nu}$$
which shows that the $h_{\d,\eta}$-invariant annulus $\cA'$ is in the basin of attraction of $\cA_{\d,\eta}$. But this implies that $\cA'\subset \cA_{\d,\eta}$ since any point of $\cA'$ is recurrent (we recall that by assumption the dynamics of $h_{\d,\eta}$ on $\cA'$ is conjugate to a rotation).

\hfill $\Box$

\subsection{Proof of Theorem \ref{tho:rank2rotdomain}]}

\begin{lemma}\label{lemma:8.10}If $\check \b\in\R$, one has 
\be h_{\d,\eta}=\phi^d_{\ti \Theta_{\d,\eta}}\circ \phi^{-b}_{\ti R_{\d,\eta}}\label{eq:formh}\ee
($b,d\in\Z$ from (\ref{abcd})).
Furthermore,
if $(\check\a,\check\b)\in\R^2$ is non resonant, one has $T_{1}/T_{2}\notin \Q$.
\end{lemma}
\begin{proof}
From (\ref{eqhthtatr}) we see that 
$\Theta_{\d,\nu}=d\ti \Theta_{\d,\eta}-b\ti R_{\d,\nu}$, $R_{\d,\eta}=-c\ti\Theta_{\d,\eta}+a\ti R_{\d,\eta}$ hence
$$\begin{cases}
&h_{\d,\eta}=\phi^d_{\ti \Theta_{\d,\eta}}\circ \phi^{-b}_{\ti R_{\d,\eta}}\\
&h_{\d,\eta}^{q_{\d}}=\phi^{\check \a d-\check \b c}_{\ti \Theta_{\d,\eta}}\circ \phi^{-\check\a b+\check \b a}_{\ti R_{\d,\eta}}.
\end{cases}$$
We thus have for some $(m_{1},m_{2})\in\Z^2$
$$\begin{cases}
&\check \a d-\check \b c=q_{\d}d+m_{1}T_{1}\\
&-\check\a b+\check \b a=-q_{\d}b+m_{2}T_{2}
\end{cases}
$$
and because $(\check \a,\check \b)$ is non resonant one has $m_{1}\ne 0$ and $m_{2}\ne 0$ hence
$$\begin{cases}
&T_{1}=\check \a \frac{d}{m_{1}}-\check \b \frac{c}{m_{1}}-\frac{q_{\d}d}{m_{1}}\\
&T_{2}=-\check\a \frac{b}{m_{2}}+\check \b \frac{a}{m_{2}}+\frac{q_{\d}b}{m_{2}}.
\end{cases}
$$
A resonance relation $l_{1}T_{1}+l_{2}T_{2}=0$, $(l_{1},l_{2})\in \Z^2$ yields
$$\bm \check \a& \check \b\em\bm d& -b\\ -c& a\em\bm l_{1}m_{2}\\ l_{2}m_{1} \em=q_{\d}(dl_{1}m_{2}\\ -bl_{2}m_{1} )$$
which implies $(l_{1},l_{2})=(0,0)$ since $(\check\a,\check\b)$ is non resonant. 
\end{proof}

Equality (\ref{eq:formh}) can be written
\be h_{\d,\eta}=\phi^{\a_{}T_{1}}_{\ti \Theta_{\d,\eta}}\circ \phi^{\b_{}T_{2}}_{\ti R_{\d,\eta}}\label{eq:formhbis}\ee
with 
$$\a_{}=d/T_{1},\qquad \b_{}=-b/T_{2}.$$
Let $\xi \in B_{0,\rho_{}}$ with  $\rho_{}>0$.
We  define
$$\Phi:\T_{\check s_{}}\times \T_{\R_{+}}\mapsto \phi_{\ti \Theta_{\d,\eta}}^{\th T_{1}}\circ \phi_{\ti R_{\d,\eta}}^{\th_{2}T_{2}}(\xi)\in {\check\cC}_{\check s,\check\nu}$$
which is possible since one can check that for $\Im\th_{2}\geq 0$ the flow $\phi_{\ti R_{\d,\eta}}^{\th_{2}T_{2}}$ sends $B_{0,\rho_{}}$ into itself.  From (\ref{eq:formhbis}) we thus get
$$\Phi^{-1}\circ h\circ \Phi: \T_{\check s_{}}\times \T_{\R_{+}}\ni (\th,\th_{2})\mapsto (\th+\a_{},\th_{2}+\b_{})\in\T_{\check s_{}}\times \T_{\R_{+}} .$$
Setting $r=e^{ 2\pi i\th_{2}}$ and $\ti\Phi(\th,r)=\Phi(\th,\th_{2})$ gives the conjugation relation
$${\ti \Phi}^{-1}\circ h\circ \ti \Psi:\T_{\check s_{}}\times \bD(0,1)\ni (\th,r)\mapsto (\th+\a_{1},e^{2\pi i\a_{2}}r)\in \T_{\check s_{}}\times \bD(0,1).$$
It is not difficult to check that $\ti \Phi$ extends as a holomorphic injective map  $\T_{\check s_{}}\times \bD(0,\check \rho)\to \check{\cC}_{\check s, \check \nu}$.

\begin{cor} Assume $(\check\a,\check \b)\in\R^2$ is non resonant.  If $\rho_{}\ne 0$ and $\xi \in B_{r,\rho_{}}$, $\rho_{}\ne 0$,  the closure of the orbit $\{h^n_{\d,\eta}(\xi)\}_{n\in\N}$ is equal to  the 2-torus $B_{r,\rho_{}}$. If $\xi \in B_{r,0}$,  the closure of the orbit $\{h^n_{\d,\eta}(\xi)\}_{n\in\N}$ is equal to the circle $B_{r,0}$. 
\end{cor}
\begin{proof}This is a consequence of the previous Lemmata  \ref{lemma:8.10} and \ref{lemma:periodsvf} and Remark \ref{rem:8.6}.
 
\end{proof}

This completes the proof of Theorem \ref{tho:rank2rotdomain}. \hfill $\Box$

\bigskip Our task in the next section is to prove the existence of a normalizing map. We shall then see in Section \ref{sec:SiegelKAMtheorems}  that, when extra parameters are introduced, Lemma \ref{lemma:9.6} holds for many values of these parameters.

\bigskip
\section{Partial normalization of   commuting pairs}\label{sec:panormpair}

We prove in this section that the commuting pair $(h_{\d,\tau},h_{\d,\tau}^{q_{\d}})$ naturally associated to the diffeomorphism $h_{\d,\tau}$ in subsection \ref{sec:commpair} can be partially normalized; Theorem \ref{cor:2.10:commpair} gives a quantitative version of this statement. Reversibility issues and dependence on parameters are considered in Sections \ref{sec:7.5} and \ref{sec:7.5-param}.

\bigskip

Let $X_{\tau}$ be a holomorphic vector field defined  in an open set $V$ of $\C^2$, depending  in a holomorphic way   on a complex parameter $\tau\in \bD_{\C^2}(\tau_{*},\rho_{})\subset\C^2$ and satisfying 
$$\sup_{\tau\in \bD_{\C^2}(\tau_{*},\rho_{}) }\|X_{\tau}\|_{V}\leq A.$$

\begin{assumption}\label{assump:10.1}
We assume that 
\begin{enumerate}
\item For all $\tau\in \bD_{\C^2}(\tau_{*},\rho_{})$,  $X_{\tau}$ has constant divergence $2\pi i\b$, $\b\in \bD(0,2)$.
\item There exist holomorphic functions $g_{}:\bD_{\C^2}(\tau_{*},\rho_{})\to \C$, $\zeta_{}:\bD_{\C^2}(\tau_{*},\rho_{})\to \C^2$ and   $s_{*}>0$ such that for all $\tau\in \bD_{\C^2}(\tau_{*},\rho_{})$,  $X_{\tau}$ has an invariant annulus 
$$\cA^{\rm vf}_{\tau}=\{\phi^{\th}_{e^{-i\arg g_{}(\tau)}X_{\tau} }(\zeta_{}(\tau))\mid  \th\in \R+i(-s_{*},s_{*})\}$$
 on which $X_{\tau}$ is conjugate to the vector field $g_{}(\tau)\pa_{\th}$ defined on $\T_{s_{*}}$. 
  We set
 $$T_{\tau}=\frac{1}{g(\tau)}.$$
 \item One has 
 \be 
 \left\{
 \begin{aligned}&g_{}(\tau_{*})\in\R^*,\\
& \forall \tau\in \bD_{\C^2}(\tau_{*},\rho),\quad {\rm rank}\  \frac{\pa g}{\pa\tau}(\tau_{})=1.
  \end{aligned}
  \right.
  \label{eq:10.57}
  \ee
 \item For all $\tau\in \bD_{\C^2}(\tau_{*},\rho_{})$, the invariant annulus $\cA^{\rm vf}_{\tau}$ intersects and is transverse to $\zeta_{}(\tau_{*})+\C e_{2}$ where $e_{2}=\bm 0\\ 1\em$.  In particular, there exists a neighborhood $U=U'\times U''$ of $\cA^{\rm vf}_{\tau_{*}}$ such that for any $\tau\in \bD_{\C^2}(\tau_{*},\rho)$, the intersection $U\cap \cA^{\rm vf}_{\tau}$ can be described as a graph $\zeta_{\tau}+\{(z,E_{\tau}(z)\mid z\in U'\}$ where $E_{\tau}:U'\to U''$ is holomorphic.

  We define (cf. (\ref{eq:definGammaX}))
\be \Gamma_{\tau}:\zeta_{\tau}+U\ni (z,w)\mapsto (z,w-E_{\tau}(z))-\zeta_{\tau}. \label{eq:definGammaXbis}\ee

 \item We also assume we are given a holomorphic family $\bD_{\C^2}(\tau_{*},\rho)\ni \tau\mapsto F_{\tau}\in \cO(V)$ such that 
\be \begin{cases}&\sup_{\tau\in \bD_{\C^2}(\tau_{*},\rho)}\|F_{\tau }\|_{V}\leq A\d^p\\
&p>2\\
\end{cases}\label{eq:7.39hedobisbis}\ee
and we  set
\be h_{\d,{\tau}}=\phi^1_{\d X_{\tau}}\circ \iota_{F_{\tau}}=\phi^1_{\d X_{\tau}}\circ (id+\eta_{\tau})\label{def:hdeltaetabis}\ee
 \end{enumerate}
  \end{assumption}
  
\bigskip
Because $X_{\tau}=X_{\tau_{*}}+O(\tau-\tau_{*})$ one has for any $\tau\in\bD_{\C^2}(\tau_{*},\d^{3/2})$
$$\d X_{\tau}=\d X_{\tau_{*}}+O(\d^{5/2})$$
and we can write
$$h_{\d,\tau}=\phi^1_{\d X_{\tau_{*}}}\circ (id+\eta^*_{\d,\tau})$$
with
\be 
\begin{aligned}
id+\eta^*_{\d,\tau}&=\phi^{-1}_{\d X_{\tau_{*}}}\circ \phi^1_{\d X_{\tau}}\circ\iota_{F_{\tau}}\\
&=id+O_{A}(\d^3)\\
&=\iota_{F^*_{\d,\tau}}
\end{aligned}
\label{e10.78}
\ee
($F_{\d,\tau}^*=O_{A}(\d^3)$).

The diffeomorphism $h_{\d,\eta}$ can  thus be written in two different ways 
\be \begin{cases}&h_{\d,{\tau}}=\phi^1_{\d X_{\tau}}\circ \iota_{F_{\tau}} =\phi^1_{\d X_{\tau_{*}}}\circ \iota_{F^*_{\d,\tau}}\\
&F_{\tau}=O(\d^p),\qquad F^*_{\d,\tau}=O(\d^3),\qquad F_{\tau_{*}}=F^*_{\d,\tau_{*}}.
\end{cases}\label{def:hdeltaetabistwodiffways}\ee
We can  apply the results of Section \ref{sec:renormandcommpairs}, in particular Proposition \ref{lemma:7.1}, to the pair $(X,\eta)$ where 
\be X=X_{\tau_{*}},\qquad id+\eta= id+\eta^*_{\d,\tau}=\iota_{F^*_{\d,\tau}};\label{choiceXeta}\ee there exists $\d_{*}>0$ and $0<s_{*}'<s_{*}$ such that for $\nu=1/3$ (for example) and  any $\d\in (0,\d_{*})$ for which 
$$\biggl\{\frac{T_{\tau_{*}}}{\d}\biggr\}\in ((1/10),(9/10))$$
we can define the renormalization $\cR_{\rm fr}^*(h_{\d,\tau})$ of $h_{\d,\tau}$ (see Subsection \ref{sec:glueing}) and consider the naturally associated commuting pair (see Subsection \ref{sec:commpair})
$$(h_{\d,\tau},h^{q_{\d}}_{\d,\tau})_{ \cW^{X_{\tau_{*}}, \eta^*_{\tau,\d}}_{\d,s_{*}',\nu}}$$
(defined  on $ \cW^{X_{\tau_{*}}, \eta^*_{\tau,\d}}_{\d,s_{*}',\nu}$, see (\ref{def:Wsigmadelta}), (\ref{def:hdeltaetat}))
where 
$$q=q_{\d}=\biggl[\frac{T_{\tau_{*}}}{\d}\biggr]\in ((1/10),(9/10)).$$

Note that if $\d$ is small enough one has 
\be \cW^{X_{\tau_{*}}, \eta^*_{\tau,\d}}_{\d,s_{*}'/4,\nu/4}\subset \cW^{X_{\tau_{}}, \eta_{\tau}}_{\d,s_{*}' /2,\nu/2}\subset \cW^{X_{\tau_{*}}, \eta^*_{\tau,\d}}_{\d,s_{*}',\nu}.\label{e10.87}\ee
In particular, the commuting pair
\be (h_{\d,\tau},h^{q_{\d}}_{\d,\tau})_{ \cW^{X_{\tau_{}}, \eta_{\tau}}_{\d,s_{*}'/2,\nu/2}}\label{e10.88}\ee
is well defined.

To keep simple notations we let 
$s_{0}=s'_{*}/2$, $\nu_{0}=\nu_/2$ and 
\be \cW^\tau_{\d,s,\nu,\rho}=\cW^{X_{\tau_{}}, \eta_{\tau}}_{\d,s,\nu,\rho}\label{cWtau}\ee
and when $s=\rho$ we remove the dependence on $\rho$.

We define
$$\cW^{*,\tau}_{\d}=\cW^{\eta^*_{\d,\tau}}_{\d,s',\nu=1/3}.$$

\begin{rem}\label{rem:theo10.1}If $g(\tau)$ is a real number, we can also apply the results of Section \ref{sec:renormandcommpairs} to the pair  $(X,\eta)$ where 
\be X=X_{\tau},\qquad id+\eta_{\tau}=\iota_{F_{\tau}}.\label{choiceXetabis}\ee
One then gets a commuting pair  $(h_{\d,\tau},h^q_{\d,\tau})$ on  the open set $\cW^{X_{\tau},\eta_{\tau}}_{\d,s,\nu}$ (see (\ref{def:Wsigmadelta}), (\ref{def:hdeltaetat}) with the choice (\ref{choiceXetabis})).

If for some $\ph_{\tau}\in\R$, $g(\tau)e^{i\ph_{\tau}}$  is real one can choose 
\be X=e^{i\ph_{\tau}}X_{\tau},\qquad id+\eta^\sharp_{\d,\tau}=\phi^{-1}_{\d e^{i\ph_{\tau}}X_{\tau_{}}}\circ \phi^1_{\d X_{\tau}}\circ\iota_{F_{\tau}} \label{choiceXetater}\ee
($X$ then has a periodic orbit but $id+\eta^\sharp_{\d,\tau}$ is not anymore symplectic) and we then define a commuting pair  $(h_{\d,\tau},h^q_{\d,\tau})$ on 
the open set $\cW^{X^\sharp_{\tau},\eta^\sharp_{\d,\tau}}_{\d,s'_{\sharp},\nu}$ (see (\ref{def:Wsigmadelta}), (\ref{def:hdeltaetat}) with the choice (\ref{choiceXetater})).

Again, if $\d$ is small enough
$$\cW^{X^\sharp_{\tau},\eta^\sharp_{\d,\tau}}_{\d,s'_{\sharp}/4,\nu/4}\subset \cW^{X_{\tau_{}}, \eta_{\tau,\d}}_{\d,s_{\sharp}'/2,\nu/2}\subset \cW^{X^\sharp_{\tau},\eta^\sharp_{\d,\tau}}_{\d,s'_{\sharp},\nu}.$$

\end{rem}

\bigskip

\begin{theo}[Partial normalization of  commuting pairs]\label{cor:2.10:commpair}There exist $0<s'<s_{0}$, $0<\nu'<\nu<\nu_{0}$ and  $\d_{*}$ such that for all  $\d\in (0,\d_{*}]$  satisfying
$$\biggl\{\frac{1}{\d g(\tau_{*})}\biggr\}\in ((1/10),(9/10))$$
the following holds. For all $\tau\in \bD_{\C^2}(\tau_{*},\d^2)$
  there exists an exact  conformal-symplectic holomorphic injective map  $N_{\d,\tau}^{\rm ec}$
  $$N_{\d,\tau}^{\rm ec}:h_{\d,\eta}^{-1}(\cW^\tau_{\d,s',\nu'})\cup \cW^\tau_{\d,s',\nu}\cup h_{\d,\eta}(\cW^\tau_{\d,s',\nu}) \to  \C^2$$
  (remember $h_{\d,\eta}^{q_{\d}}(\cW^\tau_{\d,s_{0},\nu_{0}})\subset h_{\d,\eta}^{-1}(\cW^\tau_{\d,s_{0},\nu_{0}})$)
  such that on
  $$\check \cW^\tau_{s_{0},\nu_{0}}=(N^{ec}_{\d,\tau})^{-1}\biggl((-\nu_{0},1+\nu_{0})_{s_{0}}\times \bD(0,s_{0})\biggr)$$
 the partial normalization relation
  \be 
 \begin{aligned} 
 N^{\rm ec}_{\d,\tau}\circ\bm h_{\d,\tau}\\ h^{q_{\d}}_{\d,\tau}\em\circ (N^{\rm ec}_{\d,\tau})^{-1}&=\bm \cT_{1,\d\b}\\ S_{q_{\d}\d\b}\circ \Phi_{\ti \a_{\d,\tau}w}\circ\iota_{F^{\rm vf}_{\d,\tau}}\circ \iota_{F^{\rm cor}_{\d,\tau}}\em\\
 &=\bm \Phi_{\d\b}\circ \Phi_{w}\\ S_{q_{\d}\d\b}\circ \Phi_{\ti \a_{\d,\tau}w}\circ\iota_{F^{\rm vf}_{\d,\tau}}\circ \iota_{F^{\rm cor}_{\d,\tau}}\em
\end{aligned}
\label{conc:thm7.6}
 \ee
 holds, 
where
$F^{\rm vf}_{\d,\tau}, F^{\rm cor}_{\d,\tau}\in \cO(\check \cW^\tau_{s_{0},\nu_{0}})$ satisfy 
$$\begin{cases}&F_{\d,\tau}^{\rm vf}(z,w)=O(w^2),\\
&F_{\d,\tau}^{\rm vf}(z,w)=O_{A}(1),\\
&F^{\rm cor}_{\d,\tau}=O_{A}(\d^{p-2})\end{cases}$$ and 
$$\ti\a_{\d,\tau}=-\biggl\{ \frac{1}{\d g(\tau)}\biggr\}\quad (\in \C).$$
Furthermore, one has\footnote{In what follows $\L_{\d}$  is the dilation 
$\L_{\d}:(z,w)\mapsto (\d^{-1}z,\d^{-1}w)$.}
\be \begin{cases}
&N_{\d,\tau}^{\rm ec}=\iota_{Y^{cor}_{\d,\tau}}\circ N^{\rm vf}_{\d,\tau}\\
&\textrm{with}\quad N^{\rm vf}_{\d,\tau}= \iota_{G_{\tau}}\circ \L_{\d c_{\tau}}\circ \Gamma_{\tau}\\
&(N^{\rm vf}_{\d,\tau})_{*}(\d X_{\tau})=\pa_{z}+(2\pi i\d\beta w)\pa_{w} 
\end{cases}
\label{n10.93}
\ee
and where $c_{\tau}\asymp 1$, $G_{\tau}(z,w)=O(w)$, $\iota_{G}(0,0)=(0,0)$ and $Y^{cor}_{\d,\tau}=O(\d^{p-1})$. 
\end{theo}

The following result, which is a  corollary of the proof of the previous theorem,  we be helpful in Sections \ref{sec:proofmainA} and \ref{sec:proofmainA:B}.
\begin{cor}\label{cor:10.2}There exists $0<s_{1}\leq s_{0}$ such that for $0<\nu\leq \nu_{0}/2$, and $s\in (\d^{p-2},s_{1}]$,  one has for some $C>1$ independent of $\d$
  $$\cW^\tau_{\d,s_{}/2,C^{-1}s,\nu_{}/2}\subset (N^{\rm ec}_{\d,\tau})^{-1}\biggl((-\nu_{},1+\nu_{})_{s_{}}\times \bD(0,s)\biggr)\subset\cW^\tau_{\d,2s,Cs,2\nu}$$
  ($\d$ small enough). 
  
 Moreover, $N^{\rm ec}_{\d,\tau}(\zeta_{\tau})\in \bD((0,0),\d^{p-2})$.
\end{cor}

\bigskip 
We give the proof of Theorem \ref{cor:2.10:commpair} in Subsections \ref{sec:10.1} and \ref{sec:10.2}. The proof of Corollary \ref{cor:10.2} is done in Subsection \ref{sec:10.3}. In Subsection \ref{sec:7.5} we concentrate on the reversible case.

\subsection{Proof of Theorem \ref{cor:2.10:commpair}; case $F_{\tau}\equiv 0$.} \label{sec:10.1}
 We recall
$$\Gamma_{\tau}:\zeta_{\tau}+U\ni (z,w)\mapsto (z,w-E_{\tau}(z))-\zeta_{\tau}$$
is the map sending $U\cap \cA^{\rm vf}_{\tau}$ into $\{(z,0)\mid z\in U'\}$ (and $\zeta_{\tau}$ to $(0,0)$). The map $\Gamma_{\tau}$ is exact symplectic w.r.t. the Liouville form $wdz$
 since $(w-E_{\tau}(z))dz-wdz=-E_{\tau}(z)dz=d(-\int_{*}^z E_{\tau})$. 
 
 The vector field 
 $(\Gamma_{\tau})_{*}X_{\tau}$ is then of the form
 $$(\Gamma_{\tau})_{*}X_{\tau}:(z,w)\mapsto (a_{\tau}(z,w),b_{\tau}(z,w))$$
 with 
 $$a_{\tau}(0,0)\ne 0\quad\textrm{and}\quad b_{\tau}(z,0)\equiv 0.$$
 Note that $|a_{\tau}(0,0)|\asymp 1$.
 
 Conjugating $(\Gamma_{\tau})_{*}X_{\tau}$ by the dilation 
 $$\L_{a_{\tau}(0,0)}:(z,w)\mapsto (a_{\tau}(0,0)^{-1}z,a_{\tau}(0,0)^{-1}w)$$ we can assume that
 $a_{\tau}(0,0)=1$. 
 \begin{lemma}\label{lemma:7.7}There exists an exact symplectic holomorphic diffeomorphism $\iota_{G_{\tau}}$ with $G_{\tau}=O(w)$, $\iota_{G_{\tau}}(0,0)=(0,0)$, defined on a neighborhood  
 of $(0,0)$ 
 such that $$(\iota_{G_{\tau}})_{*} (\L_{a_{\tau}(0,0)})_{*}(\Gamma_{\tau})_{*}X_{\tau}:(z,w)\mapsto (1+\mathring{a}_{\tau}(z,w),\mathring{b}_{\tau}(z,w))$$ satisfies 
 \begin{align}&\mathring{a}_{\tau}(\cdot,0)=0\qquad \mathring{b}_{\tau}(\cdot,0)=0.\label{eq:lemma103}
 \end{align}
\end{lemma}
\begin{proof}The vector field $(\L_{a_{\tau}(0,0)})_{*}(\Gamma_{\tau})_{*}X_{\tau}$ is tangent to $\{w=0\}$ and its restriction to $\{w=0\}$ 
 can be linearized into $\pa_{z}$ by some holomorphic diffeomorphism of the form $z\mapsto z+u(z)$.
For example, the inverse of the map $z= t+is \mapsto \phi^{t+is}_{(\L_{a_{\tau}(0,0)})_{*}(\Gamma_{\tau})_{*}X_{\tau}}(0,0)$ is such a linearization (in a neighborhood of $0$). 
  Let $G_{\tau}(z,w)=u(z)w$. One has $\iota_{G_{\tau}}(z,w)=(\ti z,\ti w)$ if and only if 
 $$\begin{cases}&\ti z=z+u(z),\\
 &w=\ti w+\ti w\pa u(z).\end{cases}$$
 Adding to $u$ a constant we can impose $\iota_{G_{\tau}}(0,0)=(0,0)$.
 
\end{proof}

There exists $C\geq 1,r_{0}>0$ (independent of $\d$) such that   the diffeomorphism $\iota_{G_{\tau}}\circ \L_{a_{\tau}(0,0)}\circ \Gamma_{\tau}$ is defined on $\bD_{\C^2}(0,r_{0})$ and for any $r\in [0,r_{0}]$, it sends $\bD_{\C^2}(\zeta_{\tau},r)$ onto a neighborhood of $\bD_{\C^2}((0,0),C^{-1}r)$ and its inverse sends $\bD_{\C^2}((0,0),r)$ onto a neighborhood of $\bD_{\C^2}(\zeta_{\tau},C^{-1}r)$.

If $\L_{\d}$  is the dilation 
$$\L_{\d}:(z,w)\mapsto (\d^{-1}z,\d^{-1}w)$$
the diffeomorphism $\L_{\d}\circ \iota_{G_{\tau}}\circ \L_{a_{\tau}(0,0)}\circ \Gamma_{\tau}$  sends $\bD_{\C^2}(\zeta_{\tau},r_{})$ onto a neighborhood of $\bD_{\C^2}((0,0),\d^{-1}C^{-1}r_{})$ and its inverse sends $\bD_{\C^2}((0,0),\d^{-1}r)$  onto a neighborhood of $\bD_{\C^2}(\zeta_{\tau},C^{-1}r)$:
\begin{align}
&(\L_{\d}\circ \iota_{G_{\tau}}\circ \L_{a_{\tau}(0,0)}\circ \Gamma_{\tau})(\bD_{\C^2}(\zeta_{\tau},r_{}))\supset \bD_{\C^2}((0,0),\d^{-1}C^{-1}r_{}) \label{e10.95ante}\\
&(\L_{\d}\circ \iota_{G_{\tau}}\circ \L_{a_{\tau}(0,0)}\circ \Gamma_{\tau})^{-1}(\bD_{\C^2}((0,0),\d^{-1}r))\supset \bD_{\C^2}(\zeta_{\tau},C^{-1}r).\label{e10.95}
\end{align}

Let $M\geq 1$ be such that 
\be \phi_{\d X_{\tau}}^{-1}(\cW^{X_{\tau},0}_{\d,s,\nu})\cup \cW^{X_{\tau},0}_{\d,s,\nu}\cup \phi^1_{\d X_{\tau}}(\cW^{X_{\tau},0}_{\d,s,\nu})\subset \bD_{\C^2}(\zeta_{\tau},M\d).\ee

If $\d\leq (5MC^2)^{-1}r_{0}$ we can thus consider on 
$\bD_{\C^2}((0,0),5MC)$
 the vector field
$$(\L_{\d})_{*}(\iota_{G_{\tau}})_{*}(\L_{a_{\tau}(0,0)})_{*} (\Gamma_{\tau})_{*}(\d\times X_{\tau}):(z,w)\mapsto (1+\mathring{a}_{\tau}(\d z,\d w),\mathring{b}_{\tau}(\d z,\d w))$$
which has constant divergence equal to $2\pi i\d \b$; hence we can write
$$(\L_{\d})_{*}(\iota_{G_{\tau}})_{*}(\L_{a_{\tau}(0,0)})_{*} (\Gamma_{\tau})_{*}(\d\times X_{\tau}):(z,w)\mapsto (1,2\pi i\d\b w)+\d J\nabla \ti F_{\d,\tau}(z,w))$$
with  $\ti F_{\d,\tau}\in \cO(\bD(0,5MC)\times \bD(0,5MC))$ satisfying (see (\ref{eq:lemma103}))
$$ \ti F_{\d,\tau}(z,w)=O(w^2),\qquad \ti F_{\d,\tau}^{}(z,w)=O_{A}(1).$$

\begin{lemma}There exists an exact conformal diffeomorphism $N^{\rm vf}_{\d,\tau}:\bD_{\C^2}(\zeta_{\tau},4M\d)\to \C^2$   of the form
\be \begin{cases}
&N^{\rm vf}_{\d,\tau}=\iota_{\d Y^{\rm vf}_{\d,\tau}}\circ \L_{\d}\circ \iota_{G_{\tau}}\circ\L_{a_{\tau}(0,0)}\circ \Gamma_{\tau}\\
&Y^{\rm vf}_{\d,\tau}(z,w)=O(w^2)
\end{cases}\label{formvf}
\ee
 such that one  has on $\bD(0,4MC)\times \bD(0,4MC)$
\be (N^{\rm vf}_{\d,\tau})_{*}(\d X_{\tau})= \pa_{z}+(2\pi i\d\beta w)\pa_{w} \label{Necdelta0normalizesvf}
\ee
and 
\be N^{\rm vf}_{\d,\tau}\circ \phi^1_{\d X_{\d}}\circ (N_{\d,\tau}^{\rm vf})^{-1}:(z,w)\mapsto (z+1,e^{2\pi i\d \beta_{}}w).\label{Necdelta0normalizes}\ee
\end{lemma}
\begin{proof}
By Proposition \ref{lemma:normlemma:vf} of the Appendix (on symplectic normalization of vector fields), if $\d$ is small enough,  the vector field $\pa_{z}+(2\pi i\d\b w)\pa_{w}+\d J\nabla \ti F_{\d,\tau}$ can be linearized on some neighborhood $\bD(0,4MC)\times \bD(0,4MC)$ of $(0,0)$: there exists $Y^{\rm vf}_{\d,\tau}\in \cO( \bD(0,4MC)\times \bD(0,4MC))$
\be Y^{\rm vf}_{\d,\tau}(z,w)=O(w^2)\label{eq:7.40}\ee
 such that on $ \bD(0,4MC)\times \bD(0,4MC)$
$$(\iota_{\d Y^{\rm vf}_{\d,\tau}})_{*} \biggr(\pa_{z}+(2\pi i\d\beta w)\pa_{w} +\d J\nabla \ti F_{\d,\tau}\biggl)=\pa_{z}+(2\pi i\d\beta w)\pa_{w} . $$
Let 
\be N^{\rm vf}_{\d,\tau}=\iota_{\d Y^{\rm vf}_{\d,\tau}}\circ \L_{\d}\circ \iota_{G_{\tau}}\circ\L_{a_{\tau}(0,0)}\circ \Gamma_{\tau}.\label{eq:7.41}\ee
The diffeomorphism $N_{\d,\tau}^{\rm vf}$ is defined on a neighborhood of  $\bD_{\C^2}(\zeta_{\tau}, 4M\d)$ which is sent onto $\bD_{\C^2}(0,4MC)$,
and on $ \bD(0,4MC)\times \bD(0,4MC)$ one has
\be (N^{\rm vf}_{\d,\tau})_{*}(\d X_{\tau})= \pa_{z}+(2\pi i\d\beta w)\pa_{w} \label{Necdelta0normalizesvf}
\ee
as well as 
\be N^{\rm vf}_{\d,\tau}\circ \phi^1_{\d X_{\tau}}\circ (N_{\d,\tau}^{\rm vf})^{-1}:(z,w)\mapsto (z+1,e^{2\pi i\d \beta_{}}w).\label{Necdelta0normalizes}\ee

\end{proof}

Note   that the domain $\cW^{X_{\tau},0}_{\d,s,\nu}$ defined in  (\ref{def:Wsigma})-(\ref{def:Wsigmadelta})-(\ref{cWtau}) satisfies if $\d$ is small enough
$$N_{\d,\tau}^{\rm vf}(\cW^{X_{\tau,0}}_{\d,s',\nu'})\supset (-\nu_{0},1+\nu_{0})_{s_{0}}\times \bD(0,s_{0})$$
for some $\nu_{0},s_{0}>0$.
The previous linearization result shows the following {\it normalization} result:
\begin{lemma}The diffeomorphism $N^{\rm vf}_{\d,\tau}$ is a normalization of the commuting pair $(\phi^1_{\d X_{\tau}},\phi^{q_{\d}}_{\d X_{\tau}})$ on $(N^{\rm vf}_{\d,\tau})^{-1}((-\nu_{0},1+\nu_{0})_{s_{0}}\times \bD(0,s_{0}))$.
\end{lemma}
We now give a more precise description of the diffeomorphism $N_{\d,\tau}^{\rm vf}\circ \phi^{q_{\d}}_{\d X_{\tau}}\circ (N_{\d,\tau}^{\rm vf})^{-1}$.
 \begin{lemma}\label{lemma:7.8}One has on $(-\nu_{0},1+\nu_{0})_{s_{0}}\times \bD(0,s_{0})$
\be N^{\rm vf}_{\d,\tau}\circ \phi^{q_{\d}}_{\d X_{\tau}}\circ (N_{\d,\tau}^{\rm vf})^{-1}=S_{q_{\d}\d\b}\circ \Phi_{\ti \a_{\d,\tau}w}\circ \iota_{F_{\d,\tau}^{\rm vf}}.\label{Necactionon}\ee
 for some $F_{\d,\tau}^{\rm vf}(z,w)=O(w^2)\in\C$, $F_{\d,\tau}^{\rm vf}(z,w)=O_{A}(1)$  and $\tau \mapsto \ti \a_{\d,\tau}=\ti\a_{\d,\tau_{*}}+O(\d^{})$ is holomorphic w.r.t. $\tau\in \bD_{\C^2}(\tau_{*},\d^{3/2})$.
 \end{lemma}
 \begin{proof}
 
 \mn (a) 
  Because $q_{\d}\d\asymp 1$, one has on some neighborhood  of $\cA^{\rm vf}_{\tau}$ that does not depend on $\d$
\be \phi^{q_{\d}}_{\d X_{\tau}}=\phi^{q_{\d}\d}_{X_{\tau}}=O(1).\label{eq:7.46}\ee
 Besides, since $\phi^{q_{\d}}_{\d X_{\tau}}$ leaves invariant $\cA^{\rm vf}_{\tau}$, we deduce from  (\ref{eq:7.46}) that one has  on some domain $\bD_{C^2}((0,0),r_{1})$ with $r_{1}$ independent of $\d$
 \begin{multline}(\iota_{G_{\tau}}\circ \L_{a_{\tau}(0,0)}\circ \Gamma_{\tau})\circ \phi^{q_{\d}}_{\d X_{\tau}}\circ (\iota_{G_{\tau}}\circ\L_{a_{\tau}(0,0)}\circ \Gamma_{\tau})^{-1}:\\(z,w)\mapsto (z+ c_{\d,\tau}(z,w),w+d_{\d,\tau}(z,w))\label{eq:7.52}
\end{multline}
with
$$\begin{cases}&d_{\d,\tau}(\cdot,0)=0\\ & c_{\d,\tau},d_{\d,\tau}=O_{A}(1).\end{cases}$$
Note that when $w=0$ one has for $z_{}\in \bD(0,r_{1})$ (see Lemma \ref{lemma:7.7})
$$(\iota_{G_{\tau}}\circ\L_{a_{\tau}(0,0)}\circ \Gamma_{\tau})\circ \phi^{z}_{\d X_{\tau}}\circ (\iota_{G_{\tau}}\circ\L_{a_{\tau}(0,0)}\circ \Gamma_{\tau})^{-1}:(0,0)\mapsto (z_{},0)$$
and because $\phi^z_{X_{\tau}}$ and $\phi^q_{\d X_{\tau}}$ commute one must have 
$$c_{\d,\tau}(\cdot,0)={\rm cst.}=c_{\d,\tau}(0,0);$$
indeed, if $\zeta=(\iota_{G_{\tau}}\circ\L_{a_{\tau}(0,0)}\circ \Gamma_{\tau})\circ \phi^{q_{\d}}_{\d X_{\tau}}\circ (\iota_{G_{\tau}}\circ\L_{a_{\tau}(0,0)}\circ \Gamma_{\tau})^{-1}(0,0)$, one has $\phi_{\d X_{\tau}}^{q_{\d}}(\phi_{X_{\tau}}^z(\zeta))=\phi^z_{X_{\tau}}(\phi^{q_{\d}}_{\d X_{\tau}}(\zeta))$ hence $(z+c_{\d,\tau}(z,0),0)=(z+c_{\d,\tau}(0,0),0)$.

\medskip
We define
$$\ti \a_{\d,\tau}=\d^{-1}c_{\d,\tau}(0,0)$$
so that 
\begin{align*}
&c_{\d,\tau}(z,w)=\d\ti \a_{\d,\tau}+\sum_{k=1}^\infty c_{\d,\tau,k}(z)w^k\\
&d_{\d,\tau}(z,w)=\sum_{k=1}^\infty d_{\d,\tau,k}(z)w^k.
\end{align*}

\mn (b)  If we conjugate (\ref{eq:7.52}) by the dilation $\L_{\d}:(z,w)\mapsto (\d^{-1}z,\d^{-1}w)$ we get on 
$\bD_{\C^2}((0,0),5MC)$ ($\d$ small enough)
\begin{multline*}
(\Lambda_{\d}\circ \iota_{G_{\tau}}\circ \Gamma_{\tau})\circ \phi^q_{\d X_{\d}}\circ (\Lambda_{\d}\circ \iota_{G_{\tau}}\circ \Gamma_{\tau})^{-1}:(z,w)\mapsto \\ (z+ \ti \a_{\d,\tau}+\sum_{k=1}^\infty c_{\d,\tau,k}(\d z)\d^{k-1}w^{k},w+ \sum_{k=1}^\infty d_{\d,\tau,k}(\d z)\d^{k-1}w^{k})\\ =(z+ \ti \a_{\d,\tau}+\sum_{k=1}^\infty c_{\d,\tau,k}(\d z)\d^{k-1}w^{k},w(1+d_{1}(0))+\sum_{k=2}^\infty d_{\d,\tau,k}(\d z)\d^{k-1}w^{k}))).
\end{multline*}
The diffeomorphism $(\L_{\d}\circ \iota_{G_{\tau}}\circ \L_{a_{\tau}(0,0)}\circ \Gamma_{\tau})\circ \phi^q_{\d X_{\tau}}\circ (\L_{\d}\circ \iota_{G_{\tau}}\circ \L_{a_{\tau}(0,0)}\circ \Gamma_{\tau})^{-1}$ has constant Jacobian equal to $e^{2\pi iq_{\d}\d\b}$ hence
$$1+d_{\d,\tau,1}(0)=e^{2\pi iq_{\d}\d\b}.$$
We can thus write 
\begin{multline*}(\L_{\d}\circ \iota_{G_{\tau}}\circ\L_{a_{\tau}(0,0)}\circ \Gamma_{\tau})\circ \phi^q_{\d X_{\tau}}\circ (\L_{\d}\circ \iota_{G_{\tau}}\circ(\L_{a_{\tau}(0,0)})\circ \Gamma_{\tau})^{-1}=\\ S_{q_{\d}\d\b}\circ \Phi_{\ti \a_{\d,\tau}w}\circ \iota_{O(w^2)}\end{multline*}
and (cf. (\ref{eq:7.41})) on the domain $\bD_{\C^2}(0,4MC)$ 
the equality
\be N^{\rm vf}_{\d,\tau}\circ \phi^q_{\d X_{\tau}}\circ (N^{\rm vf}_{\d,\tau})^{-1}=S_{q_{\d}\d\b}\circ \Phi_{\ti \a_{\d,\tau}w}\circ \iota_{F^{\rm vf}_{\d,\tau}}\label{eq:10.69}\ee
with 
 $F_{\d,\tau}^{\rm vf}(z,w)=O(w^2)\in\C$, $F_{\d,\tau}^{\rm vf}(z,w)=O_{A}(1)$. 

\mn c) The dependence on $\tau$ in the preceding construction is holomorphic, in particular $\tau\mapsto c_{\d,\tau}(0,0)=\d \ti\a_{\d,\tau}$ is holomorphic. 
\end{proof}

\begin{lemma} For any $\tau\in\bD_{\C^2}(\tau_{*},\d^2)$  one has 
$$\ti \a_{\d,\tau}=-\biggl\{\frac{1}{\d g(\tau)}\biggr\}.$$
\end{lemma}
\begin{proof}We first consider the case when $g(\tau)$ is a real number.

The analysis done at the beginning of the proof of Lemma \ref{lemma:7.1} on the first return map of $\phi^1_{\d X_{}^{}}$ in the {\it arc}  $I:=\{\phi^t_{X_{}^{}}(\zeta)\mid t\in [0,\d]\}$ included in the {\it circle} $\{\phi^t_{X_{}^{}}(\zeta)\mid t\in [0,T]\}$ shows that the first return map of $\phi^1_{\d X}$ in the arc $I$ is conjugate to that of the rotation $x\mapsto x+\a$, $\a=\d/T$ on the arc $[0,\a]\subset \R/\Z$.  In particular this first return map is $\R/\Z\ni x\mapsto x-\{\a^{-1}\}\in\R/\Z$. Because $g(\tau)$ is real and $X_{\tau}$ admits a $T_{\tau}$-periodic orbit $\{\phi^t_{X_{\tau}^{}}(\zeta)\mid t\in [0,T_{\tau}]\}$, this discussion also  applies to $X=X_{\tau}$ (with now $\a=\d/T_{\tau}$). The equalities (\ref{Necdelta0normalizes}) and   (\ref{eq:10.69}) restricted to $w=0$ give
\begin{align*}&N^{\rm vf}_{\d,\tau}\circ \phi^1_{\d X_{\tau}}\circ (N_{\d,\tau}^{\rm vf})^{-1}:(z,0)\mapsto (z+1,0)\\
&N^{\rm vf}_{\d,\tau}\circ \phi^{q_{\d}}_{\d X_{\tau}}\circ (N_{\d,\tau}^{\rm vf})^{-1}:(z,0)\mapsto (z+\ti\a_{\d,\tau},0)
\end{align*}
and we thus have
$$\ti \a_{\d,\tau}=-\{T_{\tau}/\d\}.$$

\bigskip We now treat the general case $\tau\in \bD_{\C^2}(\tau_{*},\d^2)$.  Let $g_{*}=g(\tau_{*})$. The second equation of (\ref{eq:10.57}) and the constant rank theorem   show that there exists a holomorphic injective map  $f:\bD(g_{*},\rho_{1})\times \bD(0,\rho_{2})\to \bD_{\C^2}(\tau_{*},\rho_{})$ such that $g(f(u,v))=u$. 
 In particular,  for each fixed $v$, the two holomorphic functions 
$$u\mapsto \ti\a_{\d,f(u,v)}\quad\textrm{and}\quad u\mapsto -\{T_{f(u,v)}/\d\}$$
coincide on $\R\cap\bD(g_{*},\d^{3/2})$, hence on $\bD(g_{*},\d^{3/2})$.  These two functions thus coincide  on  $\bD(g_{*},\d^{3/2})\times \bD(0,\rho_{2})$ and also on the connected open  set $g(\bD(\tau_{*},\d^2))$ if $\d$ is small enough.
This proves the lemma.

\end{proof}

\subsection{Proof of Theorem \ref{cor:2.10:commpair}: general case.} \label{sec:10.2}

In the general case  $F_{\d}=O(\d^p)$, one has from Lemma \ref{lemma:est7.2}
\be 
\left\{
\begin{aligned}
&h_{\d,\tau}=\phi_{\d X_{\tau}}^{1}\circ \iota_{O(\d^p)}\\
&h_{\d,\tau}^{q}=\phi_{\d X_{\tau}}^{q}\circ \iota_{O_{}(\d^{p-1})}.
\end{aligned}
\right.
\label{eq:7.34}
\ee
Because $\L_{\d}\circ \iota_{O(\d^k)}\circ \L_{\d}^{-1}=\iota_{O(\d^{k-1})}$, one has (recall $ N^{\rm vf}_{\d,\tau}=\iota_{\d Y_{\d,\tau}}\circ \L_{\d}\circ \iota_{G_{\tau}}\circ\L_{a_{\tau}(0,0)}\circ \Gamma_{\tau}$, cf. (\ref{eq:7.41}))
\begin{align*}&N^{\rm vf}_{\d,\tau}\circ \iota_{O(\d^p)}\circ (N^{\rm vf}_{\d,\tau})^{-1}=\iota_{O(\d^{p-1})}\\
&N^{\rm vf}_{\d,\tau}\circ \iota_{O(\d^{p-1})}\circ (N^{\rm vf}_{\d,\tau})^{-1}=\iota_{O_{}(\d^{p-2})}.
\end{align*}
Using (\ref{Necdelta0normalizes}), (\ref{Necactionon}) and (\ref{eq:7.34}) we thus have
\be 
\left\{
\begin{aligned}
&N^{\rm vf}_{\d,\tau}\circ h_{\d,\eta}\circ (N^{\rm vf}_{\d,\tau})^{-1}=S_{\d\b}\circ\Phi_{w}\circ  \iota_{O(\d^{p-1})}\\
&N^{\rm vf}_{\d,\tau}\circ h_{\d,\eta}^{q}\circ (N^{\rm vf}_{\d,\tau})^{-1}=S_{q_{\d}\d\b}\circ \Phi_{\ti\a_{\d}w}\circ\iota_{F^{\rm vf}_{\d,\tau}}\circ \iota_{O_{}(\d^{p-2})}.
\end{aligned}
\right.
\label{eq:7.35}
\ee

To complete the proof we have to find an exact conservative holomorphic normalization map for $$N^{\rm vf}_{\d,\tau}\circ h_{\d,\tau}\circ (N^{\rm vf}_{\d,\tau})^{-1}=S_{\d\b}\circ\Phi_{w}\circ  \iota_{O(\d^{p-1})};$$ this is the content of Proposition \ref{lemma:normlemma} on symplectic normalization of diffeomorphisms close to $\cT_{1,\d\b}$: there exists a diffeomorphism of the form $\iota_{Y^{cor}_{\d,\tau}}$, 
\be Y^{\rm cor}_{\d,\tau}=O(\d^{p-1})\label{Ycorest}\ee such that 
$$\iota_{Y^{\rm cor}_{\d,\tau}}\circ \biggl(S_{\d\b}\circ\Phi_{w}\circ  \iota_{O(\d^{p-1})}\biggr)\circ \iota_{Y^{cor}_{\d,\tau}}^{-1}=S_{\d\b}\circ\Phi_{w};$$
this also yields
$$\iota_{Y^{cor}_{\d,\tau}}\circ \biggl(S_{q_{\d}\d\b}\circ \Phi_{\ti\a_{\d}w}\circ\iota_{F^{\rm vf}_{\d,\tau}}\circ \iota_{O_{}(\d^{p-2})}\biggr)\circ \iota_{Y^{cor}_{\d,\tau}}^{-1}=S_{q_{\d}\d\b}\circ \Phi_{\ti\a_{\d}w}\circ\iota_{F^{\rm vf}_{\d,\tau}}\circ \iota_{F^{\rm cor}_{\d,\tau}}$$
with $F^{\rm cor}_{\d,\tau}=O(\d^{p-2})$.
The diffeomorphism $N^{\rm ec}_{\d,\tau}$ for which (\ref{conc:thm7.6}) holds is thus 
\be N^{\rm ec}_{\d,\tau}=\iota_{Y^{cor}_{\d,\tau}}\circ N^{\rm vf}_{\d,\tau}.\label{def:Nec}\ee
One can check from (\ref{formvf}) 
$$N^{\rm vf}_{\d,\tau}=\iota_{\d Y^{\rm vf}_{\d,\tau}}\circ \L_{\d}\circ \iota_{G_{\tau}}\circ\L_{a_{\tau}(0,0)}\circ \Gamma_{\tau},\qquad Y^{\rm vf}_{\d,\tau}(z,w)=O(w^2)$$
and $G_{\tau}=O(w)$ (Lemma \ref{lemma:7.7}) that $N_{\d,\tau}^{\rm vf}$ and  $N_{\d,\tau}^{\rm ec}$ can be written
\be \begin{cases}
&N^{\rm vf}_{\d,\tau}=\iota_{\ti G_{\tau}}\circ \L_{\d\ti a_{\tau}}\circ \Gamma_{\tau}\\
&N^{\rm ec}=\iota_{Y^{cor}_{\d,\tau}}\circ N^{\rm vf}_{\d,\tau}
\end{cases}
\ee
where $\ti a_{\tau}=a_{\tau}(0,0)\asymp 1$ and
  $\iota_{\ti G_{\tau}}=\iota_{\d Y^{\rm vf}_{\d,\tau}}\circ \L_{\d}\circ \iota_{G_{\tau}}\circ \L_{\d^{-1}}$,
 hence $\ti G_{\tau}(z,w)=O(w)$.

\medskip This completes the proof of Theorem \ref{cor:2.10:commpair} (where we used the simpler notations $c_{\tau}=\ti a_{\tau}$, $G_{\tau}=\ti G_{\tau}$).
\ \hfill$\Box$

\begin{pspicture}(-5.25,-5.25)(5.25,5.25)%
  \pscircle*[linecolor=cyan]{5}
  \psgrid[subgriddiv=0,gridcolor=lightgray,gridlabels=0pt]
  \Huge\sffamily\bfseries
  \rput(-4.5,4.5){A} \rput(4.5,4.5){B}
  \rput(-4.5,-4.5){C}\rput(4.5,-4.5){D}
  \rput(0,0){pst-pdf}
  \rmfamily
  \rput(0,-3.8){PSTricks}
  \rput(0,3.8){\LaTeX}
\end{pspicture}

\subsection{Proof of Corollary \ref{cor:10.2}}\label{sec:10.3}
From    (\ref{eq:7.41})  one has
$$(N^{\rm vf}_{\d,\tau})^{-1}=(\L_{\d}\circ \iota_{G_{\tau}}\circ \L_{a_{\tau}(0,0)}\circ \Gamma_{\tau})^{-1}\circ \iota_{\d Y_{\d,\tau}^{\rm vf}}^{-1}.$$

The estimates $Y^{\rm vf}_{\d,\tau}(z,w)=O(w^2)$  (see (\ref{formvf})) and  the inclusions   (\ref{e10.95ante})-(\ref{e10.95}) 
yield 
$$\bD_{\C^2}(\zeta_{\tau},C^{-1}\d s_{-})\subset (N^{\rm vf}_{\d,\tau})^{-1}\biggl(\bD(0,s)\times \bD(0,s_{})\biggr)\subset \bD_{\C^2}(\zeta_{\tau},C\d s_{+})$$
with $s_{\pm}=s\pm B\d s^2$ where $B$ is some constant independent of $\d$. In particular, if $s$ is small enough (the smallness being independent of $\d$) one has 
$$\bD_{\C^2}(\zeta_{\tau},(2C)^{-1}\d s)\subset (N^{\rm vf}_{\d,\tau})^{-1}\biggl(\bD(0,s)\times \bD(0,s_{})\biggr)\subset \bD_{\C^2}(\zeta_{\tau},(2C)\d s_{})$$
hence
\begin{multline*}
\bigcup_{t\in (-\nu,1+\nu)}\phi_{\d X_{\tau}}^t\biggl(\bD_{\C^2}(\zeta_{\tau},(2C)^{-1}\d s)\biggr)\subset \\
\bigcup_{t\in (-\nu,1+\nu)}\phi_{\d X_{\tau}}^t\biggl((N^{\rm vf}_{\d,\tau})^{-1}\biggl(\bD(0,s)\times \bD(0,s_{})\biggr)\biggr)\\
\subset \bigcup_{t\in (-\nu,1+\nu)}\phi_{\d X_{\tau}}^t\biggl(\bD_{\C^2}(\zeta_{\tau},(2C)^{}\d s)\biggr).
\end{multline*}
The identity, $(N^{\rm vf}_{\d,\tau})_{*}(\d X_{\tau})=\pa_{z}+(2\pi i\d\beta w)\pa_{w} $ (cf. (\ref{n10.93})), and   the fact that $\cW^\tau_{\d,s,\nu}=\cW^{X_{\tau_{}}, \eta_{\tau}}_{\d,s,\nu}$ (see (\ref{cWtau}))
show that (for some other constant $C>0$)
$$\cW^\tau_{\d,C^{-1}s,\nu/2}\subset (N^{\rm vf}_{\d,\tau})^{-1}\biggl((-\nu_{},1+\nu_{})_{s_{}}\times \bD(0,s_{})\biggr)\subset \cW^\tau_{\d,Cs,2\nu}.$$

Finally, $(N^{\rm ec}_{\d,\tau})^{-1}=(N^{\rm vf}_{\d,\tau})^{-1}\circ \iota_{Y^{cor}_{\d,\tau}}^{-1}$ and $Y^{\rm cor}_{\d,\tau}=O(\d^{p-1})$  (cf. (\ref{Ycorest}), (\ref{def:Nec}))
show that 
\begin{multline*}(N^{\rm vf}_{\d,\tau})^{-1}\biggl((-\nu_{-},1+\nu_{-})_{s_{-}}\times \bD(0,s_{-})\biggr)\subset \\(N^{\rm ec}_{\d,\tau})^{-1}\biggl((-\nu_{},1+\nu_{})_{s_{}}\times \bD(0,s_{})\biggr)\subset \\ (N^{\rm vf}_{\d,\tau})^{-1}\biggl((-\nu_{+},1+\nu_{+})_{s_{+}}\times \bD(0,s_{+})\biggr)\end{multline*}
with $s_{\pm}=s\pm B\d^{p-1}$, $\nu_{\pm}=\nu\pm B\d^{p-1}$ (for some $B>0$ independent of $\d$).

Corollary \ref{cor:10.2} is then a consequence of these two sets of inclusion (changing the value of the constant $C$).

\hfill$\Box$

\subsection{Reversibility}\label{sec:7.5}
We now assume that $\beta\in\R$ and that in addition to condition (1)-(4) of  Assumption \ref{assump:10.1} of the beginning of this section one has

\begin{enumerate}[resume]
\item  There exists a set ${\rm Rev}\subset  \bD_{\C^2}(\tau_{*},\rho)$ such that for any $\tau \in {\rm Rev}$,  there exists an anti-holomorphic involution $\s_{\d,\tau}$ defined on $V$ such that 
$$\s_{\d,\tau}\circ h_{\d,\tau}\circ \s_{\d,\tau}=h_{\d,\tau}^{-1}$$
(recall $h_{\d,\tau}=\phi^1_{\d X_{\tau}}\circ \iota_{F_{\tau}}$, cf. \ref{def:hdeltaetabis}))
\be (\s_{\d,\tau})_{*}X_{\tau}=-X_{\tau}+O(\d^p)\label{hyp:revsigma}\ee
\be \s_{\d,\tau}(\zeta_{\tau})=\zeta_{\tau}+O(\d^p)\label{hyp:revsigma2}\ee
and, for some $a_{\d,\tau}\in\R$, $b_{\d,\tau}\in \C$, $|b_{\d,\tau}|=1$,
$$\s_{\d,\tau}:(z,w)\mapsto (\bar z+a_{\d,\tau},b_{\d,\tau}\bar w)+O(\d).$$
\end{enumerate}

\begin{theo}\label{ref:theocommpairreversible}With the notations of Theorem \ref{cor:2.10:commpair}, for $\tau\in {\rm Rev}\cap\bD(\tau_{*},\d^2)$, each diffeomorphism
\be
\left\{
\begin{aligned}
&N_{\d,\tau}^{\rm ec}\circ h_{\d,\tau}\circ (N_{\d,\tau}^{\rm ec})^{-1}=S_{\d \b}\circ \Phi_{w}:(z,w)\mapsto (z+1,e^{2\pi i\d\b}w)\\
&N_{\d,\tau}^{\rm ec}\circ h^{q_{\d}}_{\d,\tau}\circ (N_{\delta,\tau}^{\rm ec})^{-1}=S_{q_{\d}\d\b}\circ \Phi_{\ti \a_{\d}w}\circ \iota_{F^{\rm vf}_{\d,\tau}}\circ \iota_{F^{\rm cor}_{\d,\tau}}
\end{aligned} 
\right.
\label{eq:7.61}
\ee
is reversible in   $\check \cW^\tau_{s_{0},\nu_{0}}=(N^{ec}_{\d,\tau})^{-1}\biggl((-\nu_{0},1+\nu_{0})_{s_{0}}\times \bD(0,s_{0})\biggr)$  w.r.t. to an anti-holomorphic  involution of the form
$$(z,w)\mapsto (-\bar z+a_{\d,\tau}\bar w,b_{\d,\tau}\bar w)+O(w^2)+O(\d^{p-1})\qquad (a_{\d,\tau},b_{\d,\tau}\in\C).$$
\end{theo}
\begin{proof}
As we saw in the proof of Theorem \ref{cor:2.10:commpair} (vector field case) the diffeomorphism $N^{\rm vf}_{\d,\tau}$ defined by (\ref{eq:7.41}) satisfies
(cf. (\ref{Necdelta0normalizesvf}) and (\ref{Necdelta0normalizes}))
$$(N^{\rm vf}_{\d,\tau})_{*}(\d X_{\tau})= \pa_{z}+(2\pi i\d\beta w)\pa_{w}$$
and for $\th\in (-1,1)+i(-s,s)$
$$N^{\rm vf}_{\d,\tau}\circ \phi^\th_{\d X_{\tau}}\circ (N_{\d,\tau}^{\rm vf})^{-1}:(z,w)\mapsto (z+\th,e^{2\pi i\th\d \beta_{\d}}w).$$
Let 
$$\s^*_{\d,\tau}=N^{\rm vf}_{\d,\tau}\circ \s_{\d,\tau}\circ (N^{\rm vf}_{\d,\tau})^{-1}.$$
Because of the approximate reversibility condition (\ref{hyp:revsigma}), one has
$$\s_{\d,\tau}\circ \phi^\th_{\d X_{\tau}}\circ \s_{\d,\tau}= \phi^{-\bar \th}_{\d X_{\tau}}\circ (id+O(\d^p))$$
hence
$$\s^*_{\d,\tau}\circ \phi^\th_{\pa_{z}+(2\pi i\d\beta w)\pa_{w}}\circ \s^*_{\d,\tau}= \phi^{-\bar \th}_{\pa_{z}+(2\pi i\d\beta w)\pa_{w}}\circ (id+O(\d^{p-1}))$$
and if we set $\s^*_{\d,\eta}(z,w)=(z',w')$
\be \s_{\d,\tau}^*(z+\th,e^{2\pi i\th\d\b}w)=(z'-\bar \th,e^{-2\pi i\bar \th\d\b}w')+O(\d^{p-1})=(z'-\bar \th,w')+O(\d^{p-1}).\label{eq:7.51}\ee
From condition  (\ref{hyp:revsigma2}) one gets
$$\s^*_{\d,\tau}(0,0)=(0,0)+O(\d^{p-1}).$$
This and (\ref{eq:7.51}) imply 
$$\s^*_{\d,\tau}(z,0)=(-\bar z+l,0)+O(\d^{p-1})$$
with $l=O(\d^{p-1})$ and translating the variable $z$ (conjugation by a translation) if necessary we can assume $l=0$.

Proceeding like in the proof of Lemma \ref{lemma:7.8} one can then  show that for some $a_{\d,\tau},b_{\d,\tau}\in\C$, $a_{\d,\tau},b_{\d,\tau}=O_{A}(1)$,
$$\s^*_{\d,\tau}(z,w)=(-\bar z+a_{\d,\tau} \bar w, b_{\d,\tau}\bar w)+O(w^2)+O(\d^{p-1});$$
the fact that $\s^*_{\d,\tau}$ is an anti-holomorphic involution  shows that $b_{\d,\tau}\bar b_{\d,\tau}=1+O(\d^{p-1})$ and $\bar a_{\d,\tau}-a_{\d,\tau} \bar b=O(\d^{p-1})$.

\end{proof}

\subsection{Dependence on parameters}\label{sec:7.5-param}

The estimates on $F^{\rm vf}_{\d,\tau}$, $F_{\d,\tau}^{\rm cor}$ given in Theorems \ref{cor:2.10:commpair} and  \ref{ref:theocommpairreversible} are uniform in $\tau\in \bD_{\C^2}(\tau_{*},\d^2)$  (they only depend  on the constant $A$).

An application of Cauchy's inequality gives the following $C^1$ estimates:

 \begin{prop}\label{prop:parameterdependence}There exists a constant $C_{A}>0$ such that  
 \begin{align*}
 &\| t\mapsto F_{\d}^{\rm vf}(t)\|_{C^1(\bD_{\C^2}(\tau_{*},\d^2),\cO(\check \cW^\tau_{s_{0},\nu_{0}}))} \leq C_{A}\d^{-2}\\
 &\| t\mapsto F_{\d}^{\rm cor}(t)\|_{C^1(\bD_{\C^2}(\tau_{*},\d^2),\cO(\check \cW^\tau_{s_{0},\nu_{0}}))}\leq C_{A}\d^{p-4}.
 \end{align*}

 \end{prop}
 
 \bigskip

\section{Conjugating partially normalized commuting pairs}\label{sec:11}
KAM theorems are obtained by successive conjugations  of diffeomorphisms close to the identity. Finding these conjugating diffeomorphisms requires solving {\it linearized} equations, the so-called {\it cohomological equations}. We present in this section, and  in the setting of partially normalized pairs, how to solve them. The main result we thus obtain is Proposition \ref{prop:8.8} which is the first step of the  KAM procedure we shall use in the proof of the KAM-Siegel theorems in section \ref{sec:SiegelKAMtheorems}.

\bigskip

We recall the following  notations: if $I$ is an interval of $\R$
$$I_{s}=I+i(-s,s)$$
$$\R_{s}=\R+i]-s,s[,\qquad \T_{s}=\T+i]-s,s[$$
$$R_{s,\rho}=(]-1/2,3/2[+i]-s,s[)\times \bD(0,\rho)$$
$$e^{-\nu}R_{s,\rho}=R_{e^{-\nu}s,e^{-\nu}\rho}.$$
If $\a,\b\in\C$ we set 
$$S_{\b}:(z,w)\mapsto (z,e^{2\pi i\b}w)$$
$$\Phi_{\a w}:(z,w)\mapsto (z+\a,w).$$
\begin{lemma}\label{lemma:8.2}If $Y\in \cO(R_{s,\rho})$ one has for $\a,\b_{2}\in \C$,
$$(S_{\b_{2}}\circ \Phi_{\a w})^{-1}\circ \iota_{Y}\circ  (S_{\b_{2}}\circ \Phi_{\a w})= \iota_{\ti Y}$$
with $\ti Y(z,w)=e^{-2\pi i\b_{2}}Y(z+\a,e^{2\pi i\b_{2}}w)$.
\end{lemma}
\begin{proof}One has 
$$\iota_{Y}:(z,w)\mapsto (\ti z, \ti w)\iff \begin{cases}&\ti z=z+\pa_{\ti w}Y(z,\ti w)\\ &w=\ti w+\pa_{z}Y(z,\ti w) \end{cases}$$
hence if $(z',w')=\iota_{Y}(z+\a,e^{ 2\pi i\b_{2}}w)$
$$\begin{cases}&z'=z+\a+\pa_{w'}Y(z+\a,w')\\ &e^{2\pi i\b_{2}}w=w'+\pa_{z}Y(z+\a, w') \end{cases}$$
and if $(z'',w'')=(S_{\b_{2}}\circ \Phi_{\a r})^{-1}(z',w')=(z'-\a,e^{-2\pi i\b_{2}}w')$
$$\begin{cases}&z''=z+\pa_{\ti w}Y(z+\a,e^{2\pi i\b_{2}}w'')\\ &w=w''+e^{-2\pi i\b_{2}}\pa_{z}Y(z+\a,e^{2\pi i\b_{2}}w'') \end{cases}$$
which can be written
$$\begin{cases}&z''=z+\pa_{w''}\ti Y(z,w'')\\ &w=w''+\pa_{z}\ti Y(z,w'') \end{cases}$$
hence $(z'',w'')=\iota_{\ti Y}(z,w)$.
\end{proof}

\subsection{Periodic representatives of partially normalized commuting pairs}

If $F:(z,w)\mapsto F(z,w)\in \C$ we set as usual
$$\iota_{F}:(z,w)\mapsto (\ti z, \ti w)\iff \begin{cases}&\ti z=z+\pa_{\ti w}F(z,\ti w)\\ &w=\ti w+\pa_{z}F(z,\ti w). \end{cases}$$

If $\b_{1}\in\C$ we introduce the map
\be \Psi=\Psi_{\b_{1}}:\C^2\ni (z,w)\mapsto (z,e^{- 2\pi i\b_{1} z }w)\in \C^2\label{def:Psi}\ee
which satisfies
$$\Psi_{\b_{1}}\circ (S_{\b_{1}}\circ \Phi_{w})\circ \Psi_{\b_{1}}^{-1}=\cT_{1,0}:(z,w)\mapsto (z+1,w).$$
Note that when $\b_{1}$ is close to 0, the diffeomorphism $\Psi_{\b_{1}}$ is close to the identity (on any fixed bounded domain):
$$\b_{1}=O(\d)\implies \Psi_{\b_{1}}=id+O(\d).$$
We assume that we are given a partially normalized commuting pair on some open set $W\supset (-\nu_{0},1+\nu_{0})_{s_{0}}\times \bD(0,s_{0})$ 
$$(f_{1},f_{2})_{W}:=\biggl(S_{\b_{1}}\circ \Phi_{r}, S_{\b_{2}}\circ \Phi_{\a r}\circ \iota_{F}\biggr)_{W}$$
with $F\in \cO(W)$.
If $R_{s,\rho}$ is such that 
$$\Psi_{\b_{1}}(R_{s,\rho})\subset W$$
we can consider the restriction to $\Psi_{\b_{1}}(R_{s,\rho})$ of the preceding pair 
$$(f_{1},f_{2})_{\Psi_{\b_{1}} (R_{s,\rho})}:=\biggl(S_{\b_{1}}\circ \Phi_{r}, S_{\b_{2}}\circ \Phi_{\a r}\circ \iota_{F}\biggr)_{\Psi_{\b_{1}}(R_{s,\rho})}.$$
where $F\in \cO(\Psi_{\b_{1}}(R_{s,\rho}))$.

\bigskip
Let us define
\be c_{F}=(e^{-2\pi i\b_{1}}-1)^{-1}(F(0,0)-e^{-2\pi i\b_{1}}F(1,0))\in \C.\label{def:cF}\ee
We assume $\b_{1}$ is small enough.
\begin{lemma}\label{lemma:commrelbis}Let $F\in \cO(\Psi_{\b_{1}}(R_{s,\rho}))$ be such that $c_{F}=0$. The following statements are equivalent. 
\begin{enumerate}
\item The pair $(S_{\b_{1}}\circ \Phi_{w},S_{\beta_{2}}\circ \Phi_{\a w}\circ \iota_{F})_{{\Psi_{\b_{1}}(R_{s,\rho})}}$ is a commuting pair.
\item  $  S_{\b_{1}}\circ \Phi_{w}$ commutes with $\iota_{F}$. 
\item The 
observable $\check F\in \cO(R_{s,\rho})$
\be \check F:(z,w)\mapsto e^{-2\pi iz \b_{1}}F(z,e^{2\pi iz \b_{1}}w)\label{def:tiFwrtF}\ee
is 1-periodic in $z$. In particular, it defines an observable in $\cO(\T_{s}\times \bD(0,\rho))$.
\end{enumerate}
\end{lemma}
\begin{proof}Because the maps $S_{\b_{1}}\circ \Phi_{w}$ and $S_{\b_{2}}\circ \Phi_{\a w}$ commute, the fact that $(S_{\b_{1}}\circ \Phi_{w},S_{\beta_{2}}\circ \Phi_{\a w}\circ \iota_{F})$ is a commuting pair is equivalent to the commutation of $\iota_{F}$ and $S_{\b_{1}}\circ \Phi_{r}$ which shows the equivalence of (1) and (2).

We now prove the equivalence of (2) and (3).
Remembering
$$\iota_{F}(z,w)=(\ti z,\ti w)\iff 
\begin{cases}
&\ti z=z+\pa_{\ti w}F( z, \ti w)\\
&w=\ti w+\pa_{ z}F( z, \ti w),
\end{cases}
$$
the commutation relation (2)  reads
$$
\begin{cases}
&\ti z+1=z+1+\pa_{\ti w}F( z+1, e^{2\pi i\b_{1}}\ti w)\\
&e^{2\pi i\b_{1}}w=e^{2\pi i\b_{1}}\ti w+\pa_{ z}F( z+1, e^{2\pi i\b_{1}}\ti w)
\end{cases}
$$
which is equivalent to 
$$
\begin{cases}
&\pa_{\ti w}F(z,\ti w)=\pa_{\ti w}F( z+1, e^{2\pi i\b_{1}}\ti w)\\
&\pa_{z}F(z,\ti w)=e^{-2\pi i\b_{1}}\pa_{ z}F( z+1, e^{2\pi i\b_{1}}\ti w).
\end{cases}
$$
This yields
$$
\begin{cases}
&\pa_{\ti w}\biggl(e^{-2\pi i\b_{1}}F(z+1,e^{2\pi i\b_{1}}\ti w)-F(z,\ti w) \biggr)=0\\
&\pa_{z}\biggl(e^{-2\pi i\b_{1}}F(z+1,e^{2\pi i\b_{1}}\ti w)-F(z,\ti w) \biggr)=0
\end{cases}
$$
hence
$$e^{-2\pi i\b_{1}}F(z+1,e^{2\pi i\b_{1}}\ti w)-F(z,\ti w)={\rm cst}=e^{-2\pi i\b_{1}}F(1,0)-F(0,0);$$
the condition $c_{F}=0$ gives (cf. (\ref{def:cF}))
$$e^{-2\pi i\b_{1}} F(z+1,e^{2\pi i\b_{1}}\ti w)-F(z,\ti w)=0.$$
Setting
\be \check F(z,w)=e^{2\pi iz\b_{1}}F(z,e^{-2\pi iz\b_{1}}w),\label{deftiFter}\ee
we thus have
$$\check F(z+1,w)-\check F(z,w)=0$$
for $(z,w)\in \T_{s}\times \bD(0,\rho)$.

\end{proof}

\begin{rem}Since for $c\in\C$, $\iota_{c+F}=\iota_{F}$, we can  assume without loss of generality  that $c_{F}=0$ without changing  $\iota_{F}$.
\end{rem}

\begin{rem}\label{rem:8.1} If $F\in \cO(\Psi_{\b_{1}}(R_{s,\rho}))$ is small, the diffeomorphism $\Psi_{\b_{1}}$ conjugates the pair $(f_{1},f_{2})$ to a commuting pair
$$(f_{1}',f_{2}'):=\Psi_{\b_{1}}\circ(f_{1}, f_{2})\circ \Psi_{\b_{1}}^{-1}=(\cT_{1,0},\cT_{\a,\check \b}\circ (id+\psi_{F}))$$
where 
$$\begin{cases}&\cT_{\a,\check \b}:(z,w)\mapsto (z+\a,e^{2\pi i\check \b}w)\\
&\check \b=\b_{2}-\a\b_{1}
\end{cases}$$
and $\psi_{F}\in \cO(R_{s,\rho})$, $\psi_{F}=\fO_{1}(F)$, is $\cT_{1,0}$-periodic i.e. satisfies $\psi_{F}\circ \cT_{1,0}=\psi_{F}$ (it is periodic in the $z$-variable). The diffeomorphism $(z,w)\mapsto (z,w)+\psi_{F}(z,w)$ is thus defined on the cylinder $\T_{s}\times \bD(0,\rho)$ . However, it is {\it not symplectic} w.r.t. the standard symplectic form $dz\wedge dw$.
\end{rem}

\subsection{Cohomological equation }
\begin{lemma}\label{lemma:8.4}Let $F,Y\in \cO(\Psi_{\b_{1}}(R_{s,\rho}))$ be such that $c_{F}=c_{Y}=0$ and define
\begin{align*}&\check F(z,w)=e^{-2\pi i\b_{1} z}F(z,e^{2\pi i\b_{1}z}w)\\
&\check Y(z,w)=e^{-2\pi i\b_{1} z}Y(z,e^{2\pi i\b_{1}z}w).
\end{align*}
The system 
\be
\forall(z,w)\in \Psi_{\b_{1}}(R_{s,\rho})\quad \left\{
\begin{aligned}&F(z+1,e^{2\pi i\b_{1}}w)=e^{2\pi i\b_{1}}F(z,w)\\
&Y(z+1,e^{2\pi i\b_{1}}w)=e^{2\pi i\b_{1}}Y(z,w)\\
&e^{-2\pi i\b_{2}}Y(z+\a,e^{2\pi i\b_{2}}w)-Y(z,w)=F(z,w)
\end{aligned}
\right.
\label{eq:8.69}
\ee
is equivalent to 
$$
\forall(z,w)\in R_{s,\rho}\quad
\left\{
\begin{aligned}
&\check F(z+1,w)=\check F(z,w)\\
&\check Y(z+1,w)=\check Y(z,w)\\
&e^{-2\pi i\check \b}\check Y(z+\a,e^{2\pi i\check \b}w)-\check Y(z,w)=\check F(z,w).
\end{aligned}
\right.
$$
where $\check \b=\b_{2}-\a \b_{1}$.
\end{lemma}
\begin{proof}
We just have to check that the equivalence 
\begin{multline*}e^{-2\pi i\b_{2}}Y(z+\a,e^{2\pi i\b_{2}}w)-Y(z,w)=F(z,w)\iff \\e^{-2\pi i \check\b} \check Y(z+\a,e^{2\pi i\check \b}w)-\check Y(z,w)=\check F(z,w)\end{multline*}
holds.
This is done the following way: the equality on the left hand side of the equivalence reads
\begin{multline*}e^{-2\pi i\b_{2}}e^{2\pi i\b_{1}(z+\a)}\check Y(z+\a,e^{-2\pi i\b_{1}(z+\a)}e^{2\pi i\b_{2}}w)-e^{2\pi i\b_{1}z}\check Y(z,e^{2\pi i\b_{1}z}w)=\\ e^{2\pi i\b_{1}z}\check F(z,e^{-2\pi i\b_{1}z}w)\end{multline*}
or equivalently
\begin{multline*}e^{-2\pi i(\b_{2}-\b_{1}\a)}\check Y(z+\a,e^{2\pi i(\b_{2}-\a\b_{1})}e^{-2\pi i\b_{1}z}w)\\ -\check Y(z,e^{-2\pi i\b_{1}z}w)=\check F(z,e^{-2\pi i\b_{1}z}w).\end{multline*}
\end{proof}

\subsection{Non-resonance and Diophantine conditions}
We say that a pair $(\a, \check \b)\in \C^2$ is {\it non resonant } if 
$$\forall(k,l,m),\in \Z\times \N\times \Z, \ k_{1}\a+(l-1)\check \b-m= 0\implies k=l=m=0..$$

If $c_{*},e_{*}$ are positive numbers, we define the closed sets $DC(c_{*},e_{*})$, $DC_{\R^2}(c_{*},e_{*})$ and $DC_{\R}(c_{*},e_{*})$ as
\begin{multline*}{DC}(c_{*},e_{*})=\biggl\{(\a,\check\beta)\in \C^2\mid \forall (k,l,m)\in \Z\times \N\times \Z,\ |k|+|l-1|\ne 0\\ \implies |k\a+(l-1)\check\beta-m|\geq \frac{c_{*}}{(|k|+|l-1|)^{e_{*}}}\biggr\}) \end{multline*}
$$DC_{\R^2}(c_{*},e_{*})=DC(c_{*},e_{*})\cap \R^2$$
and 
$$DC_{\R}(c_{*},e_{*})=\{\a\in \R\mid (\a,0)\in DC(c_{*},e_{*})\}.$$
Note that 
\be \begin{cases}&|\Im\check\b| >c_{*}\\\
&\a\in DC_{\R}(c_{*},e_{*})\end{cases}
\implies (\a,\check\b)\in DC(c_{*},e_{*}).
\label{DCRDC}
\ee

The following lemmas are easy to prove.
\begin{lemma}\label{lemma:diophn}Assume $e_{*}>3$ and let $B_{1},B_{2}\subset\C$
 be  nonempty open disks with center on $\R$ and $I_{j}=B_{j}\cap \R$, $j=1,2$.
 One has 
$$\begin{cases}
&{\rm Leb}_{\C^2}( (B_{1}\times B_{2})\setminus DC(c_{*},e_{*}))\lesssim c_{*}\\
&{\rm Leb}_{\R^2}((I_{1}\times I_{2})\setminus DC(c_{*},e_{*}))\lesssim c_{*}.  
\end{cases}$$
\end{lemma}
\begin{proof}These are classical properties of Diophantine sets.
Let's prove the first estimate by writing 
\begin{multline*}(B_{1}\times B_{2}) \setminus DC(c_{*},e_{*})\subset \\ \bigcup_{\substack{(k,l,,m)\in\Z^3\\ (k,l)\ne (0,1)}} \biggl\{(\a,\check\beta)\in B_{1}\times B_{2}\mid\ |k\Re\a+(l-1)\Re\check\beta-m|< \frac{c_{*}}{(|k|+|l-1|)^{e_{*}}}\biggr\} .
\end{multline*}
Thus,
$$\begin{cases}&{\rm Leb}_{\C^2}( (B_{1}\times B_{2})\setminus DC(c_{*},e_{*}))\lesssim c_{*}A_{e_{*}}\\
&A_{e_{*}}=\sum_{{\substack{(k,l,,m)\in\Z^3\\ (k,l)\ne (0,1)}}} (|k|+|l-1|)^{-e_{*}}<\infty.
\end{cases}$$
\end{proof}

Let $D$ be an open set of $\C^2$ and $D_{\R^2}=D\cap \R^2$.

\begin{lemma} \label{item:3n}Assume there exist a $C^1$ injective map  $\ph:D\to B_{1}\times B_{2}$ (resp. $\ph:D_{\R^2}\to I_{1}\times I_{2}$) then
$${\rm Leb}_{C^2}\biggl(\{t\in D\mid \ph(t)\notin DC(c_{*},e_{*})\}\biggr)\lesssim \sup |{\rm Jac}(\ph)|^{-1}\times c_{*}.$$
(resp. $${\rm Leb}_{R^2}\biggl(\{t\in D_{\R^2}\mid \ph(t)\notin DC_{\R^2}(c_{*},e_{*})\}\biggr)\lesssim \sup |{\rm Jac}(\ph)|^{-1}\times c_{*}.)$$
\end{lemma}
\begin{proof}Just observe that $\{t\in D\mid \ph(t)\notin DC(c_{*},e_{*})\}\subset \ph^{-1}((B_{1}\times B_{2})\setminus DC(c_{*},e_{*}))$ and use the change of variable formula and the estimate given by previous lemma.

The second inequality is proved in a similar way.
\end{proof}

\begin{notation} We shall often take in the rest of the text $e_{*}=4$ and set
$$DC(c_{*})=DC(c_{*},4).$$
\end{notation}

\subsection{Solving the cohomological equation}
For $F\in \cO(\Psi_{\b_{1}}(R_{s,\rho}))$ such that $\iota_{Y}$ commutes with $S_{\b_{1}}\circ \Phi_{w}$, we define  the following quantity which is independent of $\e>0$ (small enough):
$$\cM_{}(F)=\int_{\T}\biggl(\frac{1}{2\pi i }\int_{\pa \bD(0,\e)}\frac{e^{-2\pi i\b_{1}z}F(z,e^{2\pi i\b_{1} z}w)}{w^2}dw\biggr) dz$$
(the function under the integral is 1-periodic in $z$ by Lemma \ref{lemma:commrelbis}).

Note that when $F(z)=aw$, $a\in\C$ one has $\cM(F)=a$.
\begin{lemma}\label{lemma:8.8}Assume that  $(\a,\b_{2}-\a\b_{1})$ is in $DC(c_{*},e_{*})$. Let 
$F\in \cO(\Psi_{\b_{1}}(R_{s,\rho}))$ be such that $\iota_{F}$ and $S_{\b_1}\circ \Phi_{w}$ commute and $c_{F}=0$. 
Then, there exists  $Y\in \cO(\Psi_{\b_{1}}(R_{s,\rho}))$  ($c_{Y}=0$) such that $\iota_{Y}$ commutes with $S_{\b_1}\circ \Phi_{w}$ and  solves on $\Psi_{\b_{1}}(R_{s,\rho})$
\be e^{-2\pi i\b_{2}} Y(z+\a,e^{2\pi i\b_{2}}w)-Y(z,w)=F(z,w)-\cM(F)w.\label{eq:8.70}\ee
Moreover, for any $\nu>0$ one has 
\be \|Y\|_{\Psi_{\b_{1}}(e^{-\nu}R_{s,\rho})}\lesssim_{e_{*},s} c_{*}^{-1}\nu^{-(e_{*}+2)}\|F\|_{\Psi_{\b_{1}}(R_{s,\rho})}.\label{estYF}\ee

\end{lemma}
\begin{proof}We define 
\begin{align*}&\check F(z,w)=e^{-2\pi i\b_{1} z}F(z,e^{2\pi i\b_{1}z}w)-\cM(F)w\\
&\check Y(z,w)=e^{-2\pi i\b_{1} z}Y(z,e^{2\pi i\b_{1}z}w).
\end{align*}
From Lemma \ref{lemma:8.4}, equation (\ref{eq:8.70}) is equivalent to 
$$\forall(z,w)\in R_{s,\rho}\quad e^{-2\pi i\check \b}\check Y(z+\a, e^{2\pi i\check \b}w)-\check Y(z,w)=\check F(z,w).$$
The observables $\check Y$, $\check F$ are 1-periodic in the $z$-variable and can be seen as observables in $\T_{s}\times \bD(0,\rho)$; they can be expanded in Taylor-Fourier series.
Writing
\begin{align*}
&\check F(\th,r)=\sum_{n\in\N} \check F_{n}(\th)r^n=\sum_{n\in\N}\sum_{k\in\Z}\hat{ \check F}_{n}(k)e^{2\pi i k\th}r^n\\
& \check Y(\th,r)=\sum_{n\in\N}\check Y_{n}(\th)r^n=\sum_{n\in\N}\sum_{k\in\Z}\hat{ \check Y}_{n}(k)e^{2\pi i k\th}r^n
\end{align*}
the preceding equality  reads
$$\check F_{n}(\th)=e^{2\pi (n-1)i \check\beta} \check Y_{n}(\th+\a)- \check Y_{n}(\th)$$
and in Fourier
\be \hat{ \check F}_{n}(k)=(e^{2\pi i (k\a+(n-1)\check\beta)}-1)\hat{\check Y}_{n}(k).\label{eq:8.71}\ee
Note that 
\begin{align*}\hat{\check F}_{n=1}(k=0)&=\int_{\T}\check F_{1}(\th)d\th\\
&=\frac{1}{2\pi i} \int_{\T}\int_{C(0,\e)} \frac{\check F(\th,r)}{r^2}drd\th\\
&=\frac{1}{2\pi i} \int_{\T}\int_{C(0,\e)} \frac{e^{- 2\pi i\b_{1}z}F(z,e^{2\pi i\b_{1}z}w) -\cM(F)w}{w^2} dwdz\\
&=\frac{1}{2\pi i} \int_{\T}\int_{C(0,\e)} \frac{e^{-2\pi i\b_{1}z}F(z,e^{2\pi i\b_{1}z}w) }{w^2} dwdz-\cM(F)\\
&=0.
\end{align*}
Equations (\ref{eq:8.71}) are solved by setting 
\be
\left\{
\begin{aligned}& \hat{\check Y}_{1}(0)=0\\
&\hat{\check Y}_{n}(k)=\frac{\hat{\check F}_{n}(k)}{e^{2\pi i (k\a+(n-1)\check\beta)}-1}\quad\textrm{if}\ (n,k)\ne (1,0);
\end{aligned}
\right.
\label{def:YwrtF}
\ee
we then get for any $(n,k)\in\N\times \Z$,
$$|\hat{\check Y}_{n}(k)|\lesssim c_{*}^{-1}(|k|+|n-1|)^{e_{*}} |\hat{F}_{n}(k)|.$$
 This yields for any $n\in\N$ and any $\nu>0$
$$\|\check Y_{n}\|_{e^{-\nu/2}s}\lesssim_{e_{*},s} c_{*}^{-1}\nu^{-(e_{*}+1)}(1+\nu |n-1|)^{e_{*}} \| F_{n}\|_{s}$$
hence
$$\|\check Y\|_{e^{-\nu}W_{s,\rho}}\lesssim_{e_{*}} c_{*}^{-1}\nu^{-(e_{*}+2)}\|\check F\|_{W_{s,\rho}}.$$
This implies the estimate (\ref{estYF}).

\end{proof}

\subsection{The linearization step}
We can now apply the results of the preceding subsections to the linearization problem.
\begin{prop}[KAM-like]\label{prop:8.8}
Assume $(S_{\b_{1}}\circ \Phi_{r},S_{\b_{2}}\circ \Phi_{\a r}\circ \iota_{F})$ is a commuting pair with $F\in \cO(\Psi_{\b_{1}}(R_{s,\rho}))$. If $Y\in \cO(\Psi_{\b_{1}}(R_{s,\rho}))$ is such that $\iota_{Y}$ commutes with $S_{\b_{1}}\circ \Phi_{r}$ and  is a solution of the cohomological equation 
\be e^{-2\pi i\b_{2}}Y(z+\a,e^{2\pi i\b_{2}}w)-Y(z,w)=F(z,w)-\cM(F)w\label{eq:8.70ter}\ee
then 
\begin{align*}
&\iota_{Y}\circ \biggl(S_{\b_{1}}\circ \Phi_{w}\biggr)\circ \iota_{Y}^{-1}=S_{\b_{1}}\circ \Phi_{r}\\
&\iota_{Y}\circ \biggl(S_{\b_{2}}\circ \Phi_{\a w}\circ \iota_{F}\biggr)\circ \iota_{Y}^{-1}=S_{\b_{2}}\circ \Phi_{(\a+\cM(F)) w}\circ \iota_{\ti F}
\end{align*}
where $\ti F=\fO_{2}(F,Y)$ (in particular, for any $\nu=\frak{d}(F)$, $\ti F\in \cO(\Psi_{\b_{1}}(e^{-\nu}R_{s,\rho}))$); see the notations of subsection \ref{sec:notationfrakO}.
\end{prop}
\begin{proof}
We just have to check the second equality.

We write  
\begin{multline*}
\iota_{Y}\circ (S_{\b_{2}}\circ \Phi_{\a w}\circ \iota_{F})\circ \iota_{Y}^{-1}=\\ (S_{\b_{2}}\circ \Phi_{\a w})\circ \biggl((S_{\b_{2}}\circ \Phi_{\a w})^{-1}\circ \iota_{Y}\circ (S_{\b_{2}}\circ \Phi_{\a w}) \biggr)\circ \iota_{F}\circ \iota_{Y}^{-1}
\end{multline*}
and using Lemmata \ref{lemma:8.1} and \ref{lemma:8.2} and the notation $\ti Y(z,w)=e^{-i\b_{2}}Y(z+\a,e^{i\b_{2}}w)$
\begin{align*}\iota_{Y}\circ (S_{\b_{2}}\circ \Phi_{\a w}\circ \iota_{F})\circ \iota_{Y}^{-1}&=(S_{\b_{2}}\circ \Phi_{\a w})\circ \iota_{\ti Y+F-Y}\circ \iota_{\fO_{2}(Y,F)}\\
&=S_{\b_{2}}\circ \Phi_{\a w}\circ\iota_{\cM(F)}\circ \iota_{\fO_{2}(Y,F)}\\
&=S_{\b_{2}}\circ \Phi_{(\a+\cM(F)) w}\circ \iota_{\fO_{2}(Y,F)}.
\end{align*}

\end{proof}

\section{KAM-Siegel theorems for partially normalized commuting pairs}\label{sec:SiegelKAMtheorems}

The aim of this section is to prove the KAM-Siegel theorems we need to prove the existence of rotation domains or attracting annuli. These are Theorems \ref{theo:SiegelReversible} and \ref{theo:SiegelDissipative} that will be applied in Sections \ref{sec:proofmainA} and \ref{sec:proofmainA:B} to a partially normalized  commuting pair given by Theorem \ref{cor:2.10:commpair}.

\bigskip

We assume we are given a partially normalized commuting pair $(f_{1},f_{2})_{W}$ defined on some open set $W=\Psi_{\b_{1}}(R_{s,\rho})\supset (-\nu_{0},1+\nu_{0})_{s_{0}}\times \bD(0,s_{0})$  and that it is of the form
\be 
\left\{
\begin{aligned}
&f_{1}=S_{\b_{1}}\circ \Phi_{w}:(z,w)\mapsto (z+1,e^{2\pi i\b_{1}}w)\\
&f_{2}=S_{\b_{2}}\circ \Phi_{\a w}\circ \iota_{F^{\rm vf}_{}}\circ \iota_{F^{\rm cor}}
\end{aligned} 
\right.
\label{eq:9.78bis}
\ee
where $\a,\b_{1},\b_{2}\in\C$,  $F^{\rm vf},F^{\rm cor}\in \cO(\Psi_{\b_{1}}(R_{s,\rho}))$ and 
\be \begin{cases}
&F^{\rm vf}(z,w)=O(w^2)\\
&\|F^{\rm cor}\|_{\Psi_{\b_{1}}(R_{s,\rho})}=O(\d^{p-2}).
\end{cases}
\label{eq:9.79bis}
\ee

Note that this is the form of commuting pairs Theorem \ref{cor:2.10:commpair} yields. 

In the reversible case (then $\b_{1},\b_{2}$ are real numbers), we shall assume, in addition, that  the commuting  pair $(S_{\b_{1}}\circ \Phi_{w},S_{\beta_{2}}\circ \Phi_{\a w}\circ \iota_{F^{\rm vf}}\circ \iota_{ F^{\rm cor}})$ is reversible w.r.t. some anti-holomorphic involution $\s$, $$\s=\s_{0}\circ L_{a,b}\circ (id+\eta):(z,w)\mapsto (-\bar z+a\bar w, b\bar w)+O(w^2)+O(\d^{p-1}),$$
where  $\eta:R_{s,\rho}\to \C^2$ and $\|\eta\|_{\Psi(R_{s,\rho})}=o(\d^{p-1})$. Note that one can choose $b$ such that $|b|=1$.  See Subsection \ref{sec:7.5}.

\subsection{Putting the system into suitable KAM form}\label{sec:suitableKAMform}
We now perform a conjugation that takes our commuting pair $(f_{1},f_{2})$ to a form to which we  shall be able to apply a KAM scheme. 
\begin{prop}\label{prop:10.1}The  exact conformal holomorphic diffeomorphism $D_{\d^{(p-2)/2}}:(z,w)\mapsto (z,\d^{-(p-1)/2}w)$ conjugates  the commuting pair $(f_{1},f_{2})$ to a commuting  pair of the form
\be 
\left\{
\begin{aligned}
&f'_{1}=S_{\b_{1}}\circ \Phi_{w}:(z,w)\mapsto (z+1,e^{ 2\pi i\b_{1}}w)\\
&f'_{2}=S_{\b_{2}}\circ \Phi_{\a w}\circ  \iota_{F'}.
\end{aligned} 
\right.
\label{eq:9.78ter}
\ee
where $F'\in \cO(\Psi_{\b_{1}}(R_{s/2,\rho/2}))$ satisfies $F'=O(\d^{(p-2)/2})$.
\end{prop}
\begin{proof}
Let $D_{\d^{(p-2)/2}}:(z,w)\mapsto (z,\d^{-(p-1)/2}w)$. We then have
$$
\left\{
\begin{aligned}
&D_{\d^{(p-2)/2}}\circ (S_{\b_{1}}\circ \Phi_{w})\circ D_{\d^{(p-2)/2}}^{-1}=(S_{\b_{1}}\circ \Phi_{w})\\
&D_{\d^{(p-2)/2}}\circ (S_{\b_{2}}\circ \Phi_{\a w}\circ \iota_{F^{\rm vf}_{}}\circ \iota_{F^{\rm cor}})\circ D_{\d^{(p-2)/2)}}^{-1}\\&\hskip 5cm=S_{\b_{2}}\circ \Phi_{\a w}\circ \iota_{O(\d^{(p-2)/2})}\circ \iota_{O(\d^{(p-2)/2})}.
\end{aligned} 
\right.
$$
\end{proof}

\begin{prop}\label{prop:10.2}In the reversible case, the commuting pair $(f'_{1},f'_{2})$ of the preceding proposition is  conjugate by a map of the form $(z,w)\mapsto (z,e^{it_{} }w)$ ($t_{}\in\R$) to a commuting pair $(f_{1}'',f_{2}'')$ which is  reversible w.r.t. an anti-holomorphic involution $\s''=\s_{0}\circ (id+\eta'')$ with $\s''=O(\d^{(p-2)/2})$. Furthermore, one has
$$\bar \a-\a=O(\d^{(p-2)/2}).$$
\end{prop}
\begin{proof}After conjugation by the map $D_{\d^{(p-2)/2}}$ the anti-holomorphic involution $\s$ becomes $\s'=\s_{0}\circ \L_{\d^{(p-2)/2} a',b'}\circ(id+O(\d^{(p-2)/2}))$ ($a',b'=O(1)$). A conjugation by  $(z,w)\mapsto (z,e^{it_{}}w)$ where $t_{}\in\R$ is such that $e^{2it_{}}=b$ reduces $\s'$ to $\s''=\s_{0}\circ(id+O(\d^{(p-2)/2}))$.

Using the fact that $f_{2}''=S_{\b_{2}}\circ \Phi_{\a w}\circ \iota_{O(\d^{(p-1)/2})}:(z,w)\mapsto(z+\a,e^{i\b_{2}}w)+O(\d^{(p-2)/2})$ (with $\b_{2}\in \R$) is reversible w.r.t. $\s''$    shows $\bar \a-\a=O(\d^{(p-2)/2})$.
\end{proof}

As a corollary of Propositions \ref{prop:10.1} and \ref{prop:10.2} we can state:
\begin{cor}\label{cor:10.3}Given  a commuting pair $(f_{1},f_{2})$ of the form (\ref{eq:9.78bis}), (\ref{eq:9.79bis}, there exist $s,\rho>0$,  $F^{KAM}\in \cO(\Psi_{\b_{1}}(R_{s,\rho}))$ and a holomorphic conformal symplectic mapping that conjugates $(f_{1},f_{2})$ to a commuting pair  $(f_{1}^{KAM},f_{2}^{KAM})$ of the form 
\be 
\left\{
\begin{aligned}
&f^{KAM}_{1}=S_{\b_{1}}\circ \Phi_{w}:(z,w)\mapsto (z+1,e^{2\pi i\b_{1}}w)\\
&f^{KAM}_{2}=S_{\b_{2}}\circ \Phi_{\a w}\circ  \iota_{F^{KAM}}.
\end{aligned} 
\right.
\label{eq:9.78quater}
\ee
such that for $\d$ small enough  $\|F^{KAM}\|_{\Psi_{\b_{1}}(R_{s,\rho})}\leq \d^{(p-2)/2}$.

Moreover, in the reversible case the pair $(f_{1}^{KAM},f_{2}^{KAM})$ is reversible w.r.t. an anti-holomorphic involution $\s^{KAM}$ of the form $\s^{KAM}=\s_{0}\circ (id+\eta)$ where $\|\eta\|_{\Psi_{\b_{1}}(R_{s,\rho})}\lesssim \d^{(p-2)/2}$. 
\end{cor}

\bigskip

\subsubsection*{Dependence on parameters} 
We now assume that the commuting pair (\ref{eq:9.78bis}) depends on a parameter $t\in D_{}$, where $D_{}$ is an open disk of $\C$ or of  $\C^2$  of diameter $2\d^2$ and we suppose (like in  Proposition \ref{prop:parameterdependence})
 \begin{align*}
 &\| t\mapsto F_{\d}^{\rm vf}(t)\|_{C^1(D,\cO(R_{s,\rho}))} \leq C \d^{-2}\\
 &\| t\mapsto F_{\d}^{\rm cor}(t)\|_{C^1(D,\cO(R_{s,\rho}))}\leq C \d^{p-4}.
 \end{align*}
The commuting pairs (\ref{eq:9.78ter}) and (\ref{eq:9.78quater}) then depend on the parameter $t\in D$.
\begin{prop}\label{prop:10.4}One has
$$\|t\mapsto F^{KAM}(t)\|_{C^1(D,\cO(\Psi(R_{s,\rho})))}\lesssim_{C}\d^{(p-2)/2-2}.$$
[If necessary, $s$ and $\rho$ are modified  by an additive contant $=O(\b_{1})$.]
\end{prop}
\begin{proof}The proof is done like in Proposition \ref{prop:10.1}.
\end{proof}

\subsection{The KAM scheme}
We assume we are given a commuting pair $(f_{1},f_{2})=(f_{1}^{KAM},f_{2}^{KAM})$ satisfying the conclusion of Corollary \ref{cor:10.3}. By Proposition \ref{prop:8.8} and Lemma \ref{lemma:8.8} one has:

\begin{prop} \label{prop:KAMstep}  For any $(\a,\b_{1},\b_{2})\in\C^3$ such that  $\gamma:=(\a,\b_{2}-\a\b_{1})\in DC(c_{*},e_{*})$, there exists $\e>0$ such that for  any   $F\in \cO(\Psi_{\b_{1}}(W_{s,\rho}))$, $\|F\|_{\Psi_{\b_{1}}(R_{s,\rho})}\leq \e$, the following holds.
There exist $\nu=\frak{d}(F)$, $Y_{\g,F}, \ti{F}_{\g,F}\in \cO(\Psi_{\b_{1}}(e^{-\nu}R_{s,\rho}))$ such that 
\be \iota_{Y_{\g,F}}\circ\bm S_{\b_{1}}\circ \Phi_{w}\\ S_{\beta_{2}}\circ \Phi_{\a w}\circ \iota_{F} \em\circ  \iota_{Y_{\g,F}}^{-1}=\bm S_{\b_{1}}\circ \Phi_{w}\\ S_{\beta_{2}}\circ \Phi_{(\a+\cM({F})) w} \circ \iota_{\ti {F}_{\g,F}} \em\label{eq:conjrelation}\ee
with $Y_{\g,F}=c_{*}^{-1}\frak{O}_{1}(F)$ and $\ti{F}_{\g,F}=c_{*}^{-2}\frak{O}_{2}(F)$.

\end{prop}
\begin{proof}Using Lemma \ref{lemma:8.8} we can solve the cohomological equation
$$e^{-2\pi i\b_{2}} Y(z+\a,e^{2\pi i\b_{2}}w)-Y(z,w)=F(z,w)-\cM(F)w$$
with $Y=c_{*}^{-1}\fO_{1}(F)$ and $\iota_{Y}$ commuting with $S_{\b_{1}}\circ \Phi_{w}$. We then apply Proposition \ref{prop:8.8} to get
 $$\iota_{Y_{\g,F}}\circ\bm S_{\b_{1}}\circ \Phi_{w}\\ S_{\beta_{2}}\circ \Phi_{\a w}\circ \iota_{F} \em\circ  \iota_{Y_{\g,F}}^{-1}=\bm S_{\b_{1}}\circ \Phi_{w}\\ S_{\beta_{2}}\circ \Phi_{(\a+\cM({F})) w} \circ \iota_{\ti {F}_{\g,F}} \em$$
 with $\ti F=\fO_{2}(Y,F)=c_{*}^{-2}\fO_{2}(F)$.
\end{proof}

\subsection{Treating $\g$ as a parameter: Whitney type extensions}\label{sec:Whitney}

If $(\cE_{j},\|\cdot\|_{j})$, $j=1,2$ are two Banach spaces, $V\subset \cE_{1}$ a non empty open set   and $\ph:V\to \cE_{2}$ a $C^1$ map we denote by $\vvvert \ph\vvvert_{V,\cE_{2}}=\| \ph\|_{C^1(V,\cE_{2})}$ the $C^1$-norm of $\ph$ and shall often use the short hand notation  $\vvvert \ph\vvvert_{V}$. We refer to (\ref{defcBepsilonbis}) for the notation $\cB_{\e}(U)$.

\begin{prop}\label{prop:KAMWhitneyStepn}Let $c_{*}>0$ and $\g_{*}\in \C^2$. There exist constants $C>0$, $a>0$ such that for any $\bar \e>0$ and  $\nu>0$ satisfying
$$C(c_{*}\nu)^{-a}\bar \e\leq 1$$ there exist
$C^1$ maps
$$\bD_{\C^2}(\g_{*},\rho_{*})\times \cB_{\bar \e}(\cO(\Psi_{\b_{1}}(R_{s,\rho})))\ni (\g,F)\mapsto \begin{cases}&{Y}^{\rm Wh}_{\g,F}\\
&\ti{F}^{\rm Wh}_{\g,F}\\
&G_{\g,F}
\end{cases}\qquad\in  \cO(\Psi_{\b_{1}}(e^{-\nu}R_{s,\rho}))$$
(see (\ref{defcBepsilonbis})) and
$$\bD_{\C^2}(\g_{*},\rho_{*})\times \cB_{\bar \e}(\cO(\Psi(W_{s,\rho})))\ni  (\g,F)\mapsto \cM^{\rm Wh}_{\g,F}\in \C
$$
such that 
$$\iota_{Y^{\rm Wh}_{\g,F}}\circ \bm S_{\b_{1}}\circ \Phi_{w}\\ S_{\beta_{2}}\circ \Phi_{\a w}\circ \iota_{F^{}_{}} \circ \iota_{G_{\g,F}}\em\circ \iota_{Y^{\rm Wh}_{\g,F}}^{-1}=\bm S_{\b_{1}}\circ \Phi_{w}\\ S_{\beta_{2}}\circ \Phi_{(\a+\cM(F^{\rm Wh}_{\g,F})) w}\circ \iota_{\ti{F}^{\rm Wh}_{\g,F}}\em$$
and that satisfy, for any $0\leq \e\leq \bar\e$ the estimates
\be \begin{cases}
&\vvvert (\g,F)\mapsto \ti{F}^{\rm Wh}_{\g,F}\vvvert_{\bD_{\C^2}(\g_{*},\rho_{*})\times \cB_{\e}(\Psi_{\b_{1}}(e^{-\nu}R_{s,\rho} ))}\lesssim (c_{*}\nu)^{-a}\e^2\\
&\vvvert (\g,F)\mapsto{Y}^{\rm Wh}_{\g,F}\vvvert_{\bD_{\C^2}(\g_{*},\rho_{*})\times \cB_{\e}(\Psi_{\b_{1}}(e^{-\nu}R_{s,\rho}))}\lesssim (c_{*}\nu)^{-a}\e\\
&\vvvert (\g,F)\mapsto {G}_{\g,F}\vvvert_{\bD_{\C^2}(\g_{*},\rho_{*})\times \cB_{\e}(\Psi_{\b_{1}}(e^{-\nu}R_{s,\rho}))}\lesssim (c_{*}\nu)^{-a}\e\\
&\vvvert  (\g,F)\mapsto \cM(F^{\rm Wh}_{\g,F})\vvvert_{\bD_{\C^2}(\g_{*},\rho_{*})}\lesssim (c_{*}\nu)^{-a}\e.
\end{cases}
\label{eq:10.90}
\ee
Moreover, one has
$$\g=(\a,\b_{2}-\a\b_{1})\in {DC}_{}(c_{*})\implies \iota_{G_{\g,F}}=id.$$

\end{prop}
\begin{proof}We use the same scheme as in the proof of Proposition \ref{prop:KAMstep}  with the following modifications. 

We first  provide a Whitney-type parameter version of Lemma  \ref{lemma:8.8} (we use the notations introduced therein).

Let $\chi:\R\to [0,1]$ be a smooth function  with support in  $[-1,1]$ and equal to 1 on $[-1/2,1/2]$. Define for $\g=(\a,\check \b)\in\C^2$, $n\in\N$, $k\in\Z$
$$m_{c_{*}}(\g,n,k) =\biggl|\exp\biggl(2\pi  i (k\a+(n-1)\check\beta)\biggr)-1\biggr|^2\times c_{*}^{-2}(|k|+|n-1|)^{2e_{*}}$$
so that  for all $\g\in\C^2$, $n\in\N$, $k\in\Z$,
\be\g\in {DC}_{}(c_{*})\implies (1-\chi(m_{c_{*}}(\g,n,k)))=1\label{lincAchin}\ee
and
\be \frac{(1-\chi(m_{c_{*}}(\g,n,k)))}{|e^{2\pi i (k\a+(n-1)\check\beta_{})}-1|}\leq C\times c_{*}^{-1}(|k|+|n-1|)^{e_{*}}\label{controlY}\ee
where $C=\sup_{m\geq 0}(1-\chi(m))/m^{1/2}$.

More generally, if $D_{\g}^j$ denotes the $j$-th derivative w.r.t. $\g$ (i.e. $D_{\g}^j=(\pa_{\g}^{j_{1}}{\bar \pa}_{\g}^{j_{2}})_{(j_{1},j_{2})}$, $j_{1}+j_{2}=j$),  
\be\sup_{\g\in \bD(0,M)^2}\max_{j=0,1,2}\biggl|D_{\g}^j\biggl( \frac{(1-\chi(m_{c_{*}}(\g,n,k)))}{e^{2\pi i (k\a+(n-1)\check\beta_{})}-1}\biggr)\biggr|\lesssim_{M} (c_{*}^{-1}(|k|+|n-1|)^{e_{*}})^{A}\label{controlY}\ee
for some $A>0$.

We extend the definition 
(\ref{def:YwrtF}) of $\hat{\check Y}_{n}(k)$ by setting 
$$
\left\{
\begin{aligned}& \hat{\check Y}^{(\g,F)}_{1}(0)=0\\
&\hat{\check Y}^{(\g,F)}_{n}(k)=(1-\chi(m_{c_{*}}(\g,n,k)))\frac{\hat{\check F}_{n}(k)}{e^{ 2\pi i (k\a+(n-1)\check\beta_{})}-1}\quad\textrm{if}\ (n,k)\ne (1,0).
\end{aligned}
\right.
$$
If
$$ \check Y^{\rm Wh}_{\g,F}(\th,r)=\sum_{n\in\N}\sum_{k\in\Z}\hat{\check Y}^{(\g,F)}_{n}(k)e^{2\pi i k\th}r^n$$
(which is well defined because of (\ref{controlY}))
and
$$ \check F^{\rm Wh}_{\g,F}(\th,r)=\sum_{n\in\N}\sum_{k\in\Z}(1-\chi(m_{c_{*}}(\g,n,k)))\hat{ \check F}_{n}(k)e^{2\pi i k\th}r^n$$
we have
$$\check F^{\rm Wh}_{\g,F}(\th,r)=e^{-2\pi i\check\beta_{}} \check Y^{\rm Wh}_{\g,F}(\th+\a,e^{2\pi i \check \beta_{}}r)- \check Y^{\rm Wh}_{\g,F}(\th,r).$$
and  from (\ref{lincAchin})
\be\g\in{DC}_{}(c_{*})\implies\begin{cases} 
&\check Y_{\g,F}^{\rm Wh}=\check Y\\
& \check F_{\g,F}^{\rm Wh}=\check F.
\end{cases} \label{eq:10.89c}\ee
Moreover, for any $\g\in\bD_{\C^2}(\g_{*},\rho_{*})$
\be
\left\{
\begin{aligned}&\sup_{j=0,1,2}\| D^j_{\a}\check Y^{\rm Wh}_{\g,F}\|_{e^{-\nu}(\T_{s}\times \bD(0,\rho))}\lesssim (c_{*} \nu)^{-a} \| F\|_{\T_{s}\times \bD(0,\rho)}\\
&\sup_{j=0,1,2}\| D^j_{\g}\check F^{\rm Wh}_{\g,F}\|_{e^{-\nu}(\T_{s}\times \bD(0,\rho))}\lesssim (c_{*} \nu)^{-a} \|F\|_{\T_{s}\times \bD(0,\rho)}
\end{aligned}
\right.
\label{tiYmodtiFmodn}
\ee
for some $a>0$.

We then define 
\begin{align*}
&Y^{Wh}_{\g,F}(z,w)=e^{2\pi i\b_{1}z}\check Y^{Wh}(z,e^{-2\pi i \b_{1}z}w)\\
&F^{Wh}_{\g,F}(z,w)=e^{2\pi i\b_{1}z}\check F^{Wh}(z,e^{-2\pi i \b_{1}z}w)
\end{align*}
and the map $\ti{F}^{\rm Wh}_{\a,F}$ by the conjugation relation  
$$ \iota_{Y^{\rm Wh}_{\g,F}}\circ\bm S_{\b_{1}}\circ\Phi_{w}\\ S_{\beta_{2}}\circ \Phi_{\a w}\circ \iota_{F^{\rm Wh}_{\g,F}}\em\circ  \iota_{Y^{\rm Wh}_{\g,F}}^{-1}=   \bm S_{\b_{1}}\circ \Phi_{w}\\ S_{\mbeta_{2}}\circ \Phi_{(\a+\cM(F^{\rm Wh}_{\g,F})) r}\circ \iota_{\ti{F}^{\rm Wh}_{\g,F}}\em. $$
Like in the proof of Proposition \ref{prop:8.8}, one has on $\Psi_{\b_{1}}(e^{-\nu/2}R_{s,\rho})$ ($\nu=\frak{d}(F)$)
\begin{align}
&Y_{\g,F}^{\rm Wh}=\frak{O}_{1}(F),
\label{eq:Ymodn}\\
&F_{\g,F}^{\rm Wh}=\frak{O}_{1}(F),\\
&\ti{F}^{\rm Wh}_{\g,F}=\frak{O}_{2}(F)\qquad(Prop.\ \ref{prop:8.8}).
\label{Fmodn}
\end{align}

We finally define  $G^{}_{\g,F}\in \cO(\Psi_{\b_{1}}(e^{-\nu}R_{s,\rho}))$ by the relation
$$\bm S_{\b_{1}}\circ \Phi_{w}\\ S_{\beta_{2}}\circ \Phi_{\a w}\circ \iota_{F^{}_{}} \circ \iota_{G_{\g,F}}\em=\bm S_{\b_{1}}\circ \Phi_{w}\\S_{\beta_{2}}\circ \Phi_{\a w}\circ \iota_{F^{\rm Wh}_{\g,F}}\em$$
so that 
$$\iota_{Y^{\rm Wh}_{\g,F}}\circ \bm S_{\b_{1}}\circ \Phi_{w}\\ S_{\beta_{2}}\circ \Phi_{\a w}\circ \iota_{F^{}_{}} \circ \iota_{G_{\g,F}}\em\circ \iota_{Y^{\rm Wh}_{\g,F}}^{-1}=\bm S_{\b_{1}}\circ \Phi_{w}\\S_{\beta_{2}}\circ \Phi_{(\a+\cM(F^{\rm Wh}_{\g,F})) w}\circ \iota_{\ti{F}^{\rm Wh}_{\g,F}}\em.$$
One can verify that the maps $(\g,F)\mapsto Y^{\rm Wh}_{\g,F},\ti{F}^{\rm Wh}_{\g,F}$ are $C^1$ and that the following generalization of (\ref{eq:Ymodn}, (\ref{Fmodn})) is satisfied
$$\begin{cases}
&\vvvert (\g,F)\mapsto \ti{F}^{\rm Wh}_{\g,F}\vvvert_{\bD_{\C^2}(\g_{*},\rho_{*})\times \cB_{\e}(\Psi_{\b_{1}}(e^{-\nu}R_{s,\rho}))}\lesssim (c_{*}\nu)^{-a}\e^2\\
&\vvvert (\g,F)\mapsto{Y}^{\rm Wh}_{\g,F}\vvvert_{\bD_{\C^2}(\g_{*},\rho_{*})\times \cB_{\e}(\Psi_{\b_{1}}(e^{-\nu}R_{s,\rho}))}\lesssim (c_{*}\nu)^{-a}\e
\end{cases}
$$
(for the dependence w.r.t. $\g$ it comes from (\ref{tiYmodtiFmodn})).

Note that  (cf. (\ref{eq:10.89c}))
$$\g\in{DC}_{}(c_{*})\implies\begin{cases} 
& Y_{\g,F}^{\rm Wh}= Y\\
&  F_{\g,F}^{\rm Wh}= F
\end{cases}$$
hence 
$$\g\in {DC}_{}(c_{*})\implies \iota_{G_{\g,F}}=id.$$

\end{proof}

\subsection{KAM and Reversibility}

\begin{prop}\label{prop:KAMreversible}Let $\b_{1},\b_{2}\in\R$, $\a\in\C$, satisfy $(\a,\b_{2}-\a\b_{1})\in DC(c_{*})$ and assume that the commuting pair $(S_{\b_{1}}\circ \Phi_{w},S_{\b_{2}}\circ \Phi_{\a w}\circ \iota_{F})$, $F\in \cO(\Psi_{\b_{1}}(R_{s,\rho)})$, is  reversible w.r.t. some antiholomorphic involution\footnote{Recall $\s_{0}(z,w)=(-\bar z, \bar w)$.} $ \s=\s_{0}\circ (id+\eta):\Psi_{\b_{1}}(R_{s,\rho}):\to \Psi_{\b_{1}}(R_{s,\rho})$. Then, if $\|\eta\|_{\Psi_{\b_{1}}(R_{s,\rho})}$ and $\|F\|_{\Psi_{\b_{1}}(R_{s,\rho})}$ are small enough, one has 
\be \bar \a-\a=c_{*}^{-1}(\frak{O}_{1}(F)+\frak{O}_{2}(\eta,F))\label{estimaginaryalpha}\ee
and
there exists a conjugation of the form $\cT_{a,b}:(z,w)\mapsto (z+a,e^{2\pi ib} w)$, $a,b\in\R$, that transforms $\s$ into 
$$\ti \s:(\th,r)\mapsto (-\bar \th,\bar r)+c_{*}^{-1}(\frak{O}_{1}(F)+\frak{O}_{2}(\eta,F))$$
and the  commuting pair $(S_{\b_{1}}\circ \Phi_{w},S_{\b_{2}}\circ \Phi_{\a w}\circ \iota_{F})$, $F\in \cO(\Psi_{\b_{1}}(R_{s,\rho)})$ into a commuting pair  $(S_{\b_{1}}\circ \Phi_{w},S_{\b_{2}}\circ \Phi_{\a w}\circ \iota_{\ti F})$ with  $\ti F(z,w)=e^{2\pi ib}F(z-a,e^{-2\pi ib}w).$ \end{prop}
\begin{proof}
By Remark \ref{rem:8.1} one has 
\be \Psi_{\b_{1}}\circ(S_{\b_{1}}\circ \Phi_{w},S_{\b_{2}}\circ \Phi_{\a w}\circ \iota_{F})\circ \Psi_{\b_{1}}^{-1}=(\cT_{1,0},\cT_{\a,\check \b}\circ (id+\psi_{F}))\label{eq:p12.7}\ee
where 
$$\begin{cases}&\cT_{\a,\check \b}:(z,w)\mapsto (z+\a,e^{2\pi i\check \b}w)\\
&\check \b=\b_{2}-\a\b_{1}
\end{cases}$$
and $\psi_{F}\in \cO(\T_{s}\times \bD(0,\rho))$, $\psi_{F}=\fO_{1}(F)$, is $\cT_{1,0}$-periodic.

Using the fact that $\Psi_{\b_{1}}\circ \s_{0}\circ \Psi_{\b_{1}}^{-1}=\s_{0}$ we see that the anti-holomorphic involution $\ti\s=\Psi_{\b_{1}}\circ\s\circ  \Psi_{\b_{1}}^{-1}$ satisfies 
\begin{align*}
\ti \s&=\Psi_{\b_{1}}\circ\s\circ  \Psi_{\b_{1}}^{-1}\\
&=\Psi_{\b_{1}}\circ \s_{0}\circ \Psi_{\b_{1}}^{-1}\circ \Psi_{\b_{1}}\circ(id+\eta)\circ \Psi_{\b_{1}}^{-1}\\
&=\s_{0}\circ (id+\ti \eta)
\end{align*}
with $\ti \eta=\fO_{1}(\eta)$.

The commuting pair $(\cT_{1,0},\cT_{\a,\check \b}\circ (id+\psi_{F}))$ is reversible w.r.t. $\ti \s$.

\begin{lemma}The map $\ti \eta$ is 1-periodic in $z$: $\ti\eta=\ti \eta\circ T_{1,0}$. In other words, $\ti \eta\in \cO(\T_{s}\times \bD(0,\rho))$.
\end{lemma}
\begin{proof} 
We observe that because $S_{\b_{1}}\circ \Psi_{w}$ is reversible w.r.t. $\s$, the map $\cT_{1,0}$ is reversible w.r.t. $\ti\s$ and we write 
\begin{align*}
&\cT_{1,0}^{-1}=\ti \s\circ \cT_{1,0}\circ \ti \s\qquad(\textrm{reversibility\ of}\  \cT_{1,0})\\
&=\ti\s^{-1}\circ \cT_{1,0}\circ\ti \s \qquad(\ti \s\ \textrm{is\ an \ involution)}\\
&=(id+\ti \eta)^{-1}\circ  \s_{0}\circ \cT_{1,0} \circ \s_{0}\circ  (id+\ti \eta)\\
&=(id+\ti \eta)^{-1}\circ  \cT_{1,0}^{-1}\circ   (id+\ti \eta)
\end{align*} 
which reads $\cT_{1,0}\circ (id+\ti \eta)=(id+\ti\eta)\circ \cT_{1,0}$ and means that $\ti \eta$ is 1-periodic in the $z$-variable. 

\end{proof}

We assume that the antiholomorphic involution $\ti \s$  and the diffeomorphism $\psi_{F}$ (see (\ref{eq:p12.7})) have the form
\begin{align*}&\ti \s=\s_{0}\circ (id+\ti \eta):(\th,r)\mapsto (-\bar \th+\kappa(\bar \th,\bar r),\bar r+\l(\bar \th,\bar r))\\
&\psi_{F}:(z,w)\mapsto (z+u(z,w),w+v(z,w))
\end{align*}
with $\kappa,\lambda,u,v$ holomorphic on $\T_{s}\times \bD(0,\rho)$ and
\begin{align*}
&\kappa,\lambda=\fO_{1}(\eta)\\
&u,v=\fO(F).
\end{align*}

\mn  1) The relation $\ti \s\circ\ti \s=id$ yields 
\be 
\begin{aligned}
&\th=\th-\overline{\kappa(\bar \th,\bar r)}+\kappa(- \th,r)+ \frak{O}_{2}(\eta)\\
&r=r+\overline{\l(\bar \th,\bar r)}+\l(-\th,r)+\frak{O}_{2}(\eta).
\end{aligned}
\label{eq:sigmasigman}
\ee

\mn 2) We now use  the reversibility  relation
$$\ti\s\circ\biggl(T_{\a,\check \beta}\circ \psi_{F}\biggr)\circ \ti\s=\biggl(T_{\a,\check \beta}\circ \psi_{F}\biggr)^{-1}.$$
We write 
$$f:=\cT_{\a,\check \b}\circ \psi_{F}:(\th,r)\mapsto (\th+\a+u(\th,r), e^{2\pi i\check\beta}(r+v(\th,r)).$$
Modulo  $\frak{O}_{2}(\eta,F)$-terms we have 
\begin{align*}
&f\circ \s:(\th,r)\mapsto (-\bar \th+\kappa(\bar \th,\bar r)+\a+u(-\bar \th,\bar r), e^{2\pi i\check\beta_{}}(\bar r+\l(\bar \th,\bar r)+v(-\bar \th,\bar r)))
\end{align*}
hence
\begin{multline*}\ti\s\circ f_{}\circ\ti \s:(\th,r)\mapsto\\ \biggl(\th-\overline{\kappa(\bar \th,\bar r)}-\bar\a_{}-\overline{u(-\bar \th,\bar r)}+\kappa(-\th+\bar \a_{},e^{-2\pi i\check\beta}r),\\
e^{-2\pi i\check\beta_{}}(r+\overline{\l(\bar \th,\bar r)}+\overline{v(-\bar \th,\bar r)})+\l(-\th+\bar\a_{},e^{-2\pi i\check\beta_{}}r) \biggr)+\frak{O}_{2}(\eta,F).
\end{multline*}
Using $\ti\s\circ f\circ \ti\s=f_{}^{-1}$,  (\ref{eq:sigmasigman}) and the equality 
$$f_{}^{-1}:(\th,r)\mapsto (\th-\a_{}-u(\th-\a_{},e^{-2\pi i\check\beta_{}}r), e^{-2\pi i\check\beta_{}}r-v(\th-\a_{},e^{-2\pi i\check\beta_{}}r))+\frak{O}_{2}(F)$$
we thus get $\mod\frak{O}_{2}(\eta,F)$
\begin{align*}
&\th-\kappa(-\th,r)-\bar\a_{}-\overline{u(-\bar \th,\bar r)}+\kappa(-\th+\bar \a,e^{-2\pi i\check\beta_{}}r)=\th-\a_{}-u(\th-\a_{},e^{-2\pi i\check\beta_{}}r)\\
&e^{-2\pi i\check\beta_{}}(r-\l(-\th,r)+\overline{v(-\bar \th,\bar r)})+\l(-\th+\bar\a_{},e^{-2\pi i\check\beta_{}}r)=e^{-2\pi i\check\beta_{}}r-v(\th-\a_{},e^{-2\pi i\check\beta_{}}r)
\end{align*}
hence $\mod\frak{O}_{2}(\eta,F)$
\begin{align*}
&\kappa(-\th+\bar \a, e^{-2\pi i\check\beta}r)-\kappa(-\th,r)=\bar\a_{}-\a_{}+\overline{u(-\bar \th,\bar r)}-u(\th-\a_{},e^{-2\pi i\check\beta_{}}r)\\
&\l(-\th+\bar\a_{},e^{-2\pi i\check\beta_{}}r)-e^{-2\pi i\check\beta_{}}\l(-\th,r)=-e^{-2\pi i\check\beta_{}}\overline{v(-\bar \th,\bar r)})-v(\th-\a_{},e^{-2\pi i\check\beta_{}}r).
\end{align*}

The previous set of equations gives
\be
\left\{\begin{aligned}
&\kappa(-\th+\bar \a,e^{-2\pi i\check\beta_{}}r)-\kappa(-\th,r)=\bar\a-\a+\frak{O}_{1}(F)+\frak{O}_{2}(\eta,F)\\
&\l(-\th+\bar \a,e^{-2\pi i\check\beta_{}}r)-e^{-2\pi i\check\beta_{}}\l(-\th,r)=\frak{O}_{1}(F)+\frak{O}_{2}(\eta,F).
\end{aligned}
\right.
\label{eq:8.48bisn}
\ee

\mn 3) Using Fourier-Taylor decompositions
$$
\left\{
\begin{aligned}
&{\kappa}_{}(\th,r)=\sum_{k\in\Z}\sum_{n\in\N} \hat{{\kappa}}_{}(k,n)e^{2\pi i k\th}r^n\\
&{\lambda}_{}(\th,r)=\sum_{k\in\Z}\sum_{n\in\N} \hat{{\lambda}}_{}(k,n)e^{2\pi i k\th}r^n
\end{aligned}
\right.$$
we see that
the first equation of (\ref{eq:8.48bisn}) and the  fact  $(\bar \a, \check\b)\in DC(c_{*})$ (this  comes from $(\a, \check\b)\in DC(c_{*})$)  shows that 
\be \bar\a-\a=c_{*}^{-1}(\frak{O}_{1}(F)+\frak{O}_{2}(\eta,F))\label{baralphaminusalphan}\ee
as well as the fact that all the non constant terms of ${\kappa}_{}$ are $c_{*}^{-1}(\frak{O}_{1}(F)+\frak{O}_{2}(\eta,F))$:
$${\kappa}_{}(\th,r)=\hat{{\kappa}}_{}(0,0)+c_{*}^{-1}(\frak{O}_{1}(F)+\frak{O}_{2}(\eta,F)).$$

Besides, (\ref{baralphaminusalphan}) and  the second equation of (\ref{eq:8.48bisn})  show that 
$$e^{2\pi i\check\beta}{\l}_{}(\th-\bar \a, e^{-2\pi i\check\beta}r)-{\l}_{}(\th,r)=c_{*}^{-1}(\frak{O}_{1}(F)+\frak{O}_{2}(\eta,F))
$$
hence all the terms in ${\l}_{}$ are  $c_{*}^{-1}(\frak{O}_{1}(F)+\frak{O}_{2}(\eta,F))$ except maybe the coefficient $\hat{{\lambda}}_{}(0,1)$ of $r$; as a consequence
$${\l}_{}(\th,r)=\hat{{\lambda}}_{}(0,1)r+c_{*}^{-1}(\frak{O}_{1}(F)+\frak{O}_{2}(\eta,F)).$$
We thus have 
\be \ti \s(\th,r)=(-\bar \th+\hat \kappa(0,0),(1+\hat \l(0,1))\bar r)+c_{*}^{-1}(\frak{O}_{1}(F)+\frak{O}_{2}(\eta,F)).\label{eq:8.66n}\ee

\mn 4) Equations (\ref{eq:8.66n}) and  (\ref{eq:sigmasigman}) show that 
\begin{align*}
&\Im \hat \kappa(0,0)=c_{*}^{-1}(\frak{O}_{1}(F)+\frak{O}_{2}(\eta,F))\\
&\Re \hat \l(0,1)=c_{*}^{-1}(\frak{O}_{1}(F)+\frak{O}_{2}(\eta,F)).
\end{align*}
and we can thus write
\be \ti \s(\th,r)=(-\bar \th-2a,e^{-4\pi i b}\bar r)+c_{*}^{-1}(\frak{O}_{1}(F)+\frak{O}_{2}(\eta,F)).\label{eq:8.66bis}\ee
where $a$ and $b$ are real.

The conjugation $\cT_{a,b}:(\th,r)\mapsto (\th+a,e^{ 2\pi ib}r)$ turns $\ti \s$ into $$\s':(\th,r)\mapsto (-\bar \th,\bar r)+c_{*}^{-1}(\frak{O}_{1}(F)+\frak{O}_{2}(\eta,F))$$ and the  commuting pair  
$$(\cT_{1,0},\cT_{\a,\check\beta}\circ \psi_{F})$$ into $$(\cT_{1,0},T_{a,b}\circ(\cT_{\a,\check \b}\circ \psi_{F})\circ \cT_{a,b}^{-1})).$$
Because this pair is reversible w.r.t. $\s'$, we deduce, conjugating back by $\Psi_{\b_{1}}^{-1}$, that if
$$\Xi_{a,b}:=\Psi_{\b_{1}}^{-1}\circ \cT_{a,b}\circ \Psi_{\b_{1}}:(z,w)\mapsto (z+a,e^{2\pi i(b+\b_{1}a)}w)$$
 the commuting pair 
 $$\Xi_{a,b}\circ (S_{\b_{1}}\circ \Phi_{w},S_{\b_{2}}\circ \Phi_{\a w}\circ \iota_{F})\circ \Xi_{a,b}^{-1}$$
 is reversible w.r.t. the
 anti-holomorphic involution 
 $$\Psi_{\b_{1}}^{-1}\circ \s'\circ \Psi_{\b_{1}}=\s_{0}\circ (id+c_{*}^{-1}(\frak{O}_{1}(F)+\frak{O}_{2}(\eta,F))).$$
 By Lemma \ref{lemma:8.2}, one has 
 $$\Xi_{a,b}\circ (S_{\b_{1}}\circ \Phi_{w},S_{\b_{2}}\circ \Phi_{\a w}\circ \iota_{F})\circ \Xi_{a,b}^{-1}=(S_{\b_{1}}\circ \Phi_{w},S_{\b_{2}}\circ \Phi_{\a w}\circ \iota_{\ti F})$$
 with 
 $$\ti F(z,w)=e^{2\pi i(b+\b_{1}a)}F(z-a,e^{-2\pi i(b+\b_{1}a)}w).$$
This completes the proof of Proposition \ref{prop:KAMreversible}.

\end{proof}

\subsection{KAM-Siegel Theorem: general form}\label{subsec:10.5} Let $D_{}\subset \C^2$ be of the form 
$$D=\bD_{\C^2}(t_{*},\d^2)=\bD(t_{*,1},\d^2 )\times \bD(t_{*,2},\d^2 )$$ for some $t_{*}=(t_{*,1},t_{*,2})\in\C^2$ and $\rho>0$.

We  assume we are given  $C^1$-families
\be
\left\{
\begin{aligned}
&D\ni t\mapsto \g(t):= (\a_{t},\b_{1,t},\b_{2,t})\in\C^3\\
&D\ni t\mapsto F_{t}\in \cO(\Psi_{\b_{1,t}}(R_{s,\rho}))
\end{aligned}
\right.
\label{eq:families}
\ee
and we set
\be  D\ni t\mapsto \check\g(t):= (\a_{t},\b_{2,t}-\a_{t}\b_{1,t})\in \C^2.\label{eq:checkgamma}\ee
We make the following assumption:  let $(\a_{*},\check\b_{*})\in \R\times \C$
and assume that,  for some 
\be p>20(a+1)\label{pvsa}\ee
where $a$ is the constant appearing in Proposition \ref{prop:KAMWhitneyStepn}, one has:
\begin{enumerate}
\item The $C^1$-norm of the map $\check \g:D\to \check \g(D)$ is $\lesssim \d^{-1}$, $\check \gamma$ is invertible and the inverse map $\check \g^{-1}:\check \g(D)\to D$ has a $C^1$-norm $\lesssim 1$.
\item  There exists a point $(\a_{*},\check\b_{*})\in \R\times \C\subset \C^2$ which is contained in $\check \g(D)$.  
\item The $C^1$-norm of $D\ni t\mapsto F_{t}\in \cO(\Psi_{\b_{1,t}}(R_{s,\rho}))$ is $\lesssim \d^{(p-2)/2-2}$ (cf. Proposition \ref{prop:10.4}).
\end{enumerate}

Note that there exists $\rho_{*}$ such that $\check\g(D)\supset \bD(\a_{*},2\d^2 \rho_{*})\times \bD(\check \b_{*},2\d^2\rho_{*})$.

\begin{theo}\label{theo:SiegelGeneraln}If $\d$ is small enough,  there exists a $C^1$ map $\check\g_{\infty}^{-1}:\bD_{}(\a_{*},\rho_{*}\d^2) \times \bD(\check\b_{*},\rho_{*}\d^2)\to \C^2$ and a positive Lebesgue measure set $\cA^{(\infty)}\subset \bD_{\R}(\a_{*},\rho_{*}\d^2) \times \bD(\check\b_{*},\rho_{*}\d^2)$ such that for any $(\a,\check\b)\in \cA^{(\infty)}$ the following holds: if $t=\check\g^{-1}_{\infty}(\a,\check \b)$, there exists an exact conformal symplectic diffeomorphism $\iota_{Y_{t}^{[1,\infty]}}$,  
\be Y_{t}^{[1,\infty]}\in \cO(\Psi_{\b_{1}}(e^{-1/3}R_{s,\rho})),\qquad  \|Y_{t}\|_{\Psi_{\b_{1}}(e^{-1/3}R_{s,\rho}))}\leq \d^{(p-2)/2-a} \ee
 such that  $$\bm S_{\b_{1,t}}\circ\Phi_{w}\\S_{\beta_{2,t}}\circ \Phi_{\a_{t} w}\circ \iota_{F_{t}}\em=\iota_{Y_{t}^{[1,\infty]}}^{-1}\circ \bm S_{\b_{1,t}}\circ \Phi_{w}\\ S_{\beta_{2,t}}\circ \Phi_{\a w}\em\circ \iota_{Y_{t}^{[1,\infty]}}.$$
\end{theo}
\begin{proof}
Let 
\be \begin{cases}&c_{*}^{(n)}=2^{-(n+1)}\d^7,\\
&\nu_{n}=2^{-(n+1)}\nu.
\end{cases}
\label{con8.68}
\ee

\medskip
We use Proposition \ref{prop:KAMWhitneyStepn} to construct inductively sequences of $C^1$-maps
\be
\left\{
\begin{aligned}
&D\ni t\mapsto Y_{t}^{(n)}\in \cO(\Psi_{\b_{1}}(e^{-\sum_{k=0}^{n-1}\nu_{k}}R_{s,\rho}))\\
&D\ni t\mapsto F_{t}^{(n)}\in \cO(\Psi_{\b_{1}}(e^{-\sum_{k=0}^{n-1}\nu_{k}}R_{s,\rho}))\\
&D\ni t\mapsto G_{t}^{(n)}\in\cO(\Psi_{\b_{1}}(e^{-\sum_{k=0}^{n-1}\nu_{k}}R_{s,\rho}))\\
&D\ni t\mapsto \g_{n}(t)=(\a_{n}(t),\b_{1,t},\b_{2,t})\in \C^3
\end{aligned}
\right.
\label{con8.69}
\ee
where $\iota_{Y^{(n)}(t)}$ commutes with $S_{\b_{1,t}}\circ \Phi_{w}$,
such that 
\begin{enumerate}
\item 
\be 
\left\{
\begin{aligned}
&F^{(0)}_{t}=F_{t}\\
&\g_{0}(t)=\g_{t}\\
&Y^{(0)}_{t}=Y^{Wh}_{{\g_{t}},F_{t}}\\
&G^{(0)}_{t}=G_{\g_{t},F_{t}}; 
\end{aligned}
\right.
\label{con8.70}
\ee
\item 
\be
\left\{
\begin{aligned}
&F_{t}^{(n+1)}=\ti{F}^{\rm Wh}_{\g_{n}(t), F_{t}^{(n)}}\\
&\g_{n+1}(t)=(\a_{n}(t)+\cM(F^{\rm Wh}_{\g_{n}(t),F^{(n)}_{t}}),\b_{1,t},\b_{2,t})\\
&Y^{(n)}_{t}={Y}^{\rm Wh}_{\g_{n}(t),F^{(n)}_{t}}\\
&G^{(n)}_{t}={G}_{\g_{n}(t),F^{(n)}_{t}}.
\end{aligned}
\right.
\label{con8.71}
\ee
In particular,
$$\g_{n+1}(t)-\g_{n}(t)=(\a_{n+1}(t)-\a_{n}(t),0,0)=(\cM(F^{\rm Wh}_{\g_{n}(t),F^{(n)}_{t}}),0,0).$$
\item 
\begin{multline}
\iota_{Y^{(n)}_{t}}\circ \bm S_{\beta_{1,t}} \circ \Phi_{w}\\ S_{\b_{2,t}}\circ \Phi_{\a_{n}(t) w}\circ \iota_{F^{(n)}_{t}} \circ \iota_{G^{(n)}_{t}}\em \circ \iota_{Y^{(n)}_{t}}^{-1}=
\\  \bm S_{\beta_{1,t}} \circ \Phi_{w}\\ S_{\b_{2,t}}\circ \Phi_{\a_{n+1}(t) w}\circ \iota_{F^{(n+1)}_{t}} \em
\label{con8.72}
\end{multline}
\item 
\be
 \check\g_{n}(t)\in {DC}(c_{*}^{(n)})\implies \iota_{G_{t}^{(n)}}=id.
\label{con8.73}
\ee
\item If $\e_{n}=\vvvert t\mapsto {F}^{(n)}_{t}\vvvert_{D}$ one has for some $a>0$
\be \vvvert t\mapsto {F}^{(n+1)}_{t}\vvvert_{D}=\e_{n+1}\lesssim (c^{(n)}_{*}\nu_{n})^{-a}\e_{n}^2\label{eq:estFn+1vsFn}\ee
and 
\be \begin{cases}
&\vvvert t\mapsto{Y}^{(n)}_{t}\vvvert_{D }\lesssim (c^{(n)}_{*}\nu_{n})^{-a}\e_{n}\\
&\vvvert t\mapsto {G}^{(n)}_{t}\vvvert_{D}\lesssim (c^{(n)}_{*}\nu_{n})^{-a}\e_{n}\\
&\vvvert  t\mapsto \g_{n}(t)\vvvert_{D}\lesssim \d^{-1}\\
&\vvvert  t\mapsto (\a_{n+1}(t) -\a_{n}(t))\vvvert_{D}\lesssim (c^{(n)}_{*}\nu_{n})^{-a}\e_{n}
\end{cases}
\label{eq:8.76n}
\ee
\end{enumerate}
All these inequalities can be proved by induction  using  the estimates (\ref{eq:10.90}) and the fact (proved also inductively from (\ref{eq:estFn+1vsFn})) that there exists $C>0$ such that for $\d$ small enough
\begin{align}
& \e_{n+1}\leq C 2^{2(n+1)a}\d^{7a}\e_{n}^2\qquad(\textrm{see\ Prop.}\ \ref{lemma:quadraticconv})\label{esteplisonnn}\\
&\e_{n}\leq C\d^{7a}e^{-(3/2)^n}\label{eq:11.121}\\
&\e_{n}\leq 2^{-(2a+7)(n+1)}\d^{(p-2)/2-3}\label{esteplisonnnbis}
\end{align}
(condition (\ref{pvsa}) is also used to get these estimates).

We then observe that we can write 
$$S_{\beta_{2,t}}\circ \Phi_{\a_{n}(t)w}\circ \iota_{F^{(n)}_{t}} =\iota_{Y^{(n)}_{t}}^{-1}\circ \biggl(S_{\beta_{2,t}}\circ \Phi_{\a_{n+1}(t) w}\circ \iota_{{F}^{(n+1)}_{t}}\biggr)\circ \iota_{Y^{(n)}_{t}}\circ \iota_{G^{(n)}_{t}}^{-1}.$$

Hence
$$S_{\beta_{2,t}}\circ \Phi_{\a_{t} w}\circ \iota_{F_{t}}=\iota_{Y_{t}^{[1,n]}}^{-1}\circ \biggl(S_{\beta_{2,t}}\circ \Phi_{\a_{n+1}(t)w}\circ \iota_{{F}^{(n+1)}_{t}}\biggr)\circ \iota_{Y_{t}^{[1,n]}}\circ \iota^{-1}_{G_{t}^{[1,n]}}$$
where
\begin{align*}
&\iota_{Y_{t}^{[1,n]}}=\iota_{Y_{t}^{(n)}}\circ\cdots\circ \iota_{Y_{t}^{(1)}}\\
&\iota^{-1}_{G_{t}^{[1,n]}}=(\iota_{Y_{t}^{[1,n]}}^{-1}\circ \iota^{-1}_{G_{t}^{(n)}}\circ \iota_{Y_{t}^{[1,n]}})\circ \cdots\circ\iota^{-1}_{G_{t}^{(0)}}.
\end{align*}

The last equation of (\ref{eq:8.76n}) and (\ref{esteplisonnnbis}) show that, if $\d$ is small enough,
\be \vvvert t\mapsto \check\g_{n+1}(t)-\check\g_{n}(t)\vvvert_{D}\leq \d^{(p-2)/2-3-7a}2^{-(a+7)(n+1)}\label{gammann+1}\ee
hence (see \ref{pvsa}))
$$\vvvert t\mapsto \check\g_{n}(t)-\check\g_{}(t)\vvvert_{D}\leq \d^{7}2^{-(a+7)(n+1)}.$$
As a consequence,
$$D\ni t\mapsto \check\g_{n}(t)=(\a_{n}(t),\b_{2,t}-\a_{n}(t)\b_{1,t})\in\C^2$$
is a $C^1$-diffeomorphism onto $\bD(\a_{*},(3/2)\rho_{*}\d^2)\times \bD(\check\b_{*},(3/2)\rho_{*}\d^2)$. Moreover,
 the sequence $(\check\g_{n}(\cdot))_{n}$ converges in $C^1$ norm to some diffeomorphism $\check\g_{\infty}(\cdot)$  from $D$ onto $\bD(\a_{*},(3/2)\rho_{*}\d^2)\times \bD(\check\b_{*},(3/2)\rho_{*}\d^2)$. Let
\be \ph_{n}=\check\g_{n}\circ \check\g_{\infty}^{-1};\label{def:phinparameterexclusion}\ee
if $\d$ is small enough one has  for $\d$ small enough
\be \vvvert \ph_{n}-id\vvvert\leq \d^6\label{est:phinn}\ee
and $\ph_{n}:\bD(\a_{*},(3/2)\rho_{*}\d^2)\times \bD(\check\b_{*},(3/2)\rho_{*}\d^2)\to \C^2 $ is onto $\bD(\a_{*},\rho_{*}\d^2)\times \bD(\check\b_{*},\rho_{*}\d^2)$.
Let $B=\bD(\a_{*},(3/2)\rho_{*}\d^2)\times \bD(\check\b_{*},(3/2)\rho_{*}\d^2)$; we define
$$\cA^{(n)}=\{(\a,\check\b)\in B\cap (\R\times \C)\mid \ph_{n}(\a,\check\b)\in {DC}_{}(c_{*}^{(n)})\}$$
and 
$$\cA^{(\infty)}=\bigcap_{n\in\N}\cA^{(n)}.$$
By Lemmata  \ref{lemma:diophn}-(\ref{item:3n} and estimate (\ref{est:phinn}) one has
$${\rm Leb}_{\R\times \C}\biggl(B\setminus \cA^{(\infty)}\biggr)\lesssim \sum_{n\in\N}c_{*}^{(n)}=\sum_{n\in\N}2^{-(n+1)}c_{*}\leq c_{*}=\d^7$$ 
hence $ \cA^{(\infty)}\subset D\cap (\R\times \C)$ has positive Lebesgue measure if $c_{*}=\d^7$ is small enough.

To conclude the proof, choose $(\a,\check\b)\in \cA^{(\infty)}$ and set $t=\check\g_{\infty}^{-1}(\a,\check\b)$. For each $n\in\N$ one has
$$\check\g_{n}(t)=\ph_{n}(\a,\check \b)=:(\a_{n},\check\b_{n})\in DC(c_{*}^{(n)})$$
hence
$$\iota_{G_{t}^{(n)}}=id\quad \textrm{and}\quad \iota_{G_{t}^{[1,n]}}=id$$
so that
$$S_{\beta_{2,t}}\circ \Phi_{\a_{}(t) w}\circ \iota_{F_{t}}=\iota_{Y_{t}^{[1,n]}}^{-1}\circ \biggl(S_{\beta_{2,t}}\circ \Phi_{\a_{n+1}(t)w}\circ \iota_{{F}^{(n+1)}_{t}}\biggr)\circ \iota_{Y_{t}^{[1,n]}}.$$
Because of the first inequality  of  (\ref{eq:8.76n}) and  (\ref{eq:11.121}),
the sequence of diffeomorphisms $\iota_{Y_{t}^{[1,n]}}$ converges (with its inverse) to some $\iota_{Y_{t}^{[1,\infty]}}$; so letting $n\to \infty$ one gets
$$S_{\beta_{2,t}}\circ \Phi_{\a_{t} w}\circ \iota_{F_{t}}=\iota_{Y_{t}^{[1,\infty]}}^{-1}\circ \biggl(S_{\beta_{2,t}}\circ \Phi_{\a_{\infty}(t)w}\biggr)\circ \iota_{Y_{t}^{[1,\infty]}}.$$
Since $\iota_{Y_{t}^{[1,\infty]}}$ commutes with $S_{\b_{1,t}}\circ \Phi_{w}$ and $\a_{\infty}(t)=\a$, one has also
$$\bm S_{\b_{1,t}}\circ\Phi_{w}\\S_{\beta_{2,t}}\circ \Phi_{\a_{t} w}\circ \iota_{F_{t}}\em=\iota_{Y_{t}^{[1,\infty]}}^{-1}\circ \bm S_{\b_{1,t}}\circ \Phi_{w}\\ S_{\beta_{2,t}}\circ \Phi_{\a w}\em\circ \iota_{Y_{t}^{[1,\infty]}}.$$
This is the searched for conjugation relation.

\end{proof}

\subsection{KAM-Siegel Theorem: dissipative case}

We now suppose that $D$ is a disk  $D_{\d}=\bD(t_{*},\d^2)$,  $t_{*}\in \C$, and that we are given $t$-parameter families (\ref{eq:families}). We also define $\check \gamma$ by  (\ref{eq:checkgamma}).

Let us fix $\a_{*}$ and $\check\b_{*}\in \C$ such that 
$$\begin{cases}
&\a_{*}\in\R,\\
&\Im(\check\b_{*})\ne 0.
\end{cases}$$
As in the preceding subsection, we assume that $p$ satisfies (\ref{pvsa}) and that
\begin{enumerate}
\item The $C^1$-norm of the map $\check \g:D\to \check \g(D)$ is  $\lesssim \d^{-1}$.
\item The $C^1$-norm of the map $\a_{}:D\to \a_{}(D)$ has a $C^1$-norm $\lesssim \d^{-1}$ and the inverse map $ \a_{}^{-1}:\a_{}(D)\to D$ has a $C^1$-norm $\lesssim 1$.
\item  The point $\a_{*}\in\R$ is contained in $\a(D)$.  
\item The $C^1$-norm of $D\ni t\mapsto F_{t}\in \cO(\Psi(R_{s,\rho}))$ is $\lesssim \d^{(p-2)/2-2}$ (cf. Proposition \ref{prop:10.4}).
\end{enumerate}

Note that there exists $\rho_{*}$ such that $\a_{}(D)\supset \bD(\a_{*},2 \rho_{*}\d^2).$
\begin{theo}[Dissipative case]\label{theo:SiegelDissipative}If $\d$ is small enough,  there exists a $C^1$ embedding $\a_{\infty}^{-1}:\bD(\a_{*},\rho_{*}\d^2) \to \C$ and a positive Lebesgue measure set $\cA_{\rm dissip.}^{(\infty)}\subset \bD_{\R}(\a_{*},\rho_{*}\d^2) $ such that for any $\a\in \cA_{\rm dissip.}^{(\infty)}\subset \R$ the following holds: if $t=\a^{-1}_{\infty}(\a)$, there exists an exact conformal symplectic diffeomorphism $\iota_{Y_{t}^{[1,\infty]}}$
\be Y_{t}^{[1,\infty]}\in \cO(\Psi_{\b_{1}}(e^{-1/3}R_{s,\rho})),\qquad  \|Y\|_{\Psi_{\b_{1}}(e^{-1/3}R_{s,\rho}))}\leq \d^{(p-2)/2-a} \ee
 such that  $$\bm S_{\b_{1,t}}\circ\Phi_{w}\\S_{\beta_{2,t}}\circ \Phi_{\a_{t} w}\circ \iota_{F_{t}}\em=\iota_{Y_{t}^{[1,\infty]}}^{-1}\circ \bm S_{\b_{1,t}}\circ \Phi_{w}\\ S_{\beta_{2,t}}\circ \Phi_{\a w}\em\circ \iota_{Y_{t}^{[1,\infty]}}.$$
 One can choose $\cA_{\rm dissip.}^{(\infty)}$ so that the pair $(\a_{\infty}(t),\b_{2,t}-\a_{\infty}(t)\b_{1,t})=(\a,\b_{2,\a_\infty^{-1}(\a)}-\a\b_{1,\a_{\infty}^{-1}(\a)})$ is non-resonant (or Diophantine).
\end{theo}
\begin{proof}
We follow the proof of Theorem \ref{theo:SiegelGeneraln} with the following modifications.

Estimate (\ref{gammann+1}) shows that
\begin{align}&\vvvert t\mapsto \a_{n}(t)-\a^{}(t)\vvvert_{D}\leq \d^{(p-2)/2-3-a}2^{-(a+7)(n+1)}\notag\\
&\vvvert t\mapsto \check\b_{n}(t)-\check\b_{*}\vvvert_{D}\leq \d^{(p-2)/2-3-a}2^{-(a+7)(n+1)}\label{eq:10.115}
\end{align}
hence each 
 $D\ni t\mapsto \a_{n}(t)$ is a $C^1$ diffeomorphism onto $\bD(\a_{*},(3/2)\rho_{*}\d^2)$ and 
 the sequence $(\a_{n}(\cdot))_{n}$ converges in $C^1$ norm to some diffeomorphism $\a_{\infty}(\cdot)$  from $D$ onto $\bD(\a_{*},(3/2)\rho_{*}\d^2)$. Similar to (\ref{def:phinparameterexclusion})
 we define
 \be \ph_{n}=\a_{n}\circ \a_{\infty}^{-1};\label{def:phinparameterexclusionbis}\ee
 and
 $$\cA_{\rm dissip.}^{(n)}=\{\a\in ]\a_{*}-(3/2)\rho_{*}\d^2,\a_{*}+(3/2)\rho_{*}\d^2[ \mid \ph_{n}(\a)\in {DC}_{\R}(c_{*}^{(n)})\}$$
and 
$$\cA_{\rm dissip.}^{(\infty)}=\bigcap_{n\in\N}\cA_{\rm dissip.}^{(n)}.$$
One still has ${\rm Leb}_{\R}(\cA^{(\infty)}_{\rm dissip.})>0$ (see Lemmata  \ref{lemma:diophn}-(\ref{item:3n}).

Besides, if $\a\in \cA_{\rm dissip.}^{(\infty)}$ and $t=\a_{\infty}^{-1}(\a)$, one has for $n\in\N$
$$\a_{n}(t)=\ph_{n}(\a)=:\a_{n}\in DC_{\R}(c_{*}^{(n)})$$
and because (cf.( \ref{DCRDC}), (\ref{eq:10.115}))
$$ \begin{cases}&|\Im\check\b_{n}(t)| >c_{*}\\\
&\a_{n}(t)\in DC_{\R}(c^{(n)}_{*},e_{*})\end{cases}
\implies (\a_{n}(t),\check\b_{n}(t))\in DC(c^{(n)}_{*},e_{*}).
$$
we deduce
$$\iota_{G_{t}^{(n)}}=id\quad \textrm{and}\quad \iota_{G_{t}^{[1,n]}}=id.$$

We can then conclude the proof of the Theorem like the one of Theorem \ref{theo:SiegelGeneraln}.
\end{proof}

\subsection{KAM-Siegel Theorem: reversible case}\label{sec:reversiblecase}

We now state the version of Theorem \ref{theo:SiegelGeneraln} in the reversible case. 

Let $$D_{\R^2}=\bD_{\R^2}(t_{*},\d^2)\subset\R^2$$ and suppose we are given $t$-parameter families (\ref{eq:families}),  (\ref{eq:checkgamma}).

We assume, like in  the beginning of Subsection \ref{subsec:10.5}, that   for some $p$ satisfying (\ref{pvsa}), one has the following.
\begin{enumerate}
\item  Denoting
$$D_{\R^2}=D_{\R^2}(t_{*},\d^2)=\bD_{\R}(t_{*,1},\d^2 )\times \bD_{\R}(t_{*,2},\d^2 ),$$
the $C^1$-norm of the map 
$$\Re\check \g:D_{\R^2}\to \Re(\check \g(D_{\R^2}))$$ is $\lesssim \d^{-3/2}$ and the inverse map $(\Re\check \g)^{-1}: \Re(\check \g(D_{\R^2}))\to D_{\R^2}$ has a $C^1$-norm $\lesssim \d^{-1/2}$. 
\item The $C^1$-norm of $D_{\R^2}\ni t\mapsto F_{t}\in \cO(\Psi_{\b_{1}}(R_{s,\rho}))$ is $\lesssim \d^{(p-2)/2-2}$ (cf. Proposition \ref{prop:10.4}).
\item For all $t\in D_{\R^2}$,  the following reversibility condition holds:  the commuting pair $(S_{\b_{1,t}}\circ \Phi_{w},S_{\b_{2,t}}\circ \Phi_{\a_{t}w}\circ \iota_{F_{t}})$ is reversible w.r.t. an anti-holomorphic involution $\s_{t}=\s_{0}\circ (id+\eta_{t})$ where 
$\s_{0}:(\th,r)\mapsto (-\bar \th,\bar r)$ ($\eta_{t}$ being holomorphic on  $\Psi_{\b_{1}}(R_{s,\rho}))$ and the $C^1$-norm of $D_{\R^2}\ni t\mapsto \eta_{t}\in \cO(\Psi_{\b_{1}}(R_{s,\rho}))$ is $\lesssim \d^{(p-2)/2}$.
\end{enumerate}
Note that in particular $\b_{1,t},\b_{2,t}$ are real.

\begin{theo}[Reversible case]\label{theo:SiegelReversible} If $\d$ is small enough, there exists a set $\cB_{\rm rev.}^{(\infty)}\subset \bD_{\R^2}(t_{*},\d^2)$ with positive Lebesgue measure such that  for any $t\in \cB_{\rm rev.}^{(\infty)}$, there exist $\a_{\infty}(t)\in \R$ and  an exact conformal symplectic diffeomorphism $\iota_{Y_{t}^{[1,\infty]}}$ 
\be Y_{t}^{[1,\infty]}\in \cO(\Psi_{\b_{1}}(e^{-1/3}R_{s,\rho})),\qquad  \|Y\|_{\Psi_{\b_{1}}(e^{-1/3}R_{s,\rho}))}\leq \d^{(p-2)/2-a} \ee
such that  $$\bm S_{\b_{1,t}}\circ\Phi_{w}\\S_{\beta_{2,t}}\circ \Phi_{\a_{t} w}\circ \iota_{F_{t}}\em=\iota_{Y_{t}^{[1,\infty]}}^{-1}\circ \bm S_{\b_{1,t}}\circ \Phi_{w}\\ S_{\beta_{2,t}}\circ \Phi_{\a_{\infty,t} w}\em\circ \iota_{Y_{t}^{[1,\infty]}}.$$
 One can choose $\cB_{\rm rev.}^{(\infty)}$ so that the pair $(\a_{\infty}(t),\b_{2,t}-\a_{\infty}(t)\b_{1,t})$ is non-resonant (or Diophantine).
\end{theo}
\begin{proof}
We follow the proof and notations  of Theorem \ref{theo:SiegelGeneraln}. 

We define
$$\cB_{\rm rev.}^{(n)}=\{t\in D_{\R^2}\mid \check\g_{n}(t)\in DC(c_{*}^{(n)}\}$$
and 
$$\cB_{\rm rev.}^{(\infty)}=\bigcap_{n\in\N}\cB_{\rm rev.}^{(n)}.$$
Like in the proof of Theorem \ref{theo:SiegelGeneraln}, we can see using (\ref{gammann+1}) and Lemmata  \ref{lemma:diophn}-(\ref{item:3n} that  $\cB_{\rm rev.}^{(\infty)}\subset \R^2$ has positive Lebesgue measure if $\d$ is small enough   and that

$$t\in \cB_{\rm rev.}^{(\infty)}\implies \forall n\in\N, \quad \iota_{G_{t}^{(n)}}=id\quad \textrm{and}\quad \iota_{G_{t}^{[1,n]}}=id.$$
Hence, for all $t\in \cB_{\rm rev.}^{(\infty)}$ one has 
$$\iota_{Y_{t}^{[1,n-1]}}\circ \bm S_{\b_{1,t}}\circ \Phi_{w} \\S_{\beta_{2,t}}\circ \Phi_{\a_{}(t) w}\circ \iota_{F_{t}}\em\circ \iota_{Y_{t}^{[1,n-1]}}^{-1}= \bm S_{\b_{1,t}}\circ \Phi_{w} \\S_{\beta_{2,t}}\circ \Phi_{\a_{n}(t)w}\circ \iota_{{F}^{(n)}_{t}}\em.$$
We now check that if $t\in \cB_{\rm rev.}^{(\infty)}$ then $\a_{n}(t)$ is very close to a real number.
To do this we use inductively  Proposition \ref{prop:KAMreversible}: for each $t\in  \cB_{\rm rev.}^{(\infty)}$, one can construct a sequence of anti-holomorphic complex involution 
$$\s_{t}^{(n)}=\cT_{a_{n},b_{n}}^{-1}\circ\biggl(\s_{0}\circ (id+\eta_{t}^{(n)})\biggr)\circ \cT_{a_{n},b_{n}}$$
($\cT_{a_{n},b_{n}}:(\th,r)\mapsto (\th+a_{n},e^{ib_{n}}r)$, $a_{n},b_{n}\in\R$) with respect to which $$\bm S_{\b_{1,t}}\circ \Phi_{w} \\S_{\beta_{2,t}}\circ \Phi_{\a_{n}(t)w}\circ \iota_{{F}^{(n)}_{t}}\em.$$ is reversible and
\be \eta^{(n)}_{t}=(c_{*}^{(n)})^{-1}\frak{O}_{1}(F_{t}^{(n)}).\label{estetan}\ee
The fact that (\ref{estetan}) holds is a consequence of   the inductive estimate
$$\eta^{(n+1)}_{t}=(c_{*}^{(n)})^{-1}\biggl(\frak{O}_{1}(F_{t}^{(n)})+\frak{O}_{2}(\eta_{t}^{(n)},F_{t}^{(n)})\biggr)$$
and of the proof of Proposition \ref{lemma:quadraticconv}.

Estimate (\ref{estetan})  allows to apply (\ref{estimaginaryalpha}) of Proposition \ref{prop:KAMreversible}:
$$\Im(\a_{n}(t))=(c_{*}^{(n)})^{-1}\frak{O}(F_{t}^{(n)}).$$
We thus have 
$$\Im(\a_{\infty}(t))=\lim_{n\to\infty} \Im(\a_{n}(t))=0.$$

\end{proof}

\section[Existence of Exotic Rotation Domains]{Existence of Exotic Rotation Domains in the reversible case (Theorems \ref{main:A}, \ref{main:Aprime})} \label{sec:proofmainA}
We shall mainly give the proof of Theorem \ref{main:Aprime} since the proof of Theorem \ref{main:A} follows the same line and is indeed simpler. The only modification is to replace in what follows the function $(t,\mbeta)\mapsto \tau_{\d}(t,\mbeta)\approx 1+it$ by $(t,\mbeta)\mapsto 1+t$.
\subsection{Reduction to $h^{\rm mod}_{\a,\b}$}\label{sec:13.1}
Let 
$$h^{\textrm{Hénon}}_{\b,c}:\C^2\ni (x,y)\mapsto (e^{i\pi \b}(x^2+c)-e^{2\pi i \beta} y,x)\in \C^2,\qquad \b,c\in\C$$
where
\be
\left\{
\begin{aligned}
&\d>0,\  \textrm{small}\\
&\b=\frac{1}{3}+\d \mbeta\\
&\tau=\tau_{\d}(t,\mbeta)\\
&\a=\frac{1}{6}+\d\times (\tau-1/2)\mbeta\\
&c=-(\cos(2\pi\a))^2+2\cos(2\pi\a)\cos(\pi\b).
\end{aligned}
\right.
\label{alpha,betaMainbis}
\ee
As we saw in Section \ref{sec:Ushikisresonance} this map  is conjugated (by a linear map) in a neighborhood of one of its fixed points to the modified Hénon map 
$$ h^{\textrm{mod}}_{\a,\b}:\C^2\ni \bm z\\ w\em\mapsto  \bm \l_{1}z\\\ \l_{2}w\em+\frac{q(\l_{1}z+\l_{2}w)}{\l_{1}-\l_2}\bm 1\\ -1\em\in \C^2$$
where 
$$\l_{1}=e^{2\pi i (-\a+\b/2)},\qquad \l_{2}=e^{2\pi i (\a+\b/2)}.$$

\subsection{BNF and vector field model}\label{sec:13.2}
Let (cf. (\ref{pvsa}))
\be p>20(a+1).\label{pvsabis}\ee
By Theorem \ref{theo:approxbyvf} we know   there exists a holomorphic conformal symplectic conjugation $Z_{\d,\tau'}$ (recall $\tau'=(\tau,\mbeta)$) such that   (cf. (\ref{eq:n1.3}))
\be Z_{\d,\tau'}\circ h^{\rm mod}_{\a,\b}\circ Z_{\d,\tau'}^{-1}=\diag(1,e^{2\pi i/3})\circ \phi^1_{\d \mbeta X^{}_{\d,\tau'}}\circ \iota_{F^{\rm bnf}_{\d,\tau'}}\label{eq:n1.3bis}\ee
where 
\be F^{\rm bnf}_{\d,\tau'}=O(\d^{2m-(1/3)})=O(\d^p) \qquad(p=2m-1/3)\label{estFbnf}\ee
 and 
$$
\left\{
\begin{aligned}
&X^{}_{\d,\tau'}(z,w)=X_{\tau}(z,w)+O(\d)\\
&\textrm{with}\ X_{\tau}(z,w)=X_{0,\tau}(z,w)=2\pi i \bm (1-\tau)z+z^2/2-w^3/3\\ \tau w- zw\em.
\end{aligned}
\right.
$$
Let us denote
\be
\begin{aligned}h^{\rm bnf}_{\d,\tau'}&=Z_{\d,\tau'}\circ h^{\rm mod}_{\a,\b}\circ Z_{\d,\tau'}^{-1}\\
&=\diag(1,e^{2\pi i/3})\circ \phi^1_{\d \mbeta X^{}_{\d,\tau'}}\circ \iota_{F^{\rm bnf}_{\d,\tau'}}.
\end{aligned}
\label{defhbnf}
\ee
Because $\diag(1,e^{2\pi i/3})^3=I$, the third iterate of $h^{\rm bnf}_{\d,\tau'}$ is of the form
\be
\begin{aligned}h_{\d,\tau'}&:=(h^{\rm bnf}_{\d,\tau'})^3\\
&=\phi^1_{3\d \mbeta X^{}_{\d,\tau'}}\circ \iota_{F^{}_{\d,\tau'}}
\end{aligned}
\label{relhbnf}
\ee
with 
\be F^{}_{\d,\tau'}=O(\d^{(2m-1/3)})=O(\d^p).\label{estfdeltatau}\ee

\subsection{Use of the invariant annulus theorem}\label{sec:13.3}

Let 
$$\mbeta_{*}\in \R_{*}$$ be fixed.

By  Theorems \ref{theo:perobrit}, \ref{theogreal} and \ref{theo:invannthm} (Invariant annulus theorem), we know 
that for any $\tau'=(\tau,\mbeta)\in \bD_{\C^2}((1,\mbeta_{*}),\nu_{1})$
 the vector field $X_{\d,\tau'}$ (which has divergence $2\pi i$), is tangent to an annulus $\cA_{\d,\tau'}\simeq \T_{s}$ and that its restriction on this annulus is conjugate to  the vector field of $\T_{s}$ defined by $g_{\d}(\tau')\pa_{\th}$; furthermore
\be (0,\nu_{1}]\ni \d\mapsto g_{\d}(\cdot)\in \cO(\bD_{\C^2}((1,\mbeta_{*}),\nu_{1}))\label{contgdeltadelta}\ee
is continuous.

Furthermore, we know that $g_{0}:\tau\mapsto g_{0}(\tau)$ is holomorphic on $\bD(0,\nu_{1})$ and satisfies
\be
\left\{
\begin{aligned}
&\forall t\in (-\nu_{1},\nu_{1}),\quad g_{0}(1+it)\in\R^*\\
&\textrm{and}\\
&t\mapsto g_{0}(1+it)\ \textrm{is\ not\ constant}.
\end{aligned}\label{eq:9.92}
\right.
\ee
As a consequence there exists $\tau_{*}=\tau_{0}(t_{*})=1+i t_{*}$, $t_{*}\in \R$, such that 
 \be 
 \left\{
 \begin{aligned} 
 &g_{0}(\tau_{*})\in\R^*\\
 &  \frac{\pa g_{0}}{\pa\tau}(\tau_{*})\in \R^*.
  \end{aligned}
  \right.
  \label{eq:10.57bis}
  \ee
  where $| \frac{\pa g_{0}}{\pa\tau}(\tau_{*})|$ is bounded below by a positive constant independent of $\d$.
  By Lemma \ref{lemma:defintaudelta} and the continuity of the map (\ref{contgdeltadelta}) we deduce that the $C^1$-norm of 
  \be t\mapsto \biggl(g_{\d}(\tau_{\d}(t,\mbeta_{*}),\mbeta_{*})-g_{0}(1+it)\biggr)\label{gdeltag0}\ee
  is small; 
  henceforth there exists  $\d_{1},c,\nu_{2}>0$, such that for any $\d\in (0,\d_{1})$, and any $t\in (t_{*}-\nu_{2},t_{*}+\nu_{2})$
 \be \biggl| \frac{\pa g_{\d}(\tau_{\d}(t,\mbeta_{*}),\mbeta_{*})}{\pa t}\biggr|\geq c>0.\label{cond13189}\ee
 
 \subsection{Use of reversibility}
    By (\ref{Ttaureal}) of Proposition \ref{prop:7.7} we also  know that 
  \be \forall (t,\mbeta) \in \bD_{\R^2}((t_{*},\mbeta_{*}),\nu_{2}),\quad  \Im g_{\d}(\tau_{\d}(t,\mbeta),\mbeta)=O(\d^{p-1}).\label{Impartg}\ee
  Note that by (\ref{cond13189}) we can  choose $t_{*}$ such that in addition
  \be T_{\tau_{\d}(t_{*},\mbeta_{*}),\mbeta_{*}}:=\biggl\{\frac{1}{3\mbeta_{*} g_{\d}(\tau_{\d}(t_{*},\mbeta_{*}) ,\mbeta_{*})}\biggr\}\notin \bD(0,1/9)\cup \bD(1,1/9).\label{eq:T*} \ee
  We define
  $$\tau_{*,\d}=\tau_{\d}(t_{*},\mbeta_{*}).$$
  As a consequence, for any $(\tau,\mbeta)\in \bD_{}(\tau_{*,\d},\d^2) \times \bD_{\R}(\mbeta_{*},\d^2)$ one has 
 \be T_{\tau,\mbeta_{}}:=\biggl\{\frac{1}{3\mbeta_{} g_{\d}(\tau ,\mbeta_{})}\biggr\}\notin \bD(0,1/10)\cup \bD(1,1/10) . \label{eq:T.190} \ee
  We observe that by Proposition \ref{prop:7.7} one has
 \be \forall (t,\mbeta)\in \bD_{\R^2}((t_{*},\mbeta_{*}),\d^2),\quad \Im T_{\tau_{\d}(t,\mbeta),\mbeta}=O(\d^{p-1}).\label{ImTtmbeta}\ee

\subsection{Renormalization, commuting pairs  and normalization boxes}\label{subsec:13.5}
Recall (\ref{relhbnf})
$$h_{\d,\tau'}=\phi^1_{3\d \mbeta X^{}_{\d,\tau'}}\circ \iota_{F^{}_{\d,\tau'}}\qquad(\tau'=(\tau,\mbeta)).$$

If $$p=2m-(1/3)$$ is large enough ($>3$) 
we can   apply the results of Section \ref{sec:renormandcommpairs}, on first return maps, renormalization and commuting pairs, where $X$ and $\eta$ in Assumptions \ref{assump:8.1}-\ref{assump:8.2} are respectively (see  Remark \ref{rem:theo10.1})
\be \begin{cases}
&X:=X_{\d}^*:=X^*_{\d,\tau_{*,\d},\mbeta_{*}}=3\mbeta_{*} e^{i\ph_{\d,\tau_{*,\d},\mbeta_{*}}}X_{\delta,\tau_{*,\d},\mbeta}\\
&id+\eta:= id+\eta^*_{\d,\tau,\mbeta}=\phi^{-1}_{3\d\mbeta_{*} e^{i\ph_{\d,\tau_{*,\d},\mbeta_{*}}}X_{\d,\tau_{*,\d},\mbeta_{*}}}\circ \phi^1_{3\d\mbeta X_{\d,\tau,\mbeta}}\circ\iota_{F_{\d,\tau,\mbeta}} \\
&h_{\d,\tau,\mbeta}=\phi^1_{X^*_{\d}}\circ (id+\eta^*_{\d,\tau,\mbeta})
\end{cases}
\label{Xtaustar}
\ee
where $\ph_{\d,\tau,\mbeta}\in (-\d,\d)$ is defined by  
\be e^{i\ph_{\d,\tau,\mbeta}}g_{\d}(\tau,\mbeta)\in\R.\label{defphi}\ee
Note that $X^*_{\d}$ has a periodic orbit of period 
$$T^*_{\d}= e^{-i\ph_{\d,\tau_{*,\d},\mbeta_{*}}}T_{\tau_{*,\d}}\in \R$$
that satisfies 
$$\biggl\{\frac{T^*_{\d}}{\d}\biggr\}\in ((1/10),(9/10)).$$
We set 
$$q_{\d}=\biggl[\frac{T^*_{\d}}{\d}\biggr].$$
By (\ref{ImTtmbeta})
\be\ph_{\d,\tau_{*,\d},\mbeta_{*}}=O(\d^{p-(4/3)})=O(\d^3)\label{estetasharp}\ee
hence
\be \forall (\tau,\mbeta)\in \bD(\tau_{*,\d},\d^2)\times \bD_{\R}(\mbeta_{*},\d^2),\quad \eta^*_{\d,\tau,\mbeta}=O(\d^3).\label{esteta*}\ee
 In particular, by Proposition \ref{lemma:7.1},  
we can define for any $(\tau,\mbeta)\in \bD(\tau_{*,\d},\d^2)\times \bD_{\R}(\mbeta_{*},\d^2)$ the renormalization $\cR^*(h_{\d,\tau,\mbeta})$ associated to a first return domain ${\bar \cW}^{X^*_{\d},\eta^*_{\d,\tau,\mbeta}}_{\d, s_{}}$ of $(h_{\d,\tau,\mbeta}, {\bar \cW}^{X^*_{\d},\eta^*_{\d,\tau,\mbeta}}_{\d, s'_{}} )$ (see (\ref{defbarcW}) and Definition \ref{def:firstreturnmap}) and  we can define  the commuting pair 
$$(h_{\d,\tau,\mbeta},h_{\d,\tau,\mbeta}^{q_{\d}})_{\cW^{X^*_{\d},\eta^*_{\d,\tau,\mbeta}}_{\d,s,\nu}}$$ (see (\ref{def:Wsigmadelta})).
As a consequence of (\ref{e10.87}) we can also define (by restriction) the commuting pair
\be  (h_{\d,\tau,\mbeta},h^{q_{\d}}_{\d,\tau,\mbeta})_{ \cW^{X_{\d,\tau_{},\mbeta}, \eta_{\d,\tau,\mbeta}}_{\d,s_{*}'/2,\nu/2}}\label{commpairdefined}\ee
associated to the box $\cW^{X_{\d,\tau_{},\mbeta}, \eta_{\d,\tau,\mbeta}}_{\d,s_{*}'/2,\nu/2}$ which is defined more naturally  in terms of the vector field $X_{\d,\tau,\mbeta}$.
cf. (\ref{e10.88}).

With our notation  $\tau'=(\tau,\mbeta)\in \C^2$, we set  for short 
\begin{align}
&\cW^{*,\tau'}_{\d,s,\nu}=\cW^{X^*_{\d},\eta^*_{\d,\tau,\mbeta}}_{\d,s,\nu}\label{W*tau'}\\
&\cW^{\tau'}_{\d,s,\nu}=\cW^{X_{\d,\tau_{},\mbeta}, \eta_{\d,\tau,\mbeta}}_{\d,s,\nu}.\label{Wtau'}
\end{align}

\medskip 
As mentioned in Remark \ref{rem:theo10.1},   the results of Section \ref{sec:renormandcommpairs} also apply to the case where where $X$ and $\eta$ in Assumptions \ref{assump:8.1}-\ref{assump:8.2} are (see (\ref{choiceXetater}))
\be \begin{cases}
&X:=X^\sharp_{\d,\tau,\mbeta}=3\mbeta e^{i\ph_{\d,\tau,\mbeta}}X_{\delta,\tau,\mbeta}\\
&id+\eta:= id+\eta^\sharp_{\d,\tau,\mbeta}=\phi^{-1}_{3\d\mbeta e^{i\ph_{\d,\tau,\mbeta}}X_{\d,\tau_{},\mbeta}}\circ \phi^1_{3\d\mbeta X_{\d,\tau,\mbeta}}\circ\iota_{F_{\d,\tau,\mbeta}} \\
&h_{\d,\tau,\mbeta}=\phi^1_{X^\sharp_{\d,\tau,\mbeta}}\circ (id+\eta^\sharp_{\d,\tau,\mbeta})
\end{cases}
\label{Xtausharp}
\ee
where $\ph_{\d,\tau,\mbeta}\in\R$ is still defined by (\ref{defphi}).
By (\ref{ImTtmbeta}) we have 
\be \forall (t,\mbeta)\in \bD_{\R^2}((t_{*},\mbeta_{*}),\d^2),\quad\ph_{\d,\tau_{\d}(t,\mbeta),\mbeta}=O(\d^{p-(4/3)})=O(\d^{p^\sharp})=O(\d^3).\label{estetasharpante}\ee
hence (compare with (\ref{esteta*})) with $p^\sharp=p-2$
\be
\left\{
\begin{aligned}& \eta^\sharp_{\d,\tau_{\d}(t,\mbeta),\mbeta}=O(\d^{p^\sharp})\\
&p^\sharp=p-2.
\end{aligned}
\right.
\label{estetasharp}
\ee
The orbit $(\phi^t_{X^\sharp_{\d,\tau'}}(\zeta_{\d,\tau'}))_{t\in\R}$ is $T^\sharp_{\d,\tau'}$-periodic with $T^\sharp_{\d,\tau'}\in\R$
\be T^\sharp_{\d,\tau'}=\frac{1}{3\d\mbeta e^{i\ph_{\d,\tau'}}g_{\d}(\tau')}.\label{Tsharpante}\ee
When 
$$(\tau,\mbeta)=(\tau_{\d}(t,\mbeta),\mbeta)$$
for some $(t,\mbeta)\in \bD_{\R^2}((t_{*},\mbeta_{*}),\d^2)$,
we can define the renormalization $\cR^\sharp(h_{\d,\tau,\mbeta})$ associated to a first return domain ${\bar \cW}^{X^\sharp_{\tau,\b},\eta^\sharp_{\d,\tau,\mbeta}}_{\d, s_{}}$ of $(h_{\d,\tau,\mbeta}, {\bar \cW}^{X^\sharp_{\tau,\b},\eta^\sharp_{\d,\tau,\mbeta}}_{\d, s'_{}} )$ (see (\ref{defbarcW}) and Definition \ref{def:firstreturnmap}) and  the commuting pair 
$$(h_{\d,\tau,\mbeta},h_{\d,\tau,\mbeta}^{q_{\d}})_{\cW^{X^\sharp_{\tau,\b},\eta^\sharp_{\d,\tau,\mbeta}}_{\d,s,\nu}}$$ (see (\ref{def:Wsigmadelta})). 
We denote for short 
\be \cW^{\sharp,\tau'}_{\d,s,\nu}=\cW^{X^\sharp_{\tau,\b},\eta^\sharp_{\d,\tau,\mbeta}}_{\d,s,\nu}.\label{notationWsharp}\ee

Note that the boxes  $\cW^{*,\tau'}_{\d,s,\nu}$ and $\cW^{\sharp,\tau'}_{\d,s,\nu}$ compare with $\cW^{\tau'}_{\d,s,\nu}$ as follows:  given $s,\nu$ one has when $(\tau,\mbeta)\in \bD(\tau_{*,\d},\d^2)\times \bD_{\R}(\mbeta_{*},\d^2)$
\be
\cW^{*,\tau'}_{\d,s-O(\d),\nu-O(\d)}\subset \cW^{\tau'}_{\d,s,\nu}\subset \cW^{*,\tau'}_{\d,s+O(\d),\nu+O(\d)}\ee
and when 
$$(\tau,\mbeta)=(\tau_{\d}(t,\mbeta),\mbeta)$$
for some $(t,\mbeta)\in \bD_{\R^2}((t_{*},\mbeta_{*}),\d^2)$, one has 
\be \cW^{\sharp,\tau'}_{\d,s-O(\d^{p-2}),\nu-O(\d^{p-2})}\subset \cW^{\tau'}_{\d,s,\nu}\subset \cW^{\sharp,\tau'}_{\d,s+O(\d^{p-2}),\nu+O(\d^{p-2})}\label{comp:sharpn}
\ee
(see (\ref{estetasharp}) for the last set of inclusions).

\subsection{Linearization of the third iterate}
\bigskip Theorems \ref{main:Aprime} and \ref{main:A}  are  implied  respectively by the following statements which are proved in Subsection \ref{sec:13.8}; actually, we shall only give the proof of Theorem \ref{main:Aprime} as the proof of Theorem \ref{main:A} is similar (and simpler).

\begin{theo}[A priori hyperbolic case]\label{theo:9.1}There exist $\check \nu,\check s,\check \rho$ (which are $\asymp 1$) and, 
for any $\d$ small enough,  a measurable set $E^{\rm hyp}_{\d}\subset [-1,1]^2$ of positive Lebesgue measure 
 for which the following holds. For any $(t,\mbeta)\in E^{\rm hyp}_{\d}$ we set 
 \be \tau'=(\tau,\mbeta)=(\tau_{\d}(t,\mbeta),\mbeta);\label{eq:tauprimein}\ee
 then,
  there exists a holomorphic diffeomorphism
$$N_{\d,\tau'}^{-1}:(-\check \nu,1+\check \nu)_{\check s}\times \bD(0,\check \rho)\to \C^2$$
which  satisfies with $p^\sharp=p-2$
\be
\begin{cases}
&(i)\quad \cW^{\sharp,\tau'}_{\d,\d^{p^\sharp/2+2},\nu/2}\subset N_{\d,\tau'}^{-1}((-\check \nu,1+\check \nu)_{\check s}\times \bD(0,\check \rho))\subset \cW^{\sharp,\tau'}_{\d,s',\nu}\\
&(ii)\quad N_{\d,\tau'}^{-1}(0,0)\in\bD(\zeta_{\d,\tau'}, \d^{p^\sharp-a})\\
&(iii)\quad (N_{\d,\tau'}^{-1})_{*}\pa_{z}=\d X^\sharp_{\d,\tau'}+O(\d^{p^\sharp/2-a}).
\end{cases}
\label{cond9.62new}
\ee
and such that 
 $N_{\d,\tau'}$
 conjugates on  $N_{\d,\tau'}^{-1}((-\check \nu,1+\check \nu)_{\check s}\times \bD(0,\check \rho))$ the commuting pair $(h_{\d,\tau'},h^{q_{\d}}_{\d,\tau'})$ to a normalized pair $(\cT_{1,0},\cT_{\check\a_{\tau'},\check \b_{\tau'}})$
with $\check \a_{\tau'}\in(-1,0)$, $\check \b_{\tau'}\in\R$ and $(\check \a_{\tau'},\check \b_{\tau'})$ is non resonant. 
\end{theo}

\begin{theo}[A priori elliptic case]\label{theo:9.1:thmA}There exist $\check \nu,\check s,\check \rho$ (which are $\asymp 1$) and, 
for any $\d$ small enough,  a measurable set $E^{\rm ell}_{\d}\subset [-1,1]^2$ of positive Lebesgue measure 
 for which the following holds. For any $(\tau,\mbeta)\in E^{\rm ell}_{\d}$  the conclusions of Theorem \ref{theo:9.1} hold.
 \end{theo}

\subsection{Theorem \ref{theo:9.1}(resp. \ref{theo:9.1:thmA})  implies Theorem \ref{main:Aprime}  (resp. \ref{main:A})}\label{sec:13.7}

Proving that $h^{\textrm{Hénon}}_{\b,c}$ admits an Exotic rotation domain is equivalent to proving the same property for the modified Hénon map $h^{\rm mod}_{\a,\b}$. The classification result \cite{BS2} of Bedford and Smilie (see subsection  \ref{sec:classification}) tells us that we have to find a nonempty set $ \cE^{\rm mod}_{\a,\b}$ such that
\begin{enumerate}
\item   $ \cE^{\rm mod}_{\a,\b}$ is  a bounded, $h^{\rm mod}_{\a,\b}$-invariant connected open set. 
\item $(h^{\rm mod}_{\a,\b},\cE^{\rm mod}_{\a,\b})$ is a rank-2 rotation domain in the following sense: for a dense subset of $\xi\in  \cE^{\rm mod}_{\a,\b}$, the closure of the orbit $\{(h^{\rm mod}_{\a,\b})^n(\xi)\mid n\in\N\}$ is diffeomorphic to a (real) 2-torus ($(\R/\Z)\times(\R/\Z)$).
\item $ \cE^{\rm mod}_{\a,\b}$  is not included in any bounded, $h^{\rm mod}_{\a,\b}$-invariant, connected open  set $\Omega$ on which $h^{\rm mod}_{\a,\b}$ has a fixed point.
\end{enumerate}
\subsubsection{}\label{sec:13.7.1} To find this $h^{\rm mod}_{\a,\b}$-invariant connected open set $ \cE^{\rm mod}_{\a,\b}$ we shall exhibit an $h_{\d,\tau'}$-invariant connected open set $\cE_{\d,\tau'}$. 

\medskip We recall the decomposition (\ref{Xtausharp}) $h_{\d,\tau,\mbeta}=\phi^1_{X^\sharp_{\d,\tau,\mbeta}}\circ (id+\eta^\sharp_{\d,\tau,\mbeta})$; the orbit $(\phi^t_{X^\sharp_{\d,\tau,\mbeta}}(\zeta_{\d,\tau'}))_{t\in\R}$ is $1/(3\mbeta e^{i\ph_{\d,\tau,\mbeta}}g_{\d}(\tau') )\in\R$ periodic and $\eta^\sharp_{\d,\tau_{\d}(t,\mbeta),\mbeta}=O(\d^{p^\sharp})$ with $p^\sharp=p-2$.

Theorem \ref{theo:9.1} shows that  when 
$$\tau'=(\tau_{\d}(t,\mbeta),\mbeta),\qquad (t,\mbeta) \in E_{\d},$$
the Assumptions \ref{assump:8.3} of   Theorem \ref{tho:rank2rotdomain} (of  Section \ref{sec:criterion} giving a criterion for the existence of rotation domains) are satisfied for the the decomposition (\ref{Xtausharp}). In particular, for 
$\tau'=(\tau_{\d}(t,\mbeta),\mbeta)$, $(t,\mbeta) \in E_{\d},$
the diffeomorphism
\begin{align*}h_{\d,\tau'}&=\phi^1_{X^\sharp_{\d,\tau,\mbeta}}\circ (id+\eta^\sharp_{\d,\tau,\mbeta})\\
&=\phi^1_{3\d \mbeta X^{}_{\d,\tau'}}\circ \iota_{F^{}_{\d,\tau'}}\\
&=(h_{\d,\tau,\mbeta}^{\rm bnf})^{\circ 3}
\end{align*}
which is  the third iterate of the  diffeomorphism  $h_{\d,\tau,\mbeta}^{\rm bnf}$ (cf. (\ref{defhbnf}), (\ref{relhbnf}))
has a  rank-2 rotation domain
$$\cE_{\d,\tau'}:=\check  \cC_{\check s, \check \nu}.$$
From Remark \ref{Ccontains} (see (\ref{eq:Cccontains})) and Corollary \ref{prop:iteratesfrd}  the following  inclusions hold
$$O^\sharp_{\d,\tau'}\subset\cC^{\eta_{\d,\tau'}^\sharp}_{\d,\d^{(2/3)p^\sharp},\nu}\subset\cC^{\eta_{\d,\tau'}^\sharp}_{\d,\d^{p^\sharp/2-2},\nu}\subset \cE_{\d,\tau'}$$
 where $O^\sharp_{\d,\tau'}$ is the 
 $T^\sharp_{\d,\tau'}$-periodic orbit 
\be O^\sharp_{\d,\tau'}:=\{\phi^t_{X^\sharp_{\d,\tau'}}(\zeta_{\d,\tau'})\mid t\in\R\}\label{def:orbitO}\ee
and 
\be T^\sharp_{\d,\tau'}=\frac{1}{3\d\mbeta e^{i\ph_{\d,\tau'}}g_{\d}(\tau')}.\label{Tsharp}\ee
Because $\diag(1,j)_{*}X^\sharp_{\d,\tau'}=X^\sharp_{\d,\tau'}$ (see (\ref{eq:theo:approxbyvf}) of Theorem \ref{theo:approxbyvf} and (\ref{Xtausharp})) one has by Corollary \ref{cor:7.7} ($T^\sharp_{\d,\tau'}$ is real)
$$\diag(1,j)(O^\sharp_{\d,\tau'})=O^\sharp_{\d,\tau'}$$
and for $s\in\R$
$$\phi_{\d X^\sharp_{\d,\tau'}}^{-s}\circ \diag(1,j)(O^\sharp_{\d,\tau'})=O^{\sharp}_{\d,\tau'}.$$
Estimates (\ref{estetasharpante}) and  (\ref{estetasharp}) show that 
$$\dist\biggl(\phi^{-1}_{\d X_{\d,\tau'}}\circ \diag(1,j)(O^\sharp_{\d,\tau'}),O^\sharp_{\d,\tau'}\biggr)=O(\d^{p^\sharp)}).$$
and  by estimate (\ref{estFbnf})
$$\dist\biggl(\iota_{F^{\rm bnf}_{\d,\tau'}}^{-1}\circ \phi^{-1}_{\d X_{\d,\tau'}}\circ \diag(1,j)(O^\sharp_{\d,\tau'}),O^\sharp_{\d,\tau'}\biggr)=O(\d^{p^\sharp})$$
i.e.
$$\dist\biggl((h^{\rm bnf}_{\d,\tau'})^{-1}(O^\sharp_{\d,\tau'}),O^\sharp_{\d,\tau'}\biggr)=O(\d^{p^\sharp}).$$
This implies that for $l=0,1,2$
$$\dist\biggl((h^{\rm bnf}_{\d,\tau'})^{-l}(O^\sharp_{\d,\tau'}),O^\sharp_{\d,\tau'}\biggr)=O(\d^{p^\sharp})$$
hence by Theorem \ref{lemma:9.6}
\be (h_{\d,\tau'}^{\rm bnf})^{-l}(O^\sharp_{\d,\tau'})\subset \cV_{\d^{p^\sharp-1}}( O^\sharp_{\d,\tau'})\subset \check  \cC_{\check s, \check \nu}.\label{eqinclusion}\ee

Besides, since $\check  \cC_{\check s, \check \nu}$ is $h_{\d,\tau'}$-invariant and $h_{\d,\tau'}=(h^{\rm bnf}_{\d,\tau'})^3$ (third iterate), the set
$$\cE_{\d,\tau'}^{\rm bnf}=\bigcup_{l=0}^2 h^{l}_{\d,\tau'}(\cC_{\check s, \check \nu})$$
is $h^{\rm bnf}_{\d,\tau'}$-invariant 
and the above inclusion (\ref{eqinclusion}) yields
$$\forall l\in\{0,1,2\},\quad O^\sharp_{\d,\tau'}\subset (h^{\rm bnf}_{\d,\tau'})^l(\check\cC_{\check s, \check \nu})\subset \cE^{\rm bnf}_{\d,\tau'}.$$
Since $O^\sharp_{\d,\tau'}$ is  connected, the union $\cE^{\rm bnf}_{\d,\tau'}$ of the  $(h^{\rm bnf}_{\d,\tau'})^l(\cE_{\d,\tau'})$, $l=0,1,2$ is also connected.

\medskip The conjugation relation (\ref{eq:n1.3bis}) between $h^{\rm mod}_{\a,\b}$ and $h^{\rm bnf}_{\tau,\mbeta}$ shows that the set 
$$\cE^{\rm mod}_{\a,\b}= Z_{\mbeta,\tau,\d}^{-1}(\cE^{\rm bnf}_{\d,\tau,\mbeta})$$
is connected, $h_{\a,\b}^{\rm mod}$-invariant  and that  for a dense set of $\xi\in \cE^{\rm mod}_{\a,\b}$, the closure of any orbit $((h_{\a,\b}^{\rm mod})^{3n}(\xi))_{n\in\Z}$,  is a real 2-torus.  By Bedford-Smillie classification result\footnote{Or more general arguments.} \cite{BS2} this implies that $\cE^{\rm mod}_{\a,\b}$ is a rank-2 rotation domain.

\subsubsection{} \label{sec:13.7.2} Checking that $\cE^{\rm mod}_{\a,\b}$ is an {\it Exotic} rotation domain is obvious in the {\it a priori} hyperbolic case (the fixed points are hyperbolic) and in this case the proof of Theorem  \ref{main:Aprime} (assuming Theorem \ref{theo:9.1}) is thus complete.

\medskip In the {\it a priori} elliptic case the argument is that the frequencies associated to the rotation domain  $(h^{\rm mod}_{\a,\b},\cE^{\rm mod}_{\a,\b})$  do not match with the frequencies at the elliptic fixed points.

We proceed by contradiction. Assume the rotation domain  $(h^{\rm mod}_{\a,\b},\cE^{\rm mod}_{\a,\b})$ is not exotic; there thus exists a maximal connected rotation  domain $(h^{\rm mod}_{\a,\b},\Omega)$, $\cE^{\rm mod}_{\a,\b}\subset \Omega$, $\Omega$ containing one of the fixed points of $h^{\rm mod}_{\a,\b}$, a Reinhardt domain $D\subset\C^2$ and a biholomorphism $\psi:\Omega\to D$ such that on $D$
$$\psi\circ h^{\rm mod}_{\a,\b}\circ \psi^{-1}:D\ni (\zeta_{1},\zeta_{2})\mapsto (e^{2\pi i f_{1}}\zeta_{1},e^{2\pi i f_{2}}\zeta_{2})\in D$$
where the frequency vector $(f_{1},f_{2})$ is non-resonant.  In particular 
$$\psi\circ h_{\d,\tau'}\circ\psi^{-1}= \psi\circ (h^{\rm mod}_{\a,\b})^3\circ \psi^{-1}:D\ni (\zeta_{1},\zeta_{2})\mapsto (e^{6\pi i f_{1}}\zeta_{1},e^{6\pi i f_{2}}\zeta_{2})\in D$$
Since $D$ contains a fixed point one must have 
$$\{f_{1},f_{2}\}\subset \{\pm  \mbeta\d (\tau-1),\mbeta\d(1\pm  (1-\tau))\}\mod \Z$$
hence
\be \{3f_{1},3f_{2}\}\subset \{\pm 3 \mbeta\d (\tau-1),3 \mbeta\d(1\pm  (1-\tau))\}\mod 3\Z. \label{eq:13.229}\ee

\bigskip However, by Theorem \ref{invariantannulus} (that can be applied because the Assumption \ref{assump:8.3} is the conclusion of Theorem \ref{theo:9.1}) there exists an $h_{\d,\tau'}$-invariant annulus $\cA_{\d,\tau'}$ included in   $\cE_{\d,\tau'}$  on which the diffeomorphism $h_{\d,\tau'}$ has a rotation number ${\rm rot}(h_{\d,\tau'}\mid\cA_{\d,\tau'})$ that satisfies 
\begin{align} {\rm rot}(h_{\d,\tau'}\mid\cA_{\d,\tau'})&=\frac{\d}{T^\sharp_{\d,\tau'}}+O(\d^2)\label{eq:n13.236}\\
&=3\d\mbeta e^{i\ph_{\d,\tau'}}g_{\d}(\tau')+O(\d^2)\qquad(\textrm{cf.}\ (\ref{Tsharp}))\notag\\
&=3\d\mbeta g_{\d}(\tau')+O(\d^2)  \qquad(\textrm{cf.}\  (\ref{estetasharpante}))\notag.
\end{align}
where $T^\sharp_{\d,\tau'}$ is the period of the  orbit $(\phi^s_{X^\sharp_{\d,\tau'}}(\zeta_{\d,\tau'}))_{s\in\R}$ associated to the vector field $X^\sharp_{\d,\tau'}$.

Nevertheless, for $\d$ small enough and $t$ close to 0 (hence $\tau$ close to 1)
$$3\d\mbeta g_{\d}(\tau')+O(\d^2)\notin \{\pm 3 \mbeta\d (\tau-1),3 \mbeta\d(1\pm  (1-\tau))\}\mod 3\Z$$
because 
$$g_{0}(1)=-0.834\pm 10^{-3}\notin \{0,1\}\mod 3\Z$$
(cf. for example Theorem \ref{theogreal}).

This shows that $\cE_{\a,\b}^{\rm mod}$ is exotic and completes the proof of Theorem  \ref{main:A} (assuming Theorem \ref{theo:9.1:thmA}).
\hfill $\Box$

\subsection{Proof of Theorem \ref{theo:9.1}} \label{sec:13.8}

The facts (\ref{eq:T.190}), (\ref{ImTtmbeta}) and (\ref{gdeltag0}) show that  provided $\d$ is smaller than some $\d_{1}\leq \d_{0}$, 
the map 
\begin{multline} \R^2\supset\bD_{\R}(t_{*},\d^2)\times \bD_{\R}(\mbeta_{*},\d^2)\ni (t,\mbeta)\mapsto \\  \biggl(-\Re\biggl\{\frac{1}{3 \d\mbeta g_{\d}(\tau_{\d}(t,\mbeta),\mbeta)}\biggr\},\Re \biggl( \frac{1}{g_{\d}(\tau_{\d}(t,\mbeta),\mbeta)}\biggr) \biggr)\in  \R^2\label{eq:13.134}\end{multline}
is a diffeomorphism onto its image that has $C^1$-norm $\leq \d^{-3/2}$ and the norm of its inverse is $\leq \d^{-1/2}$.

We let $\d_{*}$ be a positive number $\leq \d_{1}$ for which  Theorems  \ref{cor:2.10:commpair},
\ref{ref:theocommpairreversible}, \ref{theo:SiegelDissipative}, \ref{theo:SiegelReversible}, Propositions
\ref{prop:parameterdependence},
\ref{prop:10.1}
\ref{prop:10.2},
\ref{prop:10.4}
and Corollary
\ref{cor:10.3}
hold for all $\d\in (0,\d_{*}]$.

We can now fix $\d\in (0,\d_{*}]$ and set 
\begin{align*}
&D=\bD_{\C^2}(t'_{*},\d^2)= \bD_{}(t_{*},\d^2)\times \bD_{}(\mbeta_{*},\d^2)
\end{align*}

\subsubsection{Applying the partial normalization Theorem}\label{sec:13.8.1}
As we saw in Subsection \ref{subsec:13.5} we can,  by applying Proposition \ref{lemma:7.1}  to the system (\ref{Xtaustar}), define for any $\tau'=(\tau,\mbeta)\in \bD(\tau_{*,\d},\d^2)\times \bD_{}(\mbeta_{*},\d^2)$ the commuting pair
(\ref{commpairdefined}) (see the notation (\ref{Wtau'}))
$$(h_{\d,\tau,\mbeta},h^{q_{\d}}_{\d,\tau,\mbeta})_{\cW^{\tau'}_{\d,s,\nu}}.$$
We can then 
 apply Theorem \ref{cor:2.10:commpair} on partial normalization of commuting pairs  to the holomorphic family (\ref{Xtaustar}): 
for all $\tau'\in (\tau,\mbeta)\in \bD(\tau_{*,\d},\d^2)\times \bD_{}(\mbeta_{*},\d^2)$, the pair $(h_{\d,\tau,\mbeta},h^q_{\d,\tau,\mbeta})$ can be partially normalized on a  domain
 $$\check \cW^{\tau'}_{\d,s_{1},\nu_{1}}=(N^{ec}_{\d,\tau'})^{-1}\biggl((-\nu_{1},1+\nu_{1})_{s_{1}}\times \bD(0,s_{1})\biggr)$$
where  $N_{\d,\tau'}^{\rm ec}$ is 
 an exact  conformal-symplectic holomorphic injective map 
  $$N_{\d,\tau'}^{\rm ec}:h_{\d,\tau'}^{q_{\d}}(\cW^{\tau'}_{\d,s_{0},\nu_{0}})\cup \cW^{\tau'}_{\d,s_{0},\nu_{0}}\cup h_{\d,\tau'}(\cW^{\tau'}_{\d,s_{0},\nu_{0}}) \to  \C^2.$$
We thus have the partial normalization relation on $\check \cW^{\tau'}_{\d,s_{1},\nu_{1}}$
 \be 
 \begin{aligned} 
 N^{\rm ec}_{\d,\tau'}\circ\bm h_{\d,\tau'}\\ h^{q_{\d}}_{\d,\tau'}\em\circ (N^{\rm ec}_{\d,\tau'})^{-1}&=\bm \cT_{1,3\d\mbeta}\\ S_{3q_{\d}\d\mbeta}\circ \Phi_{\a_{\d,\tau'}w}\circ\iota_{F^{\rm vf}_{\d,\tau'}}\circ \iota_{F^{\rm cor}_{\d,\tau'}}\em\\
 &=\bm \Phi_{3\d\mbeta}\circ \Phi_{w}\\ S_{3q_{\d}\d\mbeta}\circ \Phi_{ \a_{\d,\tau'}w}\circ\iota_{F^{\rm vf}_{\d,\tau'}}\circ \iota_{F^{\rm cor}_{\d,\tau'}}\em.
\end{aligned}
\label{conc:thm7.6new}
 \ee
  Moreover,
  $$\cW^\tau_{\d,s_{0}/2,\nu_{0}/2}\subset \check \cW^\tau_{\d,s_{1},\nu_{1}}\subset \cW^\tau_{\d,s_{0},\nu_{0}},$$
 $F^{\rm vf}_{\d,\tau'}, F^{\rm cor}_{\d,\tau'}\in \cO((-\nu_{1},1+\nu_{1})_{s_{1}}\times \bD(0,s_{1}))$ are such that 
$$\begin{cases}&F_{\d,\tau'}^{\rm vf}(z,w)=O(w^2),\\
&F_{\d,\tau'}^{\rm vf}(z,w)=O_{A}(1),\\
&F^{\rm cor}_{\d,\tau'}=O_{A}(\d^{p-2})\end{cases}$$ and 
$$\a_{\d,\tau'}=-\biggl\{ \frac{1}{3\d \mbeta g_{\d}(\tau')}\biggr\}\quad (\in \C).$$
Furthermore,
\be \begin{cases}
&N_{\d,\tau'}^{\rm ec}=\iota_{Y^{cor}_{\d,\tau'}}\circ N^{\rm vf}_{\d,\tau'}\\
&\textrm{with}\quad N^{\rm vf}_{\d,\tau'}=\iota_{G_{\d,\tau'}}\circ \L_{\d c_{\d,\tau'}}\circ \Gamma_{\d,\tau}\\
&(N^{\rm vf}_{\d,\tau})_{*}(\d X_{\tau})=\pa_{z}+(2\pi i\d\mbeta w)\pa_{w} 
\end{cases}
\label{ee13.186}
\ee
and where $c_{\d,\tau'}\asymp 1$, $G_{\d,\tau'}(z,w)=O(w)$, $\iota_{G_{\d,\tau'}}(0,0)=(0,0)$ and $Y^{cor}_{\d,\tau}=O(\d^{p-1})$.

\bigskip 
In the reversible case, i.e. when 
$$\tau'=(\tau,\mbeta)=(\tau_{\d}(t,\mbeta),\mbeta),\quad (t,\mbeta)\in D_{}$$
we know by Theorem \ref{ref:theocommpairreversible} that the pair (\ref{conc:thm7.6new}) is reversible w.r.t. an anti-holomorphic  involution of the form
$$(z,w)\mapsto (-\bar z+a_{\d,\tau'}\bar w,b_{\d,\tau'}\bar w)+O(w^2)+O(\d^{p-1})\qquad (a_{\d,\tau'},b_{\d,\tau'}\in\C).$$
Besides, by Proposition \ref{prop:parameterdependence} we have for some $C_{A}>0$
 \begin{align*}
 &\| \tau'\mapsto F_{\d,\tau'}^{\rm vf}\|_{C^1(D_{\d},\cO(R_{s,\rho}))} \leq C_{A}\d^{-2}\\
 &\| \tau'\mapsto F_{\d,\tau'}^{\rm cor}\|_{C^1(D_{\d},\cO(R_{s,\rho}))}\leq C_{A}\d^{p-4}.
 \end{align*}

\subsubsection{Putting the system into  KAM form}\label{sec:13.8.2}
Before applying  the KAM Theorem \ref{theo:SiegelReversible} we  have to put our system in suitable KAM form, see subsection \ref{sec:suitableKAMform}. 
 
 Let us set 
 \be
 \left\{
\begin{aligned}
&\b_{1,\d,\tau'}=\b_{1,\mbeta,\d}=3 \d \mbeta\\
&\b_{2,\d,\tau'}=\b_{2,\tau,\mbeta,\d}=3 \biggl[\frac{1}{3  \d\mbeta g_{\d}(\tau')}\biggr]\d \mbeta\\
&\a_{\d,\tau'}=-\biggl\{\frac{1}{3 \d\mbeta g_{\d}(\tau')}\biggr\}.
\end{aligned}
\right.
\label{13.136}
\ee
From Propositions \ref{prop:10.1},\ref{prop:10.4} and Corollary \ref{cor:10.3}, we know that we can conjugate the commuting pair (\ref{conc:thm7.6new})
to a commuting pair $(f'_{1,\d,\tau'}, f'_{2,\d,\tau'})$: 
\begin{multline}\bm f'_{1,\d,\tau'}\\ f'_{2,\d,\tau'}\em=\bm S_{\b_{1,\d,\tau'}}\Phi_{w}\\ S_{\b_{2,\d,\tau'}}\circ \Phi_{\a_{\d,\tau'}w}\circ \iota_{F'_{\d,\tau'}}\em\\ =D_{\d^{(p-1)/2}}\circ \bm \Phi_{3\d\mbeta}\circ \Phi_{w}\\ S_{3q_{\d}\d\mbeta}\circ \Phi_{ \a_{\d,\tau'}w}\circ\iota_{F^{\rm vf}_{\d,\tau'}}\circ \iota_{F^{\rm cor}_{\d,\tau'}}\em\circ D_{\d^{-(p-1)/2}},
\label{f'1f'2}
\end{multline}
where 
\begin{align*}
&D_{\d^{(p-1)/2}}:(z,w)\mapsto (z, \d^{-(p-1)/2}w)\\
&F'_{\d,\tau'}\in \cO(\Psi_{\b_{1,\d,\tau'}}(R_{s,s})),
\end{align*}
with $\Psi_{\b_{1,\d,\tau'}}(R_{s,s})\subset (-\nu_{2},1+\nu_{2})\times \bD(0,s_{2})$. 
This pair   leaves invariant an anti-holomorphic involution 
$$\s'_{\d,\tau'}=\s_{0}\circ (id+\eta'_{\d,\tau'})$$
($\s_{0}(z,w)=(-\bar z,\bar w)$).
Moreover, one has the estimates
\begin{align*}
&\|\tau'\mapsto F'_{\d,\tau'}\|_{C^1(D_{},\Psi_{\b_{1,\d,\tau'}}(R_{s,\rho}))}=O(\d^{(p-2)/2-2})\\
&\|\eta'_{\d,\tau'}\|_{\Psi_{1,\d,\tau'}(R_{s,\rho})}=O(\d^{(p-1)/2}).
\end{align*}

\subsubsection{Applying the KAM-Siegel Theorem}\label{sec:13.8.3}
We now set 
\begin{align*}\check\b_{\d,\tau'}&=\b_{2,\d,\tau'}-\a_{\d,\tau'}\b_{1,\d,\tau'}\\
&=3q_{\d}\d\mbeta+(-\a_{\d,\tau'})\times (3\d \mbeta)\\
&=3 \d\mbeta (q_{\d}-\a_{\d,\tau'})\\
&=3 \mbeta T_{\d}(\tau')\\
&=\frac{1}{g_{\d}(\tau')}\\
\end{align*}
and
$${\g}_{\d,\tau'}=({\a}_{\d,\tau'},\b_{1,\d,\tau'},\b_{2,\d,\tau'})$$
\begin{align*}{\check\g}_{\d,\tau'}&=({\a}_{\d,\tau'},{\check\b}_{\d,\tau'})\\
&=\biggl(  -\biggl\{\frac{1}{3 \d\mbeta g_{\d}(\tau')^{}}\biggr\},\frac{1}{g_{\d}(\tau')^{}} \biggr).
\end{align*}
As we have seen (cf. (\ref{eq:13.134}))
the map 
$$\R^2\supset \bD_{\R^2}(t'_{*},\d^2)\ni (t,\mbeta)\mapsto \Re( \check\g_{\d,(\tau_{\d}(t,\mbeta),\mbeta)})\subset\R^2$$
is a diffeomorphism that has $C^1$-norm $\leq \d^{-3/2}$ and the norm of its inverse is $\leq \d^{-1/2}$. 

One checks that the conditions (1)-(3) of the beginning of subsection \ref{sec:reversiblecase} are satisfied.

We can thus apply  the KAM-Siegel Theorem \ref{theo:SiegelReversible} in the reversible case: there exists a set $E_{\d}:=\cB_{\rm rev.}^{(\infty)}\subset \bD_{\R^2}(t'_{*},\d^2)$ with positive Lebesgue measure such that  for any 
$$\tau'=(\tau_{\d}(t,\mbeta),\mbeta)\qquad (t,\mbeta)\in E_{\d}
$$ there exist $\check\a^{\infty}_{\d,\tau'}\in \R$ and  an exact conformal symplectic diffeomorphism $\iota_{Y_{\d,\tau'}^{[1,\infty]}}$
\be Y_{\d,\tau'}^{[1,\infty]}\in \cO(\Psi_{\b_{1,\d,\tau'}}(e^{-1/3}R_{s,s})),\qquad  \|Y\|_{\Psi_{\b_{1,\d,\tau'}}(e^{-1/3}R_{s,s}))}\leq \d^{(p-2)/2-a} \label{13.196}\ee
 such that  on $\Psi_{\b_{1,\d,\tau'}}(e^{-1/3}R_{s,s})$
\begin{multline}\bm S_{\b_{1,\d,\tau'}}\circ\Phi_{w}\\S_{\beta_{2,\d,\tau'}}\circ \Phi_{\a_{\d,\tau'} w}\circ \iota_{F_{\d,\tau'}}\em=\\ \iota_{Y_{\d,\tau'}^{[1,\infty]}}^{-1}\circ \bm S_{\b_{1,\d,\tau'}}\circ \Phi_{w}\\ S_{\beta_{2,\d,\tau'}}\circ \Phi_{\check\a^{\infty}_{\d,\tau'} w}\em\circ \iota_{Y_{\d,\tau'}^{[1,\infty]}}.\label{e13.188}\end{multline}

Putting together the conjugation relations (\ref{conc:thm7.6new}), (\ref{f'1f'2}), and (\ref{e13.188}) we get
\begin{multline}
( \iota_{Y_{\d,\tau'}^{[1,\infty]}}\circ D_{\d^{(p-1)/2}}\circ N^{\rm ec}_{\d,\tau'})\circ\bm h_{\d,\tau'}\\ h^{q_{\d}}_{\d,\tau'}\em\circ (\iota_{Y_{\d,\tau'}^{[1,\infty]}}\circ D_{\d^{(p-1)/2}}\circ N^{\rm ec}_{\d,\tau'})^{-1}=\\ \bm S_{\b_{1,\d,\tau'}}\circ \Phi_{w}\\ S_{\beta_{2,\d,\tau'}}\circ \Phi_{\a^{\infty}_{\d,\tau'} w}\em.
\end{multline}

Conjugating by $\Psi_{\b_{1,\d,\tau'}}:(z,w)\mapsto (z,e^{-i\b_{1,\d,\tau'}z}w)$, yields the linearization relation
\be N_{\d,\tau'}\circ \bm h_{\d,\tau'}\\ h^{q_{\d}}_{\d,\tau'}\em\circ N_{\d,\tau'}^{-1}=\bm \cT_{1,0}\\ \cT_{\check \a^{\infty}_{\d,\tau'},\check \b^\infty_{\d,\tau'}}\em\label{eq:13.204}\ee
where $(\check\a^\infty_{\d,\tau'},\check \b^\infty_{\d,\tau'})\in \R^2$ is non resonant and 
\be N_{\d,\tau'}=\Psi_{\b_{1,\d,\tau'}}\circ  \iota_{Y_{\d,\tau'}^{[1,\infty]}}\circ D_{\d^{(p-1)/2}}\circ N^{\rm ec}_{\d,\tau'}.\label{e13.191}\ee

\subsubsection{Conclusion} We now check  the conclusions of Theorem \ref{theo:9.1} are satisfied, the main point being to verify (\ref{cond9.62new}) holds. 

To do this we recall (\ref{ee13.186}): one has
$$N_{\d,\tau'}^{\rm ec}=\iota_{Y^{cor}_{\d,\tau'}}\circ \iota_{G_{\d,\tau'}}\circ \L_{\d c_{\d,\tau'}}\circ \Gamma_{\d,\tau'}$$
which joined with (\ref{e13.191}) yields
\begin{align}N_{\d,\tau'}&=\Psi_{\b_{1,\d,\tau'}}\circ  \iota_{Y_{\d,\tau'}^{[1,\infty]}}\circ D_{\d^{(p-1)/2}}\circ \iota_{Y^{cor}_{\d,\tau'}}\circ N^{\rm vf}_{\d,\tau'}\label{N.3.195ante}\\ 
&=\Psi_{\b_{1,\d,\tau'}}\circ  \iota_{Y_{\d,\tau'}^{[1,\infty]}}\circ D_{\d^{(p-1)/2}}\circ \iota_{Y^{cor}_{\d,\tau'}}\circ \iota_{G_{\d,\tau'}}\circ\L_{\d c_{\d,\tau'}}\circ \Gamma_{\d,\tau'}.\label{N.3.195}
\end{align}
The expression (\ref{N.3.195ante}) of $N_{\d,\tau'}$, the last equality  of (\ref{ee13.186}), the estimates  $Y^{\rm cor}_{\d,\tau'}=O(\d^{p-1})$, (\ref{13.196}) and  the relation $(\Psi_{\b_{1,\d,\tau'}})_{*}(\pa_{z}+(2\pi i\d\mbeta w)\pa_{w})=\pa_{z}$  give, taking into account the contribution of the conjugation $D_{\d^{(p-1)/2}}$,
$$(N^{}_{\d,\tau'})_{*}(\d X_{\tau})=\pa_{z}+O(\d^{\min((p-1)-(p-1)/2,(p-2)/2-a)})=\pa_{z}+O(\d^{p/2-a-1})$$
and if we use estimate (\ref{estetasharp}) 
\be (N^{}_{\d,\tau'})_{*}(\d X^\sharp_{\d,\tau'})=\pa_{z}+O(\d^{p/2-a-1}).\label{13.192}\ee
This proves the last estimate (iii) of (\ref{cond9.62new}).

Statement  (\ref{cond9.62new})-(ii)  is proved in a similar way from (\ref{N.3.195}). Indeed, because,  
$ \iota_{Y_{\d,\tau'}^{[1,\infty]}}^{-1}\circ \Psi_{\b_{1,\d,\tau'}}^{-1}(0,0)=O(\d^{p/2-a-1})$, $D_{\d^{(p-1)/2}}^{-1}(O(\d^{p/2-a-1}))=O(\d^{p-a-2})$,  $\iota_{G_{\d,\tau'}}(0,0)=(0,0)$ and $\Gamma_{\d,\tau'}^{-1}(0,0)=\zeta_{\d,\tau'}$, we deduce that
$$N_{\d,\tau'}^{-1}(0,0)\in \bD_{C^2}(\zeta_{\d,\tau'},O(\d^{p-a-2})).$$

To prove Item  (\ref{cond9.62new})-(i) we observe that $\Psi_{\b_{1,\d,\tau'}}(e^{-1/3}R_{s})$,  the linearization domain  of (\ref{eq:13.204}),
 is sent by $D_{\d^{(p-1)/2}}^{-1}\circ  (\iota_{Y_{\d,\tau'}^{[1,\infty]}})^{-1}\circ \Psi^{-1}_{\b_{1,\d,\tau'}}$  onto a neighborhood of $(-\nu,1+\nu)_{s}\times \bD(0,\d^{(p-1)/2}s)$. Because (cf. \ref{e13.191}))
 $N_{\d,\tau'}^{-1}=(N^{\rm ec}_{\d,\tau'})^{-1}\circ (D_{\d^{(p-1)/2}}^{-1}\circ  (\iota_{Y_{\d,\tau'}^{[1,\infty]}})^{-1}\circ \Psi^{-1}_{\b_{1,\d,\tau'}})$, we get  by Corollary \ref{cor:10.2} (note that $\d^{(p-1)/2}\geq \d^{p-2}$) 
  $$\cW^\tau_{\d,C^{-1}\d^{(p-1)/2} s,\nu_{}/2}\subset (N^{}_{\d,\tau})^{-1}\biggl((-\nu_{},1+\nu_{})_{\d^{(p-1)/2}s_{}}\times \bD(0,\d^{(p-1)/2}s)\biggr).$$

 The comparison estimate (\ref{comp:sharpn}) allows us to  establish the left inclusion of   Item  (\ref{cond9.62new})-(i) 
 $$\cW^\sharp_{\d,\d^{(p-1)/2+1},\nu/2}\subset N_{\d,\tau'}^{-1}((-\check \nu,1+\check \nu)_{\check s}\times \bD(0,\check \rho)).$$
 
 The other inclusion of (\ref{cond9.62new})-(i)  is proved in a similar and  easier way.

\bigskip This completes the proof of Theorem \ref{theo:9.1} hence that of Theorem \ref{main:Aprime}.

\hfill $\Box$

\begin{figure}[h]
%\hspace{-4cm}
\includegraphics[scale=.25, left]{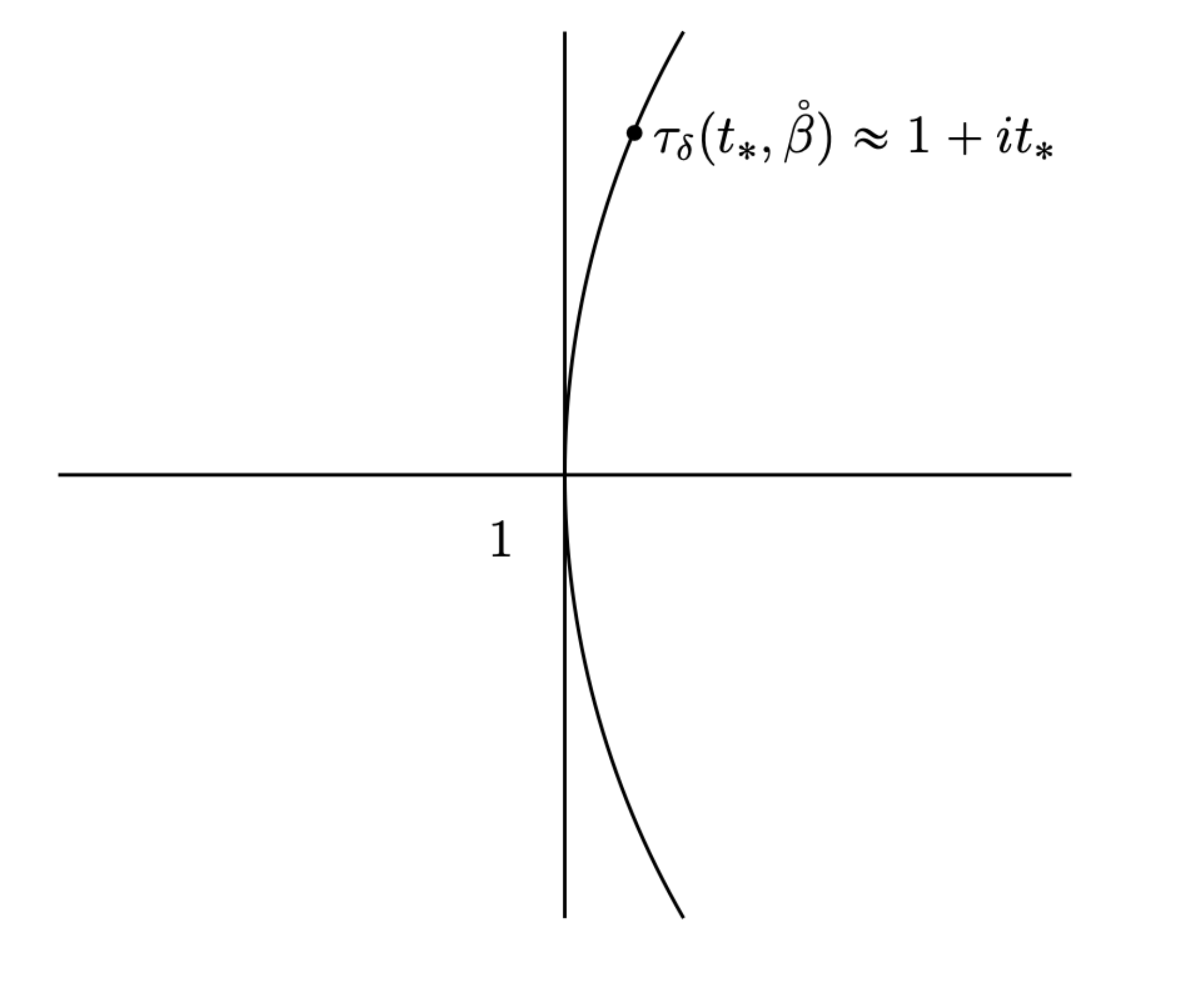}
\caption{ The point $\tau_{\d}(t_{*},\mathring{\beta})$ }\label{fig:renorm-flow-im-not2}
\end{figure}

\section{Existence of Herman Rings in the dissipative  case (Theorem \ref{main:B})}\label{sec:proofmainA:B}
The proof follows the main lines of that of Theorem \ref{main:Aprime} given in Section \ref{sec:proofmainA}.

For any $(\tau,\mbeta)$ and $\d$ small enough we can  perform the same steps described in Subsections \ref{sec:13.1}, \ref{sec:13.2}.

We then apply the procedure of Subsection \ref{sec:13.3} as follows.

We  introduce  $\tau_{*}\in (1-\nu,1+\nu)$ such that the analogue of (\ref{eq:10.57bis}) is satisfied:
 \be 
 \left\{
 \begin{aligned} 
 &g_{0}(\tau_{*})\in\R^*\\
 &  \frac{\pa g_{0}}{\pa\tau}(\tau_{*})\in \R^*.
  \end{aligned}
  \right.
  \label{eq:10.57bisnew}
  \ee
Let $\mbeta_{0}>0$.
 From the Inverse Mapping Theorem, we deduce the existence of  $\ph_{0}$, $\d_{0}$ positive such that for any $\ph\in (-\ph_{0},\ph_{0})$, $\mbeta \in (\mbeta_{0}/2,\mbeta_{0})$ and $\d\in (0,\d_{0})$  there exists   $\tau_{\#,\d}(\ph,\mbeta)$ in a neighborhood of $\tau_{*}$ such that 
  $$\Im\biggl(e^{i\ph}g_{\d}(\tau_{\#,\d}(\ph,\mbeta),e^{i\ph}\mbeta_{})\biggr)=0.$$
 The function $(\mbeta_{0}/2,\mbeta_{0})\ni \mbeta\mapsto \tau_{\#,\d}(\ph,\mbeta)-\tau_{*}$ has a $C^1$-norm which is $O(\d)$, thus, if $\d$ is small enough the function $$(\mbeta_{0}/2,\mbeta_{0})\ni \mbeta\mapsto \mbeta e^{i\ph}g_{\d}(\tau_{\#,\d}(\ph,\mbeta),e^{i\ph}\mbeta)\in \R$$
 has a derivative the absolute value of which is bounded below by some positive contant independent of $\ph\in (-\ph_{0},\ph_{0})$. Therefore, for each fixed $\d$ small enough, there exists 
 $$\mbeta_{*,\d,\ph}\in (\mbeta_{0}/2,\mbeta_{0})\subset\R$$ 
 such that 
 $$\frac{1}{3\d \mbeta_{*,\d,\ph}e^{i\ph}g_{\d}(\tau_{\#,\d}(\ph,\mbeta_{*,\d,\ph}),e^{i\ph}\mbeta_{*,\d,\ph})}\in\R\setminus\biggl(\bigcup_{k\in\Z}[k-(1/10),k+(1/10)]\biggr).$$
If we  set
$$\begin{cases}&\tau_{\#,\d,\ph}=\tau_{\#,\d}(\ph,\mbeta_{*,\d,\ph})\\
&\mbeta_{\#,\d,\ph}=e^{i\ph}\mbeta_{*,\d,\ph}\\
&g^\#_{\d,\ph}=1/T^\#_{\d,\ph}=3 \mbeta_{*,\d,\ph}e^{i\ph}g_{\d}(\tau_{\#,\d,\ph},\mbeta_{\#,\d,\ph})
\end{cases}
$$
  we thus have
 \be 
 \left\{
 \begin{aligned} 
 &\Im(\mbeta_{\#,\d,\ph})=\sin\ph\times \mbeta_{*,\d,\ph}\qquad (\mbeta_{*,\d,\ph}>0)\\
 &e^{i\ph}g_{\d}(\tau_{\#,\d,\ph},\mbeta_{\#,\d,\ph})\in \R^*\\
 &  \frac{\pa g_{\d}}{\pa\tau}(\tau_{\#,\d,\ph},\mbeta_{\#,\d,\ph})\ne 0
  \end{aligned}
  \right.
  \label{eq:10.57bisnewdelta}
  \ee
  and there exists $q_{\d}\in \Z$ such that 
  $$q_{\d}:=\biggl[\frac{T^\#_{\d,\ph}}{\d}\biggr],\qquad \biggl\{\frac{T^\#_{\d,\ph}}{\d}\biggr\}\in ((1/10),(9/10)).$$
  Note that the vector field 
  $$3\d \mbeta_{\#,\d,\ph}X_{\d,\tau_{\#,\d,\ph},\mbeta_{\#,\d,\ph}}=3\d \mbeta_{*,\d,\ph}e^{i\ph}X_{\d,\tau_{\#,\d,\ph},\mbeta_{\#,\d,\ph}}$$
  has a 
  $T^\#_{\d,\ph}$-periodic orbit where 
  $$T^\#_{\d,\ph}=\frac{1}{3 \mbeta_{*,\d,\ph}e^{i\ph}g_{\d}(\tau_{\#,\d,\ph},\mbeta_{\#,\d,\ph})}\in\R^*.$$
    
  Like in Subsection \ref{subsec:13.5}, we now apply Proposition \ref{lemma:7.1} and  Theorem \ref{cor:2.10:commpair} to the family  of holomorphic   diffeomorphisms
  $$\begin{cases}
  &h_{\d,\tau'}=\phi^1_{3\d \mbeta X^{}_{\d,\tau'}}\circ \iota_{F^{}_{\d,\tau'}}\\
 &\tau'=(\tau,\mbeta)\in \bD(\tau_{\#,\d,\ph},\d^2)\times \bD(\mbeta_{\#,\d,\ph},\d^2)
  \end{cases}$$
  to get the commuting pairs (see the notation (\ref{W*tau'}), (\ref{Wtau'}), (\ref{notationWsharp}))
  \begin{align}  
&(h_{\d,\tau,\mbeta},h^{q_{\d}}_{\d,\tau,\mbeta})_{\cW^{*,\tau'}_{\d,s,\nu} },\label{commpairdefinednew1}\\
&(h_{\d,\tau,\mbeta},h^{q_{\d}}_{\d,\tau,\mbeta})_{\cW^{\tau'}_{\d,s,\nu} },\label{commpairdefinednew}\\
&(h_{\d,\tau,\mbeta},h_{\d,\tau,\mbeta}^{q_{\d}})_{\ \cW^{\sharp,\tau'}_{\d,s,\nu}}.\label{commpairdefinednewbis3}
\end{align}

Like in Subsection \ref{sec:13.8} we can first partially renormalize the commuting  pair (\ref{commpairdefinednew}) (cf. Paragraph \ref{sec:13.8.1}) and put it into suitable KAM form (cf. Paragraph  \ref{sec:13.8.2}):
 \be 
 D_{\d^{(p-1)/2}}\circ N^{\rm ec}_{\d,\tau'}\circ\bm h_{\d,\tau'}\\ h^{q_{\d}}_{\d,\tau'}\em\circ (D_{\d^{(p-1)/2}}\circ N^{\rm ec}_{\d,\tau'})^{-1}
 = \bm S_{\b_{1,\d,\tau'}}\Phi_{w}\\ S_{\b_{2,\d,\tau'}}\circ \Phi_{\a_{\d,\tau'}w}\circ \iota_{F'_{\d,\tau'}}\em
\label{conc:thm7.6newherman}
 \ee
 with $N^{\rm ec}_{\d,\tau'}$ satisfying (\ref{ee13.186}),
  \be
 \left\{
\begin{aligned}
&\b_{1,\d,\tau'}=\b_{1,\mbeta,\d}=3 \d \mbeta\\
&\b_{2,\d,\tau'}=\b_{2,\tau,\mbeta,\d}=3 \biggl[\frac{1}{3  \d\mbeta g_{\d}(\tau')}\biggr]\d \mbeta\\
&\a_{\d,\tau'}=-\biggl\{\frac{1}{3 \d\mbeta g_{\d}(\tau')}\biggr\}.
\end{aligned}
\right.
\label{13.136herman}
\ee
and
\begin{align*}
&D_{\d^{(p-1)/2}}:(z,w)\mapsto (z, \d^{-(p-1)/2}w)\\
&F'_{\d,\tau'}\in \cO(\Psi_{\b_{1,\d,\tau'}}(R_{s,s}))\\
&\|\tau'\mapsto F'_{\d,\tau'}\|_{C^1(D_{},\Psi_{1,\d,\tau'}(R_{s,\rho}))}=O(\d^{(p-2)/2-2}).
\end{align*}

Like in Paragraph \ref{sec:13.8.3} we set
\begin{align*}\check\b_{\d,\tau'}&=\b_{2,\d,\tau'}-\a_{\d,\tau'}\b_{1,\d,\tau'}\\
&=\frac{1}{g_{\d}(\tau')}
\end{align*}
$${\g}_{\d,\tau'}=({\a}_{\d,\tau'},\b_{1,\d,\tau'},\b_{2,\d,\tau'})$$
\begin{align*}{\check\g}_{\d,\tau'}&=({\a}_{\d,\tau'},{\check\b}_{\d,\tau'})\\
&=\biggl(  -\biggl\{\frac{1}{3 \d\mbeta g_{\d}(\tau')^{}}\biggr\},\frac{1}{g_{\d}(\tau')^{}} \biggr).
\end{align*}
and we can then apply the KAM-Siegel theorem in the dissipative case, Theorem \ref{theo:SiegelDissipative}, in the following way.
Let $$\a_{\#,\d,\ph}=-\biggl\{\frac{1}{3 \d\mbeta_{\#,\d,\ph} g_{\d}(\tau_{\#,\d,\ph},\mbeta_{\#,\d,\ph})}\biggr\}=\biggl\{\frac{T^\#_{\d,\ph}}{\d}\biggr\}.$$
 For each $\ph\in (-\ph_{0},\ph_{0})$ and 
$\mbeta\in \bD(\mbeta_{\#,\d,\ph},\d^2)$ 
there exists a positive Lebesgue measure set $A_{\d,\mbeta,\ph}=\cA^{(\infty)}_{\rm dissip.}\subset \bD_{\R}(\a_{\#,\d,\ph},\rho_{*}\d^2)$ of frequencies $\a\in  \R$ and a $C^1$-embedding $\a_{\infty,\mbeta,\ph}^{-1}:\bD_{\R}(\a_{\#,\d,\ph},\rho_{*}\d^2)\to \C$ such that if 
\be \tau'=(\a_{\infty,\mbeta,\ph}^{-1}(\a),\mbeta),\qquad \a\in A_{\d,\mbeta,\ph},\quad \mbeta\in \bD(\mbeta_{\#,\d,\ph},\d^2)\label{tauprimeinf}\ee
 there exists an exact  symplectic diffeomorphism $\iota_{Y_{\d,\tau'}^{[1,\infty]}}$
\be Y_{\d,\tau'}^{[1,\infty]}\in \cO(\Psi_{\b_{1,\d,\tau'}}(e^{-1/3}R_{s,s})),\qquad  \|Y\|_{\Psi_{\b_{1,\d,\tau'}}(e^{-1/3}R_{s,s}))}\leq \d^{(p-2)/2-a} \label{13.196herman}\ee
 such that  on $\Psi_{\b_{1,\d,\tau'}}(e^{-1/3}R_{s,s})$
\begin{multline}\bm S_{\b_{1,\d,\tau'}}\circ\Phi_{w}\\S_{\beta_{2,\d,\tau'}}\circ \Phi_{\a_{\d,\tau'} w}\circ \iota_{F_{\d,\tau'}}\em=\\ \iota_{Y_{\d,\tau'}^{[1,\infty]}}^{-1}\circ \bm S_{\b_{1,\d,\tau'}}\circ \Phi_{w}\\ S_{\beta_{2,\d,\tau'}}\circ \Phi_{\a  w}\em\circ \iota_{Y_{\d,\tau'}^{[1,\infty]}}.\label{e13.188herman}\end{multline}
Hence, if (\ref{tauprimeinf}) holds, one has
 \begin{multline}
( \iota_{Y_{\d,\tau'}^{[1,\infty]}}\circ D_{\d^{(p-1)/2}}\circ N^{\rm ec}_{\d,\tau'})\circ\bm h_{\d,\tau'}\\ h^{q_{\d}}_{\d,\tau'}\em\circ (\iota_{Y_{\d,\tau'}^{[1,\infty]}}\circ D_{\d^{(p-1)/2}}\circ N^{\rm ec}_{\d,\tau'})^{-1}
\\
= \bm S_{\b_{1,\d,\tau'}}\circ \Phi_{w}\\ S_{\beta_{2,\d,\tau'}}\circ \Phi_{\a w}\em.
\label{conc:thm7.6newhermannew}
 \end{multline}
 
 The preceding discussion has  the following corollary on the following representation of $h_{\d,\tau,\mbeta}$ (see (\ref{Xtausharp}), (\ref{defphi}))
 \be \begin{cases}
&X:=X^\sharp_{\d,\tau,\mbeta}=3\mbeta e^{-i\ph_{\d,\tau,\mbeta}}X_{\delta,\tau,\mbeta}\\
&id+\eta:= id+\eta^\sharp_{\d,\tau,\mbeta}=\phi^{-1}_{3\d\mbeta e^{-i\ph_{\d,\tau,\mbeta}}X_{\d,\tau_{},\mbeta}}\circ \phi^1_{3\d\mbeta X_{\d,\tau,\mbeta}}\circ\iota_{F_{\d,\tau,\mbeta}} \\
&h_{\d,\tau,\mbeta}=X^\sharp_{\d,\tau,\mbeta}\circ (id+\eta^\sharp_{\d,\tau,\mbeta})
\end{cases}
\label{Xtausharpnew}
\ee
 \begin{cor}\label{cor:14.1}If $$\tau'=(\a_{\infty,\mbeta,\ph}^{-1}(\a),\mbeta),\qquad \a\in A_{\d,\mbeta,\ph},\quad \mbeta\in \bD(\mbeta_{\#,\d,\ph},\d^2)$$ estimate (\ref{estetasharp}) holds i.e.
 $$\eta^\sharp_{\d,\tau'}=O(\d^{p-a-1}).$$
 \end{cor}
 \begin{proof}It is sufficient to establish estimate (\ref{estetasharp}):
 $$ \ph_{\d,\tau'}=O(\d^{p-(4/3)}).\label{estetasharpnew}$$
 which follows from the fact 
 $$\Im g_{\d}(\tau')=O(\d^{p-(4/3)})$$
 that we now prove.
 
 The relations
$$ 
\begin{cases}
&N_{\d,\tau'}^{\rm ec}=\iota_{Y^{cor}_{\d,\tau}}\circ N^{\rm vf}_{\d,\tau}\\
&(N^{\rm vf}_{\d,\tau'})_{*}(\d X_{\d,\tau'})=\pa_{z}+(6\pi i\d\mbeta w)\pa_{w} 
\end{cases}
$$
(see (\ref{n10.93})) and (\ref{13.196herman}) show that the piece of invariant annulus $\cA^{\rm vf,s}_{\d,\tau'}\cap \bar \cW^{*,\tau'}_{\d,s} $ associated to the vector field $X_{\d,\tau'}$ and  lying in the renormalization box $\bar \cW^{*,\tau'}_{\d,s} $  is contained in some $O(\d^{p-a})$-neighborhood of some $\hat h_{\d,\tau'}$ invariant set $\hat \cA^{\rm }_{\d,\tau'}\subset  \bar \cW^{*,\tau'}_{\d,s} $\footnote{In the quotient manifold $\ti W^*_{\d,\tau'}$ (see subsection \ref{sec:glueing}) it is the invariant annulus $\ti\cA^{\rm }_{\d,\tau'}$ associated to the renormalized diffeomorphism $\ti h_{\d,\tau'}$.}.
Because $h_{\d,\tau'}=\phi^1_{X_{\d,\tau'}}\circ (id+O(\d^p)$ and the return times in $\bar \cW^{*,\tau'}_{\d,s} $ associated to the first return map $\hat h_{\d,\tau'}$ are $q_{\d}$ or $q_{\d}+1$ with $q_{\d}\asymp \d^{-1}$ we see that for any point $\xi \in \cA^{\rm vf,s/4}_{\d,\tau'}$ one has
$$\forall t\in [0,\d^{-(p-a-1)}], \quad \phi^t_{X_{\d,\tau'}}(\xi)\in  \cA^{\rm vf,s/2}_{\d,\tau'}.$$
Nevertheless, the dynamics of $X_{\d,\tau'}$ on $\cA^{\rm vf}_{\d,\tau'}$ is conjugate to that of the vector field $g_{\d}(\tau')\pa_{\th}$ on the annulus $\T_{s}$. This implies that 
$$|\Im g_{\d}(\tau')|\lesssim \d^{p-a-1}.$$
 
 \end{proof}
 
\begin{figure}[h]
\ \hspace{2cm}
\includegraphics[scale=.25, left]{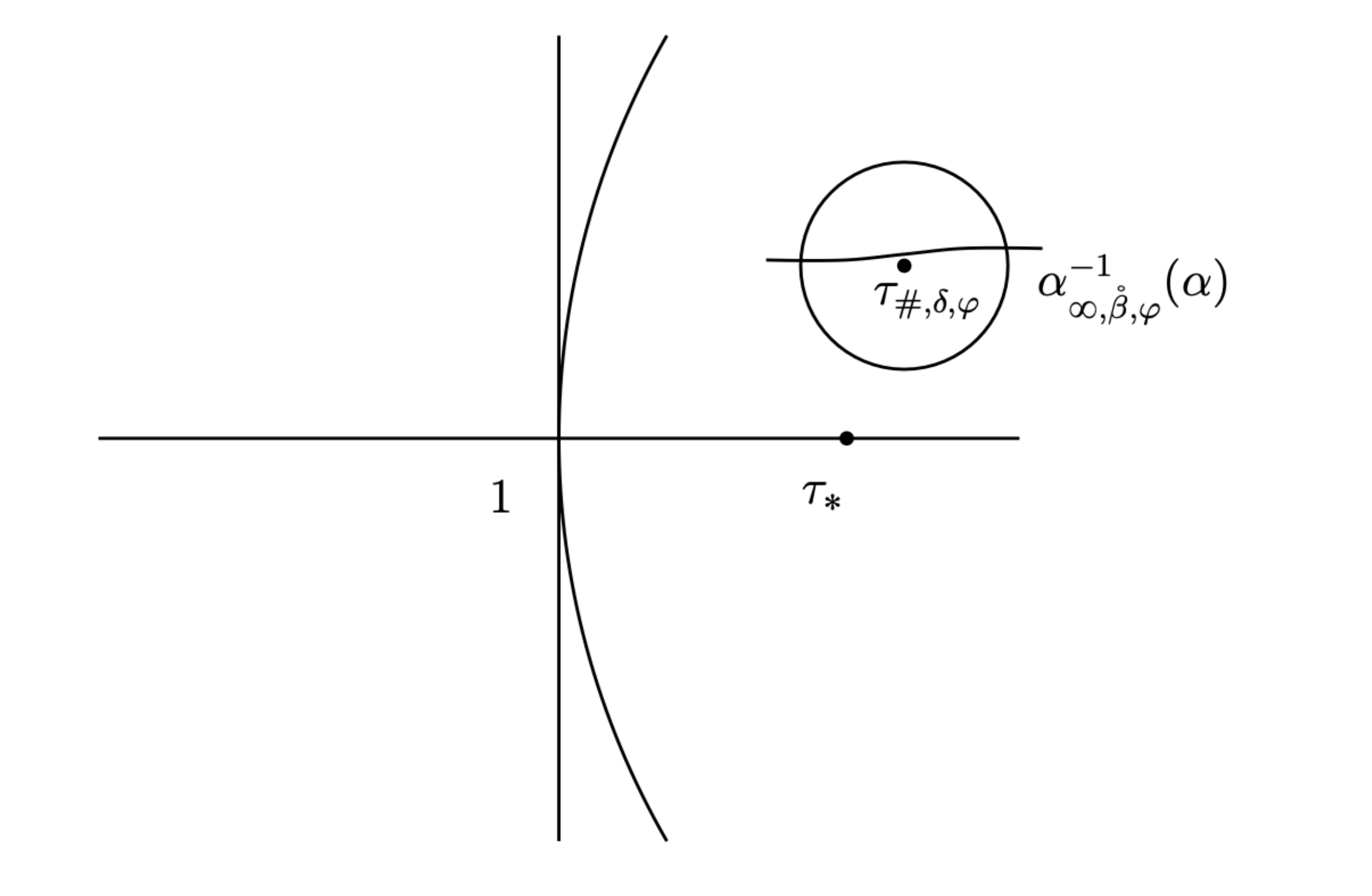}
\caption{ The point $\tau_{*}$, $\tau_{\#,\d,\varphi}$ and the curve $\a^{-1}_{\infty,\mbeta,\ph}(\a)$. }\label{fig:renorm-flow-im-not}
\end{figure}

We can now state the analog of Theorem \ref{theo:9.1} in the dissipative case.

\begin{theo}\label{theo:9.1herman}If $\tau'$ is of the form (\ref{tauprimeinf}) with $\ph\in (\ph_{0}/2,\ph_{0})$ (so that any $ \mbeta\in\bD(\mbeta_{\#,\d,\ph},\d^2)$  has positive imaginary part)
  there exists a holomorphic diffeomorphism
$$N_{\d,\tau'}^{-1}:(-\check \nu,1+\check \nu)_{\check s}\times \bD(0,\check \rho)\to \C^2$$
which  satisfies with $p^\sharp=p-2$
\be
\begin{cases}
&(i)\quad \cW^{\sharp,\tau'}_{\d,\d^{p^\sharp/2+2},\nu/2}\subset N_{\d,\tau'}^{-1}((-\check \nu,1+\check \nu)_{\check s}\times \bD(0,\check \rho))\subset \cW^{\sharp,\tau'}_{\d,s',\nu}\\
&(ii)\quad N_{\d,\tau'}^{-1}(0,0)\in\bD(\zeta_{\d,\tau'}, \d^{p^\sharp-a})\\
&(iii)\quad (N_{\d,\tau'}^{-1})_{*}\pa_{z}=\d X^\sharp_{\d,\tau'}+O(\d^{p^\sharp/2-a}).
\end{cases}
\label{cond9.62newdissip}
\ee
and such that 
 $N_{\d,\tau'}$
 conjugates on  $N_{\d,\tau'}^{-1}((-\check \nu,1+\check \nu)_{\check s}\times \bD(0,\check \rho))$ the commuting pair $(h_{\d,\tau'},h^{q_{\d}}_{\d,\tau'})$ to a normalized pair $(\cT_{1,0},\cT_{\check\a_{\tau'},\check \b_{\tau'}})$
with $\check \a_{\tau'}\in(-1,0)$, $\Im \check \b_{\tau'}>0$ and $(\check \a_{\tau'},\check \b_{\tau'})$ is non resonant. 
\end{theo}

\begin{proof}The proof of this result is the same as that of Theorem \ref{theo:9.1} provided one makes use of Corollary \ref{cor:14.1}. 
\end{proof}

\bn{\it Conclusion.--} The preceding Theorem \ref{theo:9.1herman} allows to apply the discussion of Subsection \ref{sec:13.7}  to the dissipative case.

The fact that $h^{\textrm{Hénon}}_{\b,c}$, or equivalently $h^{\rm mod}_{\a,\b}$, has a Herman ring reduces to the fact that $h^{\rm bnf}_{\d,\tau'}$ has a Herman ring.

With the notations of Paragraph \ref{sec:13.7.1} we see that  for $l=0,1,2$, the relation (\ref{eqinclusion})
$$ (h_{\d,\tau'}^{\rm bnf})^{-l}(O^\sharp_{\d,\tau'})\subset \cV_{\d^{p^\sharp-1}}( O^\sharp_{\d,\tau'})\subset \check  \cC_{\check s, \check \nu}$$
is still valid. Theorem \ref{invariantannulus} tells us $\cA_{\d,\eta}\subset \cV_{\d^{(2/3)p^\sharp+1}}(O^{\sharp}_{\d,\tau'})$ so
$$(h^{\rm bnf}_{\d,\tau'})^{-l}(\cA_{\d,\eta})\subset \cV_{\d^{(2/3)p^\sharp}}(\cA_{\d,\eta})\subset \check  \cC_{\check s, \check \nu}.$$
The set $(h^{\rm bnf}_{\d,\tau'})^{-l}(\cA_{\d,\eta})$ is an $h_{\d,\tau'}$-invariant annulus (on which the dynamics is conjugate to a rotation) included in a $\d^{(2/3)p^\sharp}$-neighborhood of $\cA_{\d,\eta}$. By Theorem \ref{invariantannulusbis} their intersection contains a non-empty $h_{\d,\tau'}$-invariant annulus and their union is thus  an $h_{\d,\tau'}$-invariant annulus on which the dynamics of $h_{\d,\tau'}$ is conjugate to an irrational rotation. This shows that the union 
$$\cA^{\rm bnf}_{\d,\tau'}:=\bigcup_{l=0}^2(h_{\d,\tau'}^{\rm bnf})^{-l}(\cA_{\d,\tau'})$$ is an annulus; it is by construction $h^{\rm bnf}_{\d,\tau'}$-invariant and attracting; it is not difficult to see that on this annulus $h^{\rm bnf}_{\d,\tau'}$ is conjugate to an irrational translation.

\medskip To check that $\cA^{\rm bnf}_{\d,\tau'}$  is a genuine annulus (and cannot be extended to  an attracting disk attached to the fixed points  of $h_{\d,\tau'}^{\rm bnf}$) we can use the relation (see (\ref{eq:n13.236}))
 $$\R\ni {\rm rot}(h_{\d,\tau'}\mid\cA_{\d,\tau'})=3\d\mbeta g_{\d}(\tau')+O(\d^2)$$
 and, like in Subsection \ref{sec:13.7.1}, check it is not compatible  with the frequencies (\ref{eq:13.229})  of $h_{\d,\tau'}$ at the fixed points of $h_{\d,\tau'}^{\rm bnf}$.
 
 \medskip
 This completes the proof of Theorem \ \ref{main:B}.
 
 \hfill$\Box$

\bigskip 
\section{Proof of the periodic orbit theorem}\label{sec:perorbthm} The aim of this Section is to provide  proofs for Theorems \ref{theo:periodicorbit} and \ref{theogreal}  of Section \ref{sec:invannulusthm} (Invariant annulus theorem).

\bigskip
Recall  our definition of the vector field 
$$ X_{\tau}(z,w)=X_{0,\tau}(z,w)= 2\pi i \bm (1-\tau)z+(1/2) z^2-(1/3)w^3\\ \tau w-zw\em,
$$
and the one obtained by conjugation by the translation $(z,w)\mapsto (z-\tau,w)$:
$$ \hat X_{\hat \tau}(z,w)=2\pi i \bm \hat \tau+z+(1/2)z^2-(1/3)w^3\\ - zw\em\label{vfmod2}
$$
with 
$$\hat \tau=\tau-(1/2)\tau^2.$$

\subsection{Fixed points and periodic orbits of $\hat X_{\hat \tau}$}

\subsubsection{Fixed points}\label{sec:14.1.1} Note that the vector field $\hat X_{\hat\tau}$ has in general  (when $\hat\tau\ne 1/2$ i.e. $\tau\ne 1$) 5 fixed points.

\begin{enumerate}
\item The  points $(z_{\pm},0)$ where $\hat \tau+z_{\pm}+(1/2)z_{\pm}^2=0$:
$$z_{\pm}=-1\pm \sqrt{1-(2\hat\tau)}=-1\pm \sqrt{1-(2\tau-\tau^2)}=-1\pm (\tau-1)=\tau-2,-\tau.$$
One has
$$D\hat X_{\hat\tau}(z_{\pm},0)=i\bm 1+z_{\pm}& 0\\ 0 &-z_{\pm}\em$$
which has eigenvalues $\pm  2\pi i(\tau-1)$ and $2\pi i(1\pm  (1-\tau))$.
\item The three  points $(0,j^k(3\hat\tau)^{1/3})$ ($k=0,1,2$). One then has
$$ D\hat X_{\hat\tau}(0,j^k(3\hat\tau)^{1/3})=2\pi i\bm 1& -(j^k(3\hat\tau)^{1/3})^2\\ -j^k(3\hat\tau)^{1/3} &0\em $$ 
($j=e^{2\pi i /3}$)
the eigenvalues of which are $2\pi ig_{\pm}$ where $g_{\pm}$ are  solutions of
$g^2-g-(3\hat\tau)=0$:
\be g_{\pm}=\frac{1\pm\sqrt{1+12\hat\tau}}{2}.\label{14.140}\ee
\end{enumerate}

\bigskip
\subsubsection{Some periodic orbits} \label{sec:perorb}The vector field $\hat X_{\hat\tau}$  admits the following  periodic orbits.
\begin{enumerate}
\item For any $c\in\R^*$,  the function $t\mapsto (z_{c}(t),0)$ is a periodic orbit of  $\hat X_{\hat\tau}$  where 
$$z_{c}(t)=\frac{z_{-}e^{2\pi i(\tau-1)t+c}-z_{+}}{e^{i(\tau-1)t+c}-1},\qquad c\notin\R$$
 is  solution of the differential equation 
\be \frac{dz}{dt}=2\pi i(\hat \tau+z+(1/2)z^2).\label{14.140bis}\ee
Indeed,
\begin{multline*}\frac{-idz}{\hat\tau+z+(1/2)z^2}=\frac{2dz}{2\hat\tau+2z+z^2}=\frac{2dz}{(z-z_{+})(z-z_{-})}\\=\frac{-2idz}{z_{+}-z_{-}} \biggl(\frac{1}{z-z_{+}}-\frac{1}{z-z_{-}}\biggr)=\frac{-idz}{\tau-1}\biggl(\frac{1}{z-z_{+}}-\frac{1}{z-z_{-}}\biggr)\\ =\frac{-i}{\tau-1}d\biggl( \ln\frac{z-z_{+}}{z-z_{-}}\biggr)\end{multline*}
so (\ref{14.140}) can be solved as
$$d\biggl( \ln\frac{z-z_{+}}{z-z_{-}}\biggr)=2\pi i(\tau-1)dt$$
equivalently
$$\frac{z_{c}-z_{+}}{z_{c}-z_{-}}=e^{2\pi i(\tau-1)t+c}$$
$$z_{c}(t)=\frac{z_{-}e^{2\pi i(\tau-1)t+c}-z_{+}}{e^{2\pi i(\tau-1)t+c}-1}.$$
This gives rise to two holomorphic functions
$$z^{\pm}:\bH_{\pm}/Z\ni \th\mapsto \frac{z_{-}e^{2\pi i\th}-z_{+}}{e^{2\pi i\th}-1}\in\C$$
($\bH_{\pm}$ are respectively the upper and lower half-planes in $\C$) solutions of the complex differential equation
$$\frac{dz}{d\th}=2\pi i(\hat \tau+z+(1/2)z^2).$$
Note that the function  $\ti z^\pm$, $\ti z^\pm(\zeta)=z^\pm(\th)$ where $\zeta=(\th-i)/(\th+i)\in \bD(0,1)$ extends to a holomorphic function defined on the open disk $\bD(0,1)$. 
\item Similarly, if
$$w_{\pm,c}(t)=e^{2\pi iz_{\pm}t}c$$
the function is a periodic solution of the ODE $dp/dt=\hat X_{\hat \tau}(p)$.
$t\mapsto (0,w_{c}(t))$ and more generally $\C\ni \th\mapsto e^{2\pi iz_{\pm}\th}c$ is a solution of the complex ODE $dp/d\th=\hat X_{\hat \tau}(p)$.
\item One can also  prove that the  vector field $\hat X_{\hat\tau}$  has periodic orbits of the form $(z_{c},w_{c})$ where
$$\begin{cases}
&z_{c}(t)=(-g_{\pm}c/(3\hat\tau)^{1/3})e^{\pm 2\pi  ig_{\pm}t}+\sum_{k\geq 2}z_{k}e^{\pm 2\pi  ig_{\pm}kt}\\
&w_{c}(t)=(3\hat \tau)^{1/3}+ce^{\pm  2\pi ig_{\pm}t}+\sum_{k\geq 2}w_{k}e^{\pm 2\pi  ig_{\pm}kt}\\
&g_{\pm}\quad\textrm{given\ by}\ (\ref{14.140})
\end{cases}$$
and $c$ is a small complex parameter. For $k=1,2$, the functions $t\mapsto (z_{c}(t), j^kw_{c}(t))$ are also  orbits of $\hat X_{\hat \tau}$ and these three solutions are distinct.   If one sets $T=1/g_{\pm}$,   the functions 
$$\bH_{+}/\Z\ni \th\mapsto (z_{c}(\th),w_{c}(\th)) $$
are solutions of the complex differential equation $dp/d\th=\hat X_{\hat\tau}(p)$.

\end{enumerate}

\subsubsection{Siegel disks}
When the fixed  points described in subsection \ref{sec:14.1.1} are Diophantine elliptic fixed points, the vector field version of Siegel's linearization theorem applies. After a holomorphic change of coordinates in some neighborhoods of these fixed points, the flow of the vector field $\hat X_{\hat \tau}$ becomes
$$(\z_{1},\zeta_{2})\mapsto (e^{2\pi it\a_{1}}\zeta_{1},e^{2\pi it\a_{2}}\zeta_{2}).$$
We can thus identify two obvious families of periodic orbits $t\mapsto  (e^{2\pi it\a_{1}}c,0)$ and  $t\mapsto (0,e^{2\pi it\a_{2}}c)$.
These correspond:
\begin{itemize}
\item In case of the fixed  points of \ref{sec:14.1.1}-(1), to the periodic orbits \ref{sec:perorb}-(1) and (2).
\item In case of the fixed  points of \ref{sec:14.1.1}-(2), to the periodic orbits \ref{sec:perorb}-$(3)_{\pm}$.
\end{itemize}
\bigskip

\subsubsection{Exotic periodic orbits}
In addition to the periodic orbits described in subsection \ref{sec:perorb} one can prove, and this is the main result of this section,  for $\hat\tau\in\R$ close to $1/2$,  the existence of another solution
$$\T_{s}\ni\th\mapsto p(\th):=(z(\th),w(\th))\in\C^2$$
 of the complex ODE $dp/d\th=\hat X_{\hat \tau}(p)$ which is $T_{\hat \tau}$-periodic, $T_{\hat\tau}\in\R^*$, in the sense that $(z(\th+T_{\hat \tau}),w(\th+T_{\hat \tau}))=(z(\th),w(\th))$.
 
 Besides, 
 as we shall see, 
 \begin{enumerate}
 \item $\hat g(\hat\tau):=(1/T_{\hat\tau})=-0.834\pm 10^{-3}$ when $\hat \tau=1/2$ ($\tau=1$).
 \item The orbit $\cA^{}_{s,\hat\tau}=\{(z(\th),w(\th))\mid \th\in\T_{s}\}$ is invariant by $(z,w)\mapsto (z,jw)$.
 \end{enumerate}
 As a consequence this orbit is not equal to the periodic orbits described in  subsection \ref{sec:perorb}: indeed, it cannot coincide with  the periodic orbits \ref{sec:perorb}-(1)-(2) because $-0.834\ne 0$ or $1$ and  it cannot coincide with  the periodic orbits \ref{sec:perorb}-$(3)_{\pm}$ because these last orbits are not preserved by $(z,w)\mapsto (z,jw)$. By Proposition \ref{exoticperorb}, the maximal invariant annulus (\ref{eq:maxmodannulus}) $\cA_{\rm max}$ associated to this periodic orbit must be exotic: its closure does not contain any fixed point of $\hat X_{\hat\tau}$. 
 
 It would be interesting to prove that  the annulus $\cA_{\rm max}$ has finite module.

\subsection{Main result}
If $e\in \C^2$ is a non zero vector we say that the line $\C e$ is transverse to the orbit of $X_{\tau}$ at a point $\zeta\in \C^2$ if 
$$\C^2=\C e \oplus \C X_{\tau}(\zeta).$$

The main result of this section is the following.
\begin{theo}[Exotic periodic orbit Theorem for $\hat X_{\hat \tau}$]\label{theo:expero}The vector field $\hat X_{\hat \tau}$   admits for $\hat \tau=1/2$  an exotic  $T_{*}=1/g_{*}$-periodic orbit  $(\phi^t_{X_{1}}(p_{*}) )_{t\in\R}$ with $g_{*}\in \R$ equal to $-0.834\pm 10^{-3}$. This orbit is invariant by $\diag(1,j)$ and more precisely for any $t\in\R$,
\be \diag(1,j)(\phi^t_{\hat X_{1/2}}(p_{*}))=\phi^{t+T_{*}/3}_{\hat X_{1/2}}(p_{*}).\label{eq:diag1jinvperorbn}
\ee
Moreover, $\hat X_{1/2}$ is reversible with respect to the anti-holomorphic involution   $\s:(z,w)\mapsto (\bar z,j^2\bar w)$ and for some $t_{*}\in \R$ one has
$$\s(p_{*})=\phi^{t_{*}}_{\hat X_{1/2}}(p_{*}).$$
Furthermore, if $\hat \tau\mapsto \hat g(\tau)$ is the map of  Theorem \ref{theogreal}:
\begin{itemize}
\item The map $\hat g$ takes real values on a small open  interval of $\R$ centered at $\hat \tau=1/2$.
\item The derivative of the map $\hat g$  at the point $1/2$ is a negative ($<0$) number which lies in the interval $(-1.9,-1.7)$.
\end{itemize}
\end{theo}
Its proof is given in Paragraph \ref{sec:proofofth15.1} of Subsection \ref{sec:15.1}.

It has an immediate corollary:
\begin{theo}[Exotic periodic orbit Theorem for $X_{\tau}$]\label{theo:experoXtau}The vector field $X_{ \tau}$   admits for $ \tau=1$  an exotic  $T_{*}=1/g_{*}$-periodic orbit  $(\phi^t_{X_{1}}(p_{*}) )_{t\in\R}$ with $g_{*}\in \R$ equal to $-0.834\pm 10^{-3}$. This orbit is invariant by $\diag(1,j)$ and more precisely for any $t\in\R$,
\be \diag(1,j)(\phi^t_{X_{1}}(p_{*}))=\phi^{t+T_{*}/3}_{X_{1}}(p_{*}).\label{eq:diag1jinvperorbnn}
\ee
Furthermore, if $ \tau\mapsto g(\tau)$ is the frequency map of $X_{\tau}\mid \cA_{\0,\tau'}^{\rm vf}$ one has:
\begin{itemize}
\item $g(\tau)=\hat g(\tau-\tau^2/2)$.
\item The map $g$ takes real values on $\biggl(\{\Im\tau=0\}\cup \{\Re \tau=1\}\biggr)\cap\bD(1,\nu)$ (some $\nu>0$).
\item The derivative of the map  $g$ on a neighborhood of 1 satisfies
$\pa g(\tau)=(1-\tau)\pa\hat g(\tau-(1/2)\tau^2)$.
\end{itemize}
\end{theo}

\subsection{On $\diag(1,e^{2\pi i/3})$-symmetry}
A first observation is that 
\be (\diag(1,e^{2\pi i /3}))_{*}X_{\tau}=X_{\tau}.\label{2.7}\ee
As we  mentioned, numerical experiments suggest that for some values of the parameters and the initial conditions, $X_{\tau}$ has periodic orbits that possess  some $(\diag(1,e^{2\pi i /3}))$-symmetry; see the Figures \ref{fig:1}, \ref{fig:2}. This is {\it a priori} surprising because this is not at all implied by  
the commutation relation (\ref{2.7}).

\medskip
This symmetry becomes less mysterious if one looks for (analytic) periodic solutions of
\be \begin{cases}
&\frac{1}{2\pi i}\dot z=  (1-\tau)z+(1/2)z^2-(1/3)w^3\\
&\frac{1}{2\pi i }\dot w= \tau w-zw
\end{cases}
\label{diffeqXnothat}
\ee
 of the form
 \be 
\begin{aligned}
&z(t)=\sum_{k\in\Z}z_{3k}e^{3ki (2\pi g) t}\\
&w(t)=\sum_{k\in\Z}w_{3k+1}e^{ (3k+1)i (2\pi g) t}\\
&g\in \C.
\end{aligned}
\label{formofzw}
\ee
These functions   are automatically  $(\diag(1,e^{2\pi i /3}))$-symmetric.
Note that if $(z(\cdot),w(\cdot))$ is a real  analytic  solution, the same is true for $(z(\cdot+t_{0}),w(\cdot+t_{0}))$, $t_{0}\in\C$, $\Im t_{0}$ small enough (this just reflects the fact that  the flow of $X$ admits then an invariant complex annulus). As a consequence, if $(z_{3k})_{k\in\Z}$, $(w_{3k+1})_{k\in\Z}$ satisfy (\ref{formofzw}) the same is true for $(z_{3k} e^{-3k(2\pi g)s})_{k\in\Z}$, $(w_{3k+1} e^{-(3k+1)(2\pi g)s})_{k\in\Z}$ for any  $s\in (-s_{0},s_{0})$ ($s_{0}$ small enough).

\medskip
The differential equation (\ref{diffeqXnothat}) is then equivalent to the  system 
\be 
\left\{
\begin{aligned}
&0=(-(3k)g+1-\tau)z_{3k}+(1/2)\sum_{\substack{l_{1}+l_{2}=k\\ (l_{1},l_{2})\in \Z^2}}z_{3l_{1}}z_{3l_{2}}\\
&\hskip 6cm -(1/3)\sum_{\substack{l_{1}+l_{2}+l_{3}=k-1\\ (l_{1},l_{2},l_{3})\in \Z^3}}w_{3l_{1}+1}w_{3l_{2}+1}w_{3l_{3}+1}\\
&0=(-(3k+1)g+\tau)w_{3k+1}-\sum_{\substack{l_{1}+l_{2}=k\\ (l_{1},l_{2})\in \Z\times\Z}}z_{3l_{1}}w_{3l_{2}+1}.
\end{aligned}
\label{setofeqnnohat}
\right.
\ee
If instead, we work with the vector field $\hat X_{\hat \tau}$ obtained by substituting $z(\cdot)$ in place of $z(\cdot)-\tau$, we get the system of ODE
\be \begin{cases}
&\frac{1}{2\pi i }\dot z=  \hat \tau+z+(1/2)z^2-(1/3)w^3\\
&\frac{1}{2\pi i }\dot w= -zw
\end{cases}
\label{diffeqX}
\ee
which is equivalent to the  system
\be 
\left\{
\begin{aligned}
&0=\hat z_{0}+(1/2)\hat z_{0}^2+\hat \tau+(1/2)\sum_{\substack{l_{1}+l_{2}=0\\ (l_{1},l_{2})\in (\Z^*)^2}}z_{3l_{1}}z_{3l_{2}}\\
&\hskip 6cm -(1/3)\sum_{\substack{l_{1}+l_{2}+l_{3}=-1\\ (l_{1},l_{2},l_{3})\in \Z^3}}w_{3l_{1}+1}w_{3l_{2}+1}w_{3l_{3}+1}\\
&0=(-(3k)g+1+\hat z_{0})z_{3k}+(1/2)\sum_{\substack{l_{1}+l_{2}=k\\ (l_{1},l_{2})\in (\Z^*)^2}}z_{3l_{1}}z_{3l_{2}}\\
&\hskip 6cm -(1/3)\sum_{\substack{l_{1}+l_{2}+l_{3}=k-1\\ (l_{1},l_{2},l_{3})\in \Z^3}}w_{3l_{1}+1}w_{3l_{2}+1}w_{3l_{3}+1}\\
&0=(-(3k+1)g-\hat z_{0})w_{3k+1}-\sum_{\substack{l_{1}+l_{2}=k\\ (l_{1},l_{2})\in \Z^*\times\Z}}z_{3l_{1}}w_{3l_{2}+1}
\end{aligned}
\label{setofeqn}
\right.
\ee
where 
$$\hat z_{0}=z_{0}-\tau.$$

Define $\cE$ the vector space of sequences  
$$\cE:=\{(\xi_{k})_{k\in\Z}=(z_{3k},w_{3k+1})_{k\in\Z}\}$$ that we endow for example with the $l^1$-norm
$$\|(\xi_{k})_{k\in\Z}\|_{l^1}=\sum_{k\in\Z}(|z_{3k}|+|w_{3k+1}|)$$
 and $\cF$ the map from $\C\times\cE$ to $\cE$ that associates to each $(g, (z_{3k})_{k\in\Z},(w_{3k+1})_{k\in\Z})$ the sequence in $\cE$ defined by the right hand side of (\ref{setofeqnnohat}) or  (\ref{setofeqn}).

We shall prove that, for $\tau$ (resp. $\hat \tau$) and $w_{1}$ conveniently chosen, one can find $g$,
$(z_{3k})_{k\in\Z}$ and $(w_{3k+1})_{k\in\Z^*}$ such that 
$$\cF(g,(z_{3k})_{k\in\Z},(w_{3k+1})_{k\in\Z})=0.$$
Numerics show that a good choice for $w_{1}$ is 
$$w_{1}=1.4.$$

\subsection{Finding a $\diag(1,e^{2\pi i /3})$-symmetric  approximate solution}Let $N$ be a positive  integer and project the system   (\ref{setofeqn}) on the finite dimensional space $\cE_{N}$ of sequences $\{(z_{3k})_{|k|\leq N},(w_{3k+1})_{|k|\leq N}\}$. We shall denote by $\cP_{N}$ this finite rank projection. 
We thus get an algebraic  map
$$\cF_{N}: \C\times \cE_{N}\ni (g,(z,w))\mapsto (\ti z,\ti w) \in \cE_{N}$$
(we replace $0$ on the l.h.s. in (\ref{setofeqn})  by $\ti z_{3k}, \ti w_{3k+1}$).
We can also fix the value of $w_{1}$ and consider the map 
$$\mathring{\cF}_{N,w_{1}}:\C\times \mathring{\cE}_{N}\ni (g, (z,\mathring{w}))\mapsto (\ti z,\ti w) \in \cE_{N}$$
where  $\mathring{\cE}_{N}$ is the set  of sequences $\{(z_{3k})_{|k|\leq N},(\mathring{w}_{3k+1})_{0<|k|\leq N}\}$  
and 
$$\mathring{\cF}_{N,w_{1}}(z,\mathring{w})=\cF_{N}(z,w)$$
where $w_{3k+1}=\mathring{w}_{3k+1}$ if $k\ne 0$ and $w_{3\times 0+1}=w_{1}$.

We shall find numerically, for $N=12$ for example, 
 a solution to the equation
$$\cF_{N,1.4}(g,z,\mathring{w})=0.$$
This means that we shall find numerical 
values 
$${\hat z}^\approx=({\hat z}^\approx_{3k})_{|k|\leq N},\quad {w}^\approx=({w}^\approx_{3k+1})_{0<|k|\leq N},\quad g_\approx$$
such that, fixing $w_{1}=1.4,$ one gets
\be \hat \cF_{N,1.4}(g_{\approx},{\hat z}^\approx,{w}^\approx)=\e\label{approxsolnum}\ee
with $\e$ small say 
\be \|\e\|_{l^1}\leq \e_{0}=10^{-7}.\ee

It turns out that when $\tau=1$ (or equivalently $\hat \tau=1/2$)  the so-found coefficients $({\hat z}^\approx_{3k})_{0\leq ||\leq N}$, $({w}^\approx_{3k+1})_{0<|k|\leq N}$  are {\it real numbers} to a very good approximation (their imaginary parts are very small), as well as $g_\approx$,  and decay exponentially fast with $|k|$, $k\in [-N,N]\cap\Z$.
As a consequence,  the $1/g_{\approx}$-periodic functions
 \be
\begin{aligned}
&\hat z_{\approx}(t)=\sum_{|k|\leq N}{\hat z}_{3k}^\approx e^{3ki (2\pi g_{\approx}) t}\\
&w_{\approx}(t)=1.4\times e^{i(2\pi g_{\approx}) t}+\sum_{0<|k|\leq N}{w}^\approx_{3k+1}e^{ (3k+1)i(2\pi g_{\approx}) t}
\end{aligned}
\ee
provide  an approximate solution up to an error of $10^{-7}$ to  the system (\ref{setofeqn}):
\be \biggl\| (I-\cP_{N})(\cF_{N,w_{1}}(g_{\approx},({\hat z}^\approx,w^\approx)))\biggr\|_{l^1(\Z)}\leq \e_{1}=10^{-7}.\label{eq:errorestimate}\ee

More specifically, we have:
\begin{prop}[Numerics] \label{eq:numerics1}Let $\hat \tau=1/2$, $w_{1}=1.4$ and $N=12$. There exist a {\rm real} number $g_{\approx}$ and sequences of {\rm real} numbers  
\be {\hat z}^\approx=({\hat z}^\approx_{3k})_{|k|\leq N},\quad {w}^\approx=({w}^\approx_{3k+1})_{0<|k|\leq N}\label{eq:sequences}\ee
satisfying 
\be \biggl\| (I-\cP_{N})(\cF_{N,w_{1}}(g_{\approx},({\hat z}^\approx,w^\approx)))\biggr\|_{l^1(\Z)}\leq \e_{1}=10^{-7},\label{errorapprox}\ee
such that the $1/g_{\approx}$-periodic functions ${\hat z}_{\approx}$ and $w_{\approx}$ defined by 
\be
\begin{aligned}
&{\hat z}_{\approx}(t)=\sum_{|k|\leq N}{\hat z}_{3k}^\approx e^{3ki (2\pi g_{\approx}) t}\\
&w_{\approx}(t)=1.4\times e^{i(2\pi g_{\approx}) t}+\sum_{0<|k|\leq N}{w}^\approx_{3k+1}e^{ (3k+1)i(2\pi g_{\approx}) t}
\end{aligned}
\label{formofzwnew}
\ee
satisfy the ODE
\be \begin{cases}
&\frac{1}{2\pi i }\dot {\hat z}_{\approx}=  (1/2)+{\hat z}_{\approx}+(1/2){\hat z}_{\approx}^2-(1/3)w_{\approx}^3+\e_{z}(t)\\
&\frac{1}{2\pi i }\dot w_{\approx}= -{\hat z}_{\approx}w_{\approx}+\e_{w}(t)
\end{cases}
\label{diffeqXapprox}
\ee
where $\e_{z}$ and $\e_{w}$ are $(1/g_{\approx})$-periodic functions satisfying 
$$\sup_{\R}\max(|\e_{z}(\cdot)|,|\e_{w}(\cdot)|)\leq \e_{\approx,1}:=10^{-7}.$$
Besides,
\be |g_\approx|=-0.835\pm 10^{-3},\qquad \sup_{\R }|\hat z_{\approx}|\leq 2.5,\qquad \sup_{\R}|w_{\approx}|\leq 2.4.\label{estgzw}\ee
and 
$$\hat X_{1/2}({\hat z}_{\approx}(0),w_{\approx}(0))=(2\pi i)\times \bm -3.728971421315655\\ -0.26938912797026227\em .$$

\end{prop}
\begin{proof}
See the subsection \ref{subsec:numerics} dedicated to numerics.
\end{proof}
Equivalently, let
\be2\pi  i\e(t)=\dot {\hat p}_{\approx}(t)-\hat X_{\hat\tau}(\hat p_{\approx}(t))\label{14.149}\ee
and recall that if $p=(z,w)$, $$\hat X_{\hat \tau}(p)=2\pi
 i\bm \hat \tau+z+(1/2)z^2-(1/3)w^3\\
 -zw
\em.
$$

\medskip Let
$$T_{\approx}=1/g_{\approx}.$$
\begin{prop}\label{prop:10.2bis}The function $\hat p_{\approx}=(\hat z_{\approx},w_{\approx})$ is a $T_{\approx}$-periodic solution with
 of the time-$T_{\approx}$-periodic vector field $X_{\approx}$ defined by
\be X_{\approx}(t,p)=\hat X_{1/2}(p)+2\pi i\e(t)\label{Xapprox} \ee
where  $\e(\cdot)$ is $T_{\approx}$-periodic and 
$$\|\e\|_{C^0(\R)}\leq \e_{\approx,1}=10^{-7}.$$
\end{prop}

\subsection*{Newton method}
From a numerical point of view the solution $$\xi^\approx=(g_{\approx},(z^\approx_{3k})_{|k|\leq N},(\mathring{w}^\approx_{3k+1})_{0<|k|\leq N})$$ of (\ref{approxsolnum}) is obtained by a simple Newton method. 
\begin{itemize}
\item One needs a {\it first guess} $\xi_{0}=\xi_{init}$;
\item then one defines the sequence 
$$\xi_{n+1}=\xi_{n}-D\cF_{N,1.4}(p_{n})^{-1}\cF_{N,1.4}(\xi_{n}).$$
\end{itemize}
Five iterations of the Newton method often provide good results and we can set $\xi^\approx=\xi_{5}$.

\subsection*{Finding the first guess}
To find the first guess $\xi_{init}$ one observes that 
approximate solutions of the form
\begin{align*}
&z(t)=\hat z_{0}+z_{-3}e^{-3i (2\pi g) t}\\
&w(t)=w_{1}e^{i\omega t}+w_{-2}e^{ -2i(2\pi g) t}
\end{align*}
already provide periodic orbits with shapes that are similar to the observed periodic orbits of $X$ (the  $w$-projections $t\mapsto w(t)$  of the observed solutions are often ``deltoid''  or ``trefoil'' like, see Figures \ref{fig:2} and \ref{fig:3} for example). Computations can be carried out explicitly in this case: the system becomes
 \be 
 \begin{aligned}
&\hat z_{0}+(1/2)\hat z_{0}^2+ \hat \tau-w_{1}^2w_{-2}=0\\
&(3g+1+\hat z_{0}) z_{-3}-w_{-2}^2w_{1}=0\\
&(-g-\hat z_{0})w_1-z_{3}w_{-2}=0\\
&(2g-\hat z_{0}) w_{-2}- z_{-3}w_{1}=0
\end{aligned}
\label{eq:hatI=-1,0n}
\ee
which admits the solution ($w_{1}$ being fixed)
\be
\left\{
\begin{aligned}
&g=g^*_{\pm}(\tau):=\frac{-4\pm\sqrt{16+ 11(2\tau-\tau^2)}}{11}=\frac{-4\pm\sqrt{16+ 22\times \hat \tau}}{11}\\
&\hat z_{0}=-g \\
&w_{-2}=\frac{3g (2g+1) }{w_{1}^2} \\
&z_{-3}(0,0)= \frac{9g^2(2g+1) }{  w_{1}^3}.
\end{aligned}
\right.
\label{2.14}
\ee
Looking for solutions of the form  (this is the case $N=1$)
\begin{align*}
&z(t)=z_{3}e^{3i (2\pi g)t}+\hat z_{0}+z_{-3}e^{-3i (2\pi g) t}\\
&w(t)=w_{4}e^{4i(2\pi g)t}+w_{1}e^{i\omega t}+w_{-2}e^{ -2i(2\pi g) t}
\end{align*}
yields the system
$$ \begin{aligned}
&(3g+1+ \hat z_{0})z_{-3}-w_{-2}^2w_{1}=0\\
&\hat z_{0}+(1/2)\hat z_{0}^2+ \hat \tau+z_{3}z_{-3}-w_{1}^2 w_{-2}- w_{4}w_{-2}^2=0\\
&(-3g+1+\hat z_{0})z_{3}-2 w_{4} w_{-2}w_{1}-(1/3)w_{1}^3=0\\
&(2g-\hat z_{0})\hat w_{-2}-\hat z_{-3}w_{1}=0\\
&(-g-\hat z_{0})w_1-z_{-3}w_{4}- z_{3}w_{-2}=0\\
&(-4g-\hat z_{0})w_{4}-z_{3}w_{1}=0
\end{aligned}
$$
which admits the solution ($w_{1}$ is fixed)
\be
\left\{
\begin{aligned}
&g\ \textrm{such\ that}\ (1/3)(5g-8)g(2g+1)-g+(1/2)g^2+\hat \tau=0\\
&\hat z_{0}=-g \\
&w_{-2}=\frac{3g (2g+1) }{w_{1}^2} \\
&z_{-3}= \frac{9g^2(2g+1) }{  w_{1}^3}\\
& z_{3}=\frac{w_{1}^3}{9}\\
&w_{4}=-\frac{w_{1}^4}{27 g}.
\end{aligned}
\right.
\label{2.14bisn}
\ee

Note that equation (\ref{2.14}) indicates that there are at least two first guesses for $\xi_{init}$ depending on whether we choose
$$g_{init}(\hat \tau)=\frac{-4+\sqrt{16+ 22\times \hat \tau}}{11}\qquad \textrm{or}\qquad \frac{-4-\sqrt{16+ 22\times \hat \tau}}{11}$$
In what follows we made the second  choice with $\hat \tau=1/2$:
$$g_{init}(1/2)= \frac{-4-\sqrt{16+ 11}}{11}\approx -0.836.$$

\subsection{Proof of Theorem \ref{theo:expero} }\label{sec:15.1}
In this Section we show how  Proposition \ref{eq:numerics1} can be used to prove Theorem \ref{theo:expero}.
\subsubsection{From approximate periodic  solutions to genuine periodic solutions}
From Proposition \ref{prop:10.2bis} we know that 
\be \frac{d\hat p_{\approx}(t)}{dt}=\hat X_{\approx}(t,\hat p_{\approx}(t))\label{15.176}\ee
where $X_{\approx}$ is the {\it time dependent}, $T_{\approx}$-periodic in time, vector field
\be X_{\approx}(t,p)=\hat X_{1/2}(p)+2\pi i\e(t)\label{Xapprox} .\ee
Our goal is to find a $T'$-periodic function $t\mapsto \hat p_{\hat\tau}(t)$, which will be close to $\hat p_{\approx}(t)$ and  which satisfies 
\be \frac{d\hat p_{\hat\tau}(t)}{dt}=\hat X_{\hat\tau}(\hat p_{\hat\tau}(t)).\label{10.32n}\ee

We first describe how one can get orbits of the vector field $\hat X_{\hat\tau}$ close to the approximate solution $\hat p_{\approx}$. In a second time, we shall prove the existence  of {\it periodic} orbits for the vector field $\hat X_{\hat\tau}$.

\subsubsection{Linearization along $\hat p_{\approx}$}
\medskip Recall 
$$ X_{\tau}(z,w)=X_{0,\tau}(z,w)=2\pi  i \bm (1-\tau)z+(1/2) z^2-(1/3)w^3\\ \tau w-zw\em,
$$
$$\hat X_{\hat \tau}(p)=
2\pi  i\bm \hat \tau+z+(1/2)z^2-(1/3)w^3\\
 -zw
\em
$$
and define
\be F(z,w)\cdot (u,v)^{\otimes_{\geq 2}} =2\pi i\bm (1/2)u^2-wv^2-(1/3)v^3\\ -uv\em\label{defF}\ee
so that 
$$\hat X_{1/2+\Delta\hat \tau}(p+\Delta p)=\hat X_{1/2}(p)+D\hat X_{1/2}(p)\cdot\Delta p+F(p)\cdot (\Delta p)^{\otimes_{\geq 2} }+2\pi i \Delta \hat\tau\bm 1\\0\em .$$
Note that if $\max(|u_{1}|,|u_{2}|)\leq \rho$, the module of the  coefficients of 
$F(z,w)\cdot (u_{1},v_{1})^{\otimes_{\geq 2}}-F(z,w)\cdot (u_{2},v_{2})^{\otimes_{\geq 2}}$ are 
$$\leq 2\pi\max( \rho |u_{1}-u_{2}|+2\rho|w| |v_{1}-v_{2}|+\rho^2|v_{1}-v_{2}|,\rho(|u_{1}-u_{2}|+|v_{1}-v_{2}|))
$$
hence from (\ref{estgzw})
\begin{multline} \|F(z,w)\cdot (u_{1},v_{1})^{\otimes_{\geq 2}}-F(z,w)\cdot (u_{2},v_{2})^{\otimes_{\geq 2}}\|\leq \\ 4\pi \rho(2\|w_{\approx}\|_{C^0(I)}+1+\rho)\biggl\|\bm u_{1}\\ v_{1} \em-\bm u_{2}\\ v_{2}\em\biggr\|\\ \biggl\|\bm u_{1}\\ v_{1} \em-\bm u_{2}\\ v_{2}\em\biggr\|\leq 76\times \rho\times \biggl\|\bm u_{1}\\ v_{1} \em-\bm u_{2}\\ v_{2}\em\biggr\| \label{estlipF}
\end{multline}
if $\max\|(u_{1},v_{1}),(u_{2},v_{2})\|\leq \rho$ and $\rho\leq 10^{-3}$.

\bigskip
Writing 
\be \hat p_{\hat \tau}(t)=\hat p_{\approx}(t)+p_{\hat\tau,cor}(t)\label{10.32n-1}\ee
equation (\ref{10.32n}) is equivalent to 
\begin{align}
\frac{dp_{\hat\tau,cor}(t)}{dt}&=\hat X_{\hat\tau}(\hat p_{\approx}(t)+p_{cor}(t))-\hat X_{\approx}(t,\hat p_{\approx}(t))\notag\\
&=\hat X_{1/2}(\hat p_{\approx}(t)+p_{\hat\tau,cor}(t))-\hat X_{1/2}(\hat p_{\approx}(t))-2\pi i\e(t)+2\pi i\Delta\hat\tau\bm 1\\ 0\em\notag\\
& =D\hat X_{1/2}(\hat p_{\approx}(t))\cdot p_{\hat\tau,cor}(t)+
 F(\hat p_{\approx}(t))\cdot p_{\hat\tau,cor}^{\otimes_{\geq 2}} (t) -2\pi i\e(t)+2\pi i\Delta\hat\tau\bm 1\\ 0\em\label{eq:10.32bis}
\end{align}
with $\Delta\hat\tau=\hat\tau-1/2$.

Besides, $\hat p_{\approx}+p_{cor}$ is $T'$-periodic for some $T'$ close to  $T_{\approx}$, provided one has  the additional condition 
$$p_{cor}(T')-p_{cor}(0)=\hat p_{\approx}(0)-\hat p_{\approx}(T').$$

Denote by $A_{\approx}:\R\to M(2,\C)$ the time-$T_{\approx}$ function defined by
\begin{align}A_{\approx}(t)&=D\hat X_{\hat \tau}(\hat p_{\approx}(t))\\
&=2\pi i\bm 1+\hat z_{\approx}(t) & -w_{\approx}(t)^2\\ -w_{\approx}(t)&-\hat z_{\approx}(t)\em.\label{eq:10.35bis}\end{align}
One has also
\begin{align}A_{\approx}(t)&=DX_{ \tau}(p_{\approx}(t))\label{14.158}\\
&=2\pi i\bm (1-\tau)+z_{\approx}(t) & -w_{\approx}(t)^2\\ -w_{\approx}(t)&\tau -z_{\approx}(t)\em.\label{eq:10.35bisbis}\end{align}
Equation (\ref{eq:10.32bis}) is then equivalent to 
\be \frac{dp_{\hat\tau,cor}(t)}{dt}=A_{\approx}(t)p_{\hat\tau,cor}(t)+F(\hat p_{\approx}(t))\cdot p_{\hat\tau,cor}^{\otimes_{\geq 2}} (t)-2\pi i\e(t)+2\pi i\Delta\hat\tau\bm 1\\ 0\em.\label{eq:pcor}\ee

\subsubsection{The resolvent $R_{A_{\approx}}$}
Denote by $R_{A_{\approx}}(t,s)$ the resolvent of the  the $T_{\approx}$-periodic   linear ODE
\be \dot Y(t)=A_{\approx}(t)Y(t).\label{eq:linode}\ee
By definition $R_{A_{\approx}}(t,s)$ is the unique linear map satisfying for all solution of (\ref{eq:linode}) the relation $Y(t)=R_{A_{\approx}}(t,s)Y(s).$  It satisfies the ODE
\be \frac{dR_{A_{\approx}}}{dt}(t,t_{0})=A_{\approx}(t)R_{A_{\approx}}(t,t_{0}),\qquad R_{A_{\approx}}(t_{0},t_{0})=I.\label{15.187}\ee
Besides, Chasles' relation $R_{A_{\approx}}(t_{2},t_{0})=R_{A_{\approx}}(t_{2},t_{1})R_{A_{\approx}}(t_{1},t_{0})$ is satisfied, and because $A_{\approx}$ is $T_{\approx}$-periodic one has
\be R_{A_{\approx}}(t_{1}+T_{\approx},t_{0}+T_{\approx})=R_{A_{\approx}}(t_{1},t_{0}).\label{periodicityresolvent}\ee

Let $I$ be an interval of $\R$ containing 0 and 
$\cK_{I}$ 
be the map 
\begin{multline}\cK_{I}:\C^2\times C^0(I,\C^2)\ni (y,b)\mapsto \\ \biggl( I\ni t\mapsto R_{A_{\approx}}(t,0)y+\int_{0}^tR_{A_{\approx}}(t,s)b(s)ds\in \C^2\biggr) \in C^1(I,\C^2). \label{defKI}
\end{multline}
The method of variation of  constants tells us that 
$\cK_{I}(y,b)$
is the solution of the affine ODE
\be\begin{cases}&\dot Y(t)=A_{\approx}(t)Y(t)+b(t)\\
&Y(0)=y.
\end{cases}
\label{defKI-bis}
\ee

\bigskip
The previous discussion remains valid if we consider the variable $t$ as a complex time in the complex domain $I_{\nu}:=I+i(-\nu,\nu)$ where $\nu>0$ is small. Formulae (\ref{15.187}), (\ref{periodicityresolvent}) make sense as well as  (\ref{defKI}), (\ref{defKI-bis}) provided we consider 
\begin{multline}\cK_{I_{\nu}}:\C^2\times \cO(I_{\nu},\C^2)\ni (y,b)\mapsto \\ \biggl( I_{\nu}\ni t\mapsto R_{A_{\approx}}(t,0)y+\int_{0}^tR_{A_{\approx}}(t,s)b(s)ds\in \C^2\biggr) \in \cO(I_{\nu},\C^2). \label{defKIcomplex}
\end{multline}

Let 
\begin{align*}&\Psi_{}(p_{}(\cdot))=\int_{0}^\cdot R_{A_{\approx}}(\cdot,s)\biggl(F(\hat p_{\approx}(s))\cdot p^{\otimes_{\geq 2}}(s)\biggr)ds\\
&\e_{\approx}(\cdot)=-2\pi i\int_{0}^\cdot R_{A_{\approx}}(\cdot,s)\e(s)ds\\
&\mu_{\approx}(\cdot)=2\pi i\int_{0}^\cdot R_{A_{\approx}}(\cdot,s)\bm 1\\0\em ds.
\end{align*}

\begin{lemma}\label{lemma:CPvsFPante} The Cauchy problem
$$
\left\{
\begin{aligned}
&\frac{d(\hat p_{\approx}+p_{\hat\tau, cor})}{dt}=\hat X_{\hat\tau}\circ \biggl((\hat p_{\approx}+p_{\hat\tau,cor})\biggr)\\
&(\hat p_{\approx}+p_{\hat\tau,cor})(0)=\hat p_{\approx}(0)+y
\end{aligned}
\right.
$$
is equivalent to the fixed point problem
\be p_{\hat\tau,cor}(\cdot)=\Psi_{} (p_{\hat\tau,cor}(\cdot))+R_{A_{\approx}}(\cdot,0)y+(\hat \tau-1/2)\mu_{\approx}(\cdot)+\e_{\approx}(\cdot).\label{fixedptpb}\ee
\end{lemma}

\subsubsection{Floquet decomposition}\label{subsec:15.5.4}
Because the linear map $A_{\approx}$ is $T_{\approx}$-periodic the resolvent $R_{A_{\approx}}$ admits a {\it Floquet decomposition}:
\be R_{A_{\approx}}(t,s)=P_{\approx}(t)e^{(t-s)M_{\approx}}P_{\approx}(s)^{-1}\label{n15.192}\ee
where 
\begin{itemize}
\item $M_{\approx}\in M(2,\C)$ is a matrix such that $e^{T_{\approx}M_{\approx}}=R_{A_{\approx}}(T_{\approx},0)$.
\item $\R\ni t\mapsto P_{\approx}(t)\in GL(2,\C)$ is $T_{\approx}$-periodic\footnote{This is a consequence of (\ref{periodicityresolvent}).} and  can be chosen equal to the $T_{\approx}$-periodic map  $t\mapsto e^{-tM_{\approx}}R_{A_{\approx}}(t,0)P(0)$ where $P(0)$ can be chosen arbitrarily in $GL(2,\C)$.
\item The function $P_{\approx}$ satisfies the equation
\be P_{\approx}^{-1}\frac{dP_{\approx}}{dt}=P_{\approx}^{-1}A_{\approx}P_{\approx}-M_{\approx}.\label{ODEP}\ee
\end{itemize}
Besides, since ${\rm tr} A_{\approx}(t)=2\pi i$, one has 
\be \det R_{A_{\approx}}(t,s)=e^{2\pi i (t-s)}.\ee

We shall see that  1 is almost an eigenvalue of $R_{A_{\approx}}(T_{\approx},0)$
because $A_{\approx}(\cdot)=D\hat X_{1/2}(p_{\approx}(\cdot))$ is  the linearization along the $T_{\approx}$-periodic  solution  $p_{\approx}$ which is almost an orbit of the autonomous ODE $\dot p=\hat X_{1/2}(p)$. As a consequence, the eigenvalues of $R_{A_{\approx}}(T_{\approx},0)$ are almost equal to 1 and $e^{2\pi i\b}$. We can thus choose $M_{\approx}$ to be conjugate to a diagonal matrix $M_{\approx}=S\diag(\l_{\approx,1},\l_{\approx,2})S^{-1}$ with 
$$\l_{\approx,1}\approx 0,\qquad \l_{\approx,2}\approx 2\pi i (1-g_{\approx})\approx 2\pi i\times 1.8345$$ and the relation (\ref{n15.192}) becomes
\be R_{A_{\approx}}(t,s)=P_{\approx}(t)e^{(t-s)M_{\approx}}P_{\approx}(s)^{-1}\label{nn15.192}\ee
with
\begin{itemize}
\item  $ M_{\approx}=\diag(\l_{\approx,1},\l_{\approx,2})$
\item $P_{\approx}(t)S$ in place of $P_{\approx}(t)$.
\end{itemize}

\subsubsection{Gauge transformation}
We now set in (\ref{eq:pcor})
$$q_{\hat\tau}(t)=P_{\approx}(t)^{-1}p_{\hat\tau,cor}(t)$$
which satisfies because of (\ref{ODEP})
\begin{multline} \frac{dq_{\hat\tau}(t)}{dt}=M_{\approx}q_{\hat\tau}(t)+P_{\approx}(t)^{-1}F(p_{\approx}(t))\cdot (P_{\approx}(t)q_{\hat\tau}(t))^{\otimes_{\geq 2}})\\ -2\pi iP_{\approx}(t)^{-1}\e(t)+2\pi i\Delta\hat\tau P_{\approx}(t)^{-1}\bm 1\\ 0\em.\label{diffeqq}\end{multline}
We define
\begin{align}&\hat\Psi_{}(q_{}(\cdot))=\int_{0}^\cdot e^{(\cdot-s)M_{\approx}}\biggl(P_{\approx}(s)^{-1}F(p_{\approx}(s))\cdot (P_{\approx}(s)q_{}(s))^{\otimes_{\geq 2}})\biggr)ds\notag\\
&\hat \e_{\approx}(\cdot)=-2\pi i\int_{0}^\cdot e^{(\cdot-s)M_{\approx}}P_{\approx}(s)^{-1}\e(s)ds\notag\\
&\hat \mu_{\approx}(\cdot)=2\pi i\int_{0}^\cdot e^{(\cdot-s)M_{\approx}}P_{\approx}(s)^{-1}\bm 1\\0\em ds.\label{def:hatmuapprox}
\end{align}
The fixed point problem (\ref{fixedptpb}) is then equivalent to 
\be q_{\hat\tau}(\cdot)=\hat \Psi_{} (q_{\hat\tau}(\cdot))+e^{(\cdot-0)M_{\approx}}y+(\hat \tau-1/2)\hat \mu_{\approx}(\cdot)+\hat \e_{\approx}(\cdot).\label{fixedptpbbis}\ee

To summarize:
\begin{lemma}\label{lemma:CPvsFP} The Cauchy problem 
$$
\left\{
\begin{aligned}
&\frac{d(\hat p_{\approx}+p_{\hat\tau, cor})}{dt}=\hat X_{\hat\tau}\circ \biggl((\hat p_{\approx}+p_{\hat\tau,cor})\biggr)\\
&(\hat p_{\approx}+p_{\hat\tau,cor})(0)=\hat p_{\approx}(0)+P_{\approx}(0)y
\end{aligned}
\right.
$$
is equivalent to the fixed point problem  \be q_{\hat\tau}(\cdot)=\hat \Psi_{} (q_{\hat\tau}(\cdot))+e^{(\cdot-0)M_{\approx}}y+(\hat \tau-1/2)\hat \mu_{\approx}(\cdot)+\hat \e_{\approx}(\cdot).\label{fixedptpbbis}\ee
with 
$$q_{\hat \tau}(t)=P_{\approx}(t)^{-1}p_{\hat\tau,cor}(t).$$
\end{lemma}

\subsubsection{Numerical values}
Let 
\be
\begin{cases}
&I=[0,1/g_{\approx}]\\
&\nu=10^{-2}\\
&I_{\nu}=I+i(-\nu,\nu).
\end{cases}
\label{eq:defintervalI}
\ee
We shall need the following numerical values.
\be
\begin{cases}& |g_\approx|=-0.835\pm 10^{-3},\\
& \sup_{\R }|\hat z_{\approx}|\leq 2.5,\qquad \sup_{\R}|w_{\approx}|\leq 2.4.\end{cases}\label{estgzwnewbis}\ee
(see Proposition \ref{eq:numerics1})
\be\begin{cases}
&M_{\approx}=2\pi i \diag(\l_{\approx,1},\l_{\approx,2})\\
&|\l_{\approx,1}|\leq 10^{-5},\qquad |\l_{\approx,2}-(1-g_{\approx})|\leq 10^{-5}
\end{cases}
\ee
(see Proposition \ref{prop:controlresolventRApprox})
\be \begin{cases}
&2\pi \times  |I|\times \max_{t,s\in I_{\nu}}\|e^{(t-s)M_{\approx}}\|\leq 8.2\\
&\sup_{t\in\R_{\nu}}\|A_{\approx}(t)\| \leq  51
\end{cases}
\label{eq:15.308}
\ee
%%%%%%%%%%%%%%%%%%%%%
%%%%%%%%%%%%%%%%%%%%%
%%%%%%%%%%%%%%%%%%%%%
\begin{comm}
\be
\left\{
\begin{aligned}
&P_{\approx}(0)^{-1}=\bm 1.2357&  -4.0592\\
 -0.2513&  3.4657\em\pm 10^{-4},\\
 & P_{\approx}(0)=\bm 1.0624& 1.2460\\ 0.0767 & 0.3793\em \pm 10^{-4}\\
&\forall t\in\R,\quad |\det  P_{\approx}(t)|\geq 0.3\quad \textrm{and}\quad \|P_{\approx}(t)\|_{op}\leq 2.6
\end{aligned}
\right.
\label{eq:15.309}
\ee
\end{comm}
%%%%%%%%%%%%%%%%%%%%%
%%%%%%%%%%%%%%%%%%%%%
%%%%%%%%%%%%%%%%%%%%%
\be
\left\{
\begin{aligned}
&P_{\approx}(0)^{-1}=\bm 1.23&  -4.05\\
 -0.25&  3.46\em\pm 10^{-2},\\
 & P_{\approx}(0)=\bm 1.06& 1.24\\ 0.07 & 0.37\em \pm 10^{-2}\\
&\forall t\in\R,\quad |\det  P_{\approx}(t)|\geq 0.3\quad \textrm{and}\quad \|P_{\approx}(t)\|_{op}\leq 2.6
\end{aligned}
\right.
\label{eq:15.309}
\ee

(see Proposition \ref{prop:controlresolventRApprox}).

\subsubsection{Contraction mapping theorem}
Let $I\subset \R$ be the interval defined in (\ref{eq:defintervalI}) and let's  introduce on $\C^2\times C^0(I,\C^2)$ the norm
$$\|(y,b)\|=\max(\|y\|,\|b\|_{C^0(I,\C^2)}).$$
We define for $\rho>0$, $\nu>0$
$$\cB_{C^0(I_{},\C^2)}(0,\rho)=\{p\in C^0(I_{},\C^2)\mid \|p\|_{C^0(I_{},\C^2)}\leq \rho\}.$$
$$\cB_{\cO(I_{\nu},\C^2)}(0,\rho)=\{p\in \cO(I_{\nu},\C^2)\mid \|p\|_{\cO(I_{\nu},\C^2)}\leq \rho\}.$$
\begin{lemma}For $0<\rho<10^{-3}$, the map 
$$\hat\Psi:\cO(I_{\nu},\C^2)\to \cO(I_{\nu},\C^2)$$
  satisfies $\hat \Psi(0)=0$ and   is $\kappa$-Lipschitz on $\cB_{\cO(I_{\nu},\C^2)}(0,\rho_{})$ with
$$\begin{cases}&\kappa=C_{\hat\Psi}\rho\\
&C_{\hat \Psi}=6548.
\end{cases}$$

\end{lemma}
\begin{proof} If $q_{1},q_{2}\in \cB_{\cO(I_{\nu},\C^2)}(0,\rho)$ one has from (\ref{estlipF}), (\ref{estgzwnewbis}),  (\ref{eq:15.308}), (\ref{eq:15.309}) and the definition of $\hat\Psi$
\begin{multline*}\|\hat \Psi(q_{1})-\hat \Psi(q_{2})\|_{\cO(I_{\nu},\C^2)}\leq \rho\times 76 \times |I|\times \max_{t,s\in I_{\nu}}\|e^{(t-s)M_{\approx}}\| \times  \\ \sup_{t\in I_{\nu}}\|P_{\approx}\|^{-1}\|P_{\approx}\|^2 \times \|q_{1}-q_{2}\|_{\cO(I_{\nu})}\\
\leq \rho\times 76\times 1.3\times (1+10^{-2})\times 9\times 2.7^2\times \|q_{1}-q_{2}\|\\
\leq 6548\times \rho\times \|q_{1}-q_{2}\|.
\end{multline*}

\end{proof}

Let 
$$C_{1}= \max(2\pi \times  |I|\times \max_{t,s\in I_{\nu}}\|e^{(t-s)M_{\approx}}\|,8.2)=8.2.$$
$$C_{A_{\approx}}=\max(51,\sup_{t\in\R}\|A_{\approx}(t)\|)= 51.$$
We note that
\begin{align*}&\|\hat \e_{\approx}\|_{C^0(I)}\leq C_{1}\times \|\e_{}\|_{C^0(I)}\\
&\|\hat \mu_{\approx}\|_{C^0(I)}\leq C_{1}.
\end{align*}

\begin{cor}\label{cor:10.4} Let $\rho$ be such that $C_{\Psi}\rho\leq 1/3$ and assume that $C_{1}\|\e\|_{C^0(I)}\leq \rho/3$. Then, for any $\hat\tau\in \C$ such that $|\hat\tau-1/2|\leq  (5C_{1})^{-1}\rho$ and any $y\in \bD(0,\rho/3)$,
there exists a unique $q^y_{\hat\tau}\in \cB_{\cO(I_{\nu},\C^2)}(0,\rho)$ such that 
$$q^y_{\hat\tau}(\cdot)=\hat \Psi_{} (q^y_{\hat\tau}(\cdot))+e^{(\cdot-0)M_{\approx}}y+(\hat \tau-1/2)\hat \mu_{\approx}(\cdot)+\hat \e_{\approx}(\cdot).$$
Moreover,  he map 
\begin{multline*}\bD(0,\rho/3)\times \bD(1/2,(5C_{1})^{-1}\rho)\ni \\ (y,\hat\tau)\mapsto q^y_{\hat\tau}(\cdot)-e^{(\cdot-0)M_{\approx}}y-(\hat \tau-1/2)\hat \mu_{\approx}(\cdot)-\hat \e_{\approx}(\cdot)\\ \in \cO(I_{\nu},\C^2)
\end{multline*}
is $C_{\Psi}\rho/2$-Lipschitz.
\end{cor}
\begin{proof}
This is a consequence of the previous Lemma, of Lemma  \ref{lem:annex2} of the Appendix (on the classical Contraction mapping principle), of the fact that $\Psi(0)=0$ and of the inequality
$$\| e^{(\cdot-0)M_{\approx}}y+(\hat \tau-1/2)\hat \mu_{\approx}(\cdot)+\hat \e_{\approx}(\cdot)\|_{\cO(I_{\nu})}\leq 1.1 \|y\|+|\hat\tau-1/2|C_{1}+C_{1}\|\e\|_{0}.$$
\end{proof}

Let $m$ be the map
\begin{multline}m:\bD(0,\rho/3)\times \bD(1/2,(5C_{1})^{-1}\rho)\times I_{\nu}\ni \\ (y,\hat\tau,s)\mapsto q^y_{\hat\tau}(t)-e^{tM_{\approx}}y-(\hat \tau-1/2)\hat \mu_{\approx}(t)-\hat \e_{\approx}(t) \in \C^2.\label{thefunctionm}
\end{multline}
It is $C^1$-w.r.t. $t$ and for $t\in I$
\be m(0,1/2,t)=q^0_{1/2}(t)-\hat \e_{\approx}(t)=\hat \Psi(q^0_{1/2})(t)\in \bD(0,\rho).\label{c15.190}\ee
\be m(y,1/2,t)=q^y_{1/2}(t) -e^{tM_{\approx}}y-\hat \e_{\approx}(t)\in \bD(0,2\rho).\label{c15.190bis}\ee
\begin{lemma}\label{lemma:n15.8}If $4C_{1}C_{\hat \Psi}\|\e\|_{C^0(I)}\leq 1$, one has for all $t\in I$, $m(0,1/2,t)\in \bD(0,2C_{1}\|\e\|_{C^0(I)})$.
\end{lemma}
\begin{proof}We apply the previous Corollary with $\rho=\rho_{*}$ where
$$\rho_{*}=2C_{1}\|\e\|_{C^0(I)}.$$
\end{proof}

\begin{lemma} \label{lemma:15.8} Assume $4C_{1}C_{\hat \Psi}\|\e\|_{C^0(I)}\leq 1$. The map $m$ is $100\times \rho$-Lipschitz on $$D_{\rho}:=\bD(0,\rho/3)\times \bD(1/2,10^{-3}\rho)\times I_{\nu}$$ and  for $t\in I_{\nu}$ one has $m(0,1/2,t)\in \bD(0,2C_{1}\|\e\|_{C^0(I)})$. 
\end{lemma}
\begin{proof}We just have to check  that for $t\in I_{\nu}$ one has $\|\pa_{t}m(y,\hat\tau,t)\|\leq C_{\hat \Psi}\rho/2$.
We see that 
\be \pa_{t} \biggl( \int_{0}^te^{(t-s)M_{\approx}}g(s)ds\biggr)=M_{\approx}\int_{0}^te^{(t-s)M_{\approx}}g(s)ds+g(t)\label{a.15.189}\ee
hence using (\ref{diffeqq})
\begin{multline*} \pa_{t}m(y,\hat\tau,t)=M_{\approx}q^y_{\hat\tau}(t)+P_{\approx}(t)^{-1}F(\hat p_{\approx}(t))\cdot (P_{\approx}(t)q_{\hat\tau}(t))^{\otimes_{\geq 2}})\\ -2\pi iP_{\approx}(t)^{-1}\e(t)+2\pi i(\hat\tau-1/2) P_{\approx}(t)^{-1}\bm 1\\ 0\em\\
-M_{\approx} y\\
-2\pi i(\hat\tau-1/2)\biggl(M_{\approx}\int_{0}^te^{(t-s)M_{\approx}}P_{\approx}(s)^{-1}\bm 1\\ 0\em ds+P_{\approx}(t)^{-1}\bm 1\\ 0\em\biggr)
\\+2\pi i \biggl(M_{\approx}\int_{0}^te^{(t-s)M_{\approx}}P_{\approx}(s)^{-1}\e(s)ds+P_{\approx}(t)^{-1}\e(t)\biggr)\end{multline*}
or
\begin{multline*} \pa_{t}m(y,\hat\tau,t)=M_{\approx}q^y_{\hat\tau}(t)+P_{\approx}(t)^{-1}F(\hat p_{\approx}(t))\cdot (P_{\approx}(t)q^y_{\hat\tau}(t))^{\otimes_{\geq 2}}
-M_{\approx} y\\
-2\pi i(\hat\tau-1/2)\biggl(M_{\approx}\int_{0}^te^{(t-s)M_{\approx}}P_{\approx}(s)^{-1}\bm 1\\ 0\em ds\biggr)
\\+2\pi i \biggl(M_{\approx}\int_{0}^te^{(t-s)M_{\approx}}P_{\approx}(s)^{-1}\e(s)ds\biggr).\end{multline*}
As a consequence
\begin{multline*}\|\pa_{t}m(y,\hat\tau,t)\|\leq  \|M_{\approx}\|\times \|q^y_{\hat\tau}\|_{\cO(I_{\nu})}+\|P_{\approx}\|^{-1}\|P_{\approx}\|^2 \times \|q^y_{\hat\tau}\|_{\cO(I_{\nu})}\\
+\|M_{\approx}\|\times \| y\|+2\pi |\hat\tau-1/2|\|M_{\approx}\|  |I|\times \sup_{t,s\in I_{\nu}}\|e^{(t-s)M_{\approx}}\|\|P_{\approx}^{-1}\|_{C^0}\\+2\pi \|M_{\approx}\||I|\times \sup_{t,s\in I_{\nu}}\|e^{(t-s)M_{\approx}}\|\|P_{\approx}^{-1}\|_{C^0}\|\e\|_{C^0}
\end{multline*}
and since $q^y_{\hat\tau}\in \cB_{\cO(I_{\nu})}(0,\rho)$, $|\hat\tau-1/2|\leq \rho/(5C_{1})$ and $y\in \bD(0,\rho/3)$
we get (we use the fact that $\|M_{\approx}\|\leq 4\pi $)
\begin{align*}\|\pa_{t}m(y,\hat\tau,t)\|&\leq 98\rho+861|\hat\tau-1/2|+ 861\|\e\|_{C^0}\\
&\leq 100\rho.
\end{align*}
\end{proof}

\begin{rem}\label{rem:lemma}
The preceding condition on $\e$ is satisfied when 
$$\rho\leq 5\times 10^{-5}$$
and 
$$\|\e\|_{0}\leq 4\times 10^{-6}.$$
\end{rem}

\subsubsection{Existence of  periodic solutions}
Referring to Lemma \ref{lemma:CPvsFP}, let 
\begin{align}t\mapsto {p}_{\hat\tau}^y(t):&=\hat p_{\approx}(t)+p_{\hat\tau,cor}^y(t)
\end{align}
be  the unique  solution of the Cauchy problem
\be
\left\{
\begin{aligned}
 &\frac{dp^y_{\hat\tau}(t)}{dt}=\hat X_{\hat \tau}(p^y_{\hat\tau}(t))\\
 &p_{\hat\tau}^y(0)=\hat p_{\approx}(0)+P_{\approx}(0)y.
 \end{aligned}
 \right.
 \label{eq:10.38nbis}
 \ee
 Lemma \ref{lemma:CPvsFP}, Corollary \ref{cor:10.4}-Lemma \ref{lemma:15.8} tell us that
 \be p_{\hat\tau}^y(t)=\hat p_{\approx}(t)+P_{\approx}(t)\biggl(e^{tM_{\approx}}y+(\hat\tau-1/2)\hat \mu_{\approx}(t)+\hat \e_{\approx}(t)+m(y,\hat\tau,t)\biggr)\label{ewppy}\ee
where   the map $m$
  has Lipschitz constant $100\rho$ on
  $$D_{\rho}:=\bD(0,\rho/3)\times \bD(1/2,(5C_{1})^{-1}\rho)\times I_{\nu}$$ and  for $t\in I_{\nu}$ one has $m(0,1/2,t)\in \bD(0,2C_{1}\|\e\|_{C^0})$. 
  
Besides, the solution of (\ref{eq:10.38nbis})  is  $T'=T_{\approx}+s$-periodic, $s\in I_{\nu}$,  if and only if 
 $$p_{\hat\tau}^{y}(T')=p_{\hat\tau}^{y}(0)$$
 i.e.
 $$\hat p_{\approx}(T_{\approx}+s)+p_{\hat\tau,cor}^y(T_{\approx}+s) =\hat p_{\approx}(0)+p_{\hat\tau,cor}^y(0)$$ 
 and because $P_{\approx}$ and $\hat p_{\approx}$ are $T_{\approx}$-periodic
  \be \hat p_{\approx}(s)+P_{\approx}(s)q_{\hat\tau}^y(T_{\approx}+s) =\hat p_{\approx}(0)+P_{\approx}(0)q_{\hat\tau}^y(0).\label{eq:n15.206}\ee
  We know from (\ref{15.176})-(\ref{Xapprox})  that 
  \begin{align*}\hat p_{\approx}(s)&=\hat p_{\approx}(0)+\int_{0}^s\pa_{s}\hat p_{\approx}(u)du\\
  &=\hat p_{\approx}(0)+s\hat X_{1/2}(p_{\approx}(0))+r(s)
  \end{align*}
  where 
 \begin{align*}r(s)&=\int_{0}^s(\hat X_{1/2}(p_{\approx}(u))-X_{1/2}(p_{\approx}(0))du+2\pi i \int_{0}^s\e(u)du\\
 &=\int_{0}^s\int_{0}^{u}D\hat X_{1/2}((1-t)p_{\approx}(0)+tp_{\approx}(u))dtdu+2\pi i \int_{0}^s\e(u)du.
 \end{align*}
 The reader can check that provided $s_{0}\leq 10^{-4}$,  the map $r$ has on $\bD(0,s_{0})$ a Lipschitz norm which is 
 \begin{align*}&\leq s_{0}\times \|D\hat X_{1/2}\|_{V}\times \|\hat X_{1/2}\|_{V}\\
 &\leq s_{0}\times 51\times 150\\
 &\leq s_{0}\times  7650
 \end{align*}
 ($V$ is some $10^{-2}$-neighborhood of the $\{\hat p_{\approx}(t)\mid t\in \R\}$) and satisfies 
 $$\sup_{\bD(0,s_{0})}\|r(s)\|\leq s_{0}\times 103.$$
 
 Equation (\ref{eq:n15.206}) can be written
 $$ s\hat X_{1/2}(p_{\approx}(0))+r(s)+P_{\approx}(s)q_{\hat\tau}^y(T_{\approx}+s) =P_{\approx}(0)q_{\hat\tau}^y(0)$$
 hence
  $$ sP_{\approx}(s)^{-1}\hat X_{1/2}(p_{\approx}(0))+P_{\approx}(s)^{-1}r(s)+q_{\hat\tau}^y(T_{\approx}+s) =P_{\approx}(s)^{-1}P_{\approx}(0)q_{\hat\tau}^y(0).$$
Introducing the function $m$ (cf. (\ref{thefunctionm})) we can write
 \begin{multline*}sP_{\approx}(s)^{-1}\hat X_{1/2}(p_{\approx}(0))+P_{\approx}(s)^{-1}r(s)+e^{(T_{\approx}+s)M_{\approx}}y\\ 
 -(\hat \tau-1/2)\hat \mu_{\approx}(T_{\approx}+s)-\hat \e_{\approx}(T_{\approx}+s)+m(T_{\approx}+s,y,\tau)\\=P_{\approx}(s)^{-1}P_{\approx}(0)\biggl(y+m(0,y,\tau)\biggr)\end{multline*}
 or in a more compact form
 \begin{multline}sP_{\approx}(0)^{-1}\hat X_{1/2}(p_{\approx}(0))+(e^{T_{\approx}M_{\approx}}-I)y=\\ -(\hat\tau-1/2)\hat\mu_{\approx}(T_{\approx})+Q(s,y,\hat\tau)\label{formcompactQ}
 \end{multline}
 where $Q$ is the map 
$$ I_{\nu}\times \bD(0,\rho/3 ) \times \bD(0,(5C_{1} )^{-1}\rho )
\ni  (s,y,\hat\tau)\mapsto Q(s,y,\hat\tau)\in \C^2
$$
defined by
\begin{multline*}Q(s,y,\hat \tau)=- e^{T_{\approx}M_{\approx}}(e^{sM_{\approx}}-I)y
\\ -P_{\approx}(s)^{-1}r(s)+(P_{\approx}(s)^{-1}P_{\approx}(0)-I)y\\ +P_{\approx}(s)^{-1}P_{\approx}(0)m(0,y,\tau)-m(T_{\approx}+s,y,\tau)\\
 (\hat \tau-1/2)\hat \mu_{\approx}(T_{\approx}+s)+\hat \e_{\approx}(T_{\approx}+s)\\
 -s(P_{\approx}(s)^{-1}-P_{\approx}(0)^{-1})\hat X_{1/2}(p_{\approx}(0)).
\end{multline*}

\begin{lemma}\label{lemma:15.9}
The map $Q$ is $0.31$-Lipschitz on 
$$D=\bD(0,10^{-6})\times \bD(0,10^{-6})\times \bD(1/2,10^{-6})$$
and $Q(0,0,1/2)\in \bD(0,4C_{1}\|\e\|_{0})$.
\end{lemma}
\begin{proof}
Using (\ref{14.149}) we see that the derivatives w.r.t. $s$  of the function $P_{\approx}(s)^{-1}$ is
$$\pa_{s}P_{\approx}^{-1}=P_{\approx}^{-1}M_{\approx}-A_{\approx}P_{\approx}^{-1}$$
hence for $s\in \bD(0,s_{0})$ ($s_{0}=10^{-6}$)
\begin{align*}
\|\pa_{s}P_{\approx}(s)^{-1}\|&\leq 9\times (4\pi+51)\\
&\leq 573.
\end{align*}
The Lipschitz norm w.r.t. $s$ of $Q$ is thus bounded above by
\begin{multline*}10^{-5}+\\
+(573\times s_{0}\times 103+9\times 7650\times s_{0})+(573\times 3\times \|y\|)\\
+(773\times 3\times 2\rho+9\times 3\times 100\rho )+(100\times \rho)\\
+|\tau-1/2|\times 853+\|\e\|_{0}\times 853\\
+(2\times 477\times s_{0}\times 150)
\end{multline*}
which is 
\begin{align*}&\leq 2.8\times 10^5\times s_{0}+8.9\times 10^3\times \rho+853\times |\hat \tau-1/2|\\
&\leq 0.29.\end{align*}

Furthermore, using Lemma \ref{lemma:15.8} (with $3\rho$ in place of $\rho$) and Remark \ref{rem:lemma}, we see that  the Lipschitz norm of $Q$ w.r.t. the variables $(y,\tau)$ is 
\begin{align*}&\leq 573\times s_{0}+300\rho\\
&\leq 10^{-3}.\end{align*}
when $(y,\hat\tau)\in \bD(0,\rho)\times \bD(1/2,3\times 10^{-3}\rho)$.

To conclude, we have by Lemma \ref{lemma:n15.8},
$$Q(0,0,1/2)\in \bD(0,4C_{1}\|e\|_{0})$$
since $4C_{1}C_{\hat \Psi}\|\e\|_{C^0}\leq 1$.

\end{proof}

\bigskip We shall prove in Section \ref{sec:15.6}-\ref{sec:15.6.8} the following result.
\begin{prop}\label{prop:15.12}
One has  (recall the definition (\ref{def:hatmuapprox}) of $\hat \mu_{\approx}(t)$)
\begin{align}
&P_{\approx}(0)\in GL(2,\R)\label{Papprox0isreal}\\
&P_{\approx}(0)^{-1}\hat X_{1/2}(\hat p_{\approx}(0))=(2\pi i)\times b_{1}\times \bm 1+a_{1} \\ a_{2}\em\label{15.190.1}\\
&\hat \mu_{\approx}(T_{\approx})= (2\pi i)\times \bm \ti \mu_{1}\\ \ti \mu_{2}\em\label{15.190.2}
\end{align}
with
\be
\begin{cases}
&\max_{1\leq j\leq 2}|a_{j}|\leq 1.2\times 10^{-2}\\
&b_{1}\approx-3.51
\end{cases}
\label{e15.322ante}
\ee
\be
\left\{
\begin{aligned}&\ti \mu_{1}=\frac{T_{\approx}}{\det \ti P(0)}\ti v_{2,1}+10^{-2}\approx -0.95,\\ &\ti v_{2,1}\approx -0.237\\
&|\ti\mu_{2}|\leq 6.
\end{aligned}
\right.
\label{e15.322}
\ee
Moreover,
\be 
\begin{cases}
&M_{\approx}=2\pi i \diag(\l_{\approx,1},\l_{\approx,2})\\
&|\l_{\approx,1}|\leq 10^{-3},\qquad |\l_{\approx,2}-(1-g_{\approx})|\leq 10^{-3}.
\end{cases}
\label{e15.326}
\ee

\end{prop}

\medskip
Coming back to (\ref{formcompactQ}) and setting 
\begin{align}
&y=\bm 0\\ \zeta\em\label{yzeta}
\end{align} one thus gets from  (\ref{15.190.1})-(\ref{15.190.2}) 
(recall (\ref{def:hatmuapprox}))
\begin{multline*}2\pi is b_{1}\bm 1+a_{1}\\ a_{2} \em+\zeta \bm a_{3}\\ e^{2\pi iT_{\approx}}-1+a_{4}\em=-2\pi i (\hat\tau-1/2)\bm \ti \mu_{1}\\ \ti \mu_{2}\em\\+Q\biggl(s,\bm 0\\ \zeta \em,\tau\biggr)\end{multline*}
where $\max_{3\leq j\leq 4}|a_{j}|\leq 10^{-3}$.

This gives
\begin{multline*}\bm 2\pi i b_{1}(1+a_{1})& a_{3}\\ 2\pi i b_{1}a_{2}&e^{2\pi i T_{\approx}}-1+a_{4}\em\bm s\\ \zeta \em= -2\pi i (\hat\tau-1/2)\bm \ti \mu_{1}\\ \ti \mu_{2}\em +\\Q\biggl(s,\bm 0\\ \zeta \em,\tau\biggr)\end{multline*}
hence
\be \bm s\\ \zeta\em=- (\hat\tau-1/2)\bm \mathring{\mu}_{1}\\ \mathring{\mu}_{2}\em+\hat Q(s,\zeta,\hat\tau)
\label{n15.193ante}
\ee
with
\begin{align} \bm\mathring{\mu}_{1}\\ \mathring{\mu}_{2}\em&=2\pi i \bm 2\pi i b_{1}(1+a_{1})& a_{3}\\ 2\pi i b_{1}a_{2}&e^{2\pi i T_{\approx}}-1+a_{4}\em^{-1}\bm \ti\mu_{1}\\ \ti \mu_{2}\em\label{n15.193}
\end{align}
and 
$$\hat Q(s,\zeta,\hat\tau)=\bm 2\pi i b_{1}(1+a_{1})& a_{3}\\ 2\pi i b_{1}a_{2}&e^{2\pi i T_{\approx}}-1+a_{4}\em^{-1} \\ Q\biggl(s,\bm 0\\ \zeta \em,\tau\biggr).$$
One has 
\begin{multline}2\pi i \bm 2\pi i b_{1}(1+a_{1})& a_{3}\\ 2\pi i b_{1}a_{2}&e^{2\pi i T_{\approx}}-1+a_{4}\em^{-1}=\\ \frac{1}{b_{1}(1+a_{1})(e^{2\pi i T_{\approx}}-1+a_{4})-b_{1}a_{2}a_{3}}\bm e^{2\pi i T_{\approx}}-1+a_{4} & -a_{3}\\ -2\pi i b_{1}a_{2}&2\pi i b_{1}(1+a_{1})\em\\
=\bm (1+a'_{1})b_{1}^{-1}& a'_{3}\\ a'_{2} &2\pi i(e^{2\pi i T_{\approx}}-1)^{-1}\em\label{muli}
\end{multline}
with 
\begin{align*}
&|a'_{1}|\leq 1.3\times 10^{-2}\\
&\max(|a'_{2}|,|a'_{3}|)\leq 10^{-3}.
\end{align*}
Note that because $e^{2\pi i T_{\approx}}-1\approx (0.319-1)-0.947 i$ one has 
$$\frac{1}{2\pi |b_{1}|}\times \frac{1}{|e^{2\pi iT_{\approx}}-1|}\leq 1/25<1/10.$$
This and Lemma \ref{lemma:15.9} imply
\begin{lemma}
The map $\hat Q$ is $0.013$-Lipschitz on $\bD(0,10^{-6})\times \bD(0,10^{-6})\times \bD(1/2,10^{-9})$ and $\hat Q(0,0,1/2)\in \bD(0,1.6\times 10^{-7})$ and 
$$\begin{cases}
&\mathring{\mu_{1}}=(1+a''_{1})b_{1}^{-1}\quad\textrm{with}\  |a''_{1}|\leq 5\times 10^{-2}\\
&|\mathring{\mu_{2}}|\leq 40
\end{cases}
$$
\end{lemma}
\begin{proof}The statement on $\hat Q$ comes from (\ref{muli}),  the fact that $Q$ is $0.31$-Lipschitz (see Lemma \ref{lemma:15.9}),   $0.3/25\leq 0.013$ and $(1/25)\times 4C_{1}\|\e\|_{0}\leq 1.32\|\e\|_{0}$.

The estimates on $\mathring{\mu}_{1},\mathring{\mu}_{2}$ is due to (\ref{n15.193}), (\ref{e15.322}) and (\ref{muli}).

\end{proof}

\begin{prop}\label{cor:14.10} For any $\hat\tau\in \bD(1/2,10^{-10})$, there exists a unique $(s_{\hat\tau},\zeta_{\hat\tau})\in \bD(0,10^{-6})\times \bD(0,10^{-6})\subset \C\times \C$ such that 
\be \bm s_{\hat\tau}\\ \zeta_{\hat\tau}\em=-(\hat\tau-1/2)\bm \mathring{\mu}_{1}\\ \mathring{\mu}_{2}\em+\hat Q\biggl(s_{\hat\tau},\bm 0\\ \zeta_{\hat\tau}\em,\hat\tau\biggr).\label{eq:14.167}\ee
Moreover, the map $\hat\tau\mapsto s_{\hat\tau}+\mathring{\mu}_{1}\hat\tau$ is $0.08$-Lipschitz.
\end{prop}
\begin{proof}The existence of the fixed point $(s_{\hat\tau},\zeta_{\hat\tau})$ is a consequence of Lemma \ref{lem:annex2} of the Appendix.  

\medskip For $\hat\tau_{1},\hat\tau_{2}$ one has
$$\bm s_{\hat\tau_{1}}\\ \zeta_{\hat\tau_{1}}\em-\bm s_{\hat\tau_{2}}\\ \zeta_{\hat\tau_{2}}\em=-(\hat\tau_{1}-\hat\tau_{2})\bm \mathring{\mu}_{1}\\ \mathring{\mu}_{2}\em+\hat Q\biggl(s_{\hat\tau},\bm 0\\ \zeta_{\hat\tau}\em,\hat\tau\biggr)$$
so
$$\bm s_{\hat\tau_{1}}-s_{\hat\tau_{2}}+(\hat\tau_{1}-\hat\tau_{2})\hat\mu_{1}\\ \zeta_{\hat\tau_{1}}-\zeta_{\hat\tau_{2}}+(\hat\tau_{1}-\hat\tau_{2})\hat\mu_{2}\em=\hat Q\biggl(s_{\hat\tau},\bm 0\\ \zeta_{\hat\tau}\em,\hat\tau\biggr).$$
From Part 3) of Lemma \ref{lem:annex2} we know that $\hat\tau\mapsto (s_{\hat\tau},\zeta_{\hat\tau})$ is $(1-0.013)^{-1}\times ( \max(|\mathring{\mu}_{1}|,|\mathring{\mu}_{2}|)+0.013)$-Lipschitz i.e. $6.1$-Lipschitz.  Hence $\hat\tau\mapsto  \hat Q\biggl(s_{\hat\tau},\bm 0\\ \zeta_{\hat\tau}\em,\hat\tau\biggr)$ is  $0.9$-Lispchitz (we used ($\max(|\mathring{\mu}_{1}|,|\hat \mathring{\mu}_{2}|)\leq 6$). As a consequence $s_{\hat\tau}+\hat\tau \mathring{\mu}_{1}$ is $0.08$-Lipschitz.

\end{proof}

\medskip
This yields:
\begin{cor}\label{cor:15.15}The derivative of the  map $\hat\tau\mapsto s_{\hat\tau}$ is non zero and 
\begin{align*}(\pa_{\hat \tau}s_{\hat\tau})_{\mid \hat\tau=1/2}&=-\mathring{\mu}_{1}\pm 0.08\\
&=-T_{\approx}\times 0.22\pm 0.09\\
&=0.26\times (1+a)\quad{\rm with}\quad |a|\leq 0.35.
\end{align*}
\end{cor}
\begin{proof}The existence of the fixed point is a direct consequence of the preceding Proposition \ref{cor:14.10} (cf.  Lemma \ref{lem:annex2}). Because the dependence on $\hat \tau$ in (\ref{eq:14.167}) is $C^1$ w.r.t. $\hat \tau$, the map $\hat \tau\mapsto (s_{\hat\tau},\zeta_{\hat\tau})$ is $C^1$. Besides, from Lemma \ref{lem:annex2} we know that the map $\hat \tau\mapsto s_{\hat\tau}+(\hat\tau-1/2)\mathring{\mu}_{1}$ has Lipschitz norm $\leq 0.08$, whence the result.

To get the estimate on $\mathring{\mu}_{1}$ we observe that from (\ref{n15.193}) one has 
\begin{align*}&\mathring{\mu}_{1}= (1+a_{1}'')b_{1}^{-1}\ti \mu_{1}\\
&= \frac{1+a_{1}''}{-3.51}\times \frac{T_{\approx}}{\det \ti P(0)}\times (-0.237)\\
&=\frac{1+a_{1}''}{-3.51}\times T_{\approx}\times  \frac{1}{0.307}\times (-0.237)\\
&=T_{\approx}\times 0.22\times(1+a''_{1})\\
&=-0.26\pm 2\times 10^{-2}.
\end{align*}

\end{proof}

Let
$$\hat  g(\hat\tau)=\frac{1}{T_{\approx}+s_{\hat\tau}}.$$

\begin{cor}\label{cor15.17}One has $\pa \hat g(1/2)= -0.18\pm 10^{-1}$ (we keep this form because the expected value of $\pa \hat g(1/2)$ is $-0.18\pm 10^{-2}$). In particular it does not vanish.
\end{cor}
\begin{proof}One has 
$$\frac{\pa_{}\hat g(1/2)}{\hat g(1/2)}=-\frac{\pa_{\hat\tau}s_{\hat\tau}}{T_{\approx}+s_{1/2}}$$
hence
\begin{align*}\pa \hat g(1/2)&=(T_{\approx}\times 0.22\pm 9\times 10^{-2})\times \frac{\hat g(1/2)}{T_{\approx}+s_{1/2}}\\
&= -0.18\pm 10^{-1}
\end{align*}
\end{proof}

\begin{rem}Note that this is in good agreement with the approximate formula given in  (\ref{2.14})
$$\hat g(\hat \tau)=\frac{-4-\sqrt{16+ 22\times \hat \tau}}{11}$$
which gives 
$$\pa_{\hat\tau}\hat g(1/2)=-\frac{22}{22\sqrt{16+11}}\approx-0.19$$

\end{rem}

\begin{rem}\label{rem-cor15.17} One could also prove that 
\begin{align*}\mathring{\mu}_{2}&\approx 2\pi i(e^{2\pi i T_{\approx}}-1)^{-1}\ti \mu_{2}\\
&\approx \frac{1}{\det \ti P(0)}\sum_{|k|\leq N}\frac{\ti v_{1,3k+1}}{(3k+1)g_{\approx}-1}\\
&\approx\frac{1.13}{\det \ti P(0)}\qquad ({\rm see}\ (\ref{e15.371}))\\
&\approx 3.68
\end{align*}
and like in Corollary \ref{cor:15.15}
that
\be\pa_{\hat \tau} \zeta_{\hat \tau}\approx -\mathring{\mu}_{2}\approx -3.68. \ee
\end{rem}

\subsubsection{Proof of Theorem \ref{theo:expero}}  \label{sec:proofofth15.1}
\

\mn 1) Finding a $1/\hat g(\hat\tau)$-periodic solution (here with complex period)
\be \frac{d\hat p_{\hat\tau}(t)}{dt}=\hat X_{\hat\tau}(\hat p_{\hat\tau}(t)).\label{10.32nbis}\ee
 of the vector field $\hat X_{\hat \tau}$ (eq. (\ref{10.32n})) is equivalent to finding $y_{\hat \tau}$ such that the Cauchy problem  (see Lemma \ref{lemma:CPvsFPante} and  the notation (\ref{10.32n-1})) 
$$
\left\{
\begin{aligned}
&\frac{d(\hat p_{\approx}+p_{\hat\tau, cor})}{dt}=\hat X_{\hat\tau}\circ \biggl((\hat p_{\approx}+p_{\hat\tau,cor})\biggr)\\
&(\hat p_{\approx}+p_{\hat\tau,cor})(0)=\hat p_{\approx}(0)+y_{\tau}
\end{aligned}
\right.
$$
or (see Lemma \ref{lemma:CPvsFP}) 
\be
\left\{
\begin{aligned}
 &\frac{dp^y_{\hat\tau}(t)}{dt}=\hat X_{\hat \tau}(p^y_{\hat\tau}(t))\\
 &p_{\hat\tau}^y(0)=\hat p_{\approx}(0)+P_{\approx}(0)y,
 \end{aligned}
 \right.
 \label{eq:10.38nbisnew}
 \ee
has a $1/\hat g(\hat\tau)$-periodic solution. As we saw in Lemma \ref{lemma:CPvsFP} this last problem  can be reduced to a fixed point question that can be brought to the form (see (\ref{eq:14.167}))
$$\bm s_{\hat\tau}\\ \zeta_{\hat\tau}\em=-(\hat\tau-1/2)\bm \mathring{\mu}_{1}\\ \mathring{\mu}_{2}\em+\hat Q\biggl(s_{\hat\tau},\bm 0\\ \zeta_{\hat\tau}\em,\hat\tau\biggr)$$ 
where $y_{\hat \tau}=\bm 0\\ \zeta_{\hat\tau}\em$ and $1/\hat g(\hat\tau)=T_{\approx}+s_{\hat\tau}$.

Proposition  \ref{cor:14.10} gives   a positive  answer to  this question at least when $\hat \tau$ is in a complex neighborhood of  $1/2$ and provides a unique $(s_{\hat\tau},\zeta_{\hat\tau})\in \bD(0,10^{-6})\times \bD(0,10^{-6})\subset \C\times \C$ solution of this fixed point problem.

\mn 2)  Let's prove that when $\hat \tau=1/2$  the frequency $\hat g(1/2)$ of  this  solution $t\mapsto \hat p_{1/2}(t)$  is real. If
\be \frac{1}{\hat g(1/2)}=T_{1/2}=T_{\approx}+s_{1/2}\ee 
one has
$$\phi^{T_{1/2}}_{\hat X_{1/2}}(\hat p_{1/2}(0))=\hat p_{1/2}(0)$$
or equivalently 
\be \phi^{T_{\approx}+\bar s_{1/2}}_{\hat X_{1/2}}\biggl(\hat p_{\approx}(0))+P_{\approx}(0)\bm 0\\ \bar  \zeta_{1/2}\em\biggr)=\hat p_{\approx}(0)+P_{\approx}(0)\bm 0\\ \bar \zeta_{1/2}\em \label{15.343}\ee
and we have to prove $T_{1/2}\in \R$.

 The key observation here is that 
 the approximate frequency $g_{\approx}$ and the sequences (\ref{eq:sequences}) of Proposition \ref{eq:numerics1} are {\it real}. This implies that the $T_{\approx}$-periodic approximate solution $\hat p_{\approx}(t)=(\hat z_{\approx},w_{\approx})$  
\be 
\begin{aligned}
&{\hat z}_{\approx}(t)=\sum_{|k|\leq N}{\hat z}_{3k}^\approx e^{3ki (2\pi g_{\approx}) t}\\
&w_{\approx}(t)=1.4\times e^{i(2\pi g_{\approx}) t}+\sum_{0<|k|\leq N}\mathring{w}^\approx_{3k+1}e^{ (3k+1)i(2\pi g_{\approx}) t}
\end{aligned}
\label{syst15.345}
\ee
satisfies
$$\hat \s(\hat p_{\approx}(t))=\hat p_{\approx}(-t)$$
where $\hat \s$ is the anti-holomorphic involution $\hat \sigma:(z,w)\mapsto (\bar z, \bar  w)$. Now, the vector field $\hat X_{1/2}$ is reversible w.r.t. $\hat \s$ (see Remark \ref{rem:rev6.1}) and one thus has
$$\phi^{\bar T_{1/2}}_{\hat X_{1/2}}(\hat\s(\hat p_{1/2}(0)))=\hat\s(\hat p_{1/2}(0)).$$
Writing $\hat p_{1/2}(0)=\hat p_{\approx}(0)+P_{\approx}(0)\bm 0\\ \zeta_{1/2}\em$
this yields (remember  $P_{\approx}(0)\in GL(2,\R)$, cf. (\ref{Papprox0isreal}) of Proposition \ref{prop:15.12}) the following fixed point property
$$\phi^{T_{\approx}+\bar s_{1/2}}_{\hat X_{1/2}}\biggl(\hat p_{\approx}(0))+P_{\approx}(0)\bm 0\\ \bar  \zeta_{1/2}\em\biggr)=\hat p_{\approx}(0)+P_{\approx}(0)\bm 0\\ \bar \zeta_{1/2}\em$$
with $(\bar s_{1/2},\bar \zeta_{1/2})\in \bD(0,10^{-6})\times \bD(0,10^{-6})$. Comparing with (\ref{15.343}) we get by uniqueness of the fixed point 
$$\begin{cases}
&\bar s_{1/2}=s_{1/2}\\
&\bar \zeta_{1/2}=\zeta_{1/2}
\end{cases}$$
hence $T_{1/2}\in \R$.

\mn 3) 
Let's verify the fact that 
$$\diag(1,j)(\phi^t_{\hat X_{1/2}}(\hat p_{1/2}(0))=\phi^{t-T_{1/2}/3}_{\hat X_{1/2}}(\hat p_{1/2}(0)).$$
The system (\ref{syst15.345}) exhibits  the obvious symmetry
$$\hat p_{\approx}(t-T_{\approx}/3)=\diag (1,j)(\hat p_{\approx}(t)).$$
Besides, because $\diag(1,j)_{*}\hat X_{1/2}=\hat X_{1/2}$ the function
$$\R\ni t\mapsto \phi^{t}_{\hat X_{1/2}}\biggl(\phi^{T_{\approx}/3}_{\hat X_{1/2}}(\diag(1,j)(\hat p_{1/2}(0))\biggr)\in \C^2$$
is $T_{1/2}$-periodic and we have
\begin{align*}\phi^{T_{\approx}/3}_{\hat X_{1/2}}(\diag(1,j)(\hat p_{1/2}(0))&=\diag(1,j)\hat p_{1/2}(T_{\approx}/3)\\
&=\diag(1,j)(\hat p_{\approx}(T_{\approx}/3))\\ &\qquad+\diag(1,j)(\hat p_{1/2}(T_{\approx}/3)-\hat p_{\approx}(T_{\approx}/3))\\
&=\hat p_{\approx}(0)+q
\end{align*}
with $q=\diag(1,j)(\hat p_{1/2}(T_{\approx}/3)-\hat p_{\approx}(T_{\approx}/3))$, 
$$\|q\|\leq \sup_{t\in \R}\|\hat p_{1/2}(t)-\hat p_{\approx}(t)\|\leq 10^{-7}.$$
We can hence write
$$\hat p_{\approx}(0)+q=\phi_{\hat X_{1/2}}^{t_{q}}\biggl(\hat p_{\approx}(0)+P_{\approx}(0)\bm 0\\ \zeta_{q}\em\biggr)$$
for some $t_{q},\zeta_{q}$ satisfying
$$\begin{cases}
&|t_{q}|<10^{-6}\\
&|\zeta_{q}|<10^{-6}.
\end{cases}
$$
Because $\phi_{\hat X_{1/2}}^{T_{1/2}}(\hat p_{\approx}(0)+q)=\hat p_{\approx}+q$ we thus get
$$\phi_{\hat X_{1/2}}^{T_{1/2}}\biggl(\hat p_{\approx}(0)+P_{\approx}(0)\bm 0\\ \zeta_{q}\em\biggr)=\biggl(\hat p_{\approx}(0)+P_{\approx}(0)\bm 0\\ \zeta_{q}\em\biggr)$$
that we compare with 
$$\phi_{\hat X_{1/2}}^{T_{1/2}}\biggl(\hat p_{\approx}(0)+P_{\approx}(0)\bm 0\\ \zeta_{1/2}\em\biggr)=\biggl(\hat p_{\approx}(0)+P_{\approx}(0)\bm 0\\ \zeta_{1/2}\em\biggr).$$
Again by uniqueness we deduce
$\zeta_{q}=\zeta_{1/2}$ whence $\hat p_{\approx}(0)+q=\phi_{\hat X_{1/2}}^{t_{q}}(\hat p_{1/2}(0))$ and 
$$\phi^{(T_{\approx}/3)-t_{q}}_{\hat X_{1/2}}(\diag(1,j)(\hat p_{1/2}(0))=\hat p_{1/2}(0).$$
This shows that for any $t\in \R$ 
$$\diag(1,j)\circ \phi^t_{\hat X_{1/2}}(\hat p_{1/2}(0))=\phi^{t+t_{q}-T_{\approx}/3}_{\hat X_{1/2}}(\hat p_{1/2}(0)).$$
We can identify $T_{*}/3:=t_{q}-T_{\approx}/3$ to $-T_{1/2}/3$. Indeed, arguing like in the proof of Corollary \ref{cor:7.7} one can prove that the action of $\psi^{-1}\circ \diag(1,j)\circ \psi$ on the annulus $\T_{s''_{*}}$ is a translation $\th\mapsto \th+a$ with $3a\equiv 0\mod \Z$. In particular, $T_{*}\equiv -T_{1/2}\mod T_{1/2}$ which yields the result ($t_{q}$ is small and $-T_{\approx}$ and $-T_{1/2}$ are very close).

\mn 4)  One can  prove that for $\hat \tau \in \R$ close to $1/2$, the frequency $\hat g(\hat\tau)$ of $\hat X_{\hat\tau}$ is real. The proof is very similar to the one of Proposition \ref{prop:7.7} and we won't repeat it. Just mention that one uses the fact that $\hat X_{\hat\tau}$ is reversible w.r.t. the anti-holomorphic involution $(z,w)\mapsto (\bar z, j^2\bar w)$ and the fact that $\s$ leaves globally invariant \footnote{This is a consequence of points 2) and 3) and of the fact that $\hat\s=\diag(1,j)\circ \s\circ \diag(1,j)^{-1}$. }the orbit $(\phi^t_{\hat X_{1/2}}(\hat p_{1/2}(0)))_{t\in \R}$. We use estimates very similar to  (\ref{eq:7.74n}), (\ref{eq:7.75}) and (\ref{eq:7.75bis})  except that the $O(\d^{2m-(5/3)})$ term is now just 0.

\mn 5) 
Furthermore, Corollary \ref{cor15.17}  shows that the derivative $\pa_{\hat\tau}\hat g(1/2)$ is non zero.

This completes the proof of Theorem \ref{theo:expero}. \hfill$\Box$

\subsection{Controlling the resolvent  $R_{A_{\approx}}$}\label{sec:15.6}
The main result of this subsection is the following:
\begin{prop}[Control on $R_{A_{\approx}}$]\label{prop:controlresolventRApprox} There exists a Floquet decomposition 
$$R_{A_{\approx}}(t,s)=P_{\approx}(t)e^{(t-s)M_{\approx}}P_{\approx}(s)^{-1}$$
where $M_{\approx}$ is diagonal  
\be\begin{cases}
&M_{\approx}=2\pi i \diag(\l_{\approx,1},\l_{\approx,2})\\
&|\l_{\approx,1}|\leq 10^{-3},\qquad |\l_{\approx,2}-(1-g_{\approx})|\leq 10^{-3}
\end{cases}
\label{prop:15.17-1}
\ee
and the gauge transformation $P_{\approx}=\bm u_{\approx,1}& u_{\approx,2}\\ v_{\approx,1}& v_{\approx,2}\em:\R/(T_{\approx}\Z)\to GL(2,\C)$ has the following properties 
\begin{align*}
&P_{\approx}(0)\in GL(2,\R)\\
 &P_{\approx}(0)^{-1}=(I\pm 1.2\times  10^{-2})\bm 1.23&  -4.05\\
 -0.25&  3.46\em,\\
 & P_{\approx}(0)=\bm 1.06& 1.24\\ 0.07 & 0.37\em(I \pm 1.2\times 10^{-2})\\
&\forall t\in\R,\quad |\det  P_{\approx}(t)|\geq 0.3\quad \textrm{and}\quad \|P_{\approx}(t)\|_{op}\leq 2.6\\
&\|u_{\approx,1}\|_{\cO(I_{\nu})}\leq 1.4,\quad \|u_{\approx,2}\|_{\cO(I_{\nu})}\leq 1.55,\\ & \|v_{\approx,1}\|_{\cO(I_{\nu})}\leq 1.03,\quad \|v_{\approx,2}\|_{\cO(I_{\nu})}\leq 1.21.
\end{align*}
Furthermore,
$$P_{\approx}(0)^{-1}\hat X_{1/2}(\hat p_{\approx}(0))=(2\pi i)\times (-3.51\pm 10^{-3})\times \bm  1+10^{-2} \\ 10^{-2}\em.$$
\end{prop}
This proposition which is proved in Paragraph \ref{sec:15.6.7}, will be   a consequence of the following two propositions.
\begin{prop}[Approximate resolvent]\label{prop:controlresolvent}There exists $T_{\approx}$-periodic functions 
$$\ti P:\R\ni t \mapsto \bm \ti u_{1}&\ti u_{2}\\ \ti v_{1}&\ti v_{2}\em\in GL(2,\Z)$$ 
and a diagonal matrix 
$\ti M=2\pi i\times \diag(\ti \l_{1}, \ti \l_{2})$ with 
\be \ti\l_{1}=\pm 10^{-5},\qquad \ti \l_{2}=1-g_{\approx}\pm 10^{-5}\label{prop:15.18-1}\ee
such that the matrix $\ti R(t,s):=\ti P(t)e^{t\ti M}\ti P(s)^{-1}$ satisfies
the ODE
\be 
\left\{
\begin{aligned}&\frac{d}{dt}\ti R_{}(t,0)=A_{\approx}(t) \ti R_{}(t,0)+E(t)\ti P(0)^{-1}\\
&\ti R(0,0)=I 
\end{aligned}
\right.
\label{eq:10.41nnew}
\ee
where $E$ satisfies
$\sup_{t\in[-10T_{\approx},10T_{\approx}]}\|E(t)\| \leq 10^{-6}.$
Moreover,
\begin{align}
&\ti P(0)\in GL(2,\R)\label{e:15.348}\\
&\ti P(0)^{-1}=\bm 1.2357&  -4.0592\notag\\
 -0.2513&  3.4657\em\pm 10^{-4},\notag\\
 & \ti P(0)=\bm 1.0624& 1.2460\\ 0.0767 & 0.3793\em \pm 10^{-4}\notag\\
&\forall t\in\R,\quad |\det \ti P(t)|\geq 0.3\quad \textrm{and}\quad \|\ti P(t)\|_{op}\leq 2.6\notag
\end{align}
and
\be \ti P_{}(0)^{-1}\hat X_{1/2}(\hat p_{\approx}(0))=(2\pi i)\times (-3.51\pm 10^{-3})\times \bm  1+10^{-3} \\ 10^{-3}\em.\label{15.327}\ee
\end{prop}
\begin{proof}See Paragraph \ref{sec:15.6.5}.
\end{proof}

\begin{prop}[Comparing $(P_{\approx},M_{\approx})$ and   $(\ti P,\ti M)$]\label{prop:comptiRRsuA}One can choose the Floquet decomposition $R_{A_{\approx}}(t,s)=P_{\approx}(t)e^{(t-s)M_{\approx}}P_{\approx}(s)^{-1}$ so that $P_{\approx}(0)=\ti P(0)$ and  for  any $t\in I$ one has
\be \begin{cases}&P_{\approx}(t)=\ti P(t)(I\pm 1.1\times 10^{-2})\\
&\| P_{\approx}(t)-\ti P(t)\|\leq 5\times 10^{-3}
\end{cases}
\label{prop:15.19-1}
\ee
and
\be \begin{cases}
&|\ti \l_{1}-\l_{\approx,1}|\leq 4.2\times 10^{-4} \\
&|\ti \l_{2}-\l_{\approx,2}|\leq 4.2\times 10^{-4}.
\end{cases}
\label{prop:15.19-2}
\ee
\end{prop}
\begin{proof} See Paragraph \ref{sec:15.6.6}.
\end{proof}

\subsubsection{On the spectrum and eigenvectors  of $R_{A_{\approx}}(T_{\approx},0)$}

The results of this subsection are not needed for the proofs of the main propositions of Subsection \ref{sec:15.6} but we thought it might  explain  some  properties of the approximate resolvent $R_{A_{\approx}}(T_{\approx},0)$.

\medskip
By definition the function 
$$t\mapsto p^y(t):=\hat p_{\approx}(t)+p_{cor}^y(t)$$ is the unique  solution of the Cauchy problem
\be
\left\{
\begin{aligned}
 &\frac{dp^y_{}(t)}{dt}=\hat X_{1/2}(p^y_{}(t))\\
 &p^y(0)=\hat p_{\approx}(0)+y.
 \end{aligned}
 \right.
 \label{eq:10.38n}
 \ee
 We define 
\be A_{}(t)=D\hat X_{1/2}(\phi^t_{\hat X_{1/2}}(\hat p_{\approx}(0)))\label{def:A}\ee
(compare with $A_{\approx}(t)=D\hat X_{1/2}(\hat p_{\approx}(t))$, cf. (\ref{14.158}))
and $R_{A_{}}$ as  the resolvent associated to the linear ODE $$\dot Y(t)=A_{}(t)Y(t).$$
Because the vector field $\hat X_{1/2}$ does not depend on time,  the Linearization theorem for ODEs tells us that 
$$R_{A_{}}(t,s)=D\phi_{\hat X_{1/2}}^{t-s}(\hat p_{\approx}(0)).$$
We list in the following lemma  some consequences of this fact.
\begin{lemma}
\begin{enumerate}
\item The determinant of $R_{A_{}}(t,s)$ is equal to $e^{2\pi i(t-s)}$.
\item One has $R_{A_{}}(t,0)\hat X_{1/2}(\hat p_{\approx}(0))=\hat X_{1/2}(\phi^t_{\hat X_{}}(\hat p_{\approx}(0)))$.
\item Let $C_{A_{\approx}}=\sup_{[0,T_{\approx}]}\|R_{A_{\approx}}(t,0)\|$ and assume that 
$$5\rho\times2\pi\times T_{\approx}C_{A_{\approx}}<10^{-1}.$$ 
Then, one has for any $t\in  I$,
$$\|R_{A}(t,0)-R_{A_{\approx}}(t,0)\|\leq 6C_{A_{\approx}}\rho.$$

\end{enumerate}
\end{lemma}
\begin{proof}The first  item is a general fact (known as Liouville's Theorem).  

The second item is a consequence of the identity $$\phi^t_{X}(\phi_{X}^{s}(p_{\approx}(0))=\phi_{X}^{t+s}(p_{\approx}(0))=\phi^s_{X}(\phi_{X}^{t}(p_{\approx}(0))$$ that we differentiate   with respect to $s$.

For the third point,  we use the estimate on $p^{y=0}_{1/2,cor}(\cdot)$ provided by  Corollary \ref{cor:10.4} and the fact that  $\phi^t_{\hat X_{1/2}}(\hat p_{\approx}(0))=\hat p_{\approx}(t)+p^0_{1/2,cor}(t)$.
Because of (\ref{estgzw}), (\ref{eq:10.35bis}) and (\ref{def:A})
We thus have 
$$\|A(t)-A_{\approx}(t)\|\leq 2\pi \times 5\rho.$$
We can then conclude by using Lemma \ref{lemma:3.4resolGron-n} from the Appendix:
$$\sup_{[0,T_{\approx}+1]}\|R_{A}(\cdot,0)-R_{A_{\approx}}(\cdot,0)\|\leq  2\pi \times 5\rho\times C_{A}^2T_{\approx}e^{2\pi\times 5\rho T_{\approx}C_{A}} .$$
\end{proof}

\begin{rem}
We expect the piece of orbit $(\phi^t_{\hat X_{1/2}}(\hat p_{\approx}(0)))_{t\in [0,T_{\approx}]}$ to be close to some $T$-{\it periodic} orbit $(\phi^t_{\hat X_{1/2}}(\hat p))_{t\in [0,T_{\approx}]}$ ($\hat p\in \C^2$) with $|T-T_{\approx}|\leq {\rm cst} \times \rho$ and $|\hat p-\hat p_{\approx}(0)|\leq {\rm cst} \times \rho$. Let 
$$A_{\hat p}(t)=D\hat X_{1/2}(\phi^t_{\hat X}(\hat p))$$
and $R_{A_{\hat p}}$ be  the associated resolvent. Then
\begin{itemize}
\item The determinant of $R_{A_{\hat p}}(t,s)$ is equal to $e^{2\pi i(t-s)}$.
\item One has $R_{A_{\hat p}}(t,0)\hat X_{1/2}(\hat p)=\hat X_{1/2}(\phi^t_{\hat X_{}}(\hat p))$.
\item The eigenvalues of $R_{A_{\hat p}}(T,0)$ are $1$ and $e^{2\pi iT}$. 
\end{itemize} 
We thus expect the eigenvalues of $R_{A}(T_{\approx},0)$, hence those of  $R_{A_{\approx}}(T_{\approx},0)$, to be close to $1$ and $e^{2\pi iT_{\approx}}$ and $X(\hat p_{\approx}(0))$ to satisfy the approximate eigenvalue equation 
$$R_{A_{\approx}}(T_{\approx},0)\hat X_{1/2}(\hat p_{\approx}(0))\approx \hat X_{1/2}(\hat p_{\approx}(0)).$$
\end{rem}
\begin{lemma}\label{lemma:15.12}One has 
\begin{multline*}\biggl\|\hat X_{1/2}(p_{\approx}(0))-R_{A_{\approx}}(T_{\approx},0)\hat X_{1/2}(p_{\approx}(0))\biggr\| \\ \leq (2\pi)^2\times |g_{\approx}|\times 3(2N+1)(1+T_{\approx}C_{R_{A_{\approx}}})\times\|\e\|_{C^0(\R)}. \end{multline*}
\end{lemma}
\begin{proof} From (\ref{15.176}), (\ref{Xapprox}) we have for any $r$ small enough
\be  \frac{dp_{\approx}}{dt}(t+r)=\hat X_{1/2}(p_{\approx}(t+r))+2\pi i \e(t+r).\label{15.193}\ee
Differentiating  with respect to $r$ yields,
$$\frac{dY}{dt}(t)=D\hat X_{1/2}(p_{\approx}(t))\cdot Y(t)+2\pi i \pa_{t}\e(t)$$
with $Y(t)=\pa_{t}p_{\approx}(t)$.
Hence by the resolvent formula
$$\pa_{t}p_{\approx}(T_{\approx})=R_{A_{\approx}}(T_{\approx},0)\pa_{t}p_{\approx}(0)+\int_{0}^{T_{\approx}}R_{A_{\approx}}(T_{\approx},s)\pa_{s}\e(s)ds$$
and from (\ref{15.193})
\begin{multline*}\hat X_{1/2}(p_{\approx}(T_{\approx}))+2\pi i \e(T_{\approx})=R_{A_{\approx}}(T_{\approx},0)(\hat X_{1/2}(p_{\approx}(0))+2\pi i \e(0))\\ +2\pi i \int_{0}^{T_{\approx}}R_{A_{\approx}}(T_{\approx},s)\pa_{s}\e(s)ds.
\end{multline*}
Because $p_{\approx}$ is $T_{\approx}$ periodic
$$\biggl\|\hat X_{1/2}(p_{\approx}(0))-R_{A_{\approx}}(T_{\approx},0)\hat X_{1/2}(p_{\approx}(0))\biggr\|\leq  2\pi (1+T_{\approx}C_{R_{A_{\approx}}})\times\|\e\|_{C^1(\R)}. $$
We observe that $\e$ is a  $T_{\approx}$-periodic trigonometric polynomial with harmonics $\leq 3(2N+1)$ hence
$$\|\e\|_{C^1(\R)}\leq (2\pi |g_{\approx}|)\times 3(2N+1)\times \|\e\|_{C^0(\R)}.$$
We finally get
\begin{multline*}\biggl\|\hat X_{1/2}(p_{\approx}(0))-R_{A_{\approx}}(T_{\approx},0)\hat X_{1/2}(p_{\approx}(0)\biggr\| \\ \leq (2\pi)^2\times |g_{\approx}|\times 3(2N+1)(1+T_{\approx}C_{R_{A_{\approx}}})\times\|\e\|_{C^0(\R)}. \end{multline*}

\end{proof}
As we shall see in the next subsection this is the case.

\subsubsection{ Floquet decomposition of $R_{A_{\approx}}$}
We explain in the next section how to get a good control on $R_{A_{\approx}}$.

\medskip
Since $A_{\approx}$ is $T_{\approx}$-periodic, its resolvent  admits a Floquet decomposition
\be R_{A_{\approx}}(t,s)=P_{\approx}(t)e^{tM_{\approx}}P_{\approx}(s)^{-1}\label{resolvapprox}\ee
where $P_{\approx}:\R\to GL(2,\C)$ is $T_{\approx}$-periodic and $M_{\approx}\in M(2,\C)$ is such  that 
$$R_{A_{\approx}}(T_{\approx},0)=e^{T_{\approx}M_{\approx}}.$$

To find $M_{\approx}$ and $P_{\approx}\in C^0_{T_{\approx}-per.}(\R,GL(2,\C))$ we try to determine $\l_{\approx,j}\in\C$, and  $u_{\approx,j}(\cdot),v_{\approx,j}(\cdot)$, $j=1,2$, 
which are 
$1/g_{\approx}$-periodic  and of the form 
$$
\begin{cases}
& u_{\approx,j}(t)=\sum_{|k|\leq N}u^\approx_{j,3k} e^{i(3k)(2\pi g_\approx) t}\\
& v_{\approx,j}(t)=\sum_{|k|\leq N}v^\approx_{j,3k+1}e^{i(3k+1) (2\pi g_\approx) t}\\
&\l_{\approx,j}\in \C,
\end{cases}
$$
and are such that 
\be
 \frac{d}{dt}\bm e^{i(2\pi \l _{\approx,j})t}u_{\approx,j}(t)\\ e^{i(2\pi \l_{\approx,j}) t}v_{\approx,j}(t)\em=A_{\approx}(t)\bm e^{i (2\pi \l _{\approx,j})t}u_{\approx,j}(t)\\ e^{i(2\pi \l_{\approx,j}) t}v_{\approx,j}(t)\em
\label{eq:10.36}
 \ee
 provided
 $$\forall t\in \R\quad\det P_{\approx}(t)\ne 0.$$
If this is possible, one can then  choose
$$P_{\approx}(t)=\bm u_{\approx,1}(t)&u_{\approx,2}(t)\\ v_{\approx,1}(t)&v_{\approx,2}(t)\em\qquad\textrm{and}\qquad M_{\approx}=\bm {i(2\pi \l_{\approx,1})}&0\\ 0& {i (2\pi \l_{\approx,2})}\em.$$
Let's mention that this choice is not unique: if $P_{\approx}$, $\l_{\approx,j}$ are a solution to (\ref{resolvapprox}) then, for any $m\in\Z$,  the same is true of
$$P_{\approx}(t)\diag(e^{ 2\pi it(\l_{\approx,1} +(3mg))},e^{2\pi it(\l_{\approx,2} +(3mg))} ),\qquad \l_{\approx,j}+(3mg).$$

Remembering the definition (\ref{14.158}) of $A_{\approx}(\cdot)$, equation (\ref{eq:10.36})
can be written (we skip the index $j$)
\be 
\left\{
\begin{aligned}
&(3kg_\approx-1+\l_{\approx}) u^\approx_{3k}=\sum_{l_{1}+l_{2}=k}z^\approx_{3l_{1}}u^\approx_{3l_{2}}-\sum_{l_{1}+l_{2}+l_{3}=k-1}w^\approx_{3l_{1}+1}w^\approx_{3l_{2}+1}v^\approx_{3l_{3}+1}\\
&[(3k+1)g_\approx+\l_{\approx}]v^\approx_{3k+1}=-\sum_{l_{1}+l_{2}=k}w^\approx_{3l_{1}+1}u^\approx_{3l_{2}}-\sum_{l_{1}+l_{2}=k}z^\approx_{3l_{1}}v^\approx_{3l_{2}+1}.
\end{aligned}
\right.
\label{eq:10.38n}
 \ee
 This is an infinite dimensional eigenvalue  problem of the form
 \be \l_{\approx}\zeta=L\zeta\label{eq:10.39n}\ee
 where $L:\cE\to\cE$ denotes the linear map in the variable $(\zeta_{k})_{k\in\Z}=((u^\approx_{3k})_{k\in\Z},(v^\approx_{3k+1})_{k\in\Z})$
 defined by the right hand side of (\ref{eq:10.38n}).
 \subsubsection{The numerical approximation $\ti R$  of $R_{A_{\approx}}$}
 If we project the eigenvalue  equation (\ref{eq:10.39n}) on $\cE_{N}$
we get  an eigenvalue equation in a {\it finite} dimensional space
$$\l_{\approx}\zeta=\cP_{N}L\zeta,\qquad \zeta\in \cE_{N}.$$
For $N=12$ we numerically find that the  $2\times(2N+1)=50$ eigenvalues of $\cP_{12}\circ L$ are distinct and that    the set  they constitute   is $10^{-6}$ close (for the Hausdorff distance) to
 a subset of
$$\{0,1-g_\approx\}+3g_\approx\Z;$$ besides, it  contains a subset   which is $10^{-6}$-close  to
$$\{0, 1-g_\approx\}.$$
Let $L_{N};\cE_{N}\to \cE_{N}$ be the linear map $\cP_{N}\circ L$.
\begin{prop}\label{prop:n15.22}There exist two linearly independent vector $$\zeta_{1}=((\ti u_{1,3k})_{-N\leq k\leq N},(\ti v_{1,3k+1})_{-N\leq k\leq N})$$ and  $\zeta_{2}=((\ti u_{2,3k})_{-N\leq k\leq }$, $(\ti v_{2,3k+1})_{-N\leq k\leq N})$ in $\cE_N$ and two complex number $\ti \l_{1},\ti\l_{2}$ such that 
\begin{align}&\forall\ j=1,2,\quad L_{N}\zeta_{j}=\ti \l_{j}\zeta_{j}\notag\\
&\ti \l_{1}=\pm 10^{-6}\quad \textrm{and}\quad \ti \l_{2}=1-g_{\approx}\pm 10^{-6}\label{esttilambda}\\ 
&\max_{j=1,2}\|(L-\ti \l_{j})\zeta_{j}\|_{l^1}\leq 10^{-6}.\notag
\end{align}
\end{prop}
If we define
$$
\begin{cases}
& \ti u_{j}(t)=\sum_{|k|\leq N}\ti u_{j,3k} e^{i(3k)(2\pi g^\approx) t}\\
& \ti v_{j}(t)=\sum_{|k|\leq N}\ti v_{j,3k+1}e^{i(3k+1) (2\pi g^\approx) t}\\
\end{cases}
$$
we thus have,
\be  \frac{d}{dt}\bm e^{i (2\pi \ti \l _{j})t}\ti u_{j}(t)\\ e^{i (2\pi \ti \l_{j}) t}\ti v_{j}(t)\em=A_{\approx}(t)\bm e^{i (2\pi \ti \l _{j})t}\ti u_{j}(t)\\ e^{i (2\pi \ti \l_{j}) t}\ti v_{j}(t)\em+E_{j}(t)
\label{eq:10.40}\ee
where $E_{j}$  satisfies 
$$\sup_{t\in[-10T_{\approx},10T_{\approx}]}\sup_{j}\|E_{j}(t)\|=\e_{\approx,2}\leq 10^{-6}.$$
Let 
$$\ti P(t)=\bm \ti u_{1}&\ti u_{2}\\ \ti v_{1}&\ti v_{2}\em\qquad \textrm{and}\qquad \ti M=\diag(2\pi i \ti \l_{1},2\pi i \l_{2}).$$
and
$$\ti R(t,s)=\ti P(t)e^{(t-s)\ti M}\ti P(s)^{-1}.$$

One can numerically check that the determinant
$$\det \ti P_{}(t)=\det \bm \ti u_{1}(t)& \ti u_{2}(t)\\ \ti v_{1}(t)&\ti v_{2}(t)\em$$
doesn't vanish, by computing
$$d(t)=\ti u_{1}(t)\ti v_{2}(t)-\ti u_{2}(t)\ti v_{1}(t)=\sum_{|k|\leq 2N}d_{3k+1}e^{i(2\pi g^\approx)(3k+1)t},$$
and checking the dominant diagonal condition:
\be |d_{1}|>\sum_{0<|k|\leq 2N}|d_{3k+1}|.\label{eq:14.178}\ee
We have the following estimates on $\ti P$.
\begin{prop}[Numerics] \label{prop:propertiesoftiP}The $T_{\approx}$-periodic map $\ti P:\R\to GL(2,\C)$ satisfies:
\ 

\begin{enumerate}
\item $\ti P(0)\in GL(2,\R)$.
\item For all $t\in\R$,
$\|\ti P(t)\|_{op}\leq 2.6.$
\item  For all $t\in\R$, $|\det \ti P(t)|\geq 0.3$.
\item For all $t,s\in [0,T_{\approx}]$, $\|\ti R(t,s)\|\leq 29$.
\item One has
$$\ti P(0)=\bm 1.062471& 1.246040\\ 0.076757 & 0.379300\em \pm 10^{-6}.$$
\item One has 
\begin{align*}
&\|\ti u_{1}\|_{\cO(I_{\nu})}\leq 1.39,\quad \|\ti u_{2}\|_{\cO(I_{\nu})}\leq 1.54,\\ & \|\ti v_{1}\|_{\cO(I_{\nu})}\leq 1.02,\quad \|\ti v_{2}\|_{\cO(I_{\nu})}\leq 1.2
\end{align*}
\item One has $\ti v_{2,1}\approx -0.237$.
\end{enumerate}

\end{prop}
\begin{proof}For example, the first estimate is obtained by evaluating the sums
\begin{align*}
&\sup_{t\in\R}|\ti u_{j}|\leq \sum_{|k|\leq N} |\ti u_{j,3k}|\\
&\sup_{t\in\R}|\ti v_{j}|\leq \sum_{|k|\leq N} |\ti v_{j,3k}|.
\end{align*}

\medskip 
The second estimate is based on (\ref{eq:14.178}). See also (\ref{15.208}) for a more precise result.

\medskip The third estimate comes from $\ti R(t,s)=\ti P(t)\diag(e^{2\pi i t\ti\l_{1}} ,e^{2\pi i t \ti \l_{2}})\ti P(s)^{-1}$, the previous estimate and the fact that the imaginary parts of $\ti \l_{1},\ti \l_{2}$ have absolute value $\leq 10^{-6}$.

\medskip The third  assertion just needs the computations of $\ti P(0)$ {\it via} 
\begin{align*}
&\ti u_{j}(0) =\sum_{|k|\leq N} \ti u_{j,3k}\\
&\ti v_{j} (0)= \sum_{|k|\leq N} \ti v_{j,3k}.
\end{align*}
\end{proof}

\subsubsection{$\ti R$ is a good approximation of $R_{A_{\approx}}$}
Let 
$$\ti R(t,s)=\ti P(t) e^{(t-s)\ti M}\ti P(s)^{-1},\qquad \ti R(t)=\ti R(t,0).$$

From (\ref{eq:10.40})
one gets (with $E(t)=(E_{ij}(t))$)
$$\frac{d}{dt}(\ti P(t) e^{t\ti M})=A_{\approx}(\ti P(t)e^{t\ti M})+E(t)
$$
hence
\be 
\left\{
\begin{aligned}&\frac{d}{dt}\ti R_{}(t,s)=A_{\approx}(t) \ti R_{}(t,s)+E(t)\ti P(s)^{-1}\\
&\ti R(s,s)=I 
\end{aligned}
\right.
\label{eq:10.41n}
\ee
with 
$$\sup_{t\in [-10T_{\approx},10T_{\approx}]}\|E_{}(t)\ti P(0)^{-1}\|=\e_{6}\times \|\ti P(0)^{-1}\|\leq 10^{-5}.$$
\begin{lemma}\label{lemma:10.9}One has for $ t,s\in \R$, $|s-t|\leq T_{\approx}$,
$$
\| \ti R(t,s)^{-1}R_{A_{\approx}}(t,s)-I\|\leq 2.5\times 10^{-4}.
$$
\end{lemma}
\begin{proof}
Fix $s$ and let 
$$\Delta(t)=\ti R(t,s)^{-1}R_{A_{\approx}}(t,s).$$
Because of (\ref{eq:10.41n}) and of  $dR_{A_{\approx}}(t,s)/dt=A_{\approx}(t)R_{A_{\approx}}(t,s)$ one has 
\begin{align*}\frac{d}{dt}\D(t)&=-\ti R(t,s)^{-1}\frac{d\ti R(t,s)}{dt}\ti R(t,s)^{-1}R_{A_{\approx}}(t,s)+\ti R(t,s)^{-1}A_{\approx}(t)R_{A_{\approx}}(t,s)\\
&=-\ti R(t,s)^{-1}\biggl(A_{\approx}(t) \ti R_{}(t,s)+E(t)\ti P(s)^{-1}\biggr)\ti R(t,s)^{-1}R_{A_{\approx}}(t,s)\\
&\hskip 1cm +\ti R(t,s)^{-1}A_{\approx}(t)R_{A_{\approx}}(t,s)\\
&=-\ti R(t,s)^{-1}E(t)\ti P(s)^{-1}\ti R(t,s)^{-1}R_{A_{\approx}}(t,s)\\
&=-\ti E(t)\D(t)
\end{align*}
with
\begin{align*}\ti E(t)&=\ti R(t,s)^{-1}E(t)\ti P(s)^{-1}\\
&=\ti P(s)e^{-(t-s)\ti M}\ti P(t)^{-1}E(t)\ti P(s)^{-1}.
\end{align*}
Using Proposition \ref{prop:propertiesoftiP} and estimates (\ref{esttilambda}) one gets
$$\sup_{t\in\R}\|\ti E(t)\|\leq 195\times\sup_{t} \|E(t)\|\leq 2\times 10^{-4}.$$
Besides, $\D(0)=I$. We thus get from Gronwall's inequality and the fact that $T_{\approx}\leq 1.21$
\begin{align*}\forall t\in [0,T_{\approx}],\qquad \|\D(t)\| &\leq \exp\biggl( \int_{0}^t \| \ti E(u)\|du\biggr)\\
&\leq \exp\biggl( 1.21\sup_{u\in\R}\| \ti E(u)\|du\biggr)\\
&\leq \exp(2.42\times 10^{-4})
\end{align*}
hence, since $\D(t)-I=-\int_{0}^t \ti E(s)\D(s)ds$, 
$$\|\D(t)-I\|\leq 1.21\times  2\times 10^{-4}\times \exp(2.42\times 10^{-4})\leq 2.5\times 10^{-4}.$$
This proves the result.

\end{proof}

Note that Lemma \ref{lemma:10.9} and Proposition \ref{prop:propertiesoftiP} imply that for $t,s\in [0,T_{\approx}]$
\be \|R_{A_{\approx}}(t,s)\|_{}\leq 23.\label{eq:10.46nante}\ee

\begin{cor}\label{cor:15.25}One has for $t,s\in I$
$$\sup_{ t\in I}\|\ti R(t,s)-R_{A_{\approx}}(t,s)\|\leq 5\times 10^{-4}.
$$
\end{cor}
\begin{proof} Fix $s$. The $j$-column vector of $\ti R(t,s)$ and $ R_{A_{\approx}}(t,s)$ are respectively $\ti R(t,s)e_{j}$ and $R_{A_{\approx}}(t,s)e_{j}$. They satisfy
$$ 
\left\{
\begin{aligned}&\frac{d}{dt}(\ti R_{}(t,s)e_{j})=A_{\approx}(t) (\ti R_{}(t,s)e_{j})+E(t)\ti P(s)^{-1}e_{j}\\
&\ti R(s,s)e_{j}=e_{j} 
\end{aligned}
\right.
$$
hence
$$\ti R(t,s)e_{j}=R_{A_{\approx}}(t,s)e_{j}+\int_{s}^t R_{A_{\approx}}(t,u)E(u)\ti P(s)^{-1}e_{j}du.$$
The estimate (\ref{eq:10.46nante}) shows that for $t\in I$
$$\|\ti R(t,s)e_{j}-R_{A_{\approx}}(t,s)e_{j}\|\leq |T_{\approx}|\times 23\times \sup_{I}\|E\|\times \frac{2.6}{0.3}\times 2^{1/2}\leq 3.53\times 10^{-4}.$$
\end{proof}

\subsubsection{Proof of Proposition \ref{prop:controlresolvent} }\label{sec:15.6.5} This is a consequence of Proposition \ref{prop:n15.22}, equation (\ref{eq:10.40})  and Proposition \ref{prop:propertiesoftiP}.
Estimate (\ref{15.327}) is a (numerical) computation.

  \hfill $\Box$

\subsubsection{Proof of Proposition \ref{prop:comptiRRsuA}}\label{sec:15.6.6}
As we saw in Subsection \ref{subsec:15.5.4}, once we know the eigenvalues  of $R_{A_{\approx}}(T_{\approx},0)$ are distinct\footnote{A fact that is ensured by the estimate of Lemma \ref{lemma:10.9}. By taking $t=T_{\approx}$ and using the fact that $P_{\approx}(T_{\approx})=P_{\approx}(0)$ we see that the matrix $e^{tM_{\approx}}$ is conjugate to a matrix that has up to an error $10^{-3}$ a separated spectrum. The same argument shows that the eigenvalues of $M_{\approx}$ are on a $10^{-3}$-neighborhood of the unit circle.}\label{footnote:24} we can find a {\it diagonal} matrix $M_{\approx}=\diag(\l_{\approx,1},\l_{\approx,2})$ and a gauge transformation $P_{\approx}:\R_{\nu}/(T_{\approx}\Z)\to GL(2,\C)$ such that 
$$R_{A_{\approx}}(t,s)=P_{\approx}(t)e^{(t-s)M_{\approx}}P_{\approx}(s)^{-1}$$
and we can impose that \be P_{\approx}(0)=\ti P(0).\ee

Note that from
Corollary \ref{cor:15.25} one has 
\be \|\ti P(t)e^{t\ti M_{}}e^{-tM_{\approx}}-P_{\approx}(t)\|\leq \|e^{-tM_{\approx}}\|\times 5\times 10^{-4}\leq  5.1\times 10^{-4},\label{15.344}\ee
 hence $\|P_{\approx}(t)\|\leq 2.6$ (see Proposition \ref{prop:propertiesoftiP} and the preceding footnote).

We can write from Corollary \ref{cor:15.25}  (with $(t,s)=(0,t)$)
$$e^{-t\ti M_{}}\ti P_{}(t)^{-1}=e^{-tM_{\approx}}P_{\approx}(t)^{-1}+E_{1}(t)$$
with $\sup_{t\in I}\|E_{1}(t)\|\leq 5\times 10^{-4}$.
Hence
\begin{align*}(\ti P(t)^{-1}P_{\approx}(t))&=e^{-t(M_{\approx}-\ti M)}+e^{t \ti M}E_{1}(t)P_{\approx}(t)\\
&=\diag(e^{2\pi i t(\ti \l_{1}-\l_{\approx,1})},e^{2\pi i t(\ti\l_{2}-\l_{\approx,2} )})+E_{2}(t)
\end{align*}
with $E_{2}(t)=e^{t \ti M}E_{1}(t)P_{\approx}(t)$ satisfying $\sup_{t\in I}\|E_{2}(t)\|\leq 5\times 10^{-4}\times 2.6\leq 1.3\times 10^{-3}.$
Since  the function on the left hand side of this equation is $T_{\approx}$-periodic and equal to identity when $t=0$,  we get 
$$\begin{cases}
&|\ti \l_{1}-\l_{\approx,1}|\leq (2\pi |T_{\approx}|)^{-1}\times 2.6\times 10^{-3}\leq 4.2\times 10^{-4} \\
&|\ti \l_{2}-\l_{\approx,2}|\leq (2\pi |T_{\approx}|)^{-1}\times 2.6\times 10^{-3} \leq 4.2\times 10^{-4}
\end{cases}
$$
and then, for any $t\in I$,
$$\|(\ti P(t)^{-1}P_{\approx}(t))-I\| \leq 3.9\times 10^{-3}.$$
This also yields
$$\|P_{\approx}(t)-\ti P(t)\|\leq 2.6\times 3.9\times 10^{-3}\leq 1.1\times 10^{-2}.$$

This  is the conclusion of Proposition \ref{prop:comptiRRsuA}.

\subsubsection{Proof of Proposition \ref{prop:controlresolventRApprox} }\label{sec:15.6.7} A direct application of Propositions  
\ref{prop:controlresolvent}
 and
\ref{prop:comptiRRsuA}.

\subsubsection{Proof of Proposition \ref{prop:15.12}}\label{sec:15.6.8} Equation (\ref{Papprox0isreal}) is a consequence of (\ref{e:15.348}) of Proposition \ref{prop:controlresolvent} and of the fact  $P_{\approx}(0)=\ti P(0)$ stated in  Proposition \ref{prop:comptiRRsuA}.
Estimates (\ref{e15.322ante}) (see (\ref{15.190.1})) are a consequence of  (\ref{15.327}) of Proposition \ref{prop:controlresolvent} and of (\ref{prop:15.19-1}) of Proposition \ref{prop:comptiRRsuA}. Estimate (\ref{e15.326}) is (\ref{prop:15.17-1}) of Proposition  \ref{prop:controlresolventRApprox}. 

We now prove Estimate (\ref{15.190.2})-(\ref{e15.322}) on $\ti \mu$.

Recall 
$$R_{A_{\approx}}(t,s)=\ti P(t)e^{(t-s)\ti M}\ti P(s)^{-1}(I+E_{3}(t,s))$$
with $\|E_{3}(t,s)\|\leq 2.6\times 10^{-4}$ (see Lemma \ref{lemma:10.9})
\begin{align*}
&\mu_{\approx}(\cdot)=2\pi i \int_{0}^\cdot R_{A_{\approx}}(\cdot,s)\bm 1\\0\em ds\\
&\hat \mu_{\approx}(T_{\approx})=P_{\approx}(T_{\approx})^{-1}\mu_{\approx}(T_{\approx}) \qquad (cf.\ (\ref{15.190.2}),\ (\ref{def:hatmuapprox}))
\end{align*}
$$| \ti \l_{1}|\leq 10^{-6},\qquad |\ti \l_{2}-(1-g_\approx)|\leq 10^{-6}$$
(cf. Proposition \ref{prop:n15.22}).
Hence
\begin{align}
\mu_{\approx}(T_{\approx})&=2\pi i \int_{0}^{T_{\approx}} R_{A_{\approx}}(\cdot,s)\bm 1\\0\em ds\notag\\
&=2\pi i \int_{0}^{T_{\approx}}\ti P(T_{\approx})e^{(T_{\approx}-s)\ti M}\ti P(s)^{-1}(I+E_{3}(t,s))\bm 1\\0\em ds.
\label{15.206}
\end{align}
Because, $\det R_{A_{\approx}}(t,0)=e^{2\pi it}$ one has
\begin{align*}
\det \ti R(t,0)&=\det R_{A_{\approx}}(t,0)\det(I+E_{3}(t,0))^{-1}\\
&=e^{2\pi i t}e^{\nu(t)}
\end{align*}
with $\|\nu\|_{C^0(I)}\leq 10^{-3}$.
Hence
$e^{2\pi i t}e^{\nu(t)}=\det \ti P(t)\det e^{t\ti M}\det\ti P(0)^{-1}$
and
\begin{align}
\det \ti P(t)&=e^{2\pi i t}e^{\nu(t)}\det e^{-t\ti M}\det \ti P(0)\notag\\
&=e^{2\pi i t}e^{\nu_{1}(t)}e^{2\pi it(g_{\approx} -1)}\det \ti P(0) \notag\\
&=e^{\nu_{1}(t)}e^{2\pi itg_\approx}\det \ti P(0)\label{15.208}
\end{align}
with $\sup_{t}|\nu_{1}(t)-\nu(t)|\leq 10^{-5}.$
Hence
$$\ti P(t)^{-1}=e^{-\nu_{1}(t)}\frac{e^{-2\pi itg_\approx}}{\det \ti P(0)}\bm \ti v_{2}(t)& -\ti u_{2}(t)\\ -\ti v_{1}(t) & \ti u_{1}(t)\em.$$
This and (\ref{15.206}) give 
\begin{multline*}\mu_{\approx}(T_{\approx})= \frac{2\pi i }{\det \ti P(0)}\ti P(T_{\approx})\int_{0}^{T_{\approx}}e^{-2\pi isg_\approx}e^{-\nu_{1}(s)}\bm e^{2\pi i \ti\l_{1}(T_{\approx}-s)}&0\\ 0&e^{2\pi i \ti \l_{2}(T_{\approx}-s)}\em \\ \bm \ti v_{2}(t)& -\ti u_{2}(t)\\ -\ti v_{1}(t) & \ti u_{1}(t)\em (I+E_{3}(T_{\approx},s))\bm  1\\ 0\em  ds\end{multline*}
i.e. ($\ti P$ is $T_{\approx}$-periodic)
\begin{multline*}\bm\ti \mu_{1}\\ \ti \mu_{2}\em= \frac{1}{\det \ti P(0)}\int_{0}^{T_{\approx}}e^{-2\pi isg_\approx}e^{-\nu_{1}(s)}\bm e^{2\pi i \ti\l_{1}(T_{\approx}-s)}&0\\ 0&e^{2\pi i \ti \l_{2}(T_{\approx}-s)}\em \\ \bm \ti v_{2}(t)& -\ti u_{2}(t)\\ -\ti v_{1}(t) & \ti u_{1}(t)\em (I+E_{3}(T_{\approx},s)\bm  1\\ 0\em  ds\end{multline*}
hence (cf. the numerical estimates of Proposition \ref{prop:propertiesoftiP})
$$\bm\ti \mu_{1}\\ \ti \mu_{2}\em
= \frac{1 }{\det \ti P(0)}\int_{0}^{T_{\approx}}\bm e^{-2\pi isg_\approx}v_{\approx,2}(s)\\ -e^{-2\pi isg_\approx}e^{2\pi i(T_{\approx}-s)(1-g_\approx)}v_{\approx,1}(s)ds\em ds+\ti E_{3}
$$
with $\|\ti E_{3}\|\leq 3\times 10^{-3}$.
Remembering 
$$\ti v_{j}(t)\approx  \ti v_{j}(t)=\sum_{|k|\leq N}\ti v_{j,3k+1}e^{2\pi i(3k+1) g_\approx t}$$
we see that 
$$
\begin{aligned}
&\ti \mu_{1}=\frac{T_{\approx}}{\det \ti P(0)}\ti v_{2,1}+_{\leq}\|\ti E_{3}\|\\
&|\ti \mu_{2}|\leq \frac{1.1\times |T_{\approx}|}{\det \ti P(0)}\times \|\ti v_{1}\|_{C^0}+\|\ti E_{3}\|\leq 6.
\end{aligned}
$$
which is (\ref{e15.322}).

\begin{rem}
Note that 
$$\ti \mu_{2}\approx -\frac{1}{\det \ti P(0)}\int_{0}^{T_{\approx}}\sum_{|k|\leq N}\ti v_{1,3k+1}e^{-2\pi i sg_{\approx}}e^{2\pi i (T_{\approx}-s)(1-g_{\approx})}e^{2\pi i (3k+1)g_{\approx}s}ds$$
Hence
\begin{align*}\ti \mu_{2}&\approx -\frac{1}{\det \ti P(0)}\sum_{|k|\leq N}e^{2\pi iT_{\approx}(1-g_{\approx}) }\frac{e^{2\pi i T_{\approx}((3k+1)g_{\approx}-1)} -1}{2\pi i ((3k+1)g_{\approx}-1)}\ti v_{1,3k+1}\\
&\approx \frac{e^{2\pi iT_{\approx}}-1}{2\pi i \det \ti P(0)}\sum_{|k|\leq N}\frac{\ti v_{1,3k+1}}{(3k+1)g_{\approx}-1}.
\end{align*}
One finds
\be \sum_{|k|\leq N}\frac{\ti v_{1,3k+1}}{(3k+1)g_{\approx}-1}\approx1.1308272097663494\label{e15.371}\ee
$$\ti \mu_{2}\approx-0.39-0.55 i.$$
\end{rem}

Numerics show that 
$$\ti v_{2,1}\approx-0.237 \ne 0.$$
and 
$$\ti \mu_{1}\approx 0.092.$$

\subsubsection{Numerics}\label{subsec:numerics}
They are done with the vector field 
$$ (z,w)\mapsto i \bm (1-\tau)z+(1/2) z^2-(1/3)w^3\\ \tau w-zw\em.
$$
which is conjugate  (by replacing $z$ by $z-\tau$) yields to  the vector field
$$ \hat X_{\hat \tau}(z,w)=i \bm \hat \tau+z+(1/2)z^2-(1/3)w^3\\ - zw\em
$$
The constant $w_{1}$ is equal to $1.4$ and $\tau$ is equal to 1 (hence $\hat \tau=1/2$).

We choose $N=12$. The Newton method is iterated 8 times.

We find $g_\approx$, $z^\approx_{3k}$, $w_{3k+1}^\approx$
that satisfies
$$\| \cF_{7,3/2}(g_\approx,z^\approx,w^\approx)\|_{l^\infty}\leq 10^{-15}.$$
Their values are
\begin{verbatim}
omega= (-0.8345538969679759+0j) %%This is g^\approx%%

z_ -21 = (-9.361639711342464e-06+0j)
z_ -18 = (2.4946528854417837e-05+0j)
z_ -15 = (-0.00044113154982777157+0j)
z_ -12 = (0.0012949576061869136+0j)
z_ -9 = (-0.020506343991298234+0j)
z_ -6 = (0.06955033704249342+0j)
z_ -3 = (-0.9357340999201847+0j)
z_ 0 = (1.8345538957052878+0j)
z_ 3 = (0.3659326628185039+0j)
z_ 6 = (0.08012016222461271+0j)
z_ 9 = (0.017550297921137367+0j)
z_ 12 = (0.0038446642861876637+0j)
z_ 15 = (0.0008422305875809696+0j)
z_ 18 = (0.00018448395265462586+0j)
z_ 21 = (3.988946876386859e-05+0j)
 
w_ -20 = (3.405792118940113e-06+0j)
w_ -17 = (2.374829735564663e-05+0j)
w_ -14 = (0.00016772866693900553+0j)
w_ -11 = (0.0012041020854550999+0j)
w_ -8 = (0.009039462814890112+0j)
w_ -5 = (0.08079546301120819+0j)
w_ -2 = (0.5426705070398815+0j)
w_ 1 = (1.5+0j)
w_ 4 = (0.2221676438547309+0j)
w_ 7 = (0.04073760439514487+0j)
w_ 10 = (0.007950275533473043+0j)
w_ 13 = (0.0015986467259251906+0j)
w_ 16 = (0.000327155574191714+0j)
w_ 19 = (6.76592753362816e-05+0j)
w_ 22 = (1.396379831468273e-05+0j)
\end{verbatim}

We find that 
$$\| (I-\cF_{12,1.4}(g^{\approx},z^{\approx},w^{\approx})\|_{l^1}\leq 10^{-6}.$$

\section{Numerics}
Numerics were done in Python.
\subsection{Approximate solution}
As we've mentioned finding approximate periodic solutions for the vector field
$$X_{\tau} (z,w)\mapsto2\pi  i \bm (1-\tau)z+(1/2) z^2-(1/3)w^3\\ \tau w-zw\em.
$$
or the closely related one ($\hat\tau=\tau-\tau^2/2$)
$$ \hat X_{\hat \tau}(z,w)=2\pi i \bm \hat \tau+z+(1/2)z^2-(1/3)w^3\\ - zw\em
$$
(obtained by conjugation by $(z,w)\mapsto (z-\tau,w)$), 
leads to the algebraic systems  (\ref{setofeqnnohat}) or (\ref{setofeqn}). The first one is a little bit simpler to implement on computers and this is the one we worked with. However the only change needed to pass from the first one to the second is to shift the $z$ variable by $\tau$ (that we choose equal to 1).

So, the solutions 
$$g_{\approx,\quad}{z}^\approx=({ z}^\approx_{3k})_{|k|\leq N},\quad {w}^\approx=({w}^\approx_{3k+1})_{0<|k|\leq N}$$
obtained for (\ref{setofeqnnohat})
are related to the solutions of (\ref{setofeqn})
$$g_{\approx,\quad}{z}^\approx=({\hat z}^\approx_{3k})_{|k|\leq N},\quad {w}^\approx={w}^\approx_{3k+1})_{0<|k|\leq N}$$
by 
$\hat z_{3k}=z_{3k}$ for $k\ne 0$ and $\hat z_{0}=z_{0}-\tau=z_{0}-1$ ($\tau=1$).

\bigskip 
The constant $w_{1}$ is equal to $1.4$ and $\tau$ is equal to 1 (hence $\hat \tau=1/2$).

We choose $N=12$. The Newton method is iterated 8 times.

Here are the numerical values we found for $g_{\approx}$ and the sequences $${\hat z}^\approx=({\hat z}^\approx_{3k})_{|k|\leq N},\quad \mathring{w}^\approx=(\mathring{w}^\approx_{3k+1})_{0<|k|\leq N}$$ of Theorem \ref{eq:numerics1} for 
$$N=12$$
and with 
$$w_{1}=1.4.$$

\medskip

\begin{verbatim}
g_approx= (-0.8345538969681955+0j)
 
z_ -36 = (2.009665997157902e-09+0j)
z_ -33 = (-4.219381983934749e-08+0j)
z_ -30 = (7.064447933123985e-08+0j)
z_ -27 = (-1.3006533396367273e-06+0j)
z_ -24 = (2.4893614590978225e-06+0j)
z_ -21 = (-4.015679003267962e-05+0j)
z_ -18 = (8.628253145015947e-05+0j)
z_ -15 = (-0.0012417234175790455+0j)
z_ -12 = (0.002963555709394564+0j)
z_ -9 = (-0.038155511936364246+0j)
z_ -6 = (0.10521522081049815+0j)
z_ -3 = (-1.1509120242609674+0j)
z_ 0_hat = (0.8345538969681958+0j)
z_ 3 = (0.2975168076254836+0j)
z_ 6 = (0.05296176874514345+0j)
z_ 9 = (0.009432254645987877+0j)
z_ 12 = (0.0016799649256645638+0j)
z_ 15 = (0.0002992221228330934+0j)
z_ 18 = (5.329548814557536e-05+0j)
z_ 21 = (9.492675514726833e-06+0j)
z_ 24 = (1.6907817298191817e-06+0j)
z_ 27 = (3.01152707141572e-07+0j)
z_ 30 = (5.363911495415405e-08+0j)
z_ 33 = (9.553782775740951e-09+0j)
z_ 36 = (1.6886903123768018e-09+0j)
 
w_ -35 = (2.4701689332964987e-09+0j)
w_ -32 = (1.3731258092209175e-08+0j)
w_ -29 = (7.639500625230014e-08+0j)
w_ -26 = (4.265475691098463e-07+0j)
w_ -23 = (2.3924987409600263e-06+0j)
w_ -20 = (1.3498311435359129e-05+0j)
w_ -17 = (7.673805064236774e-05+0j)
w_ -14 = (0.00044065115012943804+0j)
w_ -11 = (0.0025719293006946273+0j)
w_ -8 = (0.0156981486787134+0j)
w_ -5 = (0.11407831412188575+0j)
w_ -2 = (0.6229635928580461+0j)
w_ 1 = (1.4+0j)
w_ 4 = (0.16858848756603534+0j)
w_ 7 = (0.02513349734558194+0j)
w_ 10 = (0.00398795240576515+0j)
w_ 13 = (0.0006519759975477376+0j)
w_ 16 = (0.00010847984043355156+0j)
w_ 19 = (1.8259870930124524e-05+0j)
w_ 22 = (3.098937545038489e-06+0j)
w_ 25 = (5.291618963400072e-07+0j)
w_ 28 = (9.078732801177921e-08+0j)
w_ 31 = (1.5635204524629112e-08+0j)
w_ 34 = (2.699311644086258e-09+0j)
w_ 37 = (4.637527246643453e-10+0j)
\end{verbatim}

\bigskip
We also computed the $l^1$-norm of these sequences.

\medskip
\begin{verbatim}
l1 of Fourier z_hat= 2.495127140332043

l1 of Fourier w= 2.3543381748256222
\end{verbatim}

\bigskip
For the error of approximation (\ref{errorapprox})
$$ \biggl\| (I-\cP_{N})(\cF_{N,w_{1}}(g_{\approx},({\hat z}^\approx,w^\approx)))\biggr\|_{l^1(\Z)}\qquad (N=12)$$
we indeed  computed the sum for $0\leq |k|\leq 24$ of the modules of the coefficients of the preceding expression as it appears that the sum of the other coefficients is $\leq 10^{-8}$.

\begin{verbatim}
For Nprime=24: L1hh= 7.279312515841349e-08
\end{verbatim}

\medskip We also computed the value of $p_{1}(0)$ and $\hat p_{1/2}(0)$. We found

%max_operator_norm_P= 2.5931279468167325
\begin{verbatim}
z_at_0= (1.1144256218278379+0j)
z_hat_at_0=(0.1144256218278379+0j)
w_at_0= (2.3543381748256222+0j)
\end{verbatim}

\subsection{Approximate resolvent}
We refer here to Proposition \ref{prop:controlresolvent} on the properties of the approximate resolvent.

We found  an approximate solution to the system (\ref{eq:10.38n})
$$
\left\{
\begin{aligned}
&(3kg_\approx-1+\l_{\approx}) u^\approx_{3k}=\sum_{l_{1}+l_{2}=k}z^\approx_{3l_{1}}u^\approx_{3l_{2}}-\sum_{l_{1}+l_{2}+l_{3}=k-1}w^\approx_{3l_{1}+1}w^\approx_{3l_{2}+1}v^\approx_{3l_{3}+1}\\
&[(3k+1)g_\approx+\l_{\approx}]v^\approx_{3k+1}=-\sum_{l_{1}+l_{2}=k}w^\approx_{3l_{1}+1}u^\approx_{3l_{2}}-\sum_{l_{1}+l_{2}=k}z^\approx_{3l_{1}}v^\approx_{3l_{2}+1}.
\end{aligned}
\right.
\label{eq:10.38nbis}
$$
by projecting it on the 50-dimensional  vector space $(u_{3k})_{|k|\leq 12},(v_{3k+1})_{|l|\leq 12}$. We found 50 eigenvalues very close to $\{0, 1-g_{\approx}\}+3g_{\approx}\Z$ and 50 eigenvectors. We selected the eigenvectors
\begin{align*}
&((u_{\approx,0,3k})_{0\leq |k|\leq N},(v_{\approx,0,3k+1})_{0\leq |k|\leq N})\\
&((u_{\approx,1-g_{\approx},3k})_{0\leq |k|\leq N},(v_{\approx,1-g_{\approx},3k+1})_{0\leq |k|\leq N})
\end{align*}
corresponding respectively  to the (approximate) eigenvalues 0 and $1-g_{\approx}$.

We could check that among the eigenvalues one finds
\begin{verbatim}

0_approx
a0= (-7.399491394281129e-14+0j)


1-g_approx
a1mg= (1.8345538969682487+0j)
\end{verbatim}

and for the eigenvectors

\begin{verbatim}

u_approx,0
b0z=
[-1.72031100e-08-4.97235591e-24j  3.31087937e-07-1.11033543e-23j
 -5.03941422e-07-1.20160857e-22j  8.35037344e-06+8.23672642e-22j
 -1.42062606e-05+5.45696411e-21j  2.00520537e-04+5.02018284e-20j
 -3.69297151e-04+3.99073385e-20j  4.42890846e-03+4.32153250e-19j
 -8.45619356e-03+4.94919061e-19j  8.16545459e-02+1.25142372e-18j
 -1.50110266e-01-1.73301733e-18j  8.21001510e-01+0.00000000e+00j
  1.68613829e-14-2.99986840e-19j  2.12233206e-01+2.15240155e-19j
  7.55604098e-02+1.12780841e-19j  2.01854576e-02+2.24801243e-20j
  4.79360268e-03+5.84086032e-21j  1.06724845e-03+1.30376135e-21j
  2.28109579e-04+3.53899882e-22j  4.74011092e-05-1.23306864e-22j
  9.64893457e-06-2.88916784e-22j  1.93344182e-06+4.19344232e-22j
  3.82633889e-07-1.88707132e-20j  7.49669557e-08+6.38977705e-20j
  1.44554993e-08+4.85699224e-19j]

v_approx,0
b0w=
[-2.05577354e-08+2.04349012e-22j -1.04481857e-07+1.15408887e-22j
 -5.26797302e-07-3.99825208e-22j -2.63706839e-06-3.23522387e-21j
 -1.30845903e-05-1.13694848e-20j -6.41933460e-05-5.28383826e-20j
 -3.10198903e-04-3.89273400e-20j -1.46691016e-03-3.51508763e-19j
 -6.72716817e-03+7.94292922e-20j -2.98620074e-02-2.91809372e-19j
 -1.35629347e-01-2.89270944e-18j -2.96260148e-01-5.88493617e-19j
  3.32896025e-01+4.14012690e-19j  1.60349821e-01+2.25731924e-19j
  4.18342068e-02+5.40901327e-20j  9.48266788e-03+1.29442095e-20j
  2.01537345e-03+2.56013485e-21j  4.12714373e-04+5.24600079e-22j
  8.24958075e-05+1.26329795e-22j  1.62112341e-05-2.54524065e-22j
  3.14564092e-06-2.35621809e-21j  6.04454812e-07+5.02958938e-21j
  1.15251300e-07+1.73502022e-20j  2.18229023e-08-1.60590099e-19j
  4.08008028e-09-5.97365680e-21j]


u_approx,1-g_approx
b1mgz=
[ 1.33452534e-08+1.52834008e-22j -1.62242079e-08+1.18937050e-22j
  3.46753821e-07+1.90070184e-22j -4.92436700e-07+2.44336095e-22j
  8.68594931e-06-5.63138947e-22j -1.39164413e-05+1.35527338e-21j
  2.06568508e-04+9.06480021e-21j -3.60125200e-04-1.78779511e-20j
  4.48596118e-03-2.48638293e-20j -8.10180992e-03-2.68363283e-19j
  7.95073867e-02+1.00745743e-19j -1.34257531e-01-2.80361981e-20j
  6.46437882e-01+0.00000000e+00j  2.58327315e-01+3.13667058e-21j
  2.80357445e-01+7.07403882e-21j  8.95949326e-02+2.57489688e-21j
  2.30047662e-02+6.74455803e-22j  5.35202927e-03+1.56191565e-22j
  1.17669894e-03+7.78074642e-23j  2.49380596e-04-2.51474704e-23j
  5.15059672e-05+9.77123630e-22j  1.04363641e-05-2.43526554e-21j
  2.08368794e-06+2.28121727e-20j  4.11174333e-07-9.19197842e-21j
  7.96564098e-08-7.91084927e-20j]


v_approx,1-g_approx
b1mgw=
[-4.42131111e-09-8.80289523e-23j -2.23704563e-08+5.86807675e-23j
 -1.13236770e-07-1.44337648e-22j -5.69869350e-07-7.52627419e-22j
 -2.84648366e-06+1.61903506e-21j -1.40850038e-05-8.48426193e-21j
 -6.88547890e-05-3.37338545e-21j -3.31113211e-04-1.89913809e-21j
 -1.55502813e-03-1.63808885e-20j -7.05674813e-03+2.81211655e-19j
 -3.07553880e-02+3.26177320e-20j -1.32652428e-01-6.71596546e-20j
 -2.37538786e-01-2.17101880e-20j  5.38728402e-01+1.37140348e-20j
  1.90218540e-01+5.44638482e-21j  4.71219533e-02+1.34009876e-21j
  1.04546138e-02+3.07268793e-22j  2.19600137e-03+5.91926615e-23j
  4.46382214e-04-8.92814596e-24j  8.87710754e-05-1.60666972e-23j
  1.73794160e-05+5.60166953e-22j  3.36272826e-06-6.57931895e-21j
  6.44717741e-07-2.38041665e-20j  1.22611121e-07-6.95602759e-20j
  2.29985652e-08-1.99887496e-19j]
\end{verbatim}

\medskip We computed the accuracy of approximation for these values by projecting the system (\ref{eq:10.38n}) on the vector space corresponding to $|l|\leq 24$ because a quick check of the coefficients shows that this involves an error $\leq 10^{-8}$.
We found 

\begin{verbatim}
Accuracy of approx. reolvent ee1= 6.785329966569467e-07

Accuracy of approx. reolvent ee2= 8.19807942579948e-07
\end{verbatim}

\medskip The gauge transformation $\ti P(t)$ is then 
$$\ti P(t)=\bm\ti u_{1}(t)& \ti u_{2}(t)\\ v_{1}(t)& v_{2}(t)\em$$
with 
\begin{align*}&\ti u_{1}(t)=\sum_{|k|\leq N} u_{\approx,0,3k}e^{2\pi i g_{\approx}(3k)t}\\
&\ti u_{2}(t)=\sum_{|k|\leq N} u_{\approx,1-g_{\approx},3k}e^{2\pi i g_{\approx}(3k)t}\\
&\ti v_{1}(t)=\sum_{|k|\leq N} v_{\approx,0,3k+1}e^{2\pi i g_{\approx}((3k+1))t}\\
&\ti v_{2}(t)=\sum_{|k|\leq N} v_{\approx,1-g_{\approx},3k+1}e^{2\pi i g_{\approx}((3k+1))t}.
\end{align*}
We found for the Fourier coefficients of $\det \ti P$
\begin{verbatim}
Fourier coeff. of c=det Pc=
[-7.81707304e-09-1.34234391e-22j -3.78515931e-09+2.82630354e-22j
 -1.12178016e-09-2.10406028e-21j -2.92098982e-10+4.50611953e-21j
 -7.19341473e-11-1.60368538e-20j -1.72787022e-11+2.12378401e-20j
 -4.12346101e-12-5.85707840e-20j -9.93125721e-13+1.28428759e-20j
 -2.45329798e-13+4.16589656e-19j -6.34318986e-14+3.69161132e-19j
 -1.94358418e-14+3.06392334e-18j -4.21884749e-15+6.27178668e-19j
  3.07353181e-01+5.31480653e-19j -3.35842465e-15-1.12683883e-19j
 -1.79301018e-14-7.35570884e-20j -2.70755640e-14-1.49291481e-20j
 -2.22828700e-13-4.93550894e-21j -3.88097540e-13-1.57581623e-21j
 -3.59971471e-12-9.20267029e-22j -6.38054041e-12-1.41665789e-22j
 -6.08554037e-11-5.68550052e-21j -1.05531973e-10-1.22057392e-20j
 -9.53665168e-10-1.00842735e-19j -1.82936540e-09-1.85963238e-19j
 -7.94898500e-09-3.42512593e-20j]
\end{verbatim}
which shows that this determinant is almost equal to $3.07$.

\medskip We also computed $\ti P(0)$ and its inverse
\begin{verbatim}
P_at_0= [[(1.0624711709108121+1.1304784753158855e-18j), 
(1.2460400371648754-2.8110993367377925e-19j)],
 [(0.07675705954779727-3.588347080400314e-18j),
  (0.37930020754260807-8.291451300164181e-20j)]]

P_ inverse at_0= [[ 1.2340859 -1.96772526e-17j -4.05409859+6.46701856e-17j]
 [-0.2497357 +1.56023902e-17j  3.45684147-5.06848755e-17j]]
 
 
 l1_of_P= [[1.380372139365217, 1.5315078199907293], 
 [1.0174297536410741, 1.1992521832956424]]
\end{verbatim}

\medskip Finally we computed $\ti P(0)^{-1}X(p_{\approx}(0))$ 

\begin{verbatim}
tildeP(0)_inverse @ X(p_approx(0))
[-3.50973099e+00+5.59543087e-17j  5.48450161e-08-4.45267912e-17j]
\end{verbatim}

\subsection{Locating the invariant annulus ($\tau$ close to 1)}\label{sec:locinvan}
It follows from the preceding discussion that the point 
$$(z^{\approx}_{*},w^{\approx}_{*})=(1.114,2.354)$$
is close to a  point $(z_{*},w_{*})$ lying on an invariant annulus for the map $(h^{\rm bnf}_{\d,\tau'})^{\circ 3}$ for $\tau'=(\tau,\mbeta)=(1,\mbeta)$ and from the results of Sections  \ref{sec:proofmainA:B}, \ref{sec:proofmainA} it is close to some point lying on an invariant annulus for the map $h^{\rm mod}_{\a,\b}$ for 
$$\b=(1/3)+\d \mbeta,\qquad \a_{*}=(1/6)+\d (\tau-1/2)\mbeta,\quad \tau=1.$$
The frequency of $h_{\a,\b}$ on this annulus is approximately
$$\d\mbeta\times(-0.834).$$
If one varies $\tau$, $\tau=1+\Delta\tau$ so that 
$$\b=(1/3)+\d \mbeta,\qquad \a=\a_{*}+\Delta\a=\a_{*}+\d (\Delta\tau) \mbeta,$$
 this frequency becomes (see Corollary \ref{cor15.17} and  note $\pa_{\tau}g(1)=0$, $\pa^2_{\tau}g(1)=-\pa_{\hat \tau}^2\hat g(1/2)\approx0.183$)
\be \d\mbeta\times(-0.834+0.183\times (\Delta \tau)^2/2).\label{eq:16.approxfreq}\ee
Similarly, one can find a point $(z^{\approx}_{1+\Delta\tau},w^{\approx}_{1+\Delta \tau})$ close to  some point  $(z^{\rm bnf}_{1+\Delta\tau},w^{\rm bnf}_{1+\Delta \tau})$ on  the invariant annulus $\cA^{\rm bnf}_{\a_{*}+\Delta\alpha,\beta}$ which is of the form (see Remark \ref{rem-cor15.17}, (\ref{yzeta}) and Lemma \ref{lemma:CPvsFP})
\begin{align*}\bm z^{\approx}_{1+\Delta\tau}\\ w^{\approx}_{1+\Delta\tau}\em
&=\bm z^{\approx}_{*}\\ w^{\approx}_{*}\em+3.68\times ((\Delta\tau)^2/2)\times P_{\approx}(0)\bm 0 \\ 1\em\\
&=\bm z^{\approx}_{*}\\ w^{\approx}_{*}\em+3.68\times ((\Delta\tau)^2/2)\times \bm 1.24 \\ 0.37\em.
\end{align*}
In the $h^{mod}_{\a,\b}$ and $h^{\textrm{Hénon}}_{\b,c}$ model this point becomes (see Theorem \ref{theo:approxbyvf} and (\ref{TtL-1}))
\begin{align*}
&\bm z^{\approx,{\rm mod}}_{\a_{*}+\Delta\a,\beta}\\  z^{\approx,{\rm mod}}_{\a_{*}+\Delta \a,\beta}\em=\iota_{G_{\d,\tau'}}^{-1}\circ \diag((2\pi   (\sqrt{3}/2) \mbeta\d)^{},(2\pi (\sqrt{3}/2) \mbeta\d)^{2/3})\bm z^{\approx}_{1+\Delta\tau}\\w^{\approx}_{1+\Delta\tau}\em\\
&\bm z^{\approx,{\textrm{Hénon}}}_{1+\Delta\tau,\beta}\\ w^{\approx,{\textrm{Hénon}}}_{1+\Delta\tau,\beta}\em=T_{t}\circ L^{-1}\bm z^{\approx,{\rm mod}}_{\a_{*}+\Delta\a,\beta}\\  z^{\approx,{\rm mod}}_{\a_{*}+\Delta \a,\beta}\em
\end{align*}
where $\iota_{G_{\d,\tau'}}$ is obtained from BNF and in first approximation satisfies
\begin{align*}&\iota_{G_{\d,\tau'}}^{-1}=\Phi_{Y}\circ (id+O(\d))\\
&Y(z,w)=Y_{3,0}z^3+Y_{1,2}zw^2+Y_{0,3}w^3 \quad(\textrm{see}\ {\rm Remark}\ \ref{Y-5.43})\\
&Y_{3,0}=\frac{i\mu/3}{1-j}+O(\d)\quad Y_{1,2}=\frac{i\mu j}{j-1}+O(\d),\quad Y_{0,3}=\frac{\mu j^2/3}{j^2-1}+O(\d)\\
&T_{t}:\C^2\ni(x,y)\mapsto (x+t,y+t)\in\C^2,\\
&t=\cos(2\pi\a)=(1/2)-\frac{\sqrt{3}}{2}2\pi \d\malpha+O(\d^2)\\
&L^{-1}=\bm \l_{1}&\l_{2}\\ 1&1\em=\bm 1 &j\\ 1&1\em+O(\d).
\end{align*}
We see that 
$$\Phi_{Y}(z,w)=\bm z\\w \em +\bm 2Y_{1,2}zw+3Y_{0,3}w^2\\ -3Y_{3,0}z^2-Y_{1,2}w^2\em+O^3(z,w)$$
and with 
\begin{align*}
&\d_{\mbeta}=\pi\sqrt{3} \mbeta \d\\
&(a,b)=(z^{\approx}_{*},w^{\approx}_{*})=(1.114,2.354)
\end{align*}
one has
$$\Phi_{Y}(\d_{\mbeta} a, \d_{\mbeta}^{2/3} b)=\bm \d_{\mbeta} a \\ \d_{\mbeta}^{2/3} b\em+\d_{\mbeta}^{4/3}b^2\bm  3Y_{0,3}\\ -Y_{1,2}\em+O(\d^{5/3}).$$
The point 
\begin{align*}&T_{t}\circ L^{-1}\circ \Phi_{Y}(\d_{\mbeta} a,\d_{\mbeta}^{2/3}b)= \bm 1/2\\ 1/2\em +\d_{\mbeta}^{2/3}b\bm j\\ 1\em+O(\d^{4/3}_{\mbeta})\\
&b=2.354
\end{align*}
is a good initial condition (when $\tau$ is say $10^{-4}$ close to 1 and $\mbeta$ is close to 1).

\subsection{Where  to find Exotic rotation domains?}\label{sec:16.4}

\begin{itemize}
\item Fix $\mbeta$ and $\tau$ 
\item Use a first program (Newton method 1) to find approximate solution to (\ref{setofeqnnohat})-(\ref{diffeqXnothat}). This gives some $\omega(\mbeta,\tau)=\mbeta\times g_{0}(\tau)$ and an approximate initial condition by evaluating $z(t)$, $w(t)$ for $t=0$. 
\item One can check (ODE program) that the solutions of 
$$ \begin{cases}
&\frac{1}{2\pi i\mbeta}\dot z=  (1-\tau)z+(1/2)z^2-(1/3)w^3\\
&\frac{1}{2\pi i \mbeta}\dot w= \tau w-zw
\end{cases}
$$
for the initial conditions given by the first program are, to a very good approximation, periodic.
\item  Take the  initial conditions  giving periodic solutions for  the vector field and iterate  the modified Hénon map $h^{\rm mod}_{\alpha,\beta}$ with this initial condition. One may have to modify by hand the {\it initial condition} to find an invariant quasi-periodic curve (i.e. an invariant circle in the  invariant annulus).
\item Come back to the intial Hénon map.
\end{itemize}

\subsection{To find Herman rings}
To find Herman's ring is  numerically more delicate as it is very sensitive to the fact that $\Im \omega=0$.
\begin{itemize}
\item Fix $\mbeta=1/g^*(\tau^*)$.
\item Use a program (Newton method 2) based on (Newton method 1) to find $\tau$ such that the imaginary part of $\omega(\mbeta,\tau)=\mbeta\times g_{0}(\tau)$ is very small.  This gives   an approximate initial condition by evaluating $z(t)$, $w(t)$ for $t=0$. 
\item One can check (ODE program) that the solutions of 
$$ \begin{cases}
&\frac{1}{2\pi i\mbeta}\dot z=  (1-\tau)z+(1/2)z^2-(1/3)w^3\\
&\frac{1}{2\pi i\mbeta }\dot w= \tau w-zw
\end{cases}
$$
for the initial conditions given by the first program are, to a very good approximation, periodic.
\item Take the  initial conditions  giving periodic solutions for  the vector field and  iterate the modified Hénon map $h^{\rm mod}_{\alpha,\beta}$. One may have to modify by hand the {\it frequency} $\malpha$ (equivalently $\tau$)  to find an invariant quasi-periodic curve (i.e. an invariant circle in the  invariant annulus). Usually, the initial condition doesn't have to be changed (the annulus is attracting).
\item Come back to the intial Hénon map.
\end{itemize}

\vskip 1cm
\appendix

\section{Symplectic normalization}\label{sec:3.3}

We recall the notations of Section \ref{sec:11}.

$$\R_{s}=\R+i]-s,s[,\qquad \T_{s}=\T+i]-s,s[$$
$$R_{s}=(]-(2-s),(2+s)[$$
$$R_{s,\rho}=(]-(2-s),(2+s)[+i]-s,s[)\times \bD(0,\rho)$$
$$e^{-\nu}R_{s,\rho}=R_{e^{-\nu}s,e^{-\nu}\rho}.$$
If $F:(z,w)\mapsto F(z,w)\in \C$ we set as usual
$$\iota_{F}:(z,w)\mapsto (\ti z, \ti w)\iff \begin{cases}&\ti z=z+\pa_{\ti w}F(z,\ti w)\\ &w=\ti w+\pa_{z}F(z,\ti w). \end{cases}$$
We  also define
$$ \Psi=\Psi_{\b_{}}:\C^2\ni (z,w)\mapsto (z,e^{- 2\pi i\b_{} z }w)\in \C^2$$
which satisfies
$$\Psi_{\b_{}}\circ (S_{\b_{}}\circ \Phi_{w})\circ \Psi_{\b_{}}^{-1}=\cT_{1,0}:(z,w)\mapsto (z+1,w)$$
$$(\Psi_{\b})_{*}\biggl(\pa_{z}+(2\pi i\beta w) \pa_{w}\biggr)=\pa_{z}$$
and 
$$S_{\b}:(\th,r)\mapsto (\th,e^{2\pi i\b}r),\qquad \Phi_{r}:(\th,r)\mapsto (\th+1,r).$$

\subsection{For vector fields}
\begin{prop}[Symplectic normalization]\label{lemma:normlemma:vf}There exist $\e>0$, such that for any  $\b\in\bD(0,1)$ the following holds. Assume that  $F\in \cO(\Psi(R_{s,\rho}))$ is small enough: $\|F\|_{\Psi(R_{s,\rho})}\leq \e$. There exists $Y\in \cO(\Psi(e^{-1/4}R_{s,\rho}))$ such that 
\be (\Phi_{Y})_{*}\circ \biggr(\pa_{z}+(2\pi i\beta w) \pa_{w}+J\nabla F\biggl)=\pa_{z}+(2\pi i\beta w)\pa_{w} . \label{A.179}\ee
Furthermore, if $F(z,w)=O(w^2)$ one can choose $Y(z,w)$ such that $Y(z,w)=O(w^2)$.
\end{prop}
\begin{proof}

If $A,B$ are two vector fields, 
\begin{align*}
&[A,B]=DB\cdot A-DA\cdot B\\
&(\phi^1_{B})_{*}A=A+[A,B]+\frak{O}_{2}(B).
\end{align*}
The linearized equation associated to (\ref{A.179}) is thus
$$[J\nabla Y,\pa_{z}+(2\pi i\b w)\pa_{w}]=J\nabla F$$
which reads
$$[(\pa_{w} Y)\pa_{z}-(\pa_{z} Y)\pa_{w},\pa_{z}+2\pi i\b w\pa_{w}]=(\pa_{w} F)\pa_{z}-(\pa_{z} F)\pa_{w}.$$
Using the fact that $[J\nabla Y,\pa_{z}]=[J\nabla Y,J\nabla w]=J\nabla ([Y,w])$ and 
$$[J\nabla Y,w\pa w]=J\nabla(Y-w\pa_{w}Y)$$ we find the equivalent equation on $\Psi(R_{s,\rho})$
\be \pa_{z}Y-i\b Y-i2\pi \b w\pa_{w}Y=-F.\label{A.180}\ee
Setting 
$$\ti F(z,w)=e^{-2\pi i\b z }F(z,e^{2\pi i\b z}w),\qquad \ti Y(z,w)=e^{-2\pi i\b z }Y(z,e^{2\pi i\b z}w)$$
we get
$$\pa_{z}\ti Y(z,w)=	\ti F(z,w)$$
which is easily solved on $R(s,\rho)$ with the estimate $\ti Y=\fO_{1}(\ti F)$.

As a consequence, we can solve the linearized equation (\ref{A.180}) with the estimate
$$Y=\fO_{1}(F).$$
With this choice, we find
\be (\Phi_{Y})_{*}\circ \biggr(\pa_{z}+(2\pi i\beta w) \pa_{w}+J\nabla F\biggl)=\pa_{z}+(2\pi i\beta w)\pa_{w} +J\nabla F_{2}. \label{A.179bis}\ee
with 
$$F_{2}=\fO_{2}(Y,F)=\fO_{2}(F).$$
This is a quadratic scheme and we can conclude by using Proposition \ref{lemma:quadraticconv}. 

\medskip The proof also shows that if $F(z,w)=O(w^2)$ one has $Y(z,w)$ such that $Y(z,w)=O(w^2)$.

\end{proof}

\subsection{For diffeomorphisms}
\begin{prop}[Symplectic normalization]\label{lemma:normlemma}Assume that  $F\in \cO(\Psi(R_{s,\rho}))$ is small enough. There exists $Y\in \cO(\Psi(e^{-1/4}R_{s,\rho}))$ such that 
$$\iota_{Y}\circ (S_{\b}\circ \Phi_{w}\circ \iota_{F})\circ \iota_{Y}^{-1}=S_{\b}\circ \Phi_{w}:(z,w)\mapsto (z+1,e^{2\pi i\b}w). $$
\end{prop}

Before proving this proposition we observe that 
the conjugacy equation
\be \iota_{Y}\circ (S_{\beta}\circ \Phi_{w}\circ \iota_{F})\circ \iota_{Y}^{-1}=S_{\beta}\circ \Phi_{w}\label{conjeq}\ee
admits  the following linearized equation
\be  F(z,w)=e^{-2\pi  i \beta}Y(z+1,e^{2\pi  i \beta} w)-Y(z,w).\label{conjeq:8.29ante}
\ee
\begin{lemma}\label{lemma:dbar:8.2}Let $F\in \cO(\Psi(R_{\s,\rho}))$. There exists $Y\in \cO(\Psi(R_{\s,\rho}))$ such that 
\be \forall (z,w)\in \Psi(R_{s,\rho}),\qquad F(z,w)=e^{- 2\pi i \beta}Y(z+1,e^{2\pi  i \beta} w)-Y(z,w).\label{conjeq:8.29}\ee
It satisfies
$$Y=\frak{O}_{1}(F).$$
\end{lemma}
\begin{proof}
Setting 
$$\ti F(z,w)=e^{-2\pi i\b z }F(z,e^{2\pi i\b z}w),\qquad \ti Y(z,w)=e^{-2\pi i\b z }Y(z,e^{2\pi i\b z}w)$$
equation (\ref{conjeq:8.29}) reads
\begin{multline*} e^{2\pi i\b z}\ti F(z,e^{-2\pi i\b z}w)=\\ e^{- 2\pi i \beta} e^{2\pi i\b (z+1)}\ti Y(z+1,e^{ -2\pi i \beta(z+1)} e^{2\pi i\b z}w)-e^{2\pi i\b z}\ti Y(z,e^{-2\pi i\b z}w)
\end{multline*}
or equivalently
$$ \ti F(z,e^{-2\pi i\b z}w)=\ti Y(z+1,e^{ -2\pi i \beta z} w)-\ti Y(z,e^{-2\pi i\b z}w).
$$
If $(z,w)=\Psi(\th,r)$ one has $z=\th$ and $w=e^{2\pi i\b \th}r=e^{2\pi i\b z}r$ hence $e^{-2\pi i\b z}w=r$.
In other words, solving (\ref{conjeq:8.29}) on the domain $\Psi(\R_{s,\rho})$ is equivalent to solving 
\be \ti F(\th,r)=\ti Y(\th+1,r)-\ti Y(\th,r) \label{conjeq:8.31}\ee 
on $\T_{s}\times \bD(0,\rho)$.

Using the expansions
$$\ti F(\th,r)=\sum_{l\in\N}\ti F_{l}(\th)r^l,\qquad \ti Y(\th,r)=\sum_{l\in\N}\ti Y_{l}(\th)r^l$$
(\ref{conjeq:8.31}) is equivalent to 
\be \forall\ l\in\N,\quad \ti F_{l}(\th)=\ti Y_{l}(\th+1)-\ti Y_{l}(\th).\label{conjeq:8.33}\ee
These are equations of the form
$$u(\th)=v(\th+1)-v(\th)$$
that can be solved  in the analytic category using $\bar\pa$-techniques. More precisely: 
\begin{lemma}\label{lemma:1.10.36}Given $u:\T_{s}\to\C$ which is $C^1$, there exists  a $C^1$ function  $\ti v:\R_{s}\to \C$ such that  for any $\zeta\in R_{s}$ one has $u(\zeta)=\ti v(\zeta+1)-\ti v(\zeta)$. One has $\|v\|_{C^1(R_{s})}\lesssim \|u\|_{C^1(\T_{s})}$.
\end{lemma}
\begin{proof}Let $\chi: [-1/4,5/4]\to \R$ be a smooth function which is equal to 0 on $[-1/4,1/4]$ and equal to 1 on $[3/4,5/4]$. We define $\ti v_{0}$ on  $I_{0}:=(-1/4,3/4)+i(-s,s)$ by $\ti v_{0}(x+iy)=\chi(x)u(x+iy)$. 
This function satisfies the matching condition 
\be \forall \zeta\in (I_{0}-1)\cap I_{0},\quad \ti v_{0}(\zeta+1)-\ti v_{0}(\zeta)=u(\zeta).\label{matching}\ee
One then extends $\ti v_{0}$ to $\R_{s}$ by setting  for  $\th\in k+I_{0}$, $\ti v(\zeta)=\ti v_{0}(\zeta-k)+\sum_{l=0}^{k-1}u(\zeta-k+l)$, if  $k>0$ and $\ti v(\zeta)=\ti v_{0}(\zeta+k)-\sum_{l=1}^ku(\zeta+l)$ if $k<0$. The matching condition (\ref{matching}) shows this function is a well defined smooth function on $\R_{s}$ that satisfies the conditions of the lemma.

\end{proof}

Note that because $\bar\pa u=0$ the function $\bar \pa \ti v$ is 1-periodic. Then one solves the $\bar\pa$ problem on the {\it annulus} $\T_{e^{-\e}s}$
$$\begin{cases}&\bar\pa w(\th)=\bar\pa \ti v(\th),\\ &w(\cdot+1)=w(\cdot).\end{cases}$$
The function $v:=\ti v-w$ is holomorphic (its $\bar\pa$ is zero) and since $w$ is 1-periodic, one has $u=v(\cdot+1)-v(\cdot)$.
One can give an estimate on $\T_{e^{-\e}s}$ of the form
$$v=\frak{O}_{1}(u).$$

More precisely,
if $w\in C^1(\T_{s},\C)$ we set
$$\cL w=\frac{1}{\pi}\cot(\pi \cdot)*w$$
i.e.
$$(\cL w)(\th)=\frac{1}{\pi}\int_{\T_{s}} \cot(\pi(\th-\zeta))w(\zeta)d\zeta\wedge \bar d\zeta.$$
\begin{lemma}Let $w\in C^1(\T_{s},\C)$ then one has $\cL w\in C^1(\T_{s},\C)$ and 
$$\begin{cases}&\bar\pa \cL w=w\\
&\|\cL w\|_{C^1(\T_{s})}\lesssim \|w\|_{C^0(\T_{s})}
\end{cases}$$
\end{lemma}
\begin{proof}
Use the fact that $\zeta\mapsto \cot(\pi \zeta)$ is locally integrable on $\T_{s}$ and that  in the distribution space $\cD'(\T_{s})$ one has
$$\bar\pa\biggl( \frac{1}{\pi}\cot(\pi \cdot)\biggr) = \d_{0}$$
($\d_{0}$ is the Dirac measure at 0).
\end{proof}
\begin{lemma}\label{lemma:dbar:8.5}Let $u\in \cO(\T_{s})$ be such that $\|u\|_{C^1(\T_{s})}<\infty$. There exists $v_{hol}\in \cO(\T_{s})$  such that $\|v_{hol}\|_{C^1(\T_{s})}<\infty$
$$\forall\th\in \T_{s},\quad  v_{hol}(\th+1)-v_{hol}(\th)=u(\th)$$
and
$$\|v_{hol}\|_{C^1(\T_{s})}\lesssim \|u\|_{C^1(\T_{s})}.$$
\end{lemma}
\begin{proof}Use Lemma \ref{lemma:1.10.36} to find $v\in C^1(\R_{s},\C)$ such that 
$$\forall\zeta\in \R_{s},\quad v(\zeta+1)-v(\zeta)=u(\zeta).$$
Because $u$ is holomorphic one has $\bar \pa u=0$ hence
 $\bar\pa v$ is 1-periodic. The function  $\cL \bar \pa v$  solves in $\T_{s}$ the equation 
$$\bar \pa (\cL \bar\pa v)(\th)=\bar \pa v(\th).$$
The searched for  function is $v_{hol}=v-\cL(\bar \pa v)$. 
\end{proof}
We can now complete the proof of Lemma \ref{lemma:dbar:8.2}. We apply Lemma \ref{lemma:dbar:8.5} to each $\ti F_{l}$ to get 
$\ti Y_{l}\in \cO(\T_{s})$  satisfying \ref{conjeq:8.33} and 
\begin{align*}\|\ti Y_{l}\|_{\T_{\s}}&\lesssim_{s} \|\ti F_{l}\|_{\T_{s}}\\
&\lesssim_{s}  \rho^{-l} \|\ti F\|_{R_{s,\rho}}\\
&\lesssim_{s,\b} \rho^{-l} \|F\|_{\Psi(R_{s,\rho})}.
\end{align*}
Setting $\ti Y(\th,r)=\sum_{l\in\N}\ti Y_{l}(\th)r^l$, one has for any $\rho'=e^{-\nu}\rho<\rho$, the inequality 
$\|\ti Y\|_{R_{s,\rho'}}\lesssim_{s,\b}\|F\|_{\Psi(R_{s,\rho})}
$; this provides  us with   $Y(\th,r)=e^{2\pi i\b \th }\ti Y(\th,e^{-2\pi i\b \th}r)$ defined on $\Psi(R_{s,\rho'})$ which satisfies (\ref{conjeq:8.29}) and the estimates
$$\|Y\|_{\Psi(R_{s,e^{-\nu}\rho})}\lesssim_{s,\b,\rho}\nu^{-1}\|F\|_{\Psi(R_{s,\rho})}.$$
\end{proof}

\noindent{\it Proof of Proposition \ref{lemma:normlemma}.} We define $\nu_{n}=(1/8)+\sum_{k\geq 4}2^{-k}$ and inductively  sequences  $F_{n}, Y_{n}$ ($n\geq 0$) in $\cO(\Psi(e^{-\nu_{n}}R_{s,\rho}))$ such that $F_{0}=F$
$$F_{n}(z,w)=e^{- 2\pi i \beta}Y_{n}(z+1,e^{ 2\pi i \beta} w)-Y_{n}(z,w)$$
and 
$$S_{\b}\circ \Phi_{w}\circ \iota_{F_{n+1}}=\iota_{Y_{n}}\circ (S_{\b}\circ \Phi_{w}\circ \iota_{F_{n}})\circ \iota_{Y_{n}}^{-1}.$$
One has $F_{n+1}=\frak{O}_{2}(F_{n},Y_{n})=\frak{O}_{2}(F_{n})$; the scheme is thus  quadratic and we can apply Proposition \ref{lemma:quadraticconv}. It implies the  fast convergence of $Y_{n},F_{n}$ to zero if $\|F\|_{\Psi(R_{s,\rho})}$ is small enough. The searched for conjugation $Y$ of Proposition \ref{lemma:normlemma} is 
defined by
$$\iota_{Y}=\lim_{n\to\infty}\iota_{Y_{n}}\circ\cdots\circ \iota_{Y_{0}}.$$
\  \hfill $\Box$

\section{Computation of the coefficient $b_{0,4}$}\label{appendix:compnu}
In the section we compute the coefficient $b_{0,4}$ of the resonant BNF (\ref{resBNFo2}), (\ref{introdnu}) of section \ref{sec:Ushikisresonance}.

\subsection{Time-1 maps of symplectic vector fields}
If $F$ is an observable
\be F\circ \Phi_{Y}=(F+\{Y,F\}+(1/2)\{Y,\{Y,F\}\}+O(Y^3)\label{B1.1}\ee
with
$$\{Y,F\}=\pa_{w}Y\pa_{z}F-\pa_{z}Y\pa_{w}F.$$
In particular, if $F:z\mapsto z$ one has
$$\{Y,z\}=\pa_{w}Y,\qquad \{Y,w\}=-\pa_{z}Y,$$
and 
$$\{Y,\{Y,z\}\}=\{Y,\pa_{w}Y\}=\pa_{w}Y\pa^2_{wz}Y-\pa_{z}Y\pa^2_{w}Y$$
$$\{Y,\{Y,w\}\}=-\{Y,\pa_{z}Y\}=-\pa_{w}Y\pa^2_{z}Y+\pa_{z}Y\pa^2_{zw}Y.$$
We thus have
\begin{multline*}z\circ \Phi_{Y}=z+\pa_{w}Y+(1/2)(\pa_{w}Y\pa_{z}\{Y,z\}\\ -\pa_{z}Y\pa_{w}\{Y,z\})+O(Y^3)\end{multline*}
$$w\circ \Phi_{Y}=w-\pa_{z}Y+(1/2)(\pa_{w}Y\pa_{z}\{Y,w\}-\pa_{z}Y\pa_{w} \pa_{z}\{ Y,w\})+O(Y^3)$$
So, if $Y(z,w)=\sum_{k+l=2}Y_{kl}z^kw^l$
one has 
\begin{align*}&\{Y,z\}=\pa_{w}Y=\sum_{k+l=3} lY_{kl}z^kw^{l-1},\\ 
&\{Y,w\}=-\pa_{z}Y=-\sum_{k+l=3}kY_{kl}z^{k-1}w^l,
\end{align*}
and 
\begin{align*}\{Y,\{Y,z\}\}
&=\{Y,\pa_{w}Y\}=\pa_{w}Y\pa^2_{wz}Y-\pa_{z}Y\pa^2_{w}Y\\
&=\sum_{k+l=3}lY_{kl}z^kw^{l-1}\sum_{k+l=3}klY_{kl}z^{k-1}w^{l-1}\\
&\qquad {-\sum_{k+l=3}kY_{kl}z^{k-1}w^{l}\sum_{k+l=3}l(l-1)Y_{kl}z^{k}w^{l-2}}
\end{align*}
\begin{align*}\{Y,\{Y,w\}\}
&=-\{Y,\pa_{z}Y\}=-\pa_{w}Y\pa^2_{z}Y+\pa_{z}Y\pa^2_{zw}Y\\
&=\sum_{k+l=3}kY_{kl}z^{k-1}w^{l}\sum_{k+l=3}klY_{kl}z^{k-1}w^{l-1}\\
&\qquad {-\sum_{k+l=3}lY_{kl}z^{k}w^{l-1}\sum_{k+l=3}k(k-1)Y_{kl}z^{k-2}w^{l}}.
\end{align*}
Denote by $\frak{z}$ the ideal $z\cO(z,w)$. 
We have 
$$\{Y,z\}=3Y_{03}w^2\mod\frak{z}$$
$$\{Y,w\}=-Y_{12}w^2\mod\frak{z}$$
$$\{Y,\{Y,z\}\}=(3Y_{03}w^2)(2Y_{12}w)-(Y_{12}w^2)(6Y_{03}w)=0\mod\frak{z}$$
\begin{multline*}\{Y,\{Y,w\}\}=(Y_{12}w^2)(2Y_{12}w)-(3Y_{03}w^2)(2Y_{21}w) \\=(2Y_{12}^2-6Y_{03}Y_{21})w^3\mod\frak{z}\end{multline*}
We thus get (by using (\ref{B1.1}) with $F=z$ and $F=w$)
\begin{multline*}\Phi_{Y}(z,w)=\biggl(3Y_{03}w^2,-Y_{12}w^2+(2Y_{12}^2-6Y_{03}Y_{21})w^3\biggr)\\+O^4(w)\mod\frak{z}.
\end{multline*}
In other words, setting $\Phi_{Y}:(z,w)\mapsto (A(z,w),B(z,w)) $ and denoting by $A_{m}, B_{m}$ the homogeneous part of degree $m$ of  $A(z,w)=\sum_{k,l}A_{k,l}z^k w^l$, $B(z,w)=\sum_{k,l}B_{k,l}z^k w^l$ (i.e. $A_{m}(z,w)=\sum_{k+l=m}A_{k,l}z^kw^l$, $B_{m}(z,w)=\sum_{k+l=m}B_{k,l}z^kw^l$) one has
\begin{align*}
&A_{2}=3Y_{03}w^2\mod\frak{z},\qquad A_{3}=0\mod\frak{z}\\
&B_{2}=-Y_{12}w^2\mod\frak{z},\qquad B_{3}=(2Y_{12}^2-6Y_{03}Y_{21})w^3\mod\frak{z}.
\end{align*}

For further records we note that 
\be
\left\{
\begin{aligned}
&A_{02}=3Y_{03}\\
&A_{11}=2Y_{12}\\
&A_{03}=(3Y_{03})(2Y_{12})-(6Y_{12}Y_{03} )=0.
\end{aligned}
\right.
\label{B.2}
\ee

\subsection{Computation of $\Phi_{Y}^{-1}\circ  h^{\textrm{mod}}_{\a,\b}\circ \Phi_{Y}$ }
Recall (cf. (\ref{def:f}))
$$ h^{\textrm{mod}}_{\a,\b}:\C^2\ni \bm z\\ w\em\mapsto  \bm \l_{1}z\\\ \l_{2}w\em+\frac{q(\l_{1}z+\l_{2}w)}{\l_{1}-\l_2}\bm 1\\ -1\em\in \C^2$$
where 
$$q(z)=e^{i\pi \b}z^2$$ and the notation (cf. (\ref{recdefmu}))
$$i\mu_{\b}=i\mu_{\d}=\frac{e^{i\pi\b}}{\l_{1}-\l_{2}}.$$

If $\Phi_{Y}=g+O^4(z,w)$  and $\Phi_{Y}^{-1}=\Phi_{-Y}=g'+O^4(z,w)$ with  
\begin{align*}g(z,w)&=(z,w)+(\sum_{k+l\leq 3}a_{kl}z^kw^l,\sum_{k+l\leq 3}b_{kl}z^kw^l)\\
&=(z+Z_{2}+Z_{3},w+W_{2}+W_{3})
\end{align*}
and
\begin{align*}g^{-1}(U,V)&=(U,V)+(\sum_{k+l\leq 3}a'_{kl}U^kV^l,\sum_{k+l\leq 3}b'_{kl}U^kV^l)\\
&=(z+Z'_{2}+Z'_{3},w+W'_{2}+W'_{3}),
\end{align*}
one finds
\begin{multline*}h^{\textrm{mod}}_{\a,\b}\circ g=(\l_{1}z,\l_{2}w)+(\l_{1}Z_{2}+\l_{1}Z_{3},\l_{2}W_{2}+\l_{2}W_{3})+\\ i\mu_{\b}(\l_{1}z+\l_{2}w+\l_{1}Z_{2}+\l_{2}W_{2}+\l_{1}Z_{3}+\l_{2}W_{3})^2\times (1,-1)\\
 =(\l_{1}z,\l_{2}w)+(\l_{1}Z_{2}+\l_{1}Z_{3},\l_{2}W_{2}+\l_{2}W_{3})+\\ i\mu_{\b}(\l_{1}^2z^2+\l_{2}^2w^2+2\l_{1}\l_{2}zw+2\l_{1}^2zZ_{2}\\ +2\l_{1}\l_{2}zW_{2}+2\l_{2}^2wW_{2}+2\l_{1}\l_{2}wZ_{2})\times (1,-1)+O^4(z,w).
\end{multline*}
Hence
\begin{multline*}[ h^{\textrm{mod}}_{\a,\b}\circ g]_{z}
=\underbracket{\l_{1}z}_{U_{1}}+\underbracket{[\l_{1}Z_{2}+i\mu_{\b}(\l_{1}^2z^2+\l_{2}^2w^2+2\l_{1}\l_{2}zw)]}_{U_{2}}\\+\underbracket{[\l_{1}Z_{3}+2i\mu_{\b}(\l_{1}^2zZ_{2}+\l_{1}\l_{2}zW_{2}+\l_{2}^2wW_{2}+\l_{1}\l_{2}wZ_{2})]}_{U_{3}}+O^4(z,w)
\end{multline*}
and 
\begin{multline*}[ h^{\textrm{mod}}_{\a,\b}\circ g]_{w}
=\underbracket{\l_{2}w}_{V_{1}}+\underbracket{[\l_{2}W_{2}-i\mu_{\b}(\l_{1}^2z^2+\l_{2}^2w^2+2\l_{1}\l_{2}zw)]}_{V_{2}}\\+\underbracket{[\l_{2}W_{3}-2i\mu_{\b}(\l_{1}^2zZ_{2}+\l_{1}\l_{2}zW_{2}+\l_{2}^2wW_{2}+\l_{1}\l_{2}wZ_{2})]}_{V_{3}}+O^4(z,w).
\end{multline*}
Thus
\begin{multline*}
[g^{-1}\circ h^{\textrm{mod}}_{\a,\b}\circ g]_{z}=U_{1}+U_{2}+U_{3}\\+a'_{2,0}(U_{1}+U_{2}+U_{3})^2+a'_{1,1}(U_{1}+U_{2}+U_{3})(V_{1}+V_{2}+V_{3})\\+a'_{0,2}(V_{1}+V_{2}+V_{3})^2+etc.\\ =U_{1}+U_{2}+U_{3}+a'_{2,0}(U_{1}^2+2U_{1}U_{2})+a'_{1,1}(U_{1}V_{1}+U_{1}V_{2}+U_{2}V_{1})+a'_{0,2}(V_{1}^2+2V_{1}V_{2})\\
+a'_{3,0}U_{1}^3+a'_{2,1}U_{1}^2V_{1}+a'_{1,2}U_{1}V_{1}^2+a'_{0,3}V_{1}^3
\\=[U_{1}]+[U_{2}+a'_{2,0}U_{1}^2+a'_{1,1}U_{1}V_{1}\\+a'_{0,2}V_{1}^2]+[U_{3}+2a'_{2,0} U_{1}U_{2}+a'_{1,1}(U_{1}V_{2}+U_{2}V_{1})\\+2a'_{0,2}V_{1}V_{2}+a'_{3,0}U_{1}^3+a'_{2,1}U_{1}^2V_{1}+a'_{1,2}U_{1}V_{1}^2+{a'_{0,3}V_{1}^3}]+h.o.t.
\end{multline*}
and, $\mod\frak{z}$, the term of homogeneous degree 3  is equal to 
$$U_{3}+a'_{1,1}U_{2}V_{1}+2a'_{0,2}V_{1}V_{2}+a'_{0,3}V_{1}^3.$$
We have
\begin{align*}&U_{1}=0\mod\frak{z},\qquad V_{1}=\l_{2}w\mod\frak{z}\\
&U_{2}=\l_{1}3Y_{03}w^2+\l_{2}^2i\mu_{\b}w^2\mod\frak{z},\qquad V_{2}=-\l_{2}Y_{12}w^2-i\mu_{\b}\l_{2}^2w^2\mod\frak{z}\\
&U_{3}=2i\mu_{\b}(\l_{2}^2 w(-Y_{12}w^2)+\l_{1}\l_{2}w(3Y_{03}w^2))\mod\frak{z}
\end{align*}
so
\begin{multline*}[g^{-1}\circ h^{\textrm{mod}}_{\a,\b}\circ g]_{z}=2i\mu_{\b}(\l_{2}^2 w(-Y_{12}w^2)+\l_{1}\l_{2}w(3Y_{03}w^2))\\
+a'_{11}(3Y_{03}+\l_{2}^2i\mu_{\b}w^2)\l_{2}w\\
+2a'_{02}\l_{2}w(-\l_{2}Y_{12}w^2-i\mu_{\b}\l_{2}^2w^2)\\
+a'_{03}\l_{2}^3w^3\mod\frak{z}.
\end{multline*}

Using the fact that (see (\ref{B.2}))
\begin{align*}&a'_{02}=-3Y_{03}\\
&a'_{11}=-2Y_{12}\\
&a'_{03}=(3Y_{03})(2Y_{12})-(6Y_{12}Y_{03} )=0
\end{align*}
we find with $\l_{1}=1+O(\d)$, $\l_{2}=j+O(\d)$
\begin{multline*} [g^{-1}\circ h^{\textrm{mod}}_{\a,\b}\circ g]_{z}=\biggl[2i\mu_{\b}(j^2 (-Y_{12})+j(3Y_{03}))\\
-2Y_{12} j(3Y_{03}+j^2i\mu_{\b}w^2)\\\
-6Y_{03} j(-jY_{12}-i\mu_{\b}j^2) +O(\d)\biggr]w^3\mod\frak{z}
\end{multline*}

\subsection{Computation of $b_{0,4}$}
We now recall that  the first resonant BNF conjugation $Y=Y_{1}$ satisfies  (cf. (\ref{eq:coeffkl}) with $G=F'$)
$$ (e^{-2\pi i\b}\l_{1}^k\l_{2}^l-1)\hat Y_{}(k,l)=\hat {F'}(k,l)$$
where (cf. (\ref{defF'}))
$$F'(z,w)=i(\mu/3) \biggl((j^2+O(\d))z^3+3z^2w+3(j+O(\d))zw^2+(j^2+O(\d))w^3\biggr)+O^4(z,w).$$
Hence
$$Y_{03}=j^2\frac{i\mu/3}{j^2-1}+O(\d),\qquad Y_{12}=j\frac{i\mu}{j-1}+O(\d), $$
and with $g=\Phi_{Y_{1}}$
\begin{multline*}w^{-3}[g^{-1}\circ h^{\textrm{mod}}_{\a,\b}\circ g]_{z}=2i\mu_{}(j^2 (-Y_{12})+j(3Y_{03}))\\
-2Y_{12}i\mu_{}-6Y_{12}Y_{03}j\\
-6Y_{03} j(-jY_{12}-i\mu_{}j^2)+O(\d)\mod\frak{z}.
\end{multline*}
Using the fact that
$$\frac{1}{j^2-1}=\frac{j-1}{3},\qquad \frac{1}{(j-1)}=\frac{j^2-1}{3}$$
we get
\begin{multline*}w^{-3}[g^{-1}\circ h^{\textrm{mod}}_{\a,\b}\circ g]_{z}=-2\mu^2( (-\frac{1}{j-1})+(\frac{1}{j^2-1}))\\
+2\mu^2j\frac{1}{j-1} +2\mu^2j \frac{1}{j^2-1}\frac{1}{j-1}\\
-2\mu^2j^2\frac{1}{j^2-1} (\frac{1}{j-1}+1)+O(\d)\mod\frak{z}
\end{multline*}
\begin{multline*} w^{-3}[g^{-1}\circ h^{\textrm{mod}}_{\a,\b}\circ g]_{z}=-(2/3)\mu^2( j-j^2)\\
+2\mu^2j\frac{1}{j-1} +2\mu^2j \frac{1}{j^2-1}\frac{1}{j-1}\\
-2\mu^2\frac{1}{j^2-1} (\frac{1}{j-1})+O(\d)\mod\frak{z}
\end{multline*}
\begin{multline*}w^{-3}[g^{-1}\circ h^{\textrm{mod}}_{\a,\b}\circ g]_{z}=-(2/3)\mu^2( j-j^2)\\
+(2\mu^2/3)j(j^2-1) +(2\mu^2/3)j \\
-(2\mu^2/3)+O(\d)\mod\frak{z}
\end{multline*}
hence, using $j=(-1/2)+i(\sqrt{3}/2)$,
\begin{align*}w^{-3}[g^{-1}\circ h^{\textrm{mod}}_{\a,\b}\circ g]_{z}&=-(2/3)\mu^2( j-j^2) +O(\d)\\
&=(-2/3)\mu^2(2j+1)+O(\d)\\
&=(-2/3)(i/3)(\sqrt{3})\\
&=-(2/3)\frac{i}{\sqrt{3}}+O(\d)\mod\frak{z}
\end{align*}
hence
$$[\Phi_{Y_{1}}^{-1}\circ h^{\textrm{mod}}_{\a,\b}\circ \Phi_{Y_{1}}]_{z}=i\biggl(-(2/3)\frac{1}{\sqrt{3}}+O(\d)\biggr)w^3\mod\frak{z}.$$
If 
$$ \Phi_{Y_{1}}^{-1}\circ h^{\rm mod}_{\a,\b}\circ \Phi_{Y_{1}}=\diag(\l_{1},\l_{2})\circ \iota_{F''},$$
we thus have 
$$F''(z,w)=i\biggl(-(2/3)\frac{1}{\sqrt{3}}+O(\d)\biggr)w^4/4\mod\frak{z}.$$
From (\ref{resBNFo2}) we know that after the second resonant BNF conjugation $\Phi_{Y_{2}}$,
$$ \Phi_{Y_{2}}^{-1}\circ \Phi_{Y_{1}}^{-1}\circ h^{\rm mod}_{\a,\b}\circ \Phi_{Y_{1}}\circ \Phi_{Y_{2}}=\diag(\l_{1},\l_{2})\circ \iota_{F_{4}}$$
with
$$F_{4}(z,w)=b_{2,1} z^2w+b_{3,1}z^3w+b_{0,4}w^4+O^5(z,w).$$
Because the {\it resonant} term ${\rm cst}\times w^3$ in $\diag(\l_{1},\l_{2})\circ \iota_{F_{4}}$ is the same as the corresponding term in $$\Phi_{Y_{1}}^{-1}\circ h^{\textrm{mod}}_{\a,\b}\circ \Phi_{Y_{1}}$$
we thus get
$$-4ib_{0,4}=\biggl(-(2/3)\frac{1}{\sqrt{3}}+O(\d)\biggr)$$
which is (\ref{introdnu}).

\hfill$\Box$

\section{Fixed point theorem}
If $\cE$ is a normed space we denote $B(x,\d)$ (resp. $\overline{B}(x,\d)$) the open (resp. closed) ball of center $x$ and radius $\d$.

We recall the Contracting mapping Theorem:
\begin{lemma}\label{lem:annex2}
\ 
\begin{enumerate}\item\label{i} Let $\cE$ be a Banach space, $\rho>0$  and  $\Psi:\bar B(0,\rho)\to\cE$  a $\kappa$-contracting map such that  $\|\Psi(0)\|\leq \rho\times (1-\kappa)$. Then, $\Psi$ has a unique fixed point in $\bar B(0,\rho)$.
\item Assume $\d:= \rho\times (1-\kappa)-\|\psi(0)\|>0$. Then for any $p\in B(0,\d)$ the map $p+\Psi$  has a unique fixed point $x(p)$  in $\bar B(0,\rho)$ and the map $B(0,\d)\ni p\mapsto x(p)-p\in \cE$ is $\kappa/(1-\kappa)$-Lipschitz. 
\item If $\Psi_{\l}:\bar B(0,\rho)\to \cE$ is a family  of $\kappa$-contracting mappings  ($\l$ in some metric space $(X,d)$) such that $\|\Psi_{\l}(0)\|\leq \rho\times (1-\kappa)$ and if for any $x\in \bar B(0,\rho)$, $\l,\l'	\in X$ one has $\|\Psi_{\l}(x)-\Psi_{\l'}(x)\|\leq Cd(\l,\l')$, then the map $X\ni \l\mapsto x(\l)\in E$, where $x(\l)$ is the unique fixed point of $\Psi_{\l}$, is $C/(1-\kappa)$-Lipschitz.
\end{enumerate}
\end{lemma}
\begin{proof}

\noindent 1) We just have to check that $\Psi (\bar B(0,\rho))\subset \bar B(0,\rho)$. This comes from the fact that if $\|x\|\leq \rho$,
$$\|\Psi(x)\|\leq \|\Psi(0)\|+\kappa \rho\leq \rho.$$

\noindent 2) Under the hypothesis $\d>0$, the existence of $x(p)\in B(0,\rho)$ satisfying  $p+\Psi(x(p))=x(p)$ follows from 1). If $p,p'\in B(0,\d)$ a classical argument shows that $\|x(p)-x(p')\|\leq (1-\kappa)^{-1}\|p-p'\|$ hence 
$$\|(x(p)-p)-(x(p')-p')\|\leq \kappa \|x(p)-x(p')\|\leq \frac{\kappa}{1-\kappa}\|p-p'\|.$$

\noindent 3) Indeed $$x(\l)-x(\l')=(\Psi_{\l}(x({\l}))-\Psi_{\l'}(x(\l)))+(\Psi_{\l'}(x(\l))-\Psi_{\l'}(x(\l')))$$
hence
$$\|x(\l)-x(\l')\|\leq d(\l,\l')+\kappa \|x(\l)-x(\l')\|$$
which gives the result. 
\end{proof}

\section{Estimate on resolvent}
We assume $A, B\in C^0(\R,M(2,\C))$ are $T$-periodic.
\begin{lemma}\label{lemma:3.4resolGron-n}
One has with $C_{A}=\max_{[0,T]}\|R_{A}(t,0)\|$
$$\|R_{B}(t,0)\|\leq C_{A}e^{TC_{A}\sup_{[0,T]}\|B-A\|}$$
and
$$\sup_{[0,T]}\|R_{A}(\cdot,0)-R_{B}(\cdot,0)\|\leq \sup_{[0,T]}\|B-A\|\times C_{A}^2Te^{TC_{A}\sup_{[0,T]}\|B-A\|} .$$
\end{lemma}
\begin{proof}
Because
\begin{align*}\frac{d}{dt}R_{B}(t,0)&=B(t)R_{B}(t,0)\\
&=A(t)R_{B}(t,0)+(B(t)-A(t))R_{B}(t)
\end{align*}
one has 
\be R_{B}(t,0)=R_{A}(t,0)+\int_{0}^tR_{A}(t-s,0)(B(s)-A(s))R_{B}(s,0))ds\label{eq:3.31}\ee
hence
$$\|R_{B}(t,0)\|\leq C_{A}+C_{A}\sup_{[0,T]}\|B-A\|\int_{0}^t\|R_{B}(s,0)\|ds$$
and by Gronwall inequality
$$\|R_{B}(t,0)\|\leq C_{A}e^{TC_{A}\sup_{[0,T]}\|B-A\|}.$$
Equality (\ref{eq:3.31}) then yields
$$\sup_{[0,T]}\|R_{B}(\cdot,0)-R_{A}(\cdot,0)\|\leq \sup_{[0,T]}\|B-A\|\times C_{A}^2Te^{TC_{A}\sup_{[0,T]}\|B-A\|}.$$
\end{proof}

\begin{comm}

\section{Program}
\begin{verbatim}
#!/usr/bin/env python3
# -*- coding: utf-8 -*-
"""
Created on Tue Mar 19 18:51:55 2024

@author: raphaelkrikorian
"""




import numpy as np
import scipy
import matplotlib.pyplot as plt
import math
import cmath
from scipy.integrate import odeint
from scipy.integrate import solve_ivp
from scipy import linalg
from scipy import fftpack
pi=math.pi

cosh=np.cosh
sinh=np.sinh
cos=np.cos
sin=np.sin
exp=np.exp
sqrt=cmath.sqrt
#sqrt=np.sqrt




tau=1



#####################

"""
**************************************************
 PARAMETERS OF THE SYSTEM. DATA ARE MBETA AND TAU
**************************************************
"""
g=(-4-sqrt(16+11*(2*tau-tau*tau)))/11

g_p=(-4+(16+11*(2*tau-tau*tau))**(1/2))/11
g_m=(-4-(16+11*(2*tau-tau*tau))**(1/2))/11

g=g_m

"""
mbeta=(1/g)*(1/(2*pi))
malpha=(tau-(1/2))*mbeta
"""




beta=1

beta_1=(1-tau)*beta
beta_2=tau*beta


mu=0.579
mu=.5
nu=-0.3876
nu=-1/3








#*****************************
# SETTING PARAMETERS ON THE FINITE DIM SPACE
#*****************************

w_1=beta*1.2
w_1=beta*1.7
w_1=beta*1.5
w_1=beta*1.4
#w_1=beta



omega_init=beta*g
z_0_init=beta*(tau-g)/(2*mu)
z_m3_init=(-3*g*g*(2*g+1)/(4*mu*mu*nu))*(beta/w_1)**3
w_m2_init=-(g*(2*g+1)/(2*mu*nu))*(beta/w_1)**2


N=12
Nz=N # NUMBERS OF HARMONICS
Nw=N   # NUMBERS OF HARMONICS

Lz=np.array(range(0,2*Nz+1))
Lw=np.array(range(0,2*Nw+1))






def Ext(z,N,Nprime):
    zz=np.zeros((2*Nprime+1),dtype=complex)
    for i in range(Nprime-N,N+Nprime+1):
        zz[i]=z[i-(Nprime-N)]
    return zz 






def Hch(z,w,omega,N):
    Lz=np.array(range(0,2*N+1))
    Lw=np.array(range(0,2*N+1))
    Sz=np.zeros((2*N+1),dtype=complex)
    Sz1=np.zeros(2*N+1,dtype=complex)
    Sz2=np.zeros(2*N+1,dtype=complex)
        

    for k in range(0,2*N+1):
        sz2=0
        for l1 in range(0,2*N+1):
            for l2 in range(0,2*N+1):
                for l3 in range(0,2*N+1):
                    if (Lw[l1]-N)+(Lw[l2]-N)+(Lw[l3]-N)==(Lz[k]-N)-1:
                        sz2=sz2+w[l1]*w[l2]*w[l3]
        Sz2[k]=sz2
        sz1=0
        for l1 in range(0,2*N+1):
            for l2 in range(0,2*N+1): 
                if (Lz[l1]-N)+(Lz[l2]-N)==(Lz[k]-N):
                    sz1=sz1+z[l1]*z[l2]
        Sz1[k]=sz1
        Sz[k]=-(3*(Lz[k]-N))*omega*z[k]+beta_1*z[k]+mu*Sz1[k]+nu*Sz2[k]
        #Sz[k]=Sz[k]/(-(3*(Lz[k]-Nz)+1))


    Sw=np.zeros(2*N+1,dtype=complex)

        

    for k in range(0,2*N+1):
        sw=0
        for l1 in range(0,2*N+1):
            for l2 in range(0,2*N+1): 
                if (Lz[l1]-N)+(Lw[l2]-N)==(Lw[k]-N):
                    sw=sw+z[l1]*w[l2]
        Sw[k]=-(3*(Lw[k]-N)+1)*omega*w[k]+beta_2*w[k]-2*mu*sw
        #Sw[k]=Sw[k]/(-(3*(Lw[k]-Nw)+1))


 #   return np.array([Sz,Sw])
    return np.concatenate([Sz,Sw])


def Hcheck(z,w,omega,N):
    Nprime=3*N+1
    zz=Ext(z,N,Nprime)
    ww=Ext(w,N,Nprime)
    return Hch(zz,ww,omega,Nprime)


###############################################





def H(z,w,omega):
    Sz=np.zeros(2*Nz+1)
    Sz=np.zeros((2*Nz+1),dtype=complex)
    Sz1=np.zeros(2*Nz+1)
    Sz1=np.zeros(2*Nz+1,dtype=complex)
    Sz2=np.zeros(2*Nz+1)
    Sz2=np.zeros(2*Nz+1,dtype=complex)
        

    for k in range(0,2*Nz+1):
        sz2=0
        for l1 in range(0,2*Nw+1):
            for l2 in range(0,2*Nw+1):
                for l3 in range(0,2*Nw+1):
                    if (Lw[l1]-Nw)+(Lw[l2]-Nw)+(Lw[l3]-Nw)==(Lz[k]-Nz)-1:
                        sz2=sz2+w[l1]*w[l2]*w[l3]
        Sz2[k]=sz2
        sz1=0
        for l1 in range(0,2*Nz+1):
            for l2 in range(0,2*Nw+1): 
                if (Lz[l1]-Nz)+(Lz[l2]-Nz)==(Lz[k]-Nz):
                    sz1=sz1+z[l1]*z[l2]
        Sz1[k]=sz1
        Sz[k]=-(3*(Lz[k]-Nz))*omega*z[k]+beta_1*z[k]+mu*Sz1[k]+nu*Sz2[k]
        #Sz[k]=Sz[k]/(-(3*(Lz[k]-Nz)+1))


    Sw=np.zeros(2*Nz+1)
    Sw=np.zeros(2*Nz+1,dtype=complex)

        

    for k in range(0,2*Nw+1):
        sw=0
        for l1 in range(0,2*Nz+1):
            for l2 in range(0,2*Nw+1): 
                if (Lz[l1]-Nz)+(Lw[l2]-Nw)==(Lw[k]-Nw):
                    sw=sw+z[l1]*w[l2]
        Sw[k]=-(3*(Lw[k]-Nw)+1)*omega*w[k]+beta_2*w[k]-2*mu*sw
        #Sw[k]=Sw[k]/(-(3*(Lw[k]-Nw)+1))


 #   return np.array([Sz,Sw])
    return np.concatenate([Sz,Sw])





def HH(V):
    omega=V[2*Nz+1+Nw]
    z=V[0:2*Nz+1]
    w=np.array(V[2*Nz+1:2*(Nz+Nw)+2])
    w=np.array(V[2*Nz+1:2*(Nz+Nw)+2],dtype=complex)
    w[Nw]=w_1
    VV=H(z,w,omega)
    return VV
   


E=np.eye(2*(Nz+Nw)+2)
E=np.eye(2*(Nz+Nw)+2,dtype=complex)



epsilon=.0001

def DHH(V):
    VV=np.zeros([2*(Nz+Nw)+2,2*(Nz+Nw)+2])
    VV=np.zeros([2*(Nz+Nw)+2,2*(Nz+Nw)+2],dtype=complex)
    for i in range(0,2*(Nz+Nw)+2):
        VV[i]=np.array((HH(V+epsilon*E[i])-HH(V))/epsilon)
        VV[i]=np.array((HH(V+epsilon*E[i])-HH(V))/epsilon,dtype=complex)
    return VV



def ZU(z,u,omega):
    Su1=np.zeros(2*Nz+1,dtype=complex)
    for k in range(0,2*Nz+1):    
        su1=0
        for l1 in range(0,2*Nz+1):
            for l2 in range(0,2*Nz+1): 
                if (Lz[l1]-Nz)+(Lz[l2]-Nz)==(Lz[k]-Nz):
                    su1=su1+z[l1]*u[l2]
        Su1[k]=su1
    return Su1   




#def K(z,w,u,v,omega):
def K(z,w,u,v,omega,N):
    Nz=N
    Nw=N
    Lz=np.array(range(0,2*N+1))
    Lw=np.array(range(0,2*N+1))
    #Sz=np.zeros(2*Nz+1)
    Su=np.zeros((2*Nz+1),dtype=complex)
    Su1=np.zeros(2*Nz+1,dtype=complex)
    Su2=np.zeros(2*Nz+1,dtype=complex)
        

    for k in range(0,2*Nz+1):
        su2=0
        for l1 in range(0,2*Nw+1):
            for l2 in range(0,2*Nw+1):
                for l3 in range(0,2*Nw+1):
                    if (Lw[l1]-Nw)+(Lw[l2]-Nw)+(Lw[l3]-Nw)==(Lz[k]-Nz)-1:
                        su2=su2+w[l1]*w[l2]*v[l3]
        Su2[k]=su2
        su1=0
        for l1 in range(0,2*Nz+1):
            for l2 in range(0,2*Nw+1): 
                if (Lz[l1]-Nz)+(Lz[l2]-Nz)==(Lz[k]-Nz):
                    su1=su1+z[l1]*u[l2]
        Su1[k]=su1
        Su[k]=-(3*(Lz[k]-Nz))*omega*u[k]+u[k]+Su1[k]-Su2[k]
        Su[k]=(-(3*(Lz[k]-Nz))*omega+(1-tau))*u[k]+Su1[k]-Su2[k]
        #Sz[k]=Sz[k]/(-(3*(Lz[k]-Nz)+1))


    Sv=np.zeros(2*Nz+1)
    Sv=np.zeros(2*Nz+1,dtype=complex)



        

    for k in range(0,2*Nw+1):
        sv1=0
        sv2=0
        for l1 in range(0,2*Nz+1):
            for l2 in range(0,2*Nw+1): 
                if (Lz[l1]-Nz)+(Lw[l2]-Nw)==(Lw[k]-Nw):
                    sv1=sv1+z[l1]*v[l2]
                    sv2=sv2+u[l1]*w[l2]
        Sv[k]=(-(3*(Lw[k]-Nw)+1)*omega+tau)*v[k]-sv1-sv2
        #Sw[k]=Sw[k]/(-(3*(Lw[k]-Nw)+1))


 #   return np.array([Sz,Sw])
    return np.concatenate([Su,Sv])



    
    
    Sz=np.zeros(2*Nz+1)
    Sz=np.zeros((2*Nz+1),dtype=complex)
    Sz1=np.zeros(2*Nz+1)
    Sz1=np.zeros(2*Nz+1,dtype=complex)
    Sz2=np.zeros(2*Nz+1)
    Sz2=np.zeros(2*Nz+1,dtype=complex)
        

    for k in range(0,2*Nz+1):
        sz2=0
        for l1 in range(0,2*Nw+1):
            for l2 in range(0,2*Nw+1):
                for l3 in range(0,2*Nw+1):
                    if (Lw[l1]-Nw)+(Lw[l2]-Nw)+(Lw[l3]-Nw)==(Lz[k]-Nz)-1:
                        sz2=sz2+w[l1]*w[l2]*v[l3]
        Sz2[k]=sz2
        sz1=0
        for l1 in range(0,2*Nz+1):
            for l2 in range(0,2*Nw+1): 
                if (Lz[l1]-Nz)+(Lz[l2]-Nz)==(Lz[k]-Nz):
                    sz1=sz1+z[l1]*u[l2]
        Sz1[k]=sz1
        Sz[k]=-(3*(Lz[k]-Nz))*omega*u[k]+u[k]+2*mu*Sz1[k]+3*nu*Sz2[k]
        #Sz[k]=Sz[k]/(-(3*(Lz[k]-Nz)+1))


    Sw=np.zeros(2*Nz+1,dtype=complex)   
    Sw1=np.zeros(2*Nw+1,dtype=complex)
    Sw2=np.zeros(2*Nw+1,dtype=complex)

        

    for k in range(0,2*Nw+1):
        sw=0
        for l1 in range(0,2*Nz+1):
            for l2 in range(0,2*Nw+1): 
                if (Lz[l1]-Nz)+(Lw[l2]-Nw)==(Lw[k]-Nw):
                    sw=sw+z[l1]*v[l2]
        Sw1[k]=-(3*(Lw[k]-Nw)+1)*omega*v[k]-2*mu*sw
        #Sw[k]=Sw[k]/(-(3*(Lw[k]-Nw)+1))
        
    for k in range(0,2*Nw+1):
        sw=0
        for l1 in range(0,2*Nz+1):
            for l2 in range(0,2*Nw+1): 
                if (Lz[l1]-Nz)+(Lw[l2]-Nw)==(Lw[k]-Nw):
                    sw=sw+u[l1]*w[l2]
        Sw2[k]=-2*mu*sw
    
    Sw=Sw1+Sw2
 #   return np.array([Sz,Sw])
    return np.concatenate([Sz,Sw])

def KK(P,Q,omega,N):
    #omega=P[2*Nz+1+Nw]
    Nz=N
    Nw=N
    #Lz=np.array(range(0,2*N+1))
    #Lw=np.array(range(0,2*N+1))
    z=P[0:2*Nz+1]
    w=np.array(P[2*Nz+1:2*(Nz+Nw)+2],dtype=complex)
    #w[Nw]=w_1
    u=np.zeros(2*Nz+1,dtype=complex)
    u=Q[0:2*Nz+1]
    v=np.array(Q[2*Nz+1:2*(Nz+Nw)+2],dtype=complex)
    KK=K(z,w,u,v,omega,N)
    return KK
 
O=np.zeros([2*(Nz+Nw)+2,2*(Nz+Nw)+2],dtype=complex)
O1=np.zeros(2*(Nz+Nw)+2,dtype=complex)
E=np.eye(2*(Nz+Nw)+2,dtype=complex)


def DKK(V,omega,N):
    Nz=N
    Nw=N
    #Lz=np.array(range(0,2*N+1))
    #Lw=np.array(range(0,2*N+1))
    VV=np.zeros([2*(Nz+Nw)+2,2*(Nz+Nw)+2],dtype=complex)
    for i in range(0,2*(Nz+Nw)+2):
        VV[i]=np.array((KK(V,epsilon*E[i],omega,N)-KK(V,O1,omega,N))/epsilon,dtype=complex)
    return VV






def ICZ(V):
    s=0
    #omega=V[2*Nz+1+Nw]
    z=V[0:2*Nz+1]
    #w=np.array(V[2*Nz+1:2*(Nz+Nw)+2])
    #w[Nw]=w_1
    #omega=V[2*Nz+1+Nw]
    for i in range(0,2*Nz+1):
        s=s+z[i]
    return s


def ICW(V):
    s=0
    #omega=V[2*Nz+1+Nw]
    #z=V[0:2*Nz+1]
    w=np.array(V[2*Nz+1:2*(Nz+Nw)+2])
    w=np.array(V[2*Nz+1:2*(Nz+Nw)+2],dtype=complex)
    w[Nw]=w_1
    
    for i in range(0,2*Nw+1):
        s=s+w[i]
    return s



W_init=np.zeros(2*Nw+1)
W_init=np.zeros(2*Nw+1,dtype=complex)
W_init[Nw]=omega_init
W_init[Nw-1]=w_m2_init

z_init=np.zeros(2*Nz+1)
z_init=np.zeros(2*Nz+1,dtype=complex)
z_init[Nz]=z_0_init
z_init[Nz-1]=z_m3_init

z=z_init    
W=W_init

V_init=np.concatenate([z_init,W_init])






def DHH_inv(V):
    return np.linalg.inv(DHH(V))

def Delta(V):
    return -HH(V)@DHH_inv(V)
    
def Linfty(M):
    E=np.eye(2*(Nz+Nw)+2,dtype=complex)
    for i in range(0,2*(Nz+Nw)+2):
        E[i]=abs(M[i])
    return np.max(E)    

def L1(M):
    l=np.shape(M)[0]
    S=0
    for i in range(0,l):
        S=S+abs(M[i])
    return S



def MaxL1Line(M):
    l=np.shape(M)[0]
    #c=np.shape(M)[1]
    E=np.zeros(l+1,dtype=complex)
    for i in range(0,l):
        E[i]=L1(M[i])
    return np.max(E)    

def MaxL1Col(M):
    return MaxL1Line(np.transpose(M))

def L1Line(M):
    E=np.eye(2*(Nz+Nw)+2,dtype=complex)
    for i in range(0,2*(Nz+Nw)+2):
        E[i]=L1(M[i])
    return E 

def L1line(M,N):
    #E=np.eye(2*N+1,dtype=complex)
    S=0
    for i in range(0,2*N+1):
        S=S+abs(M[i])
    return S 
    

def L2(M):
    A=M*np.conj(M)
    E=np.ones(2*(Nz+Nw)+2,dtype=complex)
    return sqrt(A@np.transpose(E))
    


V=V_init
for i in range(0,7):    
    V=V+Delta(V)
    
HH_goodness=HH(V)    

#print('Nz=',Nz)
#print('Nw=',Nw)
print('tau=',tau)
print(' ')
print('w_1=',w_1)
print(' ')
print('N=',N)

  

z_fin=V[0:2*Nz+1]
w_fin=np.array(V[2*Nz+1:2*(Nz+Nw)+2],dtype=complex)
omega_fin=w_fin[Nw]
w_fin[Nw]=w_1

V=np.concatenate([z_fin,w_fin])

"""
z_m3=z_fin[0]
z_0=z_fin[1]
z_3=z_fin[2]
w_m2=w_fin[0]
w_4=w_fin[2]
g=omega_fin

dtiG_m3=[z_3,0,-2*w_1*w_4,-w_1**2,0]
dtiG_0=[0,0,0,-2*w_1*w_4,0]
dtiG_3=[0,z_m3,0,0,0]
dtiG_m2=[-w_4,0,0,-z_m3,0]
dtiG_1=[0,0,0,0,0]
dtiG_4=[0,-w_m2,0,0,-z_m3]

dtiG=[dtiG_m3,dtiG_0,dtiG_3,dtiG_m2,dtiG_1,dtiG_4]

dtiGTrans=np.transpose(dtiG)

A=dtiGTrans@DHH_inv(V)

"""


""" 
def LLz(k):
    return 3*k*g-1-hatz_0

def LLw(k):
    return (3*k+1)*g+hatz_0
"""



t_0=0
 
def Z(t):
    s=0
    for i in Lz:
        s=s+exp(3*(i-Nz)*1j*omega_fin*(t-t_0))*z_fin[i]
    return s
    
def W(t):
    s=0
    for i in Lw:
        s=s+exp((3*(i-Nz)+1)*1j*omega_fin*(t-t_0))*w_fin[i]
    return s    



Z_ic=ICZ(V)
W_ic=ICW(V)


def HS(M):
    tM=np.transpose(M)
    Mst=np.conj(tM)
    return sqrt(np.trace(Mst@M))









print('')

print('omega=',omega_fin)
    
    
print(' ')
for i in Lz:
    print('z_',3*(i-Nz),'=',z_fin[i])
print(' ')
for i in Lw:
    print('w_',3*(i-Nz)+1,'=',w_fin[i])        
    
print(' ')
print('l1 of Fourier z=',L1(z_fin))
print(' ')
print('l1 of Fourier w=',L1(w_fin))
print(' ')
print('HH(V)=')  
print(HH_goodness)      
    

##################################

#zw=Hcheck(z_fin,w_fin,omega_fin,N)
Nprime=N+12

print('')

print('Nprime=',Nprime)

print('')
#zz_fin=zw[0:2*Nprime+1]
#ww_fin=np.array(zw[2*Nprime+1:2*(Nprime+Nprime)+2],dtype=complex)
#print('z_fin=',z_fin)
#print('')
#print('Ext',Ext(z_fin,N,Nprime))
#print('Ext',Ext(w_fin,N,Nprime))
hh=Hch( Ext(z_fin,N,Nprime) ,Ext(w_fin,N,Nprime),omega_fin,Nprime)
print('For Nprime: hh=')
print(hh)
print('')
print('For Nprime: L1hh=',L1(hh))
################################@


"""
print(' ')
print('z_0-mod=',z_fin[Nz]-z_0_init)
print('z_m3-mod=',z_fin[Nz-1]-z_m3_init)
print('w_m2-mod=',w_fin[Nz-1]-w_m2_init)
print('omegamod=',omega_fin-omega_init)
print(' ')
print('beta1=',beta_1)
print('beta2=',beta_2)
print('beta1=',-beta_1)
print('beta2=',beta_2+2*beta_1)
print(' ')
"""

#print('tau=',tau)

#print( 'ICZ=',ICZ(V))
#print('ICW=',ICW(V))


#print(w_1**4/(27*omega_fin))

"""
z_0=z_fin[Nz]
hatz_0=z_0-tau

z_fin[Nz]=hatz_0
V=np.concatenate([z_fin,w_fin])
omega=omega_fin
"""
V=np.concatenate([z_fin,w_fin])


#print('w_1=',w_fin[Nw])
#print('z_0=',z_fin[Nz])

print('')

#O1=np.zeros(2*(Nz+Nw)+2,dtype=complex)
#M=KK(V,O1)


M=np.transpose(DKK(V,omega_fin,N))
#THE MATRIX M IS THE DERIVATIVE AT THE POINT Z,W

#print('M=',M)
#hattau=tau-(tau**2)/2

a,b=linalg.eig(M)
# a REPRESENTS THE EIGENVALUES OF M AND 
# b THE CORRESPONDING EIGENVECTORS




print('a=')
print(a)

print('')



A0=np.where(abs(a)<0.01)[0][0]
A1mg=np.where(abs(a-(1-omega_fin))<0.01)[0][0]
print('a[A0]',a[A0])
print('a[A1mg]',a[A1mg])





        

#a2pg=a[2*Nz+2]
#b2pg=b[:,2*Nz+2]
#a1mg=a[2*Nz+6]

a0=a[2*Nz+3] #N=7
b0=b[:,2*Nz+3] #N=7

a1mg=a[2*N+6]  #N=7
b1mg=b[:,2*N+6]  #N=7

a0=a[2*Nz+5] #N=10
b0=b[:,2*Nz+5] #N=10



a1mg=a[2*N+6]  #N=10
b1mg=b[:,2*N+6]  #N=10


a0=a[A0]
a1mg=a[A1mg]



b0=b[:,A0]
b1mg=b[:,A1mg]


#b1mg=np.transpose(b[:,3*N-1])


print('')
print('a0=',a0)
print('')
print('a1mg=',a1mg)
print('')
print('b0=')
print(b0)
print('')
print('b1mg=')
print(b1mg)
print('')
#print('a0=',a[2*Nz+3])
#print('b0=',b[:,2*Nz+3])



#print('checkk',M@b1mg-a1mg*b1mg)
print('')
#print('checkk',M@b0-a0*b0)
print('')
#b8=b[8]


b0z=b0[0:2*Nz+1]
b0w=b0[2*Nz+1:2*(Nz+Nw)+2]
u1_fin=b0z
v1_fin=b0w
# b0 IS THE EIGENVECTOR CORREPSONDING TO THE VP CLOSE TO 0
# b0z IIS z COMPONENT
# IT CORREPSONDS TO $\ti u_1$ 
# b0w ITS w COMPONENT
# IT CORREPSONDS TO $\ti v_1$ 


 

b1mgz=b1mg[0:2*Nz+1]
b1mgw=b1mg[2*Nz+1:2*(Nz+Nw)+2]
u2_fin=b1mgz
v2_fin=b1mgw
# b1mg IS THE EIGENVECTOR CORREPSONDING TO THE VP CLOSE TO 1-g
# b1mgz IIS z COMPONENT
# IT CORREPSONDS TO $\ti u_2$ 
# b1mgw ITS w COMPONENT
# IT CORREPSONDS TO $\ti v_2$ 

print('')


print('b0z=')
print(b0z)
print('')
print('b0w=')
print(b0w)
print('')
print('b1mgz=')
print(b1mgz)
print('')
print('b1mgw=')
print(b1mgw)

print('')

c=ZU(b0z,b1mgw,omega_fin)-ZU(b0w,b1mgz,omega_fin)
print('Fourier coeff. of c=det Pc=')
print(c)

print('')


###########################@

#print('check=',M@b1mg-a1mg*b1mg)

############################

#print('z_fin=',z_fin)
#print('')
#print('Ext',Ext(z_fin,N,Nprime))
#print('Ext',Ext(w_fin,N,Nprime))
uv1=K(Ext(z_fin,N,Nprime) ,Ext(w_fin,N,Nprime),  Ext( b0z,N,Nprime ), Ext(b0w,N,Nprime)   ,omega_fin,Nprime)
uv2_prel=K(Ext(z_fin,N,Nprime) ,Ext(w_fin,N,Nprime),  Ext( b1mgz,N,Nprime ), Ext(b1mgw,N,Nprime)   ,omega_fin,Nprime)
uv2=uv2_prel-a1mg*np.concatenate( [Ext( b1mgz,N,Nprime ), Ext( b1mgw,N,Nprime ) ]  )

# GOODNESS OF THE APPROXIMATION
ee1=L1(uv1)
ee2=L1(uv2)
print( 'uv1=')
print(  uv1)
print('')
print( 'uv2=')
print( uv2)
print('')
print('Goodness of approx. reolvent ee1=',ee1)
print('')
print('Goodness of approx. reolvent ee2=',ee2)
#print('GOODNESS OF APPROX L1hh=',L1(hh))

print('')


u1_at_0=np.ones(2*N+1,dtype=complex)@np.transpose(b0z)
v1_at_0=np.ones(2*N+1,dtype=complex)@np.transpose(b0w)
u2_at_0=np.ones(2*N+1,dtype=complex)@np.transpose(b1mgz)
v2_at_0=np.ones(2*N+1,dtype=complex)@np.transpose(b1mgw)

print('u1 at 0=', u1_at_0)
print('v1 at 0=', v1_at_0)
print('u2 at 0=', u2_at_0)
print('v2 at 0=', v2_at_0)

print('')

P_at_0=[[u1_at_0,u2_at_0],[v1_at_0,v2_at_0]]

print('P_at_0=',P_at_0)
print('')
print('det P_at_0',np.linalg.det(P_at_0))
print('')
eigenvalP_at_0,eigenvectP_at_0=linalg.eig(P_at_0)
#print('eigenval P_at_0', eigenvalP_at_0)

#print('eigenvecl P_at_0', eigenvectP_at_0)





print('det P-integral c=',u1_at_0*v2_at_0-u2_at_0*v1_at_0-3.09067965e-01)

print('')

###########################

z_fin_at_0=np.ones(2*N+1,dtype=complex)@np.transpose(z_fin)
w_fin_at_0=np.ones(2*N+1,dtype=complex)@np.transpose(w_fin)

X_at_0=[(1-tau)*z_fin_at_0+(1/2)*z_fin_at_0**2-(1/3)*w_fin_at_0**3,tau*w_fin_at_0-z_fin_at_0*w_fin_at_0]
print('X_at_0=', X_at_0)
print('')

print('det(X_at_0,[U1_at_0,v_1_at_0]=' , (X_at_0[0])*v1_at_0-(X_at_0[1])*u1_at_0)

print('')
###########################@

##################################
# L1 NORM OF THE COEFFICIENTS OF P

l1_of_u1=L1line(b0z,N)
l1_of_v1=L1line(b0w,N)
l1_of_u2=L1line(b1mgz,N)
l1_of_v2=L1line(b1mgw,N)

l1_of_P=[[l1_of_u1,l1_of_u2],[l1_of_u1,l1_of_u2]]
print('l1_of_P=',l1_of_P)

max_op_norm_P=np.sqrt(np.max(l1_of_u1**2+l1_of_u2**2)+np.max(l1_of_v1**2+l1_of_v2**2))
print('')
print('max_operator_norm_P=',max_op_norm_P)
#################################@

L=1
"""
t = np.linspace(0, 20, 1000000)
plt.plot(L*Z(t).real,L*Z(t).imag,'ro',markersize=1)
plt.plot(L*W(t).real,L*W(t).imag,'bo',markersize=1)
"""

for i in range(0,4*Nz+2):
    plt.plot(1*a[i].real,1*a[i].imag,'bo',markersize=1)


    
plt.xlabel(u'$z$', fontsize=10)
plt.ylabel(u'$w$', fontsize=10)
plt.axis('equal')
plt.xlim(-5, 5)
plt.ylim(-2,2)

plt.title('Eigenvalues')
plt.tight_layout()
plt.show()


#for i in range(0,3):
 #   print('z'{}'=', {}),







\end{verbatim}

\end{comm}

\begin{comm}

\section{Numerical values}

For $N=12$. (5-6 min.)



\end{comm}

\bibliographystyle{plain}

%\appendix

\begin{comm}
\newpage

\begin{figure}[h]
\vskip -6cm
\hskip -2cm
\includegraphics[scale=0.5]{UshikiOriginal-proj(zr,wr).pdf}
\caption{S. Ushiki's example. Iteration of the map $h_{\b,c}$ with  
$\beta=0.327136$, $c=0.269343$. The curve represents (after  the scaling $(z,w)\mapsto (20\times (z-0.5),20\times (w-0.58))$)   $(\Re(z),\Re(w))$ after  5000 iterations.
The initial condition is $(z_{*},w_{*})$ avec 
$z_*= 0.3512857-0.352772 \sqrt{-1}$,
$w_*= 0.3856867+0.353207\sqrt{-1}$.
%a change of coordinates (scaling factor  of the picture $0.1$). Parameters $\mbeta=(-1.8592)/3,\qquad
%\malpha=(-0.8846+2.67\sqrt{-1})/3$, $\delta=0.01$; initial condition $(z_{*},w_{*})$,  $z_{*}=2.3+3.5\sqrt{-1}$, $
%w_{*}=-3.8+7.2\sqrt{-1}$.  5000 iterations. The red (resp. blue) curve is the projection of the orbit on the $z$-coordinate (resp. $w$-coordinate). 
}
\label{fig:0}
\end{figure}

\begin{figure}[h]
\includegraphics[scale=0.5]{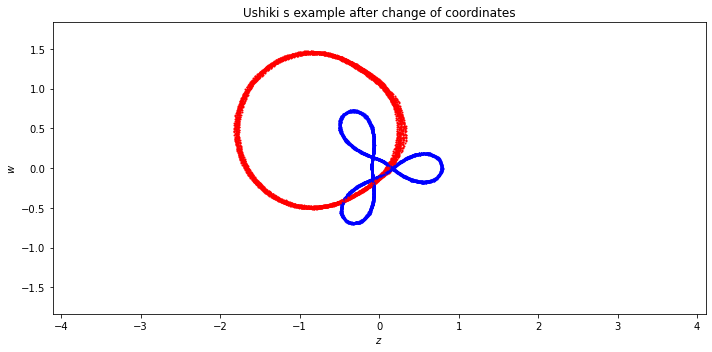}
\caption{S. Ushiki's example after a change of coordinates (BNF and scaling). Parameters $\mbeta=(-1.8592)/3,\qquad
\malpha=(-0.8846+2.67\sqrt{-1})/3$, $\delta=0.01$; initial condition $(z_{*},w_{*})$,  $z_{*}=2.3+3.5\sqrt{-1}$, $
w_{*}=-3.8+7.2\sqrt{-1}$.  5000 iterations. The red (resp. blue) curve is the projection of the orbit on the $z$-coordinate (resp. $w$-coordinate). Scaling factor  of the picture $0.1$. }\label{fig:1}
\end{figure}

\newpage
\begin{figure}[h]
\includegraphics[scale=0.5]{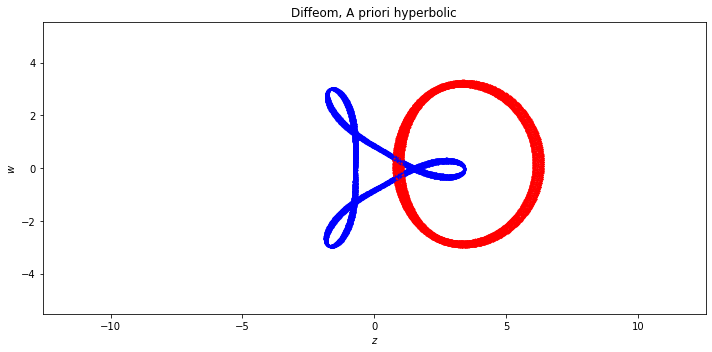}
\caption{Another Ushiki's example after a change of coordinates (scaling factor  $1$). Parameters  $\mbeta=0.311841$, $\malpha=((1/2)+10^{-1}\times\sqrt{-1}))\times \mbeta$, $\delta=0.01$; initial condition 
$(z_{*},w_{*})$,  $z_{*}=1.6+2.3\sqrt{-1}$, $
w_{*}=-1.59-2.19\sqrt{-1}$.  10000 iterations. The red (resp. blue) curve is the projection of the orbit on the $z$-coordinate (resp. $w$-coordinate). }\label{fig:2}
\end{figure}

\begin{figure}[h]
\includegraphics[scale=0.5]{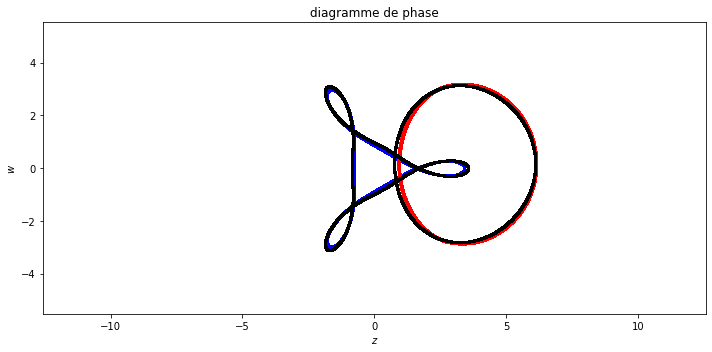}
\caption{Vector field version approximation of the previous diffeomorphism. Same parameters, same initial conditions. 
 The red (resp. blue) curve is the projection of the orbit on the $z$-coordinate (resp. $w$-coordinate). The black curves are $t\mapsto z(t)=\sum_{k=-2}^2 z_{k}e^{3ik\omega t}$, $t\mapsto w(t)=\sum_{k=-2}^1 w_{k}e^{i(3k+1)\omega t}$ for adequate choices of $z_{l},w_{l},\omega$. }\label{fig:3}
\end{figure}

\newpage
\begin{figure}[h]
\includegraphics[scale=0.5]{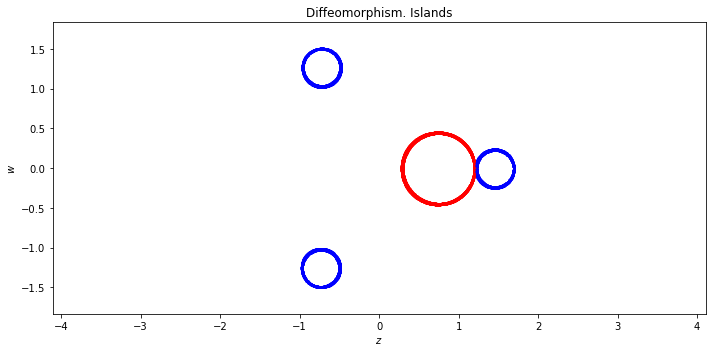}
\caption{``Elliptic Islands''. $\mbeta=0.311841$, $\malpha=-0.0535$.  $\delta=0.01$; initial condition 
$(z_{*},w_{*})$, $z_{*}=1.2$, $w_{*}=1.22$. The red (resp. blue) curve is the projection of the orbit on the $z$-coordinate (resp. $w$-coordinate). 10000 iterations.  }\label{fig:4}
\end{figure}

\begin{figure}[h]
\includegraphics[scale=0.5]{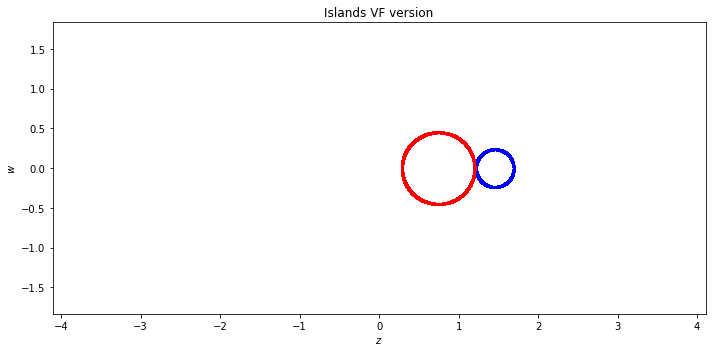}
\caption{Vector field version with the same parameters $\mbeta=0.311841$, $\malpha=-0.0535$ and the same initial condition 
$(z_{*},w_{*})$, $z_{*}=1.2$, $w_{*}=1.22$. The red (resp. blue) curve is the projection of the orbit on the $z$-coordinate (resp. $w$-coordinate). (Scaling 1).  }\label{fig:5}
\end{figure}

\begin{figure}[h]
\includegraphics[scale=0.5]{HermanRing?10000.png}
\caption{A  Herman ring  in the reduced model $f_{\malpha,\mbeta}$ (scaling factor  $0.5$). Parameters $\mbeta=0.311841+(1/3)\times 10^{-3}\sqrt{-1}$, $\malpha=(\tau-(1/2))\times \mbeta$, $\tau=0.4-.0071\sqrt{-1}$, $\delta=10^{-3}$. Initial condition $(z_{*},w_{*})$,  $z_{*}= 8.0734+0.00195\sqrt{-1}$, 
$w_{*}= 7.904-0.204\sqrt{-1}$.  10000 iterations. The red (resp. blue) curve is the projection of the orbit on the $z$-coordinate (resp. $w$-coordinate).  The cyan curve is the projection $(\Im z, \Re w)$.}\label{fig:6}
\end{figure}

\begin{figure}[h]
\includegraphics[scale=0.5]{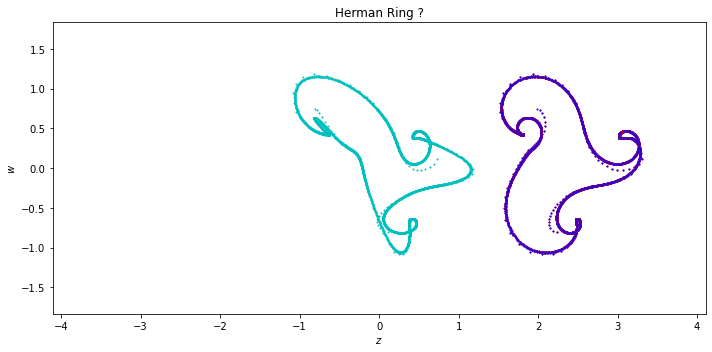}
\caption{A Herman ring for the Hénon map $h:(x,y)\mapsto (e^{i\pi \b}(x^2+c)-e^{2\pi i \beta} y,x)$,  $\beta= 0.3289999+0.0043333\sqrt{-1}$, $c=0.2619897-0.0088858\sqrt{-1}$. Initial condition $(z_{*},w_{*})$,  $z_{*}=0.44672099-0.16062292\sqrt{-1}$,
$w_{*}= 0.3961953+0.149208\sqrt{-1}$. $N=5000$ iterations. The cyan curve is the projection $(\Im z,\Im w)$ and the red and blue curves (that coincide) the projections $(\Re z,\Im z)$, $(\Re w, \Im w)$.\label{fig:7}}
\end{figure}

\begin{figure}[h]
\includegraphics[scale=0.6]{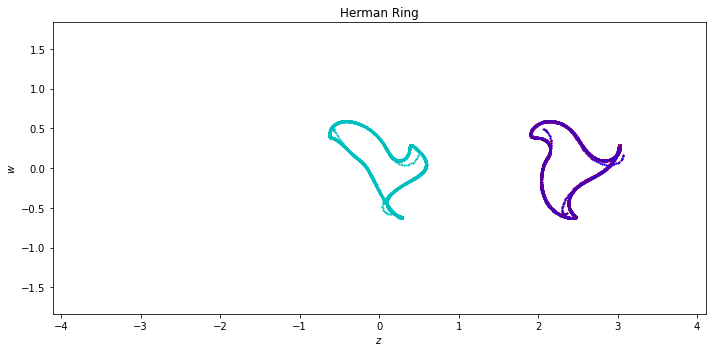}
\caption{A Herman ring for the Hénon map $h:(x,y)\mapsto (e^{i\pi \b}(x^2+c)-e^{2\pi i \beta} y,x)$,
$\beta= 0.33121126+0.00218737 \sqrt{-1}$
$c= 0.2557783-0.00497994 \sqrt{-1}$
 Initial condition $(z_{*},w_{*})$,
$z_{*}= 0.471458035-0.113447719\sqrt{-1}$
$w_{*}= 0.41305318+0.0975217\sqrt{-1}$
Number of iteration $N= 7000$.
The cyan curve is the projection $(\Im z,\Im w)$ and the red and blue curves (that coincide and give the violet curve) the projections $(\Re z,\Im z)$, $(\Re w, \Im w)$. The picture is scaled by a factor 5\label{fig:7}}
The rotation number  on the curve should be $0.0016946$.
\end{figure}

\end{comm}

\end{document}